\pdfoutput=1
\documentclass{scrartcl}

\usepackage[intlimits]{amsmath}
\usepackage[usenames,dvipsnames]{color}
\usepackage{amsthm,amssymb,amscd,mathtools,bm,ifpdf,relsize,scalerel}
\usepackage[abbrev]{amsrefs}

\usepackage[T1]{fontenc}
\usepackage{lmodern}
\usepackage{breqn}
\usepackage[inline]{enumitem}
\usepackage{microtype}
\usepackage{longtable}
\usepackage{tkz-graph}
\usepackage{subcaption}
\usepackage{floatpag}
\usepackage{stmaryrd}
\usepackage{pdflscape}
\usepackage{url}
\usepackage{siunitx}
\usepackage{hyperref}
\usepackage[normalem]{ulem}
\allowdisplaybreaks[1]

\makeatletter
\DeclareOldFontCommand{\rm}{\normalfont\rmfamily}{\mathrm}
\DeclareOldFontCommand{\sf}{\normalfont\sffamily}{\mathsf}
\DeclareOldFontCommand{\bf}{\normalfont\bfseries}{\mathbf}
\DeclareOldFontCommand{\it}{\normalfont\itshape}{\mathit}
\makeatother

\newtheorem{thm}{Theorem}[section]
\newtheorem{theorem}[thm]{Theorem}

\newtheorem{lemma}[thm]{Lemma}

\newtheorem{prop}[thm]{Proposition}
\newtheorem{proposition}[thm]{Proposition}

\newtheorem{cor}[thm]{Corollary}
\newtheorem{corollary}[thm]{Corollary}

\newtheorem{thmabc}{Theorem}

\newtheorem{corabc}[thmabc]{Corollary}

\theoremstyle{definition}
\newtheorem{defn}[thm]{Definition}
\newtheorem{definition}[thm]{Definition}

\newtheorem{ex}[thm]{Example}
\newtheorem{example}[thm]{Example}

\newtheorem{rem}[thm]{Remark}
\newtheorem{remark}[thm]{Remark}

\newtheorem{question}[thm]{Question}

\numberwithin{equation}{section}
\allowdisplaybreaks[1]

\renewcommand{\le}{\leqslant}
\renewcommand{\ge}{\geqslant}
\renewcommand{\leq}{\le}
\renewcommand{\geq}{\ge}

\renewcommand{\qedsymbol}{\ensuremath{\blacklozenge}}
\def\emptyset{\varnothing}

\def\emph{}
\DeclareTextFontCommand{\bfemph}{\bfseries}
\DeclareTextFontCommand{\itemph}{\itshape}
\def\emph{\bfemph}

\setkomafont{sectionentry}{\normalfont}
\RedeclareSectionCommand[tocbeforeskip=0pt,toclinefill=\TOCLineLeaderFill]{section}

\makeatletter
\def\blankfootnote{\xdef\@thefnmark{}\@footnotetext}

\newcommand*{\textlabel}[2]{%
  \edef\@currentlabel{#1}
  \phantomsection
  #1\label{#2}
}
\makeatother

\DeclareMathOperator{\Stair}{\mathsf{\Sigma}\Eta}
\DeclareMathOperator{\BE}{\mathsf{B}\Eta}
\DeclareMathOperator{\RE}{\mathsf{P}\Eta}

\DeclareMathOperator{\CG}{K}
\DeclareMathOperator{\DG}{\Delta}

\DeclareMathOperator{\Star}{Star}

\newcommand{\xcomm}[4]{\ensuremath{[{#1},{#4};{#3},{#2}]}}

\newcommand{\onto}{\twoheadrightarrow}
\newcommand{\into}{\rightarrowtail}
\newcommand{\incl}{\hookrightarrow}

\newcommand{\xincl}{\xhookrightarrow}
\newcommand{\xto}{\xrightarrow}

\DeclareMathOperator{\cc}{cc}
\DeclareMathOperator{\ccs}{cc}

\newcommand{\ak}{\ensuremath{\mathsf{ask}}}

\newcommand{\codis}{\textup{codis}}

\newcommand{\join}{\ensuremath{\vee}}

\DeclareMathOperator{\Pow}{\mathcal{P}}

\newcommand{\WO}{\widetilde{\textup{WO}}}
\newcommand{\WOhat}{\widehat{\textup{WO}}}

\newcommand{\Path}[1]{\operatorname{P}_{{#1}}}
\newcommand{\Cycle}[1]{\operatorname{C}_{{#1}}}

\newcommand{\MV}{\circ}
\newcommand{\MW}{\bullet}

\DeclareMathOperator{\verts}{V}
\DeclareMathOperator{\edges}{E}

\newcommand{\LEADSTO}{\ensuremath{\mathlarger{\mathlarger{\mathlarger{\mathlarger{\mathlarger{\leadsto}}}}}}}

\newcommand{\RF}{\mathfrak{K}}

\newcommand{\bil}[2]{\ensuremath{{#1} \dtimes {#2}}}
\newcommand{\baer}{\ensuremath{\diamond}}

\newcommand{\QQ}{\mathbf{Q}}
\newcommand{\FF}{\mathbf{F}}
\newcommand{\GG}{\mathbf{G}}
\newcommand{\HH}{\mathbf{H}}
\newcommand{\II}{\mathbf{I}}
\newcommand{\NN}{\mathbf{N}}
\newcommand{\N}{\NN}
\newcommand{\Z}{\ZZ}

\newcommand{\ZZ}{\mathbf{Z}}
\newcommand{\CC}{\mathbf{C}}
\newcommand{\C}{\CC}
\newcommand{\RR}{\mathbf{R}}
\newcommand{\R}{\RR}
\newcommand{\fg}{\ensuremath{\mathfrak g}}
\newcommand{\fm}{\ensuremath{\mathfrak m}}
\newcommand{\Places}{\ensuremath{\mathcal V}}

\newcommand{\hada}{\ensuremath{\star}}
\DeclareMathOperator*{\bighada}{\mathlarger{\scalerel*{{\star}}{\sum}}}

\newcommand{\Zeta}{\ensuremath{\mathsf{Z}}}
\newcommand{\Tau}{\ensuremath{\mathsf{T}}}
\newcommand{\Eta}{\ensuremath{\mathsf{H}}}

\newcommand{\DVR}{\textsf{\textup{DVR}}}

\newcommand{\WSM}{\textsf{\textup{WSM}}}

\newcommand{\fO}{\mathfrak{O}}
\newcommand{\cI}{\mathcal{I}}
\newcommand{\fP}{\mathfrak{P}}
\newcommand{\cF}{\mathcal{F}}
\newcommand{\cG}{\mathcal{G}}
\newcommand{\sA}{\mathsf{A}}
\newcommand{\sC}{\mathsf{C}}
\newcommand{\cP}{\mathcal{P}}
\newcommand{\cS}{\mathcal{S}}
\newcommand{\cX}{\mathcal{X}}

\DeclareMathOperator{\wt}{wt}

\DeclareMathOperator{\ori}{ori}
\DeclareMathOperator{\indeg}{indeg}
\DeclareMathOperator{\outdeg}{outdeg}

\DeclareMathOperator{\sgn}{sgn}

\DeclareMathOperator{\concnt}{k}
\DeclareMathOperator{\Sym}{Sym}

\newcommand{\Gl}{\ensuremath{\mathfrak{gl}}}

\newcommand{\So}{\ensuremath{\mathfrak{so}}}

\DeclareMathOperator{\Ker}{Ker}
\DeclareMathOperator{\Coker}{Coker}

\DeclareMathOperator{\diag}{diag}

\DeclareMathOperator{\Hom}{Hom}
\DeclareMathOperator{\Uni}{U}
\DeclareMathOperator{\Mat}{M}
\DeclareMathOperator{\Spec}{Spec}

\newcommand{\Orth}{\RR_{\ge 0}}
\DeclareMathOperator{\dd}{d\!}

\DeclareMathOperator{\kersize}{\mathbb{K}}
\DeclareMathOperator{\cokersize}{\mathbb{C}}

\newcommand{\normal}{\triangleleft}
\newcommand{\dtimes}{\ensuremath{\,\cdotp}}
\newcommand{\card}[1]{\lvert#1\rvert}

\DeclarePairedDelimiter{\abs}{\lvert}{\rvert}
\DeclarePairedDelimiter{\norm}{\lVert}{\rVert}

\DeclareMathOperator{\rank}{rk}

\DeclareMathOperator{\Real}{Re}
\DeclareMathOperator{\Img}{Im}

\DeclareMathOperator{\LCS}{\ensuremath{\pmb{\gamma}}}

\newcommand{\llb}{\ensuremath{[\![ }}
\newcommand{\rrb}{\ensuremath{]\!] }}

\newcommand{\ask}[1]{\operatorname{ask}({#1})}

\DeclareMathOperator{\adj}{\mathsf{adj}}
\DeclareMathOperator{\Adj}{Adj}

\DeclareMathOperator{\inc}{\mathsf{inc}}
\DeclareMathOperator{\Inc}{Inc}

\DeclareMathOperator{\NonEdges}{\mathsf{N}}
\DeclareMathOperator{\OrderedAdjPairs}{\mathcal{A}}

\newcommand{\Point}{\ensuremath{\bullet}}

\DeclareMathOperator{\KiteGraph}{Thr}
\newcommand{\Kites}{\mathsf{Threshold}}

\newcommand{\bfmu}{{\boldsymbol{\mu}}}

\newcommand{\bfa}{\ensuremath{\boldsymbol{a}}}

\newcommand{\bfm}{\mathbf{m}}
\newcommand{\bfn}{\mathbf{n}}

\newcommand{\bfs}{\ensuremath{\boldsymbol{s}}}

\newcommand{\bfV}{\ensuremath{\boldsymbol{V}}}
\newcommand{\bfX}{\ensuremath{\boldsymbol{X}}}

\newcommand{\mcI}{\mathcal{I}}

\newcommand{\mcO}{\mathcal{O}}
\newcommand{\mcP}{\mathcal{P}}

\newcommand{\mcV}{\ensuremath{\mathcal{V}}}
\newcommand{\mcZ}{\ensuremath{\mathcal{Z}}}

\DeclareMathOperator{\rk}{rk}

\newcommand{\undl}{\ensuremath{\underline}}

\newcommand{\gps}[1]{\mathrm{gp}_{#1}}
\newcommand{\gp}[1]{\mathrm{gp}\left(#1\right)}
\newcommand{\gpzero}[1]{\mathrm{gp_0}\left(#1\right)}

\newcommand{\comp}{{\textup{c}}}

\newcommand{\bfo}{{\mathbf 1}}
\newcommand{\bfz}{{\mathbf 0}}

\newcommand{\sD}{D}
\newcommand{\sd}{d}
\newcommand{\nb}{r}
\newcommand{\nc}{c}
\newcommand{\freep}{\varoast}
\newcommand{\lc}{c}

\ifpdf
  \hypersetup{pdftitle={Groups, graphs, and hypergraphs: average sizes of
      kernels of generic matrices with support constraints},
    pdfauthor={Tobias Rossmann and Christopher Voll}
  }
\fi
\title{Groups, graphs, and hypergraphs: average sizes of kernels of generic
  matrices with support constraints}

\author{Tobias Rossmann and Christopher Voll}

\date{}

\begin{document}
\thispagestyle{empty}

\maketitle

\thispagestyle{empty}

\vspace*{-2em}
\begin{abstract}
  \small
  We develop a theory of average sizes of kernels of generic matrices with
  support constraints defined in terms of graphs and hypergraphs.
  We apply this theory to study unipotent groups associated with graphs.
  In particular, we establish strong uniformity results pertaining to
  zeta functions enumerating conjugacy classes of these groups.
  We deduce that the numbers of conjugacy classes of
  $\FF_q$-points of the groups under consideration depend polynomially on $q$.
  Our approach combines group theory, graph theory, toric geometry, and
  $p$-adic integration.

  Our uniformity results are in line with a conjecture of Higman on
  the numbers of conjugacy classes of unitriangular matrix groups. 
  Our findings are, however, in stark contrast to related results by Belkale
  and Brosnan on the numbers of generic symmetric matrices of given rank
  associated with graphs. 
\end{abstract}

\blankfootnote{\noindent{\itshape 2010 Mathematics Subject Classification.}
  11M41, 
  20D15, 
  20E45, 
  15B33, 
  05A15, 
  05C50, 
  14M25, 
  11S80  
}

\setcounter{tocdepth}{1}
\tableofcontents

\section{Introduction}
\label{s:intro}

In this article, we study enumerative questions related to spaces of
matrices defined via support constraints.  Our work is motivated by
and has immediate applications to the study of (conjugacy) class
numbers of finite $p$-groups.  We will naturally touch three subjects:
rank distributions in spaces of matrices, class numbers of unipotent
groups, and zeta functions of groups.
We begin with an informal preview of our main results, deferring precise
definitions and formulations to later parts of the introduction. 

In Theorem~\ref{thm:graph_uniformity}, we obtain strong uniformity
results for new classes of zeta functions associated with graphs and
hypergraphs.
These rational generating functions enumerate, in a uniform way, the
\itemph{average sizes of kernels} of generic, symmetric, and antisymmetric
matrices with support constraints suitably defined by hypergraphs or graphs,
respectively.
The averaging process takes place over finite fields and, more generally,
over finite quotients of compact discrete valuation rings (\DVR s).
For example, Corollary~\ref{cor:bb_avg}, a consequence of
Theorem~\ref{thm:graph_uniformity}, shows that for an odd prime power $q$, the
average size of the kernel of a symmetric matrix of given size over $\FF_q$
and with support contained in a prescribed set of positions is a polynomial in
$q$.
Taking averages here is crucial.
Indeed, our uniformity results are in stark contrast to results of Belkale and
Brosnan~\cite{BB03} who showed that the variation of the numbers of 
symmetric matrices with support constraints as above and of \itemph{fixed rank}
is (in a precise sense) arbitrarily wild as a function of $q$.

We apply Theorem~\ref{thm:graph_uniformity} to a class of zeta
functions enumerating class numbers (i.e.\ total numbers of
conjugacy classes) of finite quotients of infinite \emph{graphical groups}
attached to graphs.
Given a graph $\Gamma$, we consider the maximal
nilpotent quotient $\GG_\Gamma(\ZZ)$ of class at most $2$ of the right-angled
Artin group associated with (the complement of)~$\Gamma$.
The complete graph $\Gamma=\CG_n$ on $n$ vertices, for example, gives rise to
the finitely generated free class-$2$-nilpotent group $\GG_{\CG_n}(\ZZ)$ on
$n$ generators.
By a local \emph{class counting zeta function} of $\Gamma$,
we mean an ordinary generating function which, for a given prime $p$,
enumerates the numbers of conjugacy classes of the finite graphical groups
$\GG_{\Gamma}(\ZZ/p^N\ZZ)$ for $N\ge 1$.
Corollary~\ref{cor:graphical_cc} asserts that these zeta functions
may be expressed in terms of the rational functions provided by
Theorem~\ref{thm:graph_uniformity}.
By Corollary~\ref{cor:higman},
a group-theoretic analogue of Corollary~\ref{cor:bb_avg}, the
class numbers of the groups $\GG_\Gamma(\FF_q)$ are given by a polynomial
in~$q$, depending only on the graph~$\Gamma$.
This is in line with (but presumably logically independent of) a famous
conjecture of G.~Higman on the class numbers of the full upper-unitriangular
matrix groups $\Uni_n(\FF_q)$.

Theorem~\ref{thm:master.intro} provides an explicit combinatorial
formula for the rational generating functions associated with
hypergraphs in Theorem~\ref{thm:graph_uniformity}.
This formula is of independent interest from the perspective of counting
average sizes of kernels of generic matrices with support constraints.
Remarkably, it also has direct applications to
class counting zeta functions of
\emph{cographical groups}, i.e.\ graphical groups associated with so-called 
cographs.
Cographs constitute a well-studied class of graphs that admits numerous
characterisations.
For example, cographs are precisely those graphs that do not contain a path on
four vertices as an induced subgraph. 
Theorem~\ref{thm:cograph} establishes that for every cograph $\Gamma$ there
exists an explicitly constructed hypergraph 
$\Eta=\Eta(\Gamma)$ such that the rational functions associated with
$\Gamma$ and $\Eta$ coincide.
The latter may then be described by the explicit formula in
Theorem~\ref{thm:master.intro}.
Via Corollary~\ref{cor:bb_avg}, Theorem~\ref{thm:cograph} thus yields an
explicit polynomial formula for the class numbers associated with the
cographical groups~$\GG_{\Gamma}(\FF_q)$.
Notably, by Theorem~\ref{thm:nonneg}, these polynomials have non-negative
coefficients as polynomials in $q-1$.
This is once again reminiscent of phenomena observed in the context of
Higman's conjecture. 

By a \emph{class-$2$-nilpotent free product}, we mean the maximal
nilpotent quotient of class at most $2$ of a free product of groups.
The class of cographical groups over $\ZZ$ is closed under (finite) direct
and class-2-nilpotent free products.
In \S\ref{s:cographical}, we spell out explicit consequences of our
general results on graphs and hypergraphs pertaining to class counting zeta
functions of cographical groups.
In particular, we discuss the effects of the two aforementioned
group-theoretic operations.
Along the way, we recover, generalise, and contextualise numerous previous
computations.

In Theorem~\ref{thm:ana.cograph}, we derive analytic properties of global and
local class counting zeta functions of cographical groups: remarkably, both
the global abscissae of convergence and the real parts of all local poles of
these zeta functions turn out to be integers.

To obtain our results, we apply, develop, and extend methods from toric
geometry and $p$-adic integration.
We stress that our methods are constructive.
In particular, the first author implemented practical algorithms for computing
the rational functions in Theorem~\ref{thm:graph_uniformity} as part of his
package \textsf{Zeta}~\cite{Zeta} for the computer algebra system
SageMath~\cite{SageMath}.
Using this software, the class counting zeta functions of
graphical groups associated with graphs on at most seven vertices have been
completely determined.
A complete list~\cite{cico-db} of these zeta functions is available on the
first author's home page.
In \S\ref{s:examples}, we record a substantial (but much smaller) number of
examples of zeta functions associated with graphs on few vertices. 

We now provide a more detailed discussion of the topics
related and relevant to our work and formally state our main results.

\subsection{Counting matrices of given rank}
\label{ss:intro/rank}

\paragraph{Polynomiality.}
The study of rank distributions in combinatorially defined spaces of
matrices draws on contributions from several fields of mathematical
research. Research on the subject dates back at least to the 19th
century and continues into the 21st. Indeed, the numbers of
arbitrary~\cite{Lan93} $n\times m$ matrices or of
antisymmetric~\cite[Theorem~3]{Car54},
symmetric~\cite[Theorem~2]{MacW69}, or traceless~\cite{Buc72,Ben74}
$n\times n$ matrices of a given rank over a finite field $\FF_q$ are
each given by an explicitly known polynomial in~$q$; we assume that
$q$ is odd in the (anti-)symmetric cases.  Lewis et
al.~\cite{LLMPSZ11} and Klein et al.~\cite{KLM/14} obtained further
polynomiality results for rank distributions in spaces of general,
symmetric, and antisymmetric matrices obtained by insisting that
entries in suitable positions be zero.

\paragraph{Wilderness.}
The study of rank distributions naturally involves algebro-geometric methods.
Thanks to these, much is known about ideals of minors associated
with generic, symmetric, and antisymmetric matrices~\cite{BV88,Wey03}. 

In drastic contrast to the polynomiality results above, Belkale and
Brosnan~\cite[Theorem~0.5]{BB03} demonstrated that enumerating matrices of
a given rank is a ``wild'' problem, even for spaces of combinatorial
origin.
More precisely, given $n \ge 1$ and a subset set $S$ of $\{1,\dotsc,n\}^2$,
consider the space $\Sym_n(\FF_q;S)$ of symmetric $n\times n$ matrices $[a_{ij}]$
over $\FF_q$ with $a_{ij} = 0$ whenever $(i,j) \not\in S$.
Belkale and Brosnan showed that, in a precise technical sense, enumerating
invertible matrices in $\Sym_n(\FF_q;S)$ is as difficult
as counting $\FF_q$-points on arbitrary $\ZZ$-defined varieties.
(To the authors' knowledge, it is unknown whether the same conclusion
holds for spaces of arbitrary or antisymmetric matrices with 
suitably constrained supports.)
Belkale and Brosnan used their result to refute a conjecture of Kontsevich on
the polynomiality of the numbers of $\FF_q$-points of specific hypersurfaces
associated with graphs.

Halasi and P\'{a}lfy~\cite{HalasiPalfy/11} obtained results
of a similar flavour on the numbers of matrices over
finite fields that satisfy prescribed ``rank constraints''.
In this context, too, polynomiality results
(requiring fairly restrictive combinatorial assumptions) mark the
exception from the rule of ``wild'' variation of the relevant numbers
with the prime power $q$.

\subsection{Class numbers of unipotent groups}
\label{ss:intro/class}

\paragraph{Class numbers.}
Let $\concnt(G)$ denote the number of conjugacy classes (``class number'') of
a finite group~$G$.
Let $\Uni_n(R)$ be the group of upper unitriangular $n\times n$
matrices over a ring~$R$.  In an influential paper~\cite{Hig60a},
Higman asked whether $\concnt(\Uni_n(\FF_q))$ is always given by a
polynomial in $q$.
This question has been answered affirmatively for
$n \le 13$ by Vera-L\'opez and Arregi~\cite{VLA03} and
for $n \le 16$ by Pak and Soffer~\cite{PS15}.

\paragraph{Beyond Higman's conjecture.}
We are interested in problems in the spirit of Higman's question for
other types of unipotent groups.
Let $\GG \le \Uni_n$ be a subgroup scheme---we may think of $\GG$ as a
subgroup of $\Uni_n(\CC)$ defined by the vanishing of polynomials with integer
coefficients.
What can be said about the class numbers $\concnt(\GG(\FF_q))$ as a function
of~$q$?
In particular, when does $\concnt(\GG(\FF_q))$ depend polynomially on $q$?

Questions like these have been asked and answered, to varying degrees of
generality, for numerous group schemes realising e.g.\ (Sylow subgroups of)
Chevalley groups or relatively free $p$-groups of exponent $p$,
by numerous authors, including the above-mentioned and Evseev, Goodwin,
Isaacs, Le, Lehrer, Magaard, Mozgovoy, Robinson, and R\"ohrle, to name but a
few; see, for instance,
\cite{Evseev/09, GoodwinRoehrle/09, O'BV15,Mozgovoy/19} and the references
therein.

The aforementioned problems are closely related to the enumeration of matrices
of given rank in $\ZZ$-defined spaces of matrices over $\FF_q$. We note, for
example, that the work in \cite{HalasiPalfy/11} on matrices with given rank
constraints was motivated by a study of class numbers of pattern
groups, viz.\ certain combinatorially defined subgroups of $\Uni_n(\FF_q)$.
This connection also occurs in previous work~\cite{O'BV15,ask,ask2} of both
authors.
Moreover, it turns out that if we are willing to exclude small exceptional
characteristics (depending on the group scheme in question), then the study of
$\concnt(\GG(\FF_q))$ for group schemes $\GG \le \Uni_n$ as above essentially
reduces to those of class~$2$; see Proposition~\ref{prop:cc_reduce_class2}.

\paragraph{Alternating bilinear maps.}
As a variation of the classical Baer correspondence~\cite{Bae38},
we may construct a (unipotent) group scheme $\GG_\baer$ of class at most $2$
from each alternating bilinear map $\baer\colon \ZZ^n\times \ZZ^n\to M$, where
$M$ is a free $\ZZ$-module of finite rank; see \S\ref{ss:baer} for details.
We call $\GG_\baer$ the \emph{Baer group scheme} associated with $\baer$.
Commutators in $\GG_\baer$ are given by $\baer$.
For example, if $\baer\colon \ZZ^2 \times \ZZ^2 \to \ZZ$ is the standard
symplectic form, then the associated Baer group scheme is $\Uni_3$, the
Heisenberg group scheme; see Example~\ref{exa:heisenberg}.

Rather than consider the maps $\baer$, we may equivalently use antisymmetric
matrices.
Let $\So_n(\ZZ)\subset \Mat_n(\ZZ)$ denote the module of antisymmetric
$n\times n$ matrices over~$\ZZ$.
(While we use standard Lie-theoretic notation, the Lie bracket on $\So_n(\ZZ)$
will not feature here.)
Every $\ZZ$-module homomorphism $M \xto\theta \So_n(\ZZ)$ defines an alternating
bilinear map $$[\theta]\colon \ZZ^n\times \ZZ^n\to M^* := \Hom(M,\ZZ)$$
such that for $x, y\in \ZZ^n$, the map $x[\theta]y$ is the functional
$a\mapsto x (a\theta) y^\top$ on $M$;
here and throughout this article, we usually write maps on the right.
In particular, for a submodule $M\subset \So_n(\ZZ)$, we obtain a Baer group
scheme $\GG_M := \GG_{[\iota]}$, where $\iota$ is the inclusion $M\incl
\So_n(\ZZ)$.

We note that using essentially a variation of the above construction,
finite $p$-groups associated with spaces of antisymmetric matrices over
$\FF_p$ have found applications relating graph- and group-theoretic problems
in recent work of Bei et al.~\cite{BCGQS19} and Li and Qiao~\cite{LQ20}.

\paragraph{Average sizes of kernels.}
Again, up to excluding small characteristics,
as a function of~$q$,
the study of the class numbers $\concnt(\GG_M(\FF_q))$ for modules $M\subset
\So_n(\ZZ)$ turns out to be essentially equivalent to the study of
$\concnt(\GG(\FF_q))$ for \itemph{arbitrary} unipotent group schemes~$\GG$.
By focusing on the Baer group schemes of the form $\GG_M$,
we may easily relate the study of class numbers to that of enumerating
antisymmetric matrices of given rank as in \S\ref{ss:intro/rank}.
Namely, if $A$ is a finite quotient of the ring of integers of a global (or
local) field and if $\bar M \subset\So_n(A)$ denotes the
submodule generated by the image of $M$,
then
\[
  \concnt(\GG_M(A)) =
  \card{A}^m \dtimes \frac 1 {\card {\bar M}} \sum_{a\in \bar
    M}\card{\Ker(\bar a)},
\]
where $m$ is the rank of $M$ as a $\ZZ$-module;
cf.\ Proposition~\ref{prop:antisymmetric_cc}.
That is, up to a harmless factor,
the class number of $\GG_M(A)$ is the average size of the kernels of
the elements of $\bar M$ acting on $A^n$.
Thus, if $A = \FF_q$ is a finite field, then we may express 
$\concnt(\GG_M(\FF_q))$ in terms of the numbers of $\FF_q$-points of the
$\ZZ$-defined rank loci in $M$; see \cite[\S 2.1]{ask} for details. 

\subsection{Zeta functions}
\label{ss:intro/zeta}

We have seen that the study of the class numbers $\concnt(\GG(\FF_q))$
for group schemes $\GG\le \Uni_n$ is intimately related to the study of
average sizes of kernels of matrices over $\FF_q$.
\itemph{Class counting} and \itemph{ask zeta functions} provide convenient
tools for generalising this connection to much more general finite rings,
including those of the form $\ZZ/p^k\ZZ$ ($p$ prime) and $\FF_q[z]/(z^k)$.

\paragraph{Class counting zeta functions.}
Let $R$ be the ring of integers of a local or a global field.
Let $\GG$ be a group scheme of finite type over $R$.
The \emph{class counting zeta function} of~$\GG$
is the Dirichlet series
\[
  \zeta^\cc_\GG(s) := \sum_{0\not= I\normal R} \concnt(\GG(R/I)) \dtimes \card{R/I}^{-s}.
\]
Class counting zeta functions were introduced by du~Sautoy~\cite{dS05}
for $p$-adic linear groups.
They were further studied by Berman et al.~\cite{BDOP13} for Chevalley groups
and by the first author~\cite{ask,ask2} and
Lins~\cite{Lins1/19,Lins2/20,Lins3/18} for unipotent groups; other names for
these functions in the literature are ``conjugacy class zeta functions'' and
``class number zeta functions''.
The use of zeta functions as a tool in group theory was pioneered by Grunewald
et al.~\cite{GSS88}.

\paragraph{Euler products and variation of the place.}
As we will now explain, the study of class counting zeta functions in
characteristic zero immediately reduces to a local analysis.  Let $K$
be a number field with ring of integers~$\mcO$.  Let $\mcV_K$ be the
set of non-Archimedean places of $K$.  For $v\in\mcV_K$, let $\mcO_v$
denote the valuation ring of the $v$-adic completion of $K$.  Let
$\RF_v$ denote the residue field of $\mcO_v$ and let
$q_v = \card{\RF_v}$.  Let $\GG$ be a group scheme of finite type over
$\mcO$.  Then the Chinese remainder theorem yields an Euler product
factorisation
\begin{equation}
  \label{eq:euler.intro}
  \zeta^\cc_\GG(s) = \prod_{v\in\mcV_K} \zeta^\cc_{\GG\otimes\mcO_v}(s).
\end{equation}
For a general $\GG$, it is unknown how the Euler factors
$\zeta^\cc_{\GG\otimes\mcO_v}(s)$ vary with the place $v$.  However,
if $\GG$ is a Chevalley group~\cite{BDOP13} or unipotent~\cite{ask},
then $\zeta^\cc_{\GG\otimes\mcO_v}(s)$ can, for almost all
$v\in \mcV_K$, be expressed in terms of the numbers of $\RF_v$-points
of certain $\mcO$-defined varieties and rational functions in $q_v$
and $q_v^{-s}$.
In both cases, it is generally difficult to pin down the class of varieties
required to describe $\zeta^\cc_{\GG\otimes\mcO_v}(s)$, be it for a
specific~$\GG$ or particular classes of such group schemes.

\paragraph{Uniformity.}
Among the ways that the Euler factors of a class counting zeta
function as above might depend on the place, the tamest conceivable
case has played a central role in the literature.
Namely, we say that the group scheme $\GG$ over the ring of integers~$\mcO$ of
a number field $K$ has \emph{uniform} class counting zeta
functions if there exists a rational function $W(X,T) \in \QQ(X,T)$
such that $\zeta^\cc_{\GG\otimes\mcO_v}(s) = W(q_v,q_v^{-s})$ for all
$v\in\mcV_K$.  For example, if $\zeta_K$ denotes the Dedekind zeta
function of $K$, then
\[
  \zeta^\cc_{\Uni_3\otimes\mcO}(s) =
  \frac{\zeta_K(s-1)\zeta_K(s-2)}{\zeta_K(s)} = \prod_{v\in \mcV_K}W(q_v,q_v^{-s})
\]
where $W(X,T) = \frac{1-T}{(1-XT)(1-X^2T)}$; see \cite[\S 8.2]{BDOP13}
and \cite[\S 9.3]{ask}.
A natural variation of Higman's question asks if the class counting zeta
function of each $\Uni_n$ is uniform.
While the above notion of uniformity is natural in view of
the Euler product~\eqref{eq:euler.intro}, both stronger and weaker
concepts are frequently of interest.

We wish to add a further direction by allowing local base extensions.
Namely, we say that $\GG$ as above has \emph{strongly uniform} class
counting zeta functions if there exists $W(X,T) \in \QQ(X,T)$ such
that for all compact \DVR s $\fO$ endowed with an $\mcO$-algebra
structure, we have $\zeta^\cc_{\GG\otimes\fO}(s) = W(q,q^{-s})$, where
$q$ denotes the residue field size of $\fO$.  (Note, in particular,
that we do not insist that $\fO$ have characteristic zero.)  Again,
$\Uni_3$ is an example of a group scheme with strongly uniform class
counting zeta functions; this can be verified directly or deduced from
much more general results below.

While it is relatively easy to produce examples of group schemes with
non-uniform class counting zeta functions (see
\cite[\S 7]{ask}), as in the study of class numbers over $\FF_q$ in
\S\ref{ss:intro/class}, it remains unknown just how
erratically Euler factors of class counting zeta functions may vary with the
place.

\paragraph{Ask zeta functions: analytic form.}
In the same way that the class counting zeta function $\zeta^\cc_\GG(s)$ (and
its Euler factors) of a group scheme $\GG$ encodes the collection of class
numbers $\concnt(\GG(\FF_p))$ as $p$ ranges over the primes, we may similarly
define a Dirichlet series which captures the average sizes of kernels
that appeared in \S\ref{ss:intro/class}.

Let $R$ be a commutative ring.
Consider an $R$-module homomorphism $M \xto\theta\Mat_{n\times m}(R)$, where
$M$ is a finitely generated $R$-module.
If $R$ is finite, then the \underline average \underline size
of the \underline kernel associated with $\theta$ is the rational number
\[
  \ask\theta := \frac 1{\card M}\sum_{a\in M}\card{\Ker(a\theta)}.
\]
For each $R$-algebra $S$, we obtain a map $M\otimes S
\xto{\theta^S}\Mat_{n\times m}(S)$.
Suppose that $R$ is the ring of integers of a local or global field.
The \emph{(analytic) ask zeta function}~\cite{ask,ask2} of $\theta$ is 
\[
  \zeta^\ak_\theta(s) := \sum_{0\not= I\normal R} \ask{\theta^{R/I}}\dtimes \card{R/I}^{-s}.
\]

\paragraph{Class counting and ask zeta functions.}
The following by-product of \cite{ask2} (to be proved in
\S\ref{ss:baer})
asserts that ask zeta functions associated with modules of antisymmetric matrices
essentially coincide with class counting zeta functions of associated group
schemes.

\begin{prop}
  \label{prop:antisymmetric_cc}
  Let $M\xincl\iota \So_n(\ZZ)$ be the inclusion of a submodule of
  $\ZZ$-rank $m$.  Define a unipotent group scheme $\GG_M$ as in
  \S\ref{ss:intro/class}.  Let $R$ be ring of integers of a local or
  global field of arbitrary characteristic.
  Then
  $$\zeta^\cc_{\GG_M\otimes R}(s) = \zeta^\ak_{\iota^R}(s-m).$$
\end{prop}

Beyond antisymmetric matrices, by \cite[\S 8]{ask}, ask zeta
functions associated with general $\ZZ$-module homomorphisms
$M \to \Mat_{n\times m}(\ZZ)$ are of natural group-theoretic interest:
they enumerate linear orbits of suitable groups (although perhaps not
conjugacy classes).

\paragraph{Ask zeta functions: algebraic form.}
In the local case, it will be convenient to switch freely between the above
ask zeta functions and the following algebraic counterpart.
Let~$\fO$ be a compact \DVR{} and let $M
\xto\theta\Mat_{n\times m}(\fO)$ be an $\fO$-linear map, where $M$ is a finitely
generated $\fO$-module.
Let $\fP$ be the maximal ideal of $\fO$.
Then 
\[
  \Zeta^\ak_\theta(T) :=
  \sum_{k=0}^\infty \ask{\theta^{\fO/\fP^k}}T^k \in \QQ\llb T\rrb
\]
is the \emph{(algebraic) ask zeta function} of $\theta$.
The Dirichlet series $\zeta^\ak_\theta(s)$ and ordinary generating function
$\Zeta^\ak_\theta(T)$ determine each other in the sense that
\[
  \zeta^\ak_\theta(s) = \Zeta^\ak_\theta(q^{-s}),
\]
where $q = \card{\fO/\fP}$ denotes the residue field size of $\fO$.
For this reason, we shall call each of these functions ``the'' ask
zeta function of $\theta$.

\subsection{Groups, graphs, and hypergraphs}
\label{ss:intro/graphs}

In this section, we introduce (somewhat informally) the 
protagonists of the present article: graphical groups and adjacency and
incidence representations of graphs and hypergraphs, respectively; a more
complete and rigorous account will be given in
\S\ref{s:graphs_and_hypergraphs}.

Throughout, graphs are finite without parallel edges but they
may contain loops; graphs without loops are \emph{simple}.

\paragraph{Graphical groups and negative adjacency representations.}
Let $\Gamma$ be a simple graph;
in the following, we assume that $1,\dotsc,n$ are the vertices of $\Gamma$.

The following construction of a module of antisymmetric matrices derived from
$\Gamma$ was used by Tutte~\cite{Tut47}.
Let $e_{ij}$ denote the $n\times n$ matrix with entry $1$ in position $(i,j)$
and zeros elsewhere.
Let $M^-(\Gamma) \subset \So_n(\ZZ)$ be the submodule generated by all matrices
$e_{ij} - e_{ji}$, where $(i,j)$ runs over pairs of adjacent vertices.
In the spirit of \S\ref{ss:intro/rank}, $M^-(\Gamma)$ is the module of
all antisymmetric $n\times n$ matrices over $\ZZ$ such that the support of
each matrix in $M^-(\Gamma)$ is contained in the set of pairs $(i,j)$ with $i$
adjacent to $j$ in $\Gamma$.

Let $\GG_\Gamma := \GG_{M^-(\Gamma)}$ be the group scheme associated with
$M^-(\Gamma)$ as in \S\ref{ss:intro/class}.
We call~$\GG_\Gamma$ the \emph{graphical group scheme} associated with
$\Gamma$; see \S\ref{ss:graphical_groups} for details.
We refer to the groups of points $\GG_\Gamma(R)$ over rings $R$ as the
\emph{graphical groups} associated with $\Gamma$ over~$R$.
For example, it is easy to see that if $\Path n$ denotes the path on $n$
vertices, then $\GG_{\Path n}(\ZZ)$ is the maximal nilpotent quotient of class
at most $2$ of $\Uni_{n+1}(\ZZ)$.
More generally, the groups $\GG_\Gamma(\ZZ)$ are precisely the maximal
quotients of class at most $2$ of right-angled Artin groups;
see Remark~\ref{rem:raag}. 

Among the central objects of interest in the present article are the
class counting zeta functions of graphical group schemes.
Based on what we described above, the study of these class
counting zeta functions becomes a part of the study of ask zeta functions.
Namely, define the \emph{negative adjacency representation} $\gamma_-$ of
$\Gamma$ to be the inclusion $M^-(\Gamma)\incl\So_n(\ZZ)$.
By Proposition~\ref{prop:antisymmetric_cc}, the ask zeta function of $\gamma_-$
essentially coincides with the class counting zeta function
$\zeta^\cc_{\GG_\Gamma}(s)$ of the graphical group scheme $\GG_\Gamma$.

\paragraph{Positive adjacency representations.}
As the adjective ``negative'' indicates,  the functions just defined admit
``positive'' analogues.
Suppose that $\Gamma$ is a graph as before except that we now allow $\Gamma$
to contain loops.
Let $M^+(\Gamma)$ be the submodule of the module $\Sym_n(\ZZ)$ of symmetric
$n\times n$ matrices over $\ZZ$ generated by the matrices $e_{ij} + e_{ji}$
for different adjacent vertices $i$ and $j$ and all $e_{ii}$ 
for loops $i$.
We define the \emph{positive adjacency representation} $\gamma_+$ of $\Gamma$
to be the inclusion $M^+(\Gamma)\incl \Sym_n(\ZZ)$.

Even though the ask zeta functions associated with the maps $\gamma_+$
lack an obvious group-theoretic interpretation (akin to our
interpretation of $\zeta^{\ak}_{\gamma_-}(s)$ in terms of the class
counting zeta function of $\GG_\Gamma$), they are of natural interest
in light of the results due to Belkale and Brosnan~\cite{BB03}
mentioned in \S\ref{ss:intro/rank}.  Using our present terminology,
Belkale and Brosnan showed that, as $\Gamma$ varies over all finite
graphs (with loops permitted), the number of invertible matrices in
the image of $M^+(\Gamma) \otimes\FF_q$ in $\Sym_n(\FF_q)$ is
``arbitrarily wild'' as a function of $q$.  It is therefore natural to
ask whether this wildness survives taking the average both over
$\FF_q$ and, similarly, on the level of suitable ask zeta
functions.

\paragraph{Hypergraphs and incidence representations.}
As we saw, graphs (with loops permitted) provide a combinatorial formalism for
discussing modules of antisymmetric or symmetric matrices with
support contained in a given set of positions.  In the same spirit, we
may use hypergraphs to encode modules of arbitrary rectangular
matrices with constrained support.
Here, a hypergraph $\Eta$ on the vertex set $\{1,\dotsc,n\}$ consists of
symbols $e_1,\dotsc,e_m$ called hyperedges and, for each $j=1,\dotsc,m$, a
support set $\abs{e_i}$ which is an arbitrary subset of $\{1,\dotsc,n\}$.
Define $M(\Eta)\subset \Mat_{n\times m}(\ZZ)$ to be the module of all matrices
$[a_{ij}]$ with $a_{ij} = 0$ whenever the vertex $i$ and hyperedge $e_j$ are
not incident (i.e.\ whenever $i\not\in\abs{e_j}$).
We refer to the inclusion $M(\Eta)\xincl\eta \Mat_{n\times m}(\ZZ)$ as the
\emph{incidence representation} of $\Eta$.

The ask zeta functions
$\zeta^\ak_\eta(s)$ associated with hypergraphs are of interest in
view of work of Lewis et al.~\cite{LLMPSZ11}, Klein et al.~\cite{KLM/14}, and others on rank
distributions in spaces of matrices defined in terms of support
constraints.  In addition, over the course of the present article, we
will encounter group-theoretic incentives for studying these
functions.

\subsection{Results I: strong uniformity}
\label{ss:intro/uniformity}

Our first main result establishes that whatever wild geometry can be found in
the rank loci of the modules $M^\pm(\Gamma)\subset \Mat_n(\ZZ)$ and
$M(\Eta)\subset \Mat_{n\times m}(\ZZ)$ from \S\ref{ss:intro/graphs}
disappears on average in the sense that it is invisible on the level of ask
zeta functions.
As before, $q$ denotes the residue field size of a compact \DVR{}~$\fO$.

\begin{thmabc}[Strong uniformity]
  \label{thm:graph_uniformity}
  \quad
  \begin{enumerate}
  \item
    \label{thm:graph_uniformity1}
    Let $\Eta$ be a hypergraph
    with incidence representation $\eta$ over  $\ZZ$.
    Then there exists $W_\Eta(X,T) \in \QQ(X,T)$ such that, for each
    compact \DVR{} $\fO$,
    $$\Zeta^\ak_{\eta^{\fO}}(T) = W_\Eta(q,T).$$
  \item
    \label{thm:graph_uniformity2}
    Let $\Gamma$ be a simple graph with negative adjacency representation
    $\gamma_-$ over $\ZZ$.
    Then there exists a rational function
    $W^-_\Gamma(X,T)\in \QQ(X,T)$ such that, for each compact \DVR{}
    $\fO$,
    $$\Zeta^{\ak}_{\gamma_-^\fO}(T) = W^-_\Gamma(q,T).$$
  \item
    \label{thm:graph_uniformity3}
    Let $\Gamma$ be a (not necessarily simple) graph with
    positive adjacency representation $\gamma_+$ over $\ZZ$.
    Then there exists a rational function
    $W^+_\Gamma(X,T)\in \QQ(X,T)$ such that, for each compact \DVR{}
    $\fO$ of odd residue characteristic,
    $$\Zeta^{\ak}_{\gamma_+^\fO}(T) = W^+_\Gamma(q,T).$$
  \end{enumerate}
\end{thmabc}

The conclusion of
Theorem~\ref{thm:graph_uniformity}(\ref{thm:graph_uniformity3}) does
not generally hold for compact \DVR{}s with residue field characteristic~$2$;
see Remark~\ref{rem:char.2}.

By~\cite[Theorem~1.4]{ask}, each of the generating functions
$\Zeta^\ak_{\eta^{\fO}}(T)$, $\Zeta^{\ak}_{\gamma_-^\fO}(T)$, and
$\Zeta^{\ak}_{\gamma_+^\fO}(T)$ in Theorem~\ref{thm:graph_uniformity} is
rational in $T$ provided that $\fO$ has characteristic zero.
What is remarkable is that these functions are in fact rational in both $T$
and $q$ without any restrictions on $\fO$.
This is not a general phenomenon for ask zeta functions;
see \cite[\S 7]{ask}.

The dichotomy between ``tame'' (i.e.\ strongly uniform) and ``wild'' behaviour
is a recurring theme in the study of zeta functions associated with various
group-theoretic counting problems.  Uniformity results (akin to our
Theorem~\ref{thm:graph_uniformity}) have been obtained in various situations;
see e.g.\ \cite[Theorem~2]{GSS88}, \cite[Theorem~B]{SV14} and
\cite[Theorem~1.2]{CSV/19} (equivalently, \cite[Theorem~2.2]{CSV_FPSAC2020}).

By minor abuse of notation, we refer to the rational functions $W_\Eta(X,T)$
and $W_\Gamma^\pm(X,T)$ in Theorem~\ref{thm:graph_uniformity} as the ask zeta
functions associated with $\Eta$ and $\Gamma$, respectively.
These rational functions are, to the best of our knowledge, new invariants of
graphs and hypergraphs which, as we will see, reflect interesting structural
features of the latter.
We note that while every graph $\Gamma$ is also a hypergraph, we did
not observe any meaningful relationship between the rational functions
$W_\Gamma(X,T)$ and $W_\Gamma^\pm(X,T)$.

An immediate consequence of Theorem~\ref{thm:graph_uniformity} is that upon
taking the average, the arbitrarily wild numbers of invertible matrices over
$\FF_q$ provided by Belkale and Brosnan cancel.

\begin{cor}
  \label{cor:bb_avg}
  Let $n \ge 1$ and a set $S$ be given.
  Define $\Sym_{n,r}(\FF_q;S)$ to be the set of matrices of rank $r$ in
  $\Sym_n(\FF_q;S)$ (see \S\ref{ss:intro/rank}).
  Then there exists a polynomial $f_{n,S}(X) \in \QQ[X]$ such that for
  each odd prime power $q$,
  \begin{equation}
    \label{eq:bb_avg}
    \sum_{r=0}^n \card{\Sym_{n,r}(\FF_q;S)} \, q^{n-r} = f_{n,S}(q).
  \end{equation}
\end{cor}
\begin{proof}
  Let $d$ be the $\FF_q$-dimension of $\Sym_n(\FF_q;S)$ and note that $d$ does
  not depend on~$q$.
  Let $\Gamma$ be the (not necessarily simple) graph with vertices
  $1,\dotsc,n$ and such that two vertices $i$ and $j$ of $\Gamma$ are adjacent
  if and only if $(i,j),(j,i)\in S$.
  Let $f_{n,S}(X) \in \QQ(X)$ be the coefficient of $T$ of the rational power
  series $W_\Gamma^+(X,X^d T)$ in $T$ from
  Theorem~\ref{thm:graph_uniformity}(\ref{thm:graph_uniformity3}).
  By the definition of ask zeta functions in \S\ref{ss:intro/zeta},
  \eqref{eq:bb_avg} is satisfied for all odd prime powers~$q$.
  It is a simple exercise to show that since $f_{n,S}(q)$ is an integer for
  infinitely many $q$, the rational function $f_{n,S}(X)$ is in fact a
  polynomial, as claimed.
\end{proof}

In the same way, parts~(\ref{thm:graph_uniformity1})--(\ref{thm:graph_uniformity2}) 
of Theorem~\ref{thm:graph_uniformity} imply analogous results for spaces
of general $n\times m$ and antisymmetric $n\times n$ matrices with supports
constrained by sets.

Proposition~\ref{prop:antisymmetric_cc}
and Theorem~\ref{thm:graph_uniformity}(\ref{thm:graph_uniformity2}) imply the
following group-theoretic result (see~\S\ref{ss:graphical_groups}).

\begin{corabc}[Class counting zeta functions of graphical group schemes]
  \label{cor:graphical_cc}
  \hspace*{1em}
  
  \noindent
  Let $\Gamma$ be a simple graph with $m$ edges.  Then, for each
  compact \DVR{} $\fO$ (of arbitrary characteristic) and with residue field size
  $q$,
  \[
    \zeta^\cc_{\GG_\Gamma\otimes \fO}(s) = W^-_{\Gamma}(q,q^{m-s}).
  \]
  In particular,
  graphical group schemes have strongly uniform class counting zeta
  functions.
\end{corabc}

As a very special case, we obtain the following consequence in the spirit of
Higman's question on the class numbers $\concnt(\Uni_n(\FF_q))$ for graphical
groups over $\FF_q$.

\begin{cor}
  \label{cor:higman}
  Let $\Gamma$ be a simple graph.
  Then there exists a polynomial $f_\Gamma(X) \in \QQ[X]$ such 
  that, for all prime powers $q$, we have
  $\concnt(\GG_\Gamma(\FF_q)) = f_\Gamma(q)$.
\end{cor}
\begin{proof}
  We may take $f_\Gamma(X)\in \QQ(X)$ to be the coefficient of $T$ of the
  rational power series $W^-_\Gamma(X,X^m T)$ in $T$.
  As in the proof of Corollary~\ref{cor:bb_avg},
  $f_\Gamma(X)$ is a polynomial.
\end{proof}

We note that while Corollary~\ref{cor:higman} is motivated by
and analogous to Higman's conjecture, it seems logically independent
of it since the group schemes $\Uni_n$ are not graphical for $n \ge 4$.
(Indeed, $\Uni_n(\ZZ)$ is nilpotent of class $n-1$ while graphical groups have
class at most~$2$.)
As we previously indicated, excluding small characteristics,
Proposition~\ref{prop:cc_reduce_class2} below reduces the study of the local
class counting zeta functions associated with $\Uni_n$ to the corresponding
problem for a unipotent group scheme, $\check\Uni_n$ say, of class at most $2$.
We have no reason to suspect that the group schemes $\check\Uni_n$ are graphical for
any $n\ge 4$.

The cardinalities of conjugacy classes of $\Uni_n(\FF_q)$ are all of the form
$q^k$ for integers~$k \ge 0$; see \cite[Theorem 3.5]{VLA92}.
One can show that graphical groups $\GG_\Gamma(\FF_q)$ have the same property.
Some work on Higman's conjecture suggests that already the numbers
$\cc_{q^k}(\Uni_n(\FF_q))$ of conjugacy classes of $\Uni_n(\FF_q)$ of
fixed cardinality $q^k$ may be given by polynomials;
see, for instance, \cite{VeraLopezArregi/01}.
(Higman's original conjecture only concerned the class numbers
$\concnt(\Uni_n(\FF_q)) = \sum_{k=0}^\infty \cc_{q^k}(\Uni_n(\FF_q))$.)
In Section~\ref{ss:bivariate}, we discuss related work of Lins on bivariate
refinements of the class counting zeta functions introduced in
Section~\ref{ss:intro/zeta}, designed to enumerate conjugacy classes of finite
groups derived from unipotent group schemes according to their cardinalities.
Suitable bivariate versions of our results may be used to strengthen
Corollary~\ref{cor:higman} to obtain polynomiality of the numbers of conjugacy
classes of graphical groups $\GG_{\Gamma}(\FF_q)$ of size~$q^k$ for
fixed~$k\ge 0$; cf.\ Remark~\ref{rem:lins.not}(\ref{rem:paula.3}).

\paragraph{Ingredients of the proof of Theorem~\ref{thm:graph_uniformity}.}
While our proof of
Theorem~\ref{thm:graph_uniformity}(\ref{thm:graph_uniformity1}) can be
recast in terms of existing machinery from the theory of zeta
functions (``monomial $p$-adic integrals'' as in
\S\ref{ss:intro/hypergraphs}), parts
(\ref{thm:graph_uniformity2})--(\ref{thm:graph_uniformity3}) involve
the development of several new tools that are likely to have further
applications beyond the present article.  These include
\begin{enumerate*}[label=(\alph*)]
\item
  a new type of zeta function associated with modules over
  polynomial rings (see \S\ref{s:polynomial_modules}) and, more generally, over 
  toric rings (see \S\ref{s:toric_modules}),
\item
  a notion of ``torically combinatorial'' modules
  (see~\S\ref{ss:combinatorial_modules})
  which provides an algebraic explanation of uniformity, and
\item
  a novel blend of graph theory and toric geometry in \S\ref{s:uniformity}.
\end{enumerate*}

We note that the first author previously used toric geometry in the study of
zeta functions of groups and related structures; see \cite{topzeta,topzeta2}.
However, in that work, it turned out to be extremely challenging to characterise
those groups or algebras that are amenable to toric methods.
In contrast, in the present setting,
every graph provides an example of such a group (scheme) 
via Theorem~\ref{thm:graph_uniformity}(\ref{thm:graph_uniformity2}) and
Corollary~\ref{cor:graphical_cc}.

\paragraph{Beyond uniformity.}
Apart from being surprising in light of what is known about rank
distributions and matrices with restricted support,
Theorem~\ref{thm:graph_uniformity} also raises intriguing follow-up
questions.  Which general features do the rational functions
$W^\pm_\Gamma(X,T)$ and $W_\Eta(X,T)$ possess?  How do they depend on
the graph $\Gamma$ and hypergraph $\Eta$, respectively?  Do they
afford a meaningful combinatorial interpretation?
Can they be computed?

Our proof of Theorem~\ref{thm:graph_uniformity} is constructive and will thus
provide an affirmative answer to the last of these questions.
Regarding the first question, general results on ask zeta functions from
\cite{ask} have consequences such as the following:

\begin{cor}[Functional equations]
  \label{cor:feqn}
  Let $W(X,T)$ be one of the rational functions $W_\Eta(X,T)$ or
  $W^\pm_\Gamma(X,T)$ associated with a hypergraph or graph on $n$
  vertices.
  Then:
      $$W(X^{-1},T^{-1}) = -X^nT \, W(X,T).$$
\end{cor}
\begin{proof}
  Combine \cite[Theorem~4.18]{ask} and \cite[\S 4]{stability}.
\end{proof}

In general, functional equations for local ask zeta functions
associated with globally defined module representations fail at finitely many
places.
In the combinatorial setup of this paper, however,
Theorem~\ref{thm:graph_uniformity} implies that this set of ``bad'' places is
empty or at worst, in the case of positive adjacency representations of
graphs, a subset of the set of places dividing the prime $2$ (but see
Remark~\ref{rem:char.2}).
The situation that the combinatorial origin of a global structure allows for a
tight control of the remit of associated local functional equations 
also arises in \cite[Theorem~1.2]{CSV/19} (equivalently,
\cite[Theorem~2.2]{CSV_FPSAC2020}).

The following corollary is concerned with a limit ``$q \to 1$'' in
the setting of Theorem~\ref{thm:graph_uniformity}. 

\begin{cor}[Reduced zeta functions]
  Let the notation be as in Corollary~\ref{cor:feqn}.
  Then $W(1,T) = 1/(1-T)$.
\end{cor}
\begin{proof}
  Apply \cite[\S 4.6]{ask}.
\end{proof}

The preceding corollary reflects the
fact that for a compact \DVR{} $\fO$ with maximal ideal $\fP$ and residue
field~$\RF$ of size $q$, the natural action of the group $\RF^\times$ on the
set of non-zero $n\times m$ matrices over $\fO/\fP^k$ is free and preserves
kernels.
Informally speaking, when considering ask zeta functions, the associated
orbits (each of size $q-1$) on non-zero matrices thus disappear in the
limit ``$q \to 1$''; see \cite[\S 4.6]{ask}.

Theorem~\ref{thm:graph_uniformity} and general results on zeta
functions of algebraic structures (cf.\ e.g.~\cite[Theorem~4.10]{ask})
imply that each of the rational functions in
Theorem~\ref{thm:graph_uniformity} can be written in the form
$f(X,T)/g(X,T)$, where $f(X,T)\in \QQ[X^{\pm 1},T]$ and $g(X,T)$ is a
product of factors of the form $1 - X^a T^b$ for $a,b\in \ZZ$ with
$b \ge 0$.
As we will see, we can often be much more precise here.  In particular, our
next main results will cast light on the rational functions $W_\Eta(X,T)$ for
arbitrary hypergraphs and on the rational functions $W_\Gamma^-(X,T)$ (and
hence associated class counting zeta functions) for certain graphs, namely the
so-called \itemph{cographs}.

\subsection{Results II: weak orders and explicit formulae for hypergraphs}
\label{ss:intro/hypergraphs}

While constructive, the intricate recursive nature of our proof of
Theorem~\ref{thm:graph_uniformity}(\ref{thm:graph_uniformity2})--(\ref{thm:graph_uniformity3})
provides few indications as to how the rational functions obtained
depend on the graph in question.  In contrast, in the case of
hypergraphs we make the uniformity
statement in
Theorem~\ref{thm:graph_uniformity}(\ref{thm:graph_uniformity1}) fully
explicit, as our next main result shows.

Up to isomorphism, a hypergraph $\Eta$ as in \S\ref{ss:intro/graphs} is
completely determined by a vertex set $V$ and, for each subset $I\subset V$, a
``hyperedge multiplicity'' $\mu_I$ which counts how many
hyperedges of $\Eta$ have support $I$.
We can explicitly describe $W_\Eta(X,T)$ in terms of these multiplicities.
Let $\WOhat(V)$ denote the poset of flags of subsets of $V$,
i.e.\ (essentially) the poset of weak orders on $V$; see
Definition~\ref{def:weak.order}. 

\begin{thmabc}[Ask zeta functions of hypergraphs and weak orders]
  \label{thm:master.intro}
  Let $\Eta$ be a hypergraph with vertex set $V$ and given by a family
  $\bfmu=(\mu_I)_{I\subset V}\in\N_0^{\Pow(V)}$ of hyperedge
  multiplicities.
  Then
  \begin{equation}\label{equ:master.intro}
    W_{\Eta}(X,T) = \sum_{y \in \WOhat(V)}(1-X^{-1})^{\card{\sup(y)}}
    \prod_{J\in y} \frac{X^{\card{J}-\sum_{I\cap J \neq
          \varnothing}\mu_I}T}{1-X^{\card{J}-\sum_{I\cap J \neq
          \varnothing}\mu_I}T}.
  \end{equation}
\end{thmabc}

The number of summands in \eqref{equ:master.intro} grows rather
quickly.  Indeed, let $n = \card V$.  As explained in
Remark~\ref{rem:fubini}, $\card{\WOhat(V)} = 4 f_n$, where $f_n$ is
the $n$th {Fubini} (or {ordered Bell}) {number}, enumerating weak
orders on~$V$.
In particular,
\begin{equation}
  \label{eq:fubini_growth}
  f_n \sim \frac{n!}{2(\log 2)^{n+1}}
\end{equation}
grows super-exponentially as a function of $n$;
see \cite{Bar80} or \cite[\S 5.2]{Wil06}. 
The value of Theorem~\ref{thm:master.intro} lies not primarily in
providing an algorithm for computing $W_\Eta(X,T)$ but in the rich
combinatorial structure of these functions that it reveals.

\paragraph{Consequences.}
We exhibit three main applications of Theorem~\ref{thm:master.intro}.
First, it imposes severe restrictions on the denominators of the
functions $W_\Eta(X,T)$.
This turns out to have remarkable consequences for analytic properties of ask zeta
functions associated with hypergraphs; see Theorem~\ref{thm:ana}.
Second, for specific families of hypergraphs of special interest here,
we will obtain more manageable versions of Theorem~\ref{thm:master.intro}; see
\S\S\ref{subsec:staircase}, \ref{subsubsec:disjoint}, and
\ref{subsubsec:codisjoint}.
Finally, Theorem~\ref{thm:master.intro} will allow us to capture the effects
of several natural operations for hypergraphs on the level of the rational
functions $W_\Eta(X,T)$; see \S\ref{subsec:hyp.ops}.
This will prove to be particularly valuable when combined with the results
in~\S\ref{ss:intro/cographs}. 

\paragraph{Ingredients of the proof of Theorem~\ref{thm:master.intro}.}
Let $\fO$ be a compact \DVR{}.  Beginning with the integral formalism for ask
zeta functions from \cite{ask}, our proof of Theorem~\ref{thm:master.intro} is
based on a formula akin to~\eqref{equ:master.intro} for multivariate monomial integrals such
as
\begin{equation}
  \label{eq:intro.monomial}
  \mcZ(\bfs) :=
  \int\limits_{\fO^n
    \times \fO} \prod_{I = \{ i_1,i_2,\dotsc\} \subset \{1,\dots,n\}} \norm{x_{i_1}, x_{i_2}, \dotsc\!; y}^{s_I}
  \,\dd\mu(x,y),
\end{equation}
where $\bfs = (s_I)_{I\subset\{1,\dots,n\}}$ is a family of complex
variables, $\norm\dtimes$ denotes the suitably normalised maximum norm (see
\S\ref{ss:intro/notation}), and $\mu$ denotes the additive Haar measure
on~$\fO^{n+1}$ with $\mu(\fO^{n+1})=1$.
(For aesthetic reasons, we used a semicolon rather than a comma to separate
the two types of variables in \eqref{eq:intro.monomial}.)
Weak orders on a set encode the possible rankings of its elements that
allow for ties.
Given non-zero $x_1,\dotsc,x_n,y \in \fO$, their valuations give rise
to such a ranking via the usual order.
In this way, weak orders naturally arise in the study of the integrals
\eqref{eq:intro.monomial}.
Our formula for \eqref{eq:intro.monomial} in Theorem~\ref{thm:master} is of
the same type as \eqref{equ:master.intro} in the sense that the right-hand
side has a very similar combinatorial structure.

We remark that the integrals~\eqref{eq:intro.monomial} arise as special cases
of (specialisations of) the $p$-adic integrals associated with general
hyperplane arrangements introduced and studied in~\cite{MV/21}. For the
relevant cases at hand, viz.\ the Boolean arrangements, these integrals are
essentially equal to the weak order zeta functions introduced
in~\cite[Definition~2.9]{SV1/15}; see \cite[\S\S 4.5, 4.8]{MV/21} for
details.

\subsection{Results III: cographs and their models}
\label{ss:intro/cographs}

Most of what we will learn about ask zeta functions associated with
hypergraphs rests upon explicit formulae such as
\eqref{equ:master.intro}.  As indicated above, the starting point of
these formulae is an expression for the local ask zeta functions
(i.e.\ those over compact \DVR{}s) associated with a hypergraph by
means of a monomial integral as in \eqref{eq:intro.monomial}.
We have no reason to expect that such an approach will succeed
for adjacency representations of graphs.
(Example~\ref{ex:ninja} will show that the integrals in
\eqref{eq:intro.monomial} cannot suffice.)
This explains why our proof of parts
(\ref{thm:graph_uniformity2})--(\ref{thm:graph_uniformity3}) of
Theorem~\ref{thm:graph_uniformity} is vastly more involved than that of
part~(\ref{thm:graph_uniformity1}).

Our next main result exhibits an unexpected connection between the
rational functions $W^-_\Gamma(X,T)$ associated with certain simple graphs
and the rational functions $W_\Eta(X,T)$ associated with hypergraphs in
Theorem~\ref{thm:master.intro}.

\paragraph{Cographs.}
The class of graphs known as \emph{cographs} admits numerous
equivalent characterisations; see~\S\ref{ss:cographs}.  For instance,
it is the smallest class of graphs which contains an isolated vertex
and which is closed under both disjoint unions (denoted by $\oplus$)
and ``joins'' (denoted by $\join$) of graphs; here, the \emph{join} of
two graphs $\Gamma_1$ and $\Gamma_2$ is obtained from their disjoint
union by inserting edges connecting each vertex of $\Gamma_1$ to each
vertex of~$\Gamma_2$.  Equivalently, cographs are precisely those
graphs that do not contain a path on four vertices as an induced
subgraph.

\vspace*{1em}

\begin{thmabc}[Cograph Modelling Theorem]
  \label{thm:cograph}
  Let $\Gamma$ be a cograph.
  Then there exists an explicit hypergraph $\Eta$ on the same vertex set
  as $\Gamma$ such that $$W_\Gamma^-(X,T) = W_\Eta(X,T).$$
\end{thmabc}

Informally, we think of the hypergraph $\Eta$ in
Theorem~\ref{thm:cograph} as a ``model'' of $\Gamma$ in the sense
that, through the techniques that we develop here, the former allows
us to determine and study the rational function $W^-_\Gamma(X,T)$ much
more easily than by using the methods underpinning
Theorem~\ref{thm:graph_uniformity}(\ref{thm:graph_uniformity2}).
In particular, for a cograph $\Gamma$, Theorem~\ref{thm:cograph} allows
to express $W_\Gamma^-(X,T)$ via Theorem~\ref{thm:master.intro}.
We will construct a particular hypergraph $\Eta$ as in
Theorem~\ref{thm:cograph} for each cograph $\Gamma$;
we refer to this hypergraph as ``the'' model of $\Gamma$ in the following.

Our construction reveals a number of specific properties of models.
For instance, models always have fewer hyperedges than vertices.
Moreover, the sum over the entries of an incidence matrix of a model
is always even (this will follow from Remark~\ref{rem:models.joins.freep}),
just as for graphs.
These conditions further illustrate the level of generality of
Theorem~\ref{thm:master.intro}.

We note that the special case of Theorem~\ref{thm:cograph} obtained by
taking $\Gamma$ to be a complete graph $\Gamma$ on $n$ vertices and
$\Eta$ to be a hypergraph on $n$ vertices with $n-1$ hyperedges, the
support of each is the set of all vertices, was (implicitly) proved in
\cite[Proposition~5.11]{ask}.

\paragraph{Ingredients of the proof of Theorem~\ref{thm:cograph}.}
In the same way that our proof of Theorem~\ref{thm:graph_uniformity} goes
beyond merely establishing uniformity of zeta functions by elucidating the
structure of certain modules, the cograph modelling theorem is based on more
than a mere coincidence of rational functions. Deferring precise definitions
to \S\ref{ss:incidence}, the \emph{incidence module} of a hypergraph $\Eta$
is a module encoding algebraically the incidence relations between vertices
and hyperedges of~$\Eta$.
In \S\ref{ss:two_adjacencies}, we similarly define the \emph{negative
  adjacency module} of a simple graph~$\Gamma$ by encoding adjacency relations
between vertices.
In this parlance, Theorem~\ref{thm:cograph} is a consequence of a structural
counterpart (Theorem~\ref{thm:cograph_model}) of Theorem~\ref{thm:cograph}
which establishes that for each cograph $\Gamma$, there exists an (explicit)
hypergraph $\Eta$ such that the negative adjacency module of $\Gamma$ and the
incidence module of $\Eta$, while generally non-isomorphic, are ``torically
isomorphic'' (up to a well-understood direct summand). Informally speaking,
this means that the two modules become isomorphic over the toric rings
associated with all cones in a suitable fan;
see \S\ref{ss:combinatorial_modules} for a precise definition.

Our proof of Theorem~\ref{thm:cograph} involves once again a blend of graph
theory and toric geometry.
Cographs arise since our arguments proceed by building up graphs
from smaller graphs using joins and disjoint unions.
We note that the assumption that $\Gamma$ is a cograph cannot be removed
entirely from Theorem~\ref{thm:cograph}; see Example~\ref{ex:ninja}.
We further note that we have found no evidence that would point towards a
modelling theorem for the rational functions $W^+_\Gamma(X,T)$ associated with
an interesting class of (not necessarily simple) graphs $\Gamma$.

\paragraph{Group-theoretic applications.}
By a \emph{cographical group (scheme)}, we mean a graphical group (scheme) (see
\S\ref{ss:intro/graphs}) arising from a cograph.
By combining Corollary~\ref{cor:graphical_cc}, Theorem~\ref{thm:master.intro},
and Theorem~\ref{thm:cograph}, we obtain an explicit formula for (local) class
counting zeta functions of cographical group schemes in terms of the
associated modelling hypergraphs.

In particular, many of our results on ask zeta functions of
hypergraphs (e.g.\ explicit formulae and information on analytic
properties) have immediate applications to the class counting zeta
functions of the associated cographical group schemes. These are recorded
in \S\ref{s:cographical}.
For instance, as a substantial generalisation of several previously known
formulae, we explicitly determine the (local) class counting zeta functions of
the cographical group schemes associated with the following classes of
cographical groups over $\ZZ$:
\begin{enumerate}
\item The class of finite direct products of finitely
  generated free class-$2$-nilpotent groups.
\item The class of class-$2$-nilpotent free products of free abelian groups of
  finite rank. 
\item
  The smallest class of groups which contains $\ZZ$
  and which is closed under both direct products with $\ZZ$ and class-$2$-nilpotent free products with $\ZZ$.
\end{enumerate}

The class counting zeta functions of the cographical group schemes associated
with free class-$2$-nilpotent groups and class-$2$-nilpotent
free products of two free abelian groups have been previously determined by
Lins~\cite[Corollary 1.5]{Lins2/20}.

As we noted above, right-angled Artin groups are close relatives of
our graphical groups.
Right-angled Artin groups associated with cographs have
e.g.\ been studied in \cite{Ima17,KK18}.

\subsection{A recurring example}
\label{ss:intro/example}

We illustrate Theorems \ref{thm:graph_uniformity},
\ref{thm:master.intro}, and \ref{thm:cograph} by means of a simple yet
instructive example that we will repeatedly revisit throughout this
paper.

\begin{ex}\label{exa:K3_K3_K2}
  Let $\Gamma$ be the following simple graph:
  
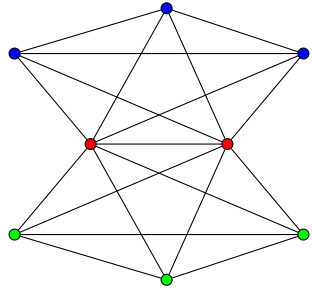
\begin{center}
    \begin{tikzpicture}[scale=0.2]
      \tikzstyle{Black Vertex}=[fill=black, draw=black, shape=circle, scale=0.4]
      \tikzstyle{Blue Vertex}=[fill=blue, draw=black, shape=circle, scale=0.4]
      \tikzstyle{Red Vertex}=[fill=red, draw=black, shape=circle, scale=0.4]
      \tikzstyle{Green Vertex}=[fill=green, draw=black, shape=circle, scale=0.4]
      \node [style=Blue Vertex] (0) at (-12, 6) {};
      \node [style=Blue Vertex] (1) at (-2, 9) {};
      \node [style=Blue Vertex] (2) at (7, 6) {};
      \node [style=Green Vertex] (3) at (-2, -9) {};
      \node [style=Green Vertex] (4) at (7, -6) {};
      \node [style=Green Vertex] (5) at (-12, -6) {};
      \node [style=Red Vertex] (6) at (-7, 0) {};
      \node [style=Red Vertex] (7) at (2, 0) {};
      \draw (1) to (2);
      \draw (0) to (1);
      \draw (0) to (2);
      \draw (5) to (3);
      \draw (5) to (4);
      \draw (3) to (4);
      \draw (0) to (6);
      \draw (6) to (1);
      \draw (6) to (2);
      \draw (6) to (5);
      \draw (6) to (3);
      \draw (6) to (4);
      \draw (7) to (2);
      \draw (7) to (1);
      \draw (7) to (0);
      \draw (7) to (3);
      \draw (7) to (5);
      \draw (7) to (4);
      \draw (6) to (7);
    \end{tikzpicture}
  \end{center}

  Using the constructive arguments underpinning
  Theorem~\ref{thm:graph_uniformity}(\ref{thm:graph_uniformity2})--(\ref{thm:graph_uniformity3}),
  we may explicitly compute the rational functions $W_\Gamma^\pm(X,T)$
  (see \S\ref{s:examples}):
  \begin{align}
    W_\Gamma^-(X,T) & = \frac{1 + X^{-6}T - 2 X^{-4}T - 2 X^{-3} T + X^{-1} T +
                      X^{-7} T^2}{(1-T)^2(1-XT)} \text{\quad and}
    \label{eq:H332}\\
    W_\Gamma^+(X,T) & = {F(X,T)}/\bigl(
                      (1 - X^{-7} T^2)(1- X^{-5} T^2)(1 - X^{-5} T)
                      (1 - X^{-4} T)\nonumber\\&\qquad\qquad\qquad(1 - X^{-3} T^2)(1 - X^{-2}
                      T)^5(1-T^{-2}X^2)(1 - T)^2\bigr),
                      \label{eq:K3_K3_K2}
  \end{align}
  where the (unwieldy) numerator $F(X,T)$ of $W_\Gamma^+(X,T)$ is
  recorded in Table~\ref{tab:CG3_CG3_CG2} on p.\ \pageref{tab:CG3_CG3_CG2}.

  Alternatively,
  the first of these rational functions can be found using
  Theorems~\ref{thm:master.intro}--\ref{thm:cograph}.
  Indeed, $\Gamma$ is a cograph for the subgraph
  induced by all vertices excluding those two depicted on the central
  horizontal edge is a disjoint union of two complete graphs on three
  vertices each.  As the aforementioned central vertices are connected
  to all other vertices, it follows that $\Gamma$ is a cograph; in
  fact, we have just shown that $\Gamma$ is isomorphic to
  $(\CG_3 \oplus \CG_3) \join \CG_2$, where $\CG_n$ denotes the
  complete graph with vertices $1,\dotsc,n$.
  
  Let $\Eta$ be a hypergraph on $8$ vertices with $7$ hyperedges and incidence
  matrix
  \begin{equation}\label{eq:M11}
    \begin{bmatrix}
      1 & 1 & 1 & 1 & 1 & 1 & 1 \\
      1 & 1 & 1 & 1 & 1 & 1 & 1 \\
      1 & 1 & 1 & 1 & 0 & 0 & 0 \\
      1 & 1 & 1 & 1 & 0 & 0 & 0 \\
      1 & 1 & 1 & 1 & 0 & 0 & 0 \\
      1 & 1 & 0 & 0 & 1 & 1 & 0 \\
      1 & 1 & 0 & 0 & 1 & 1 & 0 \\
      1 & 1 & 0 & 0 & 1 & 1 & 0 \\
    \end{bmatrix}
    \in \Mat_{8\times 7}(\ZZ).
  \end{equation}
  Write $[n] = \{1,\dotsc,n\}$.
  Using the notion of hyperedge multiplicities from
  \S\ref{ss:intro/hypergraphs},
  up to isomorphism,  $\Eta$ is thus given by the family
  $\bfmu = (\mu_I)_{I\subset [8]}$ with
  $$\mu_{[8]} = \mu_{[5]} = \mu_{\{1,2,6,7,8\}} = 2, \quad \mu_{[2]} =
  1,$$ and $\mu_I = 0$ for all remaining subsets $I\subset [8]$.  Then
  the explicit form of Theorem~\ref{thm:cograph} (see
  \S\ref{s:models}) shows that $W_\Gamma^-(X,T) = W_\Eta(X,T)$; see
  Example~\ref{exa:K3_K3_K2_pt3}.  In particular, the
  formula~\eqref{eq:H332} for $W_\Gamma^-(X,T)$ is, in principle,
  given by Theorem~\ref{thm:master.intro}, as a sum indexed by the
  poset $\WOhat([8])$.  Rather than handle a sum over the
  \num[group-separator={,}]{2183340} elements of this poset directly,
  it is far more convenient to apply some of the tools for recursively
  computing ask zeta functions associated with hypergraphs that we
  will develop in \S\ref{s:master}.  For details of this short
  computation of $W_\Eta(X,T)$, see Example~\ref{exa:K3_K3_K2_pt2}.
\end{ex}

\subsection{Results IV and open problems}
\label{ss:intro/open}

We collect consequences of our main results from above---to be proved in
\S\ref{ss:proofs_iv}---that are both closely related to topics of interest in
asymptotic and finite group theory and that seem likely to provide promising
avenues for further research.

\paragraph{Non-negativity.}
Let $\Gamma$ be a simple graph with $m$ edges.
Let $W_\Gamma^-(X,T)$ be as in
Theorem~\ref{thm:graph_uniformity}(\ref{thm:graph_uniformity2}) and
expand
\[
  W_\Gamma^-(X,X^m T) = \sum_{k=0}^\infty f_{\Gamma,k}(X) T^k
\]
for $f_{\Gamma,k}(X) \in \QQ(X)$.
By Corollary~\ref{cor:graphical_cc}, $f_{\Gamma,k}(q)$ 
is the class number of the graphical group $\GG_\Gamma(\fO/\fP^k)$
for each compact \DVR{} $\fO$ with maximal ideal $\fP$ and residue field size
$q$.
In particular, $f_{\Gamma,1}(X)$ is precisely the polynomial (!) that we
denoted by $f_\Gamma(X)$ in Corollary~\ref{cor:higman}.
Our proof of the latter result implies that, in fact, each $f_{\Gamma,k}(X)$
is a polynomial in $X$.
Inspired by Lehrer's conjecture~\cite{Leh74} on character degrees and similar
results on class numbers of the groups $\Uni_n(\FF_q)$ by Vera-L\'opez et al.\
(see, for instance,~\cite{VeraLopez/08}), both refinements of Higman's
conjecture from \S\ref{ss:intro/class}, we obtain the following.

\begin{thmabc}
  \label{thm:nonneg}
  Let $\Gamma$ be a cograph.
  Then the coefficients of each $f_{\Gamma,k}(X)$ as a polynomial in $X-1$ are
  non-negative.
\end{thmabc}

Our proof of Theorem~\ref{thm:nonneg} relies on
Theorems~\ref{thm:master.intro}--\ref{thm:cograph}, and to invoke the latter,
we need to require that $\Gamma$ is a cograph.

\begin{question}
  \label{qu:nonneg}
  For which simple graphs $\Gamma$ does the conclusion of
  Theorem~\ref{thm:nonneg} hold?
\end{question}

\paragraph{Analytic properties.}
The most fundamental analytic invariant of a (non-negative) Dirichlet series
is its abscissa of convergence which encodes the precise degree of polynomial
growth of the series's partial sums.
In a seminal paper~\cite{dSG00}, du~Sautoy and Grunewald showed that subgroup
zeta functions associated with nilpotent groups have rational abscissae of
convergence.
The same turns out to be true for class counting zeta functions of arbitrary
Baer group schemes.
(For a proof, combine Proposition~\ref{prop:baer_cc} below and
\cite[Theorem~4.20]{ask}.)
For cographical group schemes, we can do much better.
Using Theorem~\ref{thm:cograph} and an analysis of the denominator in
Theorem~\ref{thm:master.intro}, we obtain the following.

\begin{thmabc}\label{thm:ana.cograph}
  Let $\Gamma$ be a cograph with $n$ vertices and $m$ edges.
  There exists a positive integer
  $\alpha(\Gamma) \leq n + m + 1$ such that
  if  $\mcO_K$ is the ring of integers of an arbitrary number field $K$,
  then the abscissa of convergence of $\zeta^\cc_{\GG_\Gamma\otimes
    \mcO_K}(s)$ is equal to $\alpha(\Gamma)$.
  Moreover, if $\fO$ is a compact \DVR{}, then the real part
  of each pole of $\zeta^\cc_{\GG_\Gamma\otimes\fO}(s)$ is a positive integer.
\end{thmabc}
By \cite[Theorem~4.20]{ask},
for an arbitrary simple graph $\Gamma$, there is a (unique) positive rational
number $\alpha(\Gamma)$ with properties as in Theorem~\ref{thm:ana.cograph}.
However, it is not clear if~$\alpha(\Gamma)$ is always an integer.
The positivity of local poles in Theorem~\ref{thm:ana.cograph} is related to
\cite[Question~9.4]{ask}.
The integrality of local poles in Theorem~\ref{thm:ana.cograph} does not carry
over to arbitrary graphs.
For instance, for the graph $\Gamma$ in Example~\ref{ex:ninja}, the function
$\zeta^\cc_{\GG_\Gamma\otimes\fO}(s)$ has a pole at $3/2$.

\begin{question}
  \label{qu:half.int}
  Let $\Gamma$ be a simple graph.
  \begin{enumerate}
  \item Is $\alpha(\Gamma)$ always an integer?
  \item
    Are the real parts of the poles of $\zeta^\cc_{\GG_\Gamma\otimes\fO}(s)$ for
    compact \DVR{}s $\fO$ always half-integers?
  \item Is there a meaningful combinatorial formula
    (in the spirit of Theorem~\ref{thm:master.intro})  for the functions
    $W_\Gamma^\pm(X,T)$ which is valid for \itemph{all} graphs on a given vertex set?
  \end{enumerate}
\end{question}

Apart from these concrete points, we also raise the following
open-ended questions.

\begin{question}
  \label{qu:meaning}
  What do the numbers $\alpha(\Gamma)$ and the poles of class counting zeta
  functions of graphical group schemes tell us about a graph? How are they
  related to other graph-theoretic invariants?
\end{question}

\subsection{Outline}

\paragraph{Section~\ref{s:polynomial_modules}.}
In \S\ref{s:polynomial_modules}, we collect basic facts about ask zeta
functions including, in particular, the crucial duality operations
from \cite{ask2}.  Along the way, in \S\ref{ss:baer}, we formally
define Baer group schemes and relate their class counting zeta
functions to ask zeta functions attached to alternating bilinear maps.
Apart from reviewing background material, we also develop a ``cokernel
formalism'' (see \S\ref{ss:cokernel}) for expressing ask zeta
functions in terms of $p$-adic integrals.
In~\S\ref{ss:polynomial_module_zeta}, we use this to interpret ask
zeta functions as special cases of a more general class of zeta
functions attached to modules over polynomial rings.

\paragraph{Section~\ref{s:graphs_and_hypergraphs}.}
After reviewing basic constructions and terminology pertaining to
graphs and hypergraphs in \S\ref{ss:graphs} we define, in
\S\S\ref{ss:incidence}--\ref{ss:two_adjacencies}, the adjacency and
incidence representations informally described in
\S\ref{ss:intro/graphs}.  We further define adjacency and incidence
\itemph{modules} and relate their zeta functions in the sense of
\S\ref{ss:polynomial_module_zeta} to the ask zeta functions associated
with adjacency and incidence representations.  In
\S\ref{ss:graphical_groups}, we formally define graphical groups and
group schemes and relate class counting zeta functions of the latter
to ask zeta functions of adjacency representations.

\paragraph{Section~\ref{s:toric_modules}.}
Toric geometry enters the scene in \S\ref{s:toric_modules}.
We begin by collecting basic facts from convex geometry in \S\ref{ss:cones}
and on toric rings and schemes in \S\ref{ss:toric_points}. 
In \S\ref{ss:toric_module_zetas}, we further enlarge the class of zeta
functions introduced in \S\ref{ss:polynomial_module_zeta} (which, as we saw,
includes ask zeta functions) by attaching zeta functions to modules over toric
rings.
In \S\ref{ss:combinatorial_modules}, we prove
Theorem~\ref{thm:graph_uniformity}(\ref{thm:graph_uniformity1})
and introduce the key concept of ``torically combinatorial'' modules that will
also form the basis of our proof of
Theorem~\ref{thm:graph_uniformity}(\ref{thm:graph_uniformity2})--(\ref{thm:graph_uniformity3}).

\paragraph{Section~\ref{s:master}.}
\S\ref{s:master} is devoted to a detailed analysis of the rational
functions $W_\Eta(X,T)$ attached to hypergraphs $\Eta$ via
Theorem~\ref{thm:graph_uniformity}(\ref{thm:graph_uniformity1}).
In \S\ref{ss:master}, we prove (a slightly more general version of)
Theorem~\ref{thm:master.intro}.
The remainder of \S\ref{s:master} then focuses on two main themes.
First, for several classes of hypergraphs of interest, we provide more manageable
forms of Theorem~\ref{thm:graph_uniformity}(\ref{thm:graph_uniformity1}).
These classes are the ``staircase hypergraphs'' in
\S\ref{subsec:staircase}, disjoint unions of ``block hypergraphs'' in
\S\ref{subsubsec:disjoint}, and the ``reflections'' of the latter
family in \S\ref{subsubsec:codisjoint}.
Second, as we will explore and exploit throughout
\S\ref{subsec:hyp.dis.uni}--\ref{subsec:hyp.ops}, the general formula
provided by Theorem~\ref{thm:master.intro} behaves very well with
respect to natural operations on hypergraphs.
Finally we deduce, in \S\ref{subsec:hyp.ana}, consequences for analytic
properties of ask zeta functions of hypergraphs.

Later on, our results from~\S\ref{s:master} will find group-theoretic
applications in \S\ref{s:cographical} via the Cograph Modelling Theorem
(Theorem~\ref{thm:cograph}, proved in \S\ref{s:models}).
In particular, the hypergraph operations alluded to above will translate to
natural group-theoretic operations.

\paragraph{Section~\ref{s:uniformity}.}
In \S\ref{s:uniformity}, we prove
Theorem~\ref{thm:graph_uniformity}(\ref{thm:graph_uniformity2})--(\ref{thm:graph_uniformity3})
and also Corollary~\ref{cor:graphical_cc}.  Our proof considers
positive and negative adjacency representations of graphs
simultaneously by means of a common generalisation, the ``weighted
signed multigraphs'' (\WSM{}s) introduced in \S\ref{ss:wsm}.
Multigraphs are more general than graphs in that they allow parallel edges.
Each \WSM{} gives rise to an adjacency module (over a suitable toric
ring) which generalises the positive and negative adjacency modules of
graphs from \S\ref{s:graphs_and_hypergraphs}.  In \S\ref{ss:surgery},
we describe a number of ``surgical procedures'' for \WSM{}s that do
not affect the associated adjacency modules.  Even when the original
multigraph was a graph, these procedures may introduce parallel
edges---this justifies introducing the concept of \WSM{}s.  After some
technical preparations in \S\ref{ss:torically_torically}, we use these
procedures in \S\ref{ss:proof_torically_combinatorial} to give an
inductive proof of
Theorem~\ref{thm:graph_uniformity}(\ref{thm:graph_uniformity2})--(\ref{thm:graph_uniformity3}).

\paragraph{Section~\ref{s:models}.}
The Cograph Modelling Theorem (Theorem~\ref{thm:cograph}) is the
subject of \S\ref{s:models}.
We first recall basic facts about cographs in \S\ref{ss:cographs}.
In \S\ref{ss:module_comparison}, we then explain how Theorem~\ref{thm:cograph}
follows from a structural comparison result (Theorem~\ref{thm:cograph_model})
relating adjacency modules of cographs and incidence modules of hypergraphs.
Extending upon ideas underpinning the proof of
Theorem~\ref{thm:graph_uniformity}(\ref{thm:graph_uniformity2})--(\ref{thm:graph_uniformity3}),
the remainder of \S\ref{s:models} is then devoted to proving
Theorem~\ref{thm:cograph_model}.
An overview of the ingredients featuring in our proof and of our overall
strategy is given in~\S\ref{ss:model_overview}.
This is followed by an implementation of this strategy in
\S\S\ref{ss:orientation}--\ref{ss:proof_model_join}.

\paragraph{Section~\ref{s:cographical}.}
In this section we combine Theorem~\ref{thm:graph_uniformity},
Corollary~\ref{cor:graphical_cc}, Theorem~\ref{thm:master.intro},
Theorem~\ref{thm:cograph}, and our further analysis of the rational
functions $W_\Eta(X,T)$ associated with hypergraphs from
\S\ref{s:master} to deduce structural properties and produce explicit
formulae for class counting zeta functions associated with
``cographical group schemes'', i.e.\ graphical group schemes arising
from cographs.  We consider, in particular, the cographical group
schemes associated with the families of nilpotent groups listed in the final 
part of \S\ref{ss:intro/cographs} and also relate our results to work of
Lins~\cite{Lins1/19,Lins2/20} on bivariate conjugacy class zeta
functions.

\paragraph{Sections~\ref{s:examples}--\ref{s:open.problems}.}
Based on our constructive proof of Theorem~\ref{thm:graph_uniformity}
and computational techniques developed by the first author,
in \S\ref{s:examples},
we provide further examples of the rational functions $W^\pm_\Gamma(X,T)$
associated with graphs $\Gamma$ on few vertices.
Many of these examples are not covered by
Theorems~\ref{thm:master.intro}--\ref{thm:cograph}.
Motivated by such computational evidence, in~\S\ref{s:open.problems},
we pose and discus a number of questions for further
research beyond those already mentioned in \S\ref{ss:intro/open}.

\subsection{Notation}
\label{ss:intro/notation}

\paragraph{Sets.}
The symbol ``$\subset$'' signifies not necessarily strict inclusion.
We write $\sqcup$ for the disjoint union (= coproduct) of sets. Throughout,
$V$ denotes a finite set, typically of vertices of a graph or hypergraph and
of cardinality $n$.
The power set of $V$ is denoted by~$\Pow(V)$.
We write $\NN = \{ 1,2,\dotsc\}$ and $\NN_0 = \NN \cup \{0\}$.
The complement of $I$ within some ambient set $V$ is denoted by $I^\comp :=
V\setminus I$.
We write $[n] = \{1,2,\dotsc,n\}$ and $[n]_0 = [n] \cup \{0\}$.

\paragraph{Rings and modules.}
Rings are assumed to be associative, commutative, and unital.
By a ring map, we mean a unital ring homomorphism.
Let $R$ be a ring.
The unit group of $R$ is denoted by $R^\times$.
The dual of an $R$-module $M$ is $M^* = \Hom(M,R)$.
An {$R$-algebra} consists of a ring $S$ together with a ring map $R\to S$.

For a set~$V$, let $RV = \bigoplus\limits_{v\in V}Rv$ be the free module on $V$;
we extend this notation to subsets of $R$ in the evident way and e.g.\ write
$\Orth\, V = \left\{ \sum\limits_{v\in V}{\lambda_v v}  : \forall v\in
  V, \lambda_v \in \Orth\right\} \subset \RR V$, where $\Orth = \{
x\in \RR : x\ge 0\}$.
For $x\in RV$, we use the suggestive notation $x = \sum\limits_{v\in
  V}x_v v = (x_v)_{v\in V}$.
Often, $X = (X_v)_{v\in V}$ denotes a family of algebraically independent
elements over $R$.

We let  $\Mat_{n\times m}(R)$ (resp.\ $\Mat_n(R)$) denote the module of all
$n\times m$ (resp.\ $n\times n$) matrices over $R$.
The transpose of a matrix $a$ is denoted by $a^\top$.

\paragraph{Discrete valuation rings.}
Throughout this article, $\fO$ denotes a compact discrete valuation ring
(\DVR) with maximal ideal $\fP$.
We write $\fO_k = \fO/\fP^k$ and $(\dtimes)_k = (\dtimes)
\otimes{\fO_k}$.
Let $q = \card{\fO/\fP}$ denote the size of the residue field of $\fO$.
We write $(\undl{k})$ for the number $1-q^{-k}$.

Let $\nu\colon \fO \to \NN_0 \cup \{\infty\}$ denote the (surjective)
normalised valuation on $\fO$ and let $\abs\dtimes$ be the absolute
value $\abs a = q^{-\nu(a)}$ on $\fO$.
For a non-empty collection $C$ of elements of~$\fO$, write $\norm C =
\max\{\abs a: a\in C\}$ for the norm of $C$.
If $y\in \fO$, then we write $\norm{C; y}$ for the norm of the
collection formed by the elements of $C$ together with the single element~$y$;
here, $C$ can be empty.
Given two collections $C$, $C'$ of elements of $\fO$, at least one of which is
non-empty, we write $\norm{C; C'}$ for the norm of the collection formed
by the elements of both~$C$ and $C'$.
For a free $\fO$-module $M$ of finite rank, $\mu_M$ denotes the additive Haar
measure on $M$ with $\mu_M(M) = 1$.
For a non-zero $\fO$-module $M$, we write $M^\times = M\setminus \fP M$;
we let $\{0\}^\times = \{0\}$.

\paragraph{Miscellaneous.}
Maps are usually written on the right.
In particular, we regard an $n\times m$ matrix over a ring $R$ as a linear map
$R^n\to R^m$.

The $n\times n$ identity matrix is denoted by $1_n$.  In contrast,
$\bfo_{n\times m}$ (and $\bfo_n=\bfo_{n\times n}$) denotes the
respective all-one matrix.
The all-one vector of length $n$ is denoted $\bfo^{(n)}$.
The free nilpotent group of class at most $c$ on $d$
generators is denoted by $F_{c,d}$.

We write $\delta_{ij}$ for the usual ``Kronecker delta''; more
generally, for a Boolean value $P$, we let $\delta_P = 1$ if $P$ is
true and $\delta_P = 0$ otherwise.

\paragraph{Further notation}
\quad
\vspace*{0.5em}

\noindent
{\begin{longtable}{r|l|c}
  Notation\phantom{$1_1$} & comment & reference \\
  \hline
  $\zeta^\cc_{\GG}(s)$ & class counting zeta function & \S\ref{ss:intro/zeta} \\

  $W_\Eta(X,T)$, $W^\pm_\Gamma(X,T)$ & ask zeta functions of (hyper)graphs & Thm~\ref{thm:graph_uniformity} \\
                                                                              
  $\theta^S, \eta^{\fO}, \gamma_{\pm}^R, \dotsc$ & base change of module representations & \S\ref{ss:mreps} \\
  
  $\theta^\MV$, $\theta^\MW$ & Knuth duals & \S\ref{ss:mreps} \\
  
  $\sA^{U,V,W}_\theta(Z)$, $\sC^{U,V,W}_\theta(X)$ & matrices of a module representations & \S\ref{ss:A_and_C_matrices} \\
  
  $\ask\theta$, $\zeta^\ak_\theta(s)$, $\Zeta^\ak_\theta(T)$ & average size of kernel, ask zeta functions & \S\ref{ss:ask_reminder} \\
  
  $\zeta_M(s)$ & zeta function associated with a module & \S\S\ref{ss:polynomial_module_zeta}, \ref{ss:toric_module_zetas} \\

  $\abs\dtimes, \norm\dtimes$ & absolute value and norm;
  support functions
  & \S\ref{ss:intro/notation}; \S\ref{ss:graphs}\\
  
  $\verts(\Eta)$, $\edges(\Eta)$ & vertex and (hyper)edge set of a
  (hyper)graph & \S\ref{ss:graphs} \\

  $v \sim v'$, $v\sim e$ & adjacency and incidence relation &
  \S\ref{ss:graphs} \\

  $\Eta(\bfmu)$, $\Eta( V \mid\sum \mu_I I)$ & hypergraph with given hyperedge
  multiplicities & Def.\ \ref{def:hyp.mu} \\

  $\Eta_1\oplus \Eta_1$, $\Eta_1\freep \Eta_2$ & disjoint resp.\ complete union & \S\ref{ss:graphs}\\

  $\inc$, $\Inc$, $\eta$ & incidence modules and representations & \S\ref{ss:incidence}\\

  $\adj$, $\Adj$, $\gamma_{\pm }$ & adjacency modules and representations & \S\S\ref{ss:two_adjacencies}, \ref{ss:scaffolds}\\

  $G_1\freep G_2$ & free class-$2$-nilpotent product of groups & \eqref{eq:freep_of_groups} \\

  $\GG_\Gamma$ & graphical group scheme & \S\ref{ss:graphical_groups} \\
  
  $\card{\cF}$ & support of a fan  & \S\ref{ss:cones} \\
  
  $\cF_1 \wedge \cF_2$ & coarsest common refinement of fans & \S\ref{ss:cones} \\

  $\le_\sigma$ & preorder defined by a cone & \S\ref{ss:cones} \\  
  
  $R_\sigma$ & toric ring & \S\ref{ss:toric_points} \\

  $\sigma(\fO)$ & ``rational points'' of a cone over a \DVR{} & \S\ref{ss:cones} \\

  $\gp{x}$ (resp.\ $\gpzero{x}$) &  $\frac{x}{1-x}$ (resp.\ $\frac{1}{1-x}$) & \S\ref{ss:master} \\
  
  $\WOhat(V)$, $\WO(V)$ & weak orders & Def.~\ref{def:weak.order} \\
  
  $F(T) \hada G(T)$ & Hadamard product & \S\ref{subsec:hyp.dis.uni} \\

  $\Eta^{\bfo}, \Eta_{\bfo}, \Eta^\bfz, \Eta_\bfz$ & insert all-one or all-zero row or column& Def.~\ref{def:hyp.inflation} \\

  $\bm\Gamma$ & weighted signed multigraph (\WSM) & \S\ref{ss:wsm} \\
    
  $\Eta(\cS)$, $\bm\Gamma(\cS)$ & hypergraph and \WSM{} defined by a scaffold & \S\ref{ss:scaffolds}
  \\
    
  $\KiteGraph(k)$ & threshold graph & \S\ref{ss:kites}
\end{longtable}

\addtocounter{table}{-1}

\section{Ask zeta functions and modules over polynomial rings}
\label{s:polynomial_modules}

In this section, we recall background material on module representations and
associated ask zeta functions from \cite{ask,ask2}.
We also relate the latter functions to class counting zeta functions
associated with Baer group schemes.
Finally, we develop a ``cokernel formalism'' for ask zeta functions which
allows us to view the latter as special cases of a more general class of
functions attached to modules over polynomial rings.

Throughout, let $R$ be a ring.

\subsection{Module representations}
\label{ss:mreps}

In \S\ref{ss:intro/zeta}, we attached ask zeta functions to 
module homomorphisms $M\to\Mat_{n\times m}(R)$.
Rather than focusing on such parameterisations of modules of matrices,
we will use the following coordinate-free approach from \cite[\S 2]{ask2};
as shown in \cite{ask2}, disposing of coordinates elucidates duality
phenomena.

By a \emph{module representation} over $R$, we mean a homomorphism
$A \xto\theta\Hom(B,C)$, where $A$, $B$, and $C$ are $R$-modules.

\paragraph{Base change.}
For a ring map $R \xto\lambda S$, an $R$-module $M$, and an $S$-module $N$,
we let $M^\lambda = M \otimes_R S$ (resp.\ $N_\lambda$) denote the {extension}
(resp.\ {restriction}) {of scalars} of $M$ (resp.~$N$) along $\lambda$; this is
an $S$-module (resp.\ an $R$-module).
When the reference to $\lambda$ is clear, we simply write $M^S = M^\lambda$
and $N_R = N_\lambda$.

Let $A \xto\theta\Hom(B,C)$ be a module representation over $R$.
Given a ring map $R\xto\lambda S$, extension of scalars along $\lambda$ yields
a module representation 
$$A^\lambda \xto{\theta^\lambda} \Hom(B^\lambda, C^\lambda)$$
over $S$.
When the reference to $\lambda$ is clear, we write $\theta^S = \theta^\lambda$.

\paragraph{Knuth duals.}
Given a module representation $A \xto\theta\Hom(B,C)$, let $\theta^\MV$
denote the module representation $B \to \Hom(A,C)$ with $a(b\theta^\MV) = b
(a\theta)$ for $a\in A$ and $b\in B$.

Write $(\dtimes)^* = \Hom(\dtimes,R)$.
Apart from $\theta^\MV$, the module representation $\theta$ also
gives rise to a module representation $C^*\xto{\theta^\MW}\Hom(B,A^*)$
defined by $a(b(\psi\theta^\MW)) = (b(a\theta))\psi$ for $a\in A$, $b\in B$,
and $\psi \in C^*$.

Note that $(\theta^\lambda)^\MV = (\theta^\MV)^\lambda$
for each ring map $R \xto\lambda S$.
Moreover, if $A$ and $C$ are both finitely generated and projective, then we
may identify $(\theta^\MW)^\lambda = (\theta^\lambda)^\MW$.
For more on the operations $\theta \mapsto \theta^\MV$ and
$\theta\mapsto \theta^\MW$, see \cite[\S\S 4--5]{ask2}.

\paragraph{Direct sums, homotopy, and isotopy.}
Let $A \xto\theta\Hom(B,C)$ and
$A' \xto{\theta'}\Hom(B',C')$ be module representations over $R$.
The \emph{direct sum} of $\theta$ and $\theta'$ is the module representation
\[
  A \oplus A' \xto{\theta\oplus\theta'} \Hom(B\oplus B',C\oplus C'),
  \quad
  (a,a') \mapsto a\theta \oplus a'\theta'.
\]
A \emph{homotopy} $\theta \to \theta'$ is a triple of homomorphisms
$(A \xto{\alpha} A', B \xto{\beta} B', C \xto{\gamma} C')$ such that
the following diagram commutes for  each $a\in A$:
\[
  \begin{CD}
    {B} @>{a\theta}>> {C}\\
    @V{\beta}VV @VV{\gamma}V\\
    {B'} @>{(a\alpha)\theta'}>> C'.
  \end{CD}
\]
Module representations over $R$ together with homotopies as morphisms
naturally form a category.
An invertible homotopy is called an \emph{isotopy}.

\subsection{Matrices associated with module representations involving
  free modules}
\label{ss:A_and_C_matrices}

Up to isotopy, a module representation involving \itemph{free} modules of
finite rank can be equivalently expressed in terms of a matrix of linear
forms.
In detail, let $A\xto\theta \Hom(B,C)$ be a module representation over $R$.
Suppose that each $A$, $B$, and $C$ is free of finite rank.
By choosing bases $U$, $V$, and $W$ of $A$, $B$, and $C$,
respectively, we may identify $A = R U$, $B = R V$, and $C = R W$.

Let $Z = (Z_u)_{u\in U}$ consist of algebraically independent variables over $R$.
Define an $R[Z]$-linear map
\[
  \sA^{U,V,W}_\theta(Z) := \Bigl(\sum_{u\in U}Z_u u\Bigr)
  \theta^{R[Z]}
  \in \Hom\bigl(R[Z] V,\, R[Z] W\bigr).
\]
Informally, $\sA^{U,V,W}_\theta(Z)$ is the image of a ``generic
element'' of $A = RU$ under $\theta$.
The matrix of $\sA^{U,V,W}_\theta(Z)$ with respect to the
bases $V$ and $W$ of $R[Z] V$ and $R[Z] W$, respectively, is 
the matrix of linear forms associated with $\theta$ (and the chosen
bases) from \cite[\S 4.4]{ask2}.

As we will now explain, by specialising $\sA^{U,V,W}_\theta(Z)$, we may
recover $\theta$ (and $\theta^\lambda$ for each ring map $R\xto\lambda S$).

\begin{lemma}
  \label{lem:Amatrix_specialisation}
  Let $S$ be an $R$-algebra and let $z\in SU$.
  Let $S_z$ denote $S$ regarded as an $R[Z]$-algebra via $Z_u s = z_u s$
  ($u\in U, s\in S$). 
  \begin{enumerate}
  \item
    \label{lem:Amatrix_specialisation1}
    $\sA^{U,V,W}_\theta(z) := \sA^{U,V,W}_\theta(Z) \otimes_{R[Z]} S_z \in \Hom(S V, S
    W)$ coincides with $z \theta^S$.
  \item
    \label{lem:Amatrix_specialisation2}
    $\Coker(z \theta^S) \approx_S \Coker(\sA^{U,V,W}_\theta(Z)) \otimes_{R[Z]} S_z$.
  \end{enumerate}
\end{lemma}
\begin{proof}
  Part~(\ref{lem:Amatrix_specialisation1}) is clear.
  Part~(\ref{lem:Amatrix_specialisation2}) follows from
  (\ref{lem:Amatrix_specialisation1}) and right exactness of tensor products.
\end{proof}

\begin{rem}
    The isomorphism types of the cokernels
    in Lemma~\ref{lem:Amatrix_specialisation}(\ref{lem:Amatrix_specialisation2})
    (over $S$ and $R[Z]$, respectively) clearly
    only depend on the isotopy type (see \S\ref{ss:mreps}) of $\theta$.
\end{rem}

Recall the definition of $\theta^\MV$ from \S\ref{ss:mreps}.
In subsequent sections, we will analyse and compute certain zeta
functions associated with $\theta$ by studying the sizes of
$\Coker( x (\theta^\MV)^S)$ for suitable finite $R$-algebras $S$ and
$x\in S V$.
For this purpose, it will be convenient to use explicit (finite) presentations
of these cokernels.
Let $X = (X_v)_{v\in V}$ consist of algebraically independent variables over $R$.
In accordance with the notation in \cite[\S 4.3.5]{ask},
let
\begin{equation}
  \label{eq:Cmatrix}
  \sC^{U,V,W}_\theta(X) :=
  \sA^{V,U,W}_{\theta^\MV}(X) = 
  \left(\sum_{v\in V}X_v v\right)
  (\theta^\MV)^{R[X]}
  \in
  \Hom\left(R[X]U, \,R[X]W  \right).
\end{equation}

Informally, the image of $\sC^{U,V,W}_\theta(X)$ is the ``additive
orbit'' $x M = \{ xa:a\in M\}$, where $x$ is a ``generic element'' of
$R V$ and $M$ denotes the image of $\theta$.
We can read off explicit generators for the image of $\sC^{U,V,W}_\theta(X)$
and hence a presentation for the latter's cokernel.

\begin{lemma}
  \label{lem:ask_image_presentation}
  $\displaystyle 
   \Img(\sC^{U,V,W}_\theta(X)) =
   \left\langle
     \Bigl(\sum\limits_{v\in V} X_v v\Bigr)
     \bigl(u \theta^{R[X]}\bigr) :
     u\in U\right\rangle$.
   \qed
\end{lemma}

Note that we may identify $u \theta^{R[X]} = u\theta\otimes_R R[X]$
for $u\in U$.

\subsection{Reminder: ask zeta functions}
\label{ss:ask_reminder}

Let $A\xto\theta \Hom(B,C)$ be a module representation over a ring $R$.
If $A$ and $B$ are both finite as sets,
then the \emph{average size of the kernel} of $A$ acting on $B$ via
$\theta$ is the number
\[
  \ask\theta := \frac 1 {\card A} \sum_{a\in A}\card{\Ker(a\theta)}.
\]

Suppose that $R$ admits only finitely many ideals of a given finite index.
Further suppose that $A$, $B$, and $C$ are finitely generated.
The \emph{(analytic) ask zeta function} associated with $\theta$ is the
Dirichlet series
\[
  \zeta_\theta^\ak(s) := \sum_{I\normal R} \ask{\theta^{R/I}}\card{R/I}^{-s},
\]
where $s$ is a complex variable and the summation extends over those ideals
$I\normal R$ with $\card{R/I} < \infty$.
If $R$ is the ring of integers of a global or local field,
then $\zeta_\theta^\ak(s)$ defines an analytic function on
a right half-plane $\{ s\in \CC : \Real(s) > \alpha\}$;
cf.\ \cite[\S\S 3.2--3.3]{ask}.
The infimum of all such real numbers $\alpha > 0$ is the \emph{abscissa
  of convergence} $\alpha_\theta$ of $\zeta_\theta^\ak(s)$.

Let $\fO$ be a compact \DVR{} with maximal ideal $\fP$.
We write $\fO_k = \fO/\fP^k$ and,
more generally, $(\dtimes)_k = (\dtimes)\otimes_{\fO}\fO_k$.
Let $A\xto\theta\Hom(B,C)$ be a module representation over $\fO$.
Suppose that each of $A$, $B$, and $C$ is free of finite rank.
We may identify $\theta_k = \theta\otimes_{\fO}\fO_k$ and
$\theta^{\fO_k}$ (see \S\ref{ss:mreps}).
Consider the generating function
\[
  \Zeta^{\ak}_\theta(T) :=
  \sum_{k=0}^\infty \ask{\theta_k}T^k.
\]
By slight abuse of terminology, we also refer to $\Zeta^\ak_\theta(T)$ as the
\emph{(algebraic) ask zeta function} of $\theta$.
By \cite[Theorem~1.4]{ask}, if $\fO$ has characteristic zero, then
$\Zeta^\ak_\theta(T) \in \QQ(T)$.

Note that the analytic function $\zeta^\ak_\theta(s)$ and its algebraic
counterpart $\Zeta^\ak_\theta(T)$ determine one another:
$\zeta^\ak_\theta(s) =  \Zeta^\ak_\theta(q^{-s})$, where $q$ denotes the
residue field size of $\fO$.
As explained in \cite[\S 8]{ask}, ask zeta functions (of either type) arise
naturally in the enumeration of orbits and conjugacy classes of unipotent
groups.
Most of the main results of this article
(e.g.~Theorems~\ref{thm:graph_uniformity},\ref{thm:master.intro}--\ref{thm:cograph}) 
are stated in terms of the generating functions
$\Zeta^\ak_\theta(T)$ while the analytic functions $\zeta^\ak_\theta(s)$
feature in our proofs by means of suitable $p$-adic integrals.

We will rely heavily on the fact that the duality operations
$\theta\mapsto\theta^\MV$ and $\theta\mapsto\theta^\MW$
(see~\S\ref{ss:mreps}) have the following tame effects on ask zeta functions.

\begin{thm}[{\cite[Corollary 5.6]{ask2}}]
  \label{thm:ask_duality}
  $\Zeta^\ak_{\theta}(T)
  = \Zeta^\ak_{\theta^\MV}\left(q^{\rank(B) - \rank(A)}T\right)
  = \Zeta^\ak_{\theta^\MW}(T)$.
\end{thm}

We can now formalise the reduction of the study of class numbers of arbitrary
unipotent groups to the case of nilpotency class $2$ alluded to in the
introduction.

\begin{prop}
  \label{prop:cc_reduce_class2}
  Let $\GG$ be a $\ZZ$-form of a unipotent algebraic group of dimension $d$.
  Then there exists a $\ZZ$-form $\HH$ of a unipotent algebraic group of
  dimension $2d$ and nilpotency class at most $2$ such that for all
  sufficiently large primes $p$ and each compact \DVR{} $\fO$
  with residue characteristic $p$,
  we have $\zeta^\cc_{\GG\otimes \fO}(s-d) = \zeta^\cc_{\HH \otimes \fO}(s)$. 
  In particular, for each power $q> 1$ of a sufficiently large prime,
  we have $q^d \concnt(\GG(\FF_q)) = \concnt(\HH(\FF_q))$.
\end{prop}
\begin{proof}
  Let $\fg$ be a $\ZZ$-form of the Lie algebra of $\GG \otimes \QQ$
  which is free of finite rank $d$ as a $\ZZ$-module.  Let
  $\alpha\colon \fg \to \Gl(\fg)$ be the adjoint representation of
  $\fg$.  Then for some integer $N \ge 1$, $\GG\otimes \ZZ[1/N]$ is
  isomorphic over $\ZZ[1/N]$ to the unipotent group scheme
  $\exp(\fg \otimes (\dtimes))$ constructed as
  in \cite[\S 2.1.2]{SV14} by means of the Hausdorff series.  It
  follows from \cite[Proposition~6.4]{ask2} that for all sufficiently
  large primes $p$ and each compact \DVR{}~$\fO$ with residue
  characteristic $p$, we have $\zeta^{\cc}_{\GG \otimes\fO}(s)
  = \zeta^{\ak}_{\alpha^\fO}(s)$.  Let $\HH := \mathsf G_\alpha$ be
  the group scheme over $\ZZ$ associated with $\alpha$ as defined
  in \cite[\S 7.1]{ask2}.  Then $\HH \otimes \QQ$ is a unipotent
  algebraic group of nilpotency class at most $2$ and dimension $2d$.
  By \cite[Corollary~7.2]{ask2}, for each compact \DVR{} $\fO$ with
  odd residue characteristic, $\zeta^{\cc}_{\HH \otimes\fO}(s)
  = \zeta^{\ak}_{\alpha^\fO}(s-d)$.
\end{proof}

\subsection{Application: class counting zeta functions of Baer group schemes}
\label{ss:baer}

Consider an alternating bilinear map $\baer\colon A\times A\to B$ of
$\ZZ$-modules.
Suppose that $B$ is uniquely $2$-divisible (i.e.\ a $\ZZ[1/2]$-module).
The \emph{Baer group} associated with $\baer$ is the  nilpotent group $G_\baer$ of
class at most $2$ on the set $A\times B$ with multiplication
\[
  (a,b) * (a',b') = \Bigl(a + a', b + b' + \frac 1 2 (a \baer a')\Bigr);
\]
this construction is part of the \emph{Baer correspondence}~\cite{Bae38}.
We now describe a version of the operation $\baer \leadsto G_\baer$ for group
schemes.

Let $V$ and $E$ be disjoint finite sets and
let $\baer\colon \ZZ V\times \ZZ V\to \ZZ E$ be an alternating bilinear
map.
We obtain a nilpotent Lie $\ZZ$-algebra (``Lie ring'') $\fg_\baer$ 
of class at most~$2$ with underlying $\ZZ$-module $\ZZ V \oplus \ZZ E$ and Lie
bracket $[x+c,y+d] = x \baer y$ for $x,y \in \ZZ V$ and $c,d\in \ZZ E$.

Let $\sqsubseteq$ be a total order on $V\cup E$.
By \cite[\S 2.4.1]{SV14}, there exists a (unipotent) group scheme
$\GG_\baer = \GG_{\baer,\sqsubseteq}$ over $\ZZ$
such that for each ring $R$, we may identify $\GG_{\baer}(R) = R V \oplus RE$
as sets and such that group commutators in $\GG_{\baer}(R)$ coincide with Lie
brackets in $\fg_{\baer} \otimes_{\ZZ} R$. 
The multiplication $*$ on $\GG_\baer(R)$ is characterised by the following properties:
\begin{enumerate}
\item
  For $v_1 \sqsubseteq \dotsb \sqsubseteq v_k$ in $V$ and
  $r_1,\dotsc, r_k \in R$,
  $(r_1 v_1) * \dotsb * (r_k v_k) = r_1 v_1 + \dotsb + r_k v_k$.
\item
  For $v,w \in V$ with $v \sqsubseteq w$ and $r,s\in R$, we have
  $(sw) * (rv) = rv + sw - rs (v \baer w)$.
\item
  $R E$ is a central subgroup of $\GG_\baer(R)$ and $x * c = x + c$ for
  $x \in \GG_\baer(R)$ and $c \in R E$.
\end{enumerate}

Up to isomorphism, $\GG_{\baer}$ does not depend on
$\sqsubseteq$.
We call $\GG_\baer$ the \emph{Baer group scheme} associated with $\baer$.
For each ring $R$ in which $2$ is invertible, $\GG_\baer(R)$ is isomorphic to
the Baer group attached to the alternating bilinear map $RV\times RV \to RE$
obtained from $\baer$.

\begin{prop}
  \label{prop:baer_cc}
  Let $\baer\colon \ZZ V\times \ZZ V \to \ZZ E$ be an alternating bilinear map.
  Let $\ZZ V \xto\alpha \Hom(\ZZ V, \ZZ E)$ be the module representation with
  $v (w\alpha) = v\baer w$ for $v,w\in V$.
  Let $m = \card E$.
  Let $R$ be the ring of integers of a local or global field of arbitrary 
  characteristic.
  Let $\GG_\baer$ be the Baer group scheme associated with $\baer$ as above.
  Then $\zeta^\cc_{\GG_\baer \otimes R}(s) =
  \zeta^\ak_{\alpha^R}(s-m)$.
\end{prop}
\begin{proof}
  For each finite ring $A$,
  commutators in $\GG_\baer(A)$ are given by Lie brackets in
  $\fg_\baer\otimes_\ZZ A$.
  By the same reasoning as in \cite[Lemma~7.1]{ask2},
  we see that
  $\concnt(\GG_\baer(A)) = \card A^{m} \ask{\alpha^A}$.
\end{proof}

\begin{proof}[Proof of Proposition~\ref{prop:antisymmetric_cc}]
  Let $\baer\colon \ZZ^n\times \ZZ^n\to M^*$ be the alternating bilinear map
  given by $a(x\baer y) = xay^\top$ from \S\ref{ss:intro/class}.
  Let $\ZZ^n\xto\alpha \Hom(\ZZ^n,M^*)$ with $x(y\alpha) = x \baer y$.
  By Proposition~\ref{prop:baer_cc},
  $\zeta^\cc_{\GG_M\otimes R}(s) = \zeta^\ak_{\alpha^R}(s-m)$.
  We claim that $\alpha^\MW$ (see \S\ref{ss:mreps}) is
  isotopic to $M \xincl\iota \So_n(\ZZ)$.
  By \cite[Proposition\ 4.8]{ask2}, this is equivalent to $\iota^\MW$ being isotopic
  to $\alpha$ which is easily verified using the isomorphism $(\ZZ^n)^*
  \approx \ZZ^n$ given by matrix transposition.
  Using Theorem~\ref{thm:ask_duality}, we conclude that
  $\zeta^\ak_{\alpha^R}(s) = \zeta^\ak_{\iota^R}(s)$
  provided that the field of fractions of $R$ is a local field.
  Finally, using \cite[Remark~5.5]{ask2}, the global case reduces to the local
  one.
\end{proof}

\begin{example}[The Heisenberg group scheme]\label{exa:heisenberg}
  Let $v$, $w$, and $e$ be distinct symbols.
  Let $V = \{ v,w \}$ and $E = \{e\}$.
  Let $\sqsubseteq$ be the partial order on $V\cup E$
  with $v\sqsubset w\sqsubset e$.
  Let $\baer\colon \Z V \times \Z V \rightarrow \Z E$ be the (unique)
  alternating bilinear map with $v \baer w = e$.
  Then the Baer group scheme $\GG_{\baer}$ is isomorphic to the
  \emph{Heisenberg group scheme} $\Uni_3$ of upper unitriangular $3\times 3$
  matrices.
  Indeed, for each ring $R$, we obtain a group isomorphism
  \begin{align*}
    \GG_{\baer}(R) \to \Uni_3(R), &\quad
    a v + b w + c e \mapsto
    \begin{bmatrix}
      1 & b & -c \\ 0 & 1 & a\\ 0 & 0 & 1
    \end{bmatrix} & (a,b,c\in R),
  \end{align*}
  and this gives rise to an isomorphism $\GG_\baer \to \Uni_3$ of group schemes.

  The group scheme $\GG_\baer$ is a special case of a general construction, explained and
  explored in \S\ref{ss:graphical_groups}.
  This construction attaches what we will call a \emph{graphical group scheme}
  $\GG_\Gamma$ to each simple graph $\Gamma$.
  The graph yielding the Baer group scheme $\GG_\baer \approx \Uni_3$ is the
  complete graph $\CG_2$ on two vertices.
\end{example}

\subsection{Cokernel integrals}
\label{ss:cokernel}

From now on, let $\fO$ be a compact \DVR{}.
Let $A\xto\theta\Hom(B,C)$ be a module representation over $\fO$, where each
of $A$, $B$, and $C$ is free of finite rank.
For $a\in A$ and $y\in \fO$, the map $B \xto{a\theta}C$
gives rise to an induced map $B\otimes_{\fO}\fO/y \xto{a\theta\otimes\fO/y
  \,=\,
(a\otimes 1)\theta^{\fO/y}}
C\otimes_{\fO} \fO/y$.
Define
\[
  \cokersize_\theta(a,y) := \card{\Coker(a\theta\otimes\fO/y)}.
\]
The following is equivalent to \cite[Theorem~4.5]{ask},
but with kernels replaced by cokernels.

\begin{thm}
  \label{thm:ask_integral_via_coker}
  For $\Real(s) \gg 0$,
  \[
    \zeta_\theta^\ak(s) =
    (1-q^{-1})^{-1}\,
    \int\limits_{A \times \fO}
    \abs{y}^{s - \rank(B) + \rank(C) - 1}
    \,
    \cokersize_\theta(a,y)\,\,
    \dd\mu_{A\times\fO}(a,y).
  \]
\end{thm}
\begin{proof}
  As in \cite[Definition~4.4]{ask} (see \cite[\S 3.5]{ask2}), write
  $\kersize_\theta(a,y) = \card{\Ker(a\theta\otimes\fO/y)}$.
  The claim then follows immediately from \cite[Theorem~4.5]{ask} (cf.\
  \cite[Theorem~3.5]{ask2}) and the fact that
  by the first isomorphism theorem,
  $\kersize_\theta(a,y) = \cokersize_\theta(a,y) \dtimes
  \abs{y}^{\rank(C)-\rank(B)}$ for $y\not= 0$.
\end{proof}
\begin{rem}
  In the present article, we express ask zeta functions in
  terms of sizes of {co}kernels, rather than kernels;
  see Corollary~\ref{cor:ask_zeta_Cmatrix}.
  Cokernels turn out to be more convenient here 
  since they commute with base change (both being colimits).
\end{rem}

\paragraph[Explicit dual form.]{Explicit dual form of Theorem~\ref{thm:ask_integral_via_coker}.}
We can make Theorem~\ref{thm:ask_integral_via_coker} more explicit by choosing bases.
Let $A = \fO U$, $B = \fO V$, and $C = \fO W$, where $U$, $V$, and $W$ are finite sets.
As in \S\ref{ss:A_and_C_matrices}, let $Z = (Z_u)_{u\in U}$ consist of algebraically
independent elements.
By Lemma~\ref{lem:Amatrix_specialisation},
\begin{equation}
  \label{eq:cokersize_via_Amatrix}
  \cokersize_\theta(a,y) = \card{\Coker( \sA^{U,V,W}_\theta(Z) \otimes_{\fO[Z]} (\fO/y)_a)},
\end{equation}
where $(\fO/y)_a$ denotes $\fO/y$ with the $\fO[Z]$-module
structure $Z_u r = a_u r$ ($u \in U$, $r\in \fO/y$).

It will be convenient to express $\zeta_\theta^\ak(s)$ in terms of
$\theta^\MV$ via Theorem~\ref{thm:ask_duality}.
Let $X = (X_v)_{v\in V}$ consist of algebraically independent variables over~$\fO$.
Recall from \eqref{eq:Cmatrix} the definition of the $\fO[X]$-homomorphism
$\sC^{U,V,W}_\theta(X)$.

\begin{cor}
  \label{cor:ask_zeta_Cmatrix}
  For $\Real(s) \gg 0$,
  \[
    \zeta_\theta^\ak(s) =
    (1-q^{-1})^{-1}
    \!\!\!\!
    \int\limits_{\fO V\times\fO}
    \!\!\!
  \abs{y}^{s - \card V + \card W - 1}
  \,
  \card{\Coker(\sC^{U,V,W}_\theta(X))\otimes_{\fO[X]} (\fO/y)_x}
  \,\,
  \dd\mu_{\fO V\times \fO}(x,y).
  \]
  \end{cor}
\begin{proof}
  Combine Theorem~\ref{thm:ask_duality},
  equation~\eqref{eq:cokersize_via_Amatrix},
  and Theorem~\ref{thm:ask_integral_via_coker}.
\end{proof}

\subsection{Zeta functions associated with modules over polynomial rings}
\label{ss:polynomial_module_zeta}

We will study ask zeta functions by considering a sequence of
successive generalisations of the integrals featuring in
Corollary~\ref{cor:ask_zeta_Cmatrix}.
As our first step in this direction, we consider integrals obtained
from Corollary~\ref{cor:ask_zeta_Cmatrix} by allowing
$\Coker(\sC^{U,V,W}_\theta(X))$ to be a more general type of $\fO[X]$-module.

As before, we write $X = (X_v)_{v\in V}$, where the $X_v$ are algebraically independent variables over $\fO$ (or
whichever base ring we consider).

Let $M$ be a finitely generated $\fO[X]$-module.
The example of primary interest to us at this point is the case $M =
\Coker(\sC^{U,V,W}_\theta(X))$ (see \eqref{eq:Cmatrix}), where
$\theta$ is a suitable module representation over $\fO$.
As in \S\ref{ss:A_and_C_matrices}, for $x\in \fO V$, let $\fO_x$ denote $\fO$
endowed with the $\fO[X]$-module structure $X_v r = x_v r$
($v\in V$, $r\in \fO$).
More generally, for an arbitrary $\fO$-module $N$, we let $N_x$ denote
the $\fO[X]$-module $N\otimes_{\fO} \fO_x$
(cf.\ Lemma~\ref{lem:natural_specialisation}).

\begin{defn}
  \label{d:zeta_polynomial_module}
  Define a zeta function 
  \begin{equation}
    \label{eq:zeta_polynomial_module}
    \zeta_{M}(s) := \int\limits_{\fO V\times\fO}
    \abs{y}^{s-1}
    \dtimes \card{M_x\otimes_{\fO} \fO/y}
    \,\,
    \dd\mu_{\fO V\times \fO}(x,y).
  \end{equation}
\end{defn}

The following simple observation will be used repeatedly throughout this
article.

\begin{lemma}
  \label{lem:natural_specialisation}
  Let $R$ be a ring.
  Let $S$ and $S'$ be $R$-algebras.
  Let $\cX' = \Spec(S')/R$ and let $x\in \cX'(S)$.
  Let $S'\xto\chi S$ be the $R$-algebra map corresponding to $x$.
  Let $M$ be an $S$-module and let $M'$ be an $S'$-module.
  Let $M'_x := (M')^\chi$ and $M_x := M_\chi$ (see \S\ref{ss:mreps});
  these are both $(S,S')$-bimodules.
  Then $M_x \otimes_{S'} M'$ and $M \otimes_{S} M'_x$ both carry naturally
  isomorphic $(S,S')$-bimodule structures.
\end{lemma}
\begin{proof}
  Apply \cite[Ch.~II, \S 3, no.\ 8]{Bou70}
  with $A = S$, $B = S'$, $E = M$, $F = S_\chi$, and $G = M'$.
\end{proof}

\begin{rem}
  \label{rem:two_point_modules}
  Note that for $x \in \fO V = \Spec(\fO[X])(\fO)$, the definitions of the
  $\fO$-modules $M_x$ and $(\fO/y)_x$ provided by
  Lemma~\ref{lem:natural_specialisation} coincide with those given above.
  In particular, Lemma~\ref{lem:natural_specialisation} allows us to identify
  $M_x \otimes_{\fO}\fO/y = M \otimes_{\fO[X]} (\fO/y)_x$ in
  \eqref{eq:zeta_polynomial_module}.
  (For a proof, relabel $M\leftrightarrow M'$ and take $S = \fO$, $S'=\fO[X]$,
  and $M = \fO/y$ in Lemma~\ref{lem:natural_specialisation}.)
\end{rem}

\begin{rem}
  If $\fO$ has characteristic zero, then standard arguments from $p$-adic
  integration show that $\zeta_M(s) \in \QQ(q^{-s})$ and also establish
  ``Denef-type formulae'' in a global setting;
  cf.\ \cite[\S 4.3.3]{ask}.
\end{rem}

The following is immediate from Corollary~\ref{cor:ask_zeta_Cmatrix}.

\begin{cor}
  \label{cor:dual_coker_int}
  Let $U$, $V$, and $W$ be finite sets
  and $X = (X_v)_{v\in V}$ as before.
  Let $\fO U \xto{\theta} \Hom(\fO V, \fO W)$ be a module representation over $\fO$.
  Let $M = {\Coker\bigl(\sC^{U,V,W}_\theta(X)\bigr)}$; see~\eqref{eq:Cmatrix}.
  Then
  \[
    \zeta^{\ak}_\theta(s) = 
    (1-q^{-1})^{-1}
    \, \zeta_M(s - \card{V}+ \card{W}).
    \pushQED{\qed}
    \qedhere
    \popQED
  \]
\end{cor}

In particular, the zeta functions attached to modules over polynomial rings
in Definition~\ref{d:zeta_polynomial_module} generalise local ask zeta
functions.
We will further generalise the former functions by suitably replacing
polynomial rings by toric rings; see Definition~\ref{d:zeta_toric_module}.
As we will see, this greater generality will provide us with the means to
study ``toric properties'' of ask zeta functions by purely combinatorial
means.

\section{Modules and module representations from (hyper)graphs}
\label{s:graphs_and_hypergraphs}

In this section, we begin by fixing our notation for various concepts related
to graphs and hypergraphs.
Formalising and generalising our discussion from \S\ref{ss:intro/graphs},
for each hypergraph~$\Eta$ and (simple) graph $\Gamma$, we define the
incidence representation $\eta$ of $\Eta$ and the adjacency representations
$\gamma_{\pm}$ of $\Gamma$, as well as corresponding incidence and
adjacency modules $\Inc(\Eta)$ and $\Adj(\Gamma;\pm 1)$.
We also define graphical group schemes and express their class
counting zeta functions in terms of ask zeta functions.

Throughout, let $R$ be a ring.

\subsection{Graphs, multigraphs, and hypergraphs}
\label{ss:graphs}

For a general reference on hypergraph theory, see e.g.~\cite{Bre13}.

\paragraph{Hypergraphs.}
A \emph{hypergraph} is a triple $\Eta = (V,E,\abs\dtimes)$
consisting of a finite set $V$ of \emph{vertices}, a finite set $E$ of
\emph{hyperedges}, and a \emph{support function}
$E\xto{\abs{\dtimes\,}}\Pow(V)$.
We often tacitly assume that $V \cap E = \emptyset$.
When confusion is unlikely, we often omit $\abs\dtimes$ (and occasionally even
$V$ and $E$) from our notation. We write $\verts(\Eta) = V$, $\edges(\Eta) =
E$, and $\abs\dtimes_\Eta = \abs\dtimes$.
Two hyperedges $e$ and $e'$ are \emph{parallel} if $\abs e = \abs{e'}$.
An \emph{edge} is a hyperedge $e$ with $\#\abs e \in \{1,2\}$.
An edge $e$ with $\#\abs e = 1$ is a \emph{loop}.
The \emph{reflection} of $\Eta$ is the hypergraph
$\Eta^\comp = (V,E,\abs\dtimes^\comp)$
with $\abs{e}^\comp = V\setminus\abs e$.
An \emph{isomorphism} between hypergraphs $\Eta$ and $\Eta'$
consists of
bijections $\verts(\Eta) \xto\phi \verts(\Eta')$
and $\edges(\Eta)\xto\psi \edges(\Eta')$ such that
${\abs e}_\Eta \Pow(\phi) = \abs{e\psi}_{\Eta'}$ for each $e\in \edges(\Eta)$,
where $\Pow(\phi)$ is the direct image map $\Pow(\verts(\Eta)) \to
\Pow(\verts(\Eta'))$ induced by $\phi$.

\paragraph{Incidence matrices.}
Let $\Eta = (V,E,\abs\dtimes)$ be a hypergraph.
A vertex $v\in V$ and hyperedge $e\in E$ are \emph{incident}, written
$v\sim_\Eta e$ or simply $v\sim e$, if $v\in \abs e$.
Let $\II(\Eta)$ denote the set of all pairs $(v,e)\in V\times E$ with
$v\sim_\Eta e$.
Write $V = \{ v_1,\dotsc,v_n\}$ and $E = \{ e_1,\dotsc,e_m\}$, where $n =
\card V$ and $m = \card E$.
The \emph{incidence matrix} of $\Eta$ with respect to the given orderings of
the elements of $V$ and $E$ is the $n\times m$ matrix $[a_{ij}]$ with
$a_{ij} = 1$ if $v_i \sim_{\Eta} e_j$ and $a_{ij}=0$ otherwise.

\paragraph{Hypergraph operations.}
The \emph{disjoint union} $\Eta_1\oplus \Eta_2$ of hypergraphs
$\Eta_i = (V_i,E_i,\abs\dtimes_i)$ ($i=1,2$) is the hypergraph
on the vertex set $V_1\sqcup V_2$ (disjoint union) with hyperedge set
$E_1\sqcup E_2$ and support function
$\norm{e_i} = \abs{e_i}_i$ for $e_i \in E_i$.
If $A_i\in\Mat_{n_i\times m_i}(\Z)$ is an incidence matrix of
$\Eta_i$, then the block-diagonal matrix
\[
  \begin{bmatrix}
    A_1 & 0 \\
    0 & A_2
  \end{bmatrix}
  \in \Mat_{(n_1+n_2)\times (m_1+m_2)}(\Z)
\]
is an incidence matrix of $\Eta_1\oplus \Eta_2$.

The \emph{complete union} $\Eta_1\freep\Eta_2$ is the hypergraph on
the vertex set $V_1\sqcup V_2$ with hyperedge set
$E_1\sqcup E_2$ and support function
$\norm{e_i} = \abs{e_i}_i \sqcup V_j$ for $e_i \in E_i$ and $i + j = 3$.
If $A_i\in\Mat_{n_i\times m_i}(\Z)$ is an incidence matrix of
$\Eta_i$, then
\[
  \begin{bmatrix}
    A_1 & \bfo_{n_1\times m_2}\\\
    \bfo_{n_2\times m_1} & A_2
  \end{bmatrix}
  \in \Mat_{(n_1+n_2)\times (m_1+m_2)}(\Z)
\]
is an incidence matrix of $\Eta_1\freep\Eta_2$;
recall that $\bfo_{n\times m}$ denotes the $n\times m$ all-one matrix.

Both disjoint unions and complete unions naturally extend to families of more
than two hypergraphs.
These two operations are related by reflections of hypergraphs via the
identity $(\Eta_1 \oplus \Eta_2)^\comp = \Eta_1^\comp \freep \Eta_2^\comp$.

\paragraph{(Multi)graphs.}
A \emph{multigraph} is a hypergraph $\Gamma = (V,E,\abs\dtimes)$ all of whose
hyperedges are edges.
A \emph{graph} is a multigraph without parallel edges;
note that we allow graphs to contain loops.
Two vertices $v$ and $v'$ of a graph $\Gamma$ are \emph{adjacent} if there
exists an edge $e\in E$ with $\abs e = \{v,v'\}$; we write $v\sim_\Gamma v'$ or $v
\sim_\Gamma v'$ in that case.
(This notation is unambiguous whenever $V \cap E = \emptyset$ which we tacitly
assume.)
The set of \emph{neighbours} of $v\in V$ in $\Gamma$ is $\{
w\in V: v\sim w\}$.
A graph is \emph{simple} if it contains no loops.
The \emph{join} $\Gamma_1 \join \Gamma_2$ of two simple graphs $\Gamma_1$ and
$\Gamma_2$ is the simple graph obtained from the disjoint union $\Gamma_1
\oplus \Gamma_2$ by adding an edge between each pair of vertices $(v_1,v_2)
\in V_1\times V_2$. 

\paragraph{Parameterising hypergraphs.}
Up to isomorphism, a hypergraph $\Eta = (V,E,\abs\dtimes)$ determines and is
determined by the cardinalities of the fibres
$\mu_I := \# \{ e \in E : \abs e = I\} \in\N_0$ 
for $I\subset \verts(\Eta)$.
More formally:

\begin{defn}\label{def:hyp.mu}
  Let $V$ be a finite set.
  Given a vector $\bfmu=(\mu_I)_{I \subset V}\in \N_0^{\Pow(V)}$ of
  non-negative multiplicities, define a hypergraph $\Eta(\bfmu)$ with
  \[
    \verts(\Eta(\bfmu)) = V,
    \quad
    \edges(\Eta(\bfmu)) = \bigl\{ (I,j) :
    I \subset V, j \in [\mu_I]\bigr\},
    \quad
    \abs{(I,j)}_{\Eta(\bfmu)} = I.
  \]
\end{defn}

In other words, $\Eta(\bfmu)$ contains precisely $\mu_I$
hyperedges with support $I$ for each $I\subset V$.
We often write $m=\sum_{I\subset V}\mu_I$ for the total number of hyperedges
of~$\Eta(\bfmu)$.
Clearly, for each hypergraph $\Eta$, there exists a unique vector $\bfmu$ as
in Definition~\ref{def:hyp.mu} such that $\Eta$ and $\Eta(\bfmu)$ are isomorphic
by means of an isomorphism fixing each vertex.

We will use the following shorthand notation for the
hypergraphs~$\Eta(\bfmu)$.  For a finite set~$V$, suppose
that we are given numbers $\mu_I \in \NN_0$ for some but perhaps not
all subsets $I \subset V$.
We may then extend the collection of
these $\mu_I$ to a family $\bfmu$ as in Definition~\ref{def:hyp.mu} by
setting $\mu_J = 0$ for the previously missing subsets $J\subset V$.
We set
$$\Eta\Bigl(V \Bigm\vert \sum_{I} \mu_I I\Bigr) := \Eta(\bfmu);$$
to further simplify our notation, we often drop coefficients $\mu_I = 1$
and summands $\mu_I I$ with $\mu_I = 0$ from the left-hand side.

\paragraph{Important families of (hyper)graphs.}
The following hypergraphs will feature in several places throughout
this article;
most have vertex set $V = [n]$.\\

\noindent
The \emph{discrete} (\emph{hyper})\emph{graph} on $n$
vertices (often called an ``empty graph'' in the literature) is
\begin{equation}\label{def:disc.hyp}
  \DG_n := \Eta([n] \mid 0) := \Eta([n] \mid 0[n]).
 \end{equation}
The $n\times m$ \emph{block hypergraph}
is the hypergraph
\begin{equation}\label{def:block.hyp}
  \mathsf{BH}_{n,m} := \Eta([n] \mid m [n]).
\end{equation}
We denote the reflection of $\BE_{n,m}$ by $\RE_{n,m}$;
that is, $\RE_{n,m} = \Eta([n] \mid m\emptyset)$.
More generally, given $\bfn = (n_1,\dots,n_\nb)\in\N^{\nb}$ and
$\bfm=(m_1,\dots,m_\nb)\in\N_0^{\nb}$, let
\begin{align}
  \label{def:block.hyp.multi}
  \BE_{\bfn,\bfm} & := \BE_{n_1,m_1} \oplus \dotsb \oplus \BE_{n_r,m_r} \quad \text{and}\\
  \RE_{\bfn,\bfm}
  & := \BE_{\bfn,\bfm}^\comp =
    \RE_{n_1,m_1} \freep \dotsb \freep \RE_{n_r,m_r}.
    \label{def:block.codis.multi}  
\end{align}
The \emph{complete graph} on $n$ vertices is
\begin{equation}\label{def:comp.hyp}
\CG_n := \Eta([n] \mid \sum_{1 \le i < j \le n} \{i,j\}).
\end{equation}
The \emph{star graph} on $n$ vertices (with centre $1$) is
\begin{equation}\label{def:star.graph}
  \Star_n := \Eta([n] \mid \sum_{1 < i \le n} \{1,i\}).
\end{equation}
The \emph{path graph} on $n$ vertices is
\begin{equation}\label{def:path.graph}
  \Path n := \Eta([n] \mid \sum_{1 \le i < n}\{i,i+1\}).
\end{equation}
The \emph{cycle graph} on $n \ge 3$ vertices is
\begin{equation}\label{def:cycle.graph}
  \Cycle n := \Eta([n] \mid \sum_{1 \le i < n} \{i,i+1\} + \{1,n\}).
\end{equation}
The \emph{staircase hypergraph} associated with $\bm m =
(m_0,\dotsc,m_n)\in\N_0^{n+1}$ is 
\begin{equation}\label{def:staircase.graph}
  \Stair_{\bm m}:= \Eta\! \left( [n] \mid \sum_{i=0}^n m_i[i])\right).
\end{equation}

\subsection{The incidence representation and module associated with a hypergraph}
\label{ss:incidence}

Let $\Eta = (V,E,\abs\dtimes)$ be a hypergraph.
We construct a module representation $\eta^R$ over~$R$ which we call the
\emph{incidence representation} of $\Eta$ over $R$.

\paragraph{Description of $\eta$ in terms of hypergraph coordinates.}
For $(v,e) \in V\times E$, let $[ve]$ be the $R$-homomorphism
$RV \to RE$ which satisfies
$u[ve] = \delta_{uv} \dtimes e$ ($u\in V$).
Recall that $\II(\Eta) = \{ (v,e) \in V\times E : v\sim_\Eta e\}$.
Define $\eta^R$ to be the module representation
\[
  R\, \II(\Eta) \to \Hom(RV,RE), \quad
  (v,e)\mapsto [ve];
\]
write $\eta = \eta^\ZZ$.
We refer to $\eta$ as \itemph{the} (absolute) incidence representation
of $\Eta$.
Note that the notation $\eta^R$ is unambiguous: if $R\to S$ is a ring map,
then $(\eta^R)^S = \eta^S$.

\paragraph{Description of $\eta$ in terms of matrices.}
Write $m = \card E$, $n = \card V$, $V = \{ v_1,\dotsc,v_n\}$, and
$E = \{e_1,\dotsc,e_m\}$.  Let
\[
  M = \Bigl\{
    [a_{ij}] \in \Mat_{n\times m}(R) :
    a_{ij} = 0 \text{ whenever } v_i \not\in \abs{e_j}
    \Bigr\}.
\]
Then $\eta^R$, as defined above, is isotopic (see \S\ref{ss:mreps}) to the
inclusion $M\incl \Mat_{n\times m}(R)$. 

\paragraph{Incidence modules.}
Let $X = (X_v)_{v\in V}$ consist of algebraically independent
variables over $R$.
Define
\[
  \inc(\Eta;R) :=
  \Bigl\langle X_v e :  v \sim_\Eta e \,\, (v\in V, \, e\in E) \Bigr\rangle \le R[X]\, E.
\]

The \emph{incidence module} of $\Eta$ over $R$ is
\[
  \Inc(\Eta;R) := \frac{R[X]\,E}{\inc(\Eta;R)};
\]
the (absolute) incidence module of $\Eta$ is $\Inc(\Eta) :=
\Inc(\Eta;\ZZ)$.
Clearly,
\begin{equation}
  \label{eq:Inc_sum}
  \Inc(\Eta;R) \approx_{R[X]}  \mathlarger{\mathlarger{\bigoplus\limits}}_{e\in E}
  \frac{R[X]}{\langle X_v : v \in \abs e\rangle}.
\end{equation}

\begin{lemma}
  $\Inc(\Eta;S) = \Inc(\Eta;R)^{S[X]}$ for each ring map $R\to S$.
\end{lemma}
\begin{proof}
  Immediate from the right exactness of tensor products.
\end{proof}

The incidence module of $\Eta$ determines the ask zeta functions associated with $\eta$:
\begin{prop}
  \label{prop:Inc_zeta}
    For each compact \DVR{} $\fO$,
    \[
      \zeta_{\eta^\fO}^\ak(s) =
      (1-q^{-1})^{-1} \, \zeta_{\Inc(\Eta;\fO)}(s - \card V + \card E).
    \]
\end{prop}
\begin{proof}
  By Lemma~\ref{lem:ask_image_presentation}, 
  $\Img(\sC^{\II(\Eta),V,E}_{\eta^\fO}(X)) = \inc(\Eta;\fO)$ (see \eqref{eq:Cmatrix})
  whence $\Inc(\Eta;\fO) = \Coker\!\left(\sC^{\II(\Eta),V,E}_{\eta^\fO}(X)\right)$ .
  The claim thus follows from Corollary~\ref{cor:dual_coker_int}.
\end{proof}

For any suitable ring $R$, we refer to $\zeta^{\ak}_{\eta^R}(s)$ as the
\emph{ask zeta function} of $\Eta$ over $R$.
We can use the structure of $\Inc(\Eta;R)$ in \eqref{eq:Inc_sum} to
make Proposition~\ref{prop:Inc_zeta} more explicit.
\begin{prop}
  \label{prop:hypergraph_int}
  For each compact \DVR{} $\fO$,
  \[
    \zeta_{\eta^\fO}^\ak(s)
    = (1-q^{-1})^{-1}
    \,
    \int\limits_{\fO V \times \fO}
    {\abs{y}^{s - \card V + \card E - 1}}
    \,
    {\prod\limits_{e\in E}\norm{x_e;y}}^{-1}
    \,\dd\mu_{\fO V \times\fO}(x,y),
  \]
  where $x_e := \{ x_v : v\in \abs e\}$.
\end{prop}
\begin{proof}
  For ideals $\mathfrak a,\mathfrak b\normal R$ of a ring
  $R$, there is a natural $R$-module isomorphism
  $R/\mathfrak a \otimes_R R/\mathfrak b \approx R/(\mathfrak a+ \mathfrak
  b)$; this follows e.g.\ from \cite[Ch.\ I, \S 2, no.\ 8]{Bou85}.
  Equivalently, given a surjective ring map $R\xto\lambda S$, we have
  $R/\mathfrak a \otimes_R S \approx_S
  S/(\mathfrak a\lambda)$.
  Let $x\in \fO V$ and $0\not= y\in \fO$.
  Using \eqref{eq:Inc_sum} and Remark~\ref{rem:two_point_modules},
  we then obtain $\fO$-module isomorphisms
  \[
    \Inc(\Eta;\fO)\otimes_{\fO[X]} (\fO/y)_x \approx
    \mathlarger{\mathlarger{\bigoplus\limits}}_{e\in E}
    \frac{\fO[X]}{\langle X_v : v\in \abs e\rangle}
    \otimes_{\fO[X]} (\fO/y)_x
    \approx
    \mathlarger{\mathlarger{\bigoplus\limits}}_{e\in E}
    \frac{\fO}{\langle x_e; y \rangle}.
  \]
  The common cardinality of these modules is therefore
  $\prod\limits_{e\in E} \norm{x_e;y}^{-1}$.
  The claim now follows from Corollary~\ref{cor:dual_coker_int}.
\end{proof}

\begin{rem}
  Suppose that $\Eta=\BE_{n,m}$ is the $n\times m$ block
  hypergraph; see \eqref{def:block.hyp}. In terms of matrices,
  $\eta^R$ parameterises all of $\Mat_{n\times m}(R)$.  A formula for
  the ask zeta function of (the identity on) $\Mat_{n\times m}(\fO)$
  is given in \cite[Proposition\ 1.5]{ask};
  see Example~\ref{exa:staircase}(\ref{exa:BEnm}).
  The integral in the proof of
  this proposition is exactly the corresponding
  special case of Proposition~\ref{prop:hypergraph_int} here.
  Similarly, the determination of the ask zeta functions of modules
  of (strictly) upper triangular matrices over~$\fO$ in
  \cite[Proposition\ 5.15]{ask} proceeded by computing the integral in the
  corresponding special case of
  Proposition~\ref{prop:hypergraph_int};
  see Example~\ref{exa:staircase}(\ref{exa:trn}).
  We thus recognise various ad hoc arguments from \cite{ask} as instances of
  the cokernel formalism here.
\end{rem}

As we will see in \S\ref{s:master}, Proposition~\ref{prop:hypergraph_int}
can be used to produce an explicit formula for the ask zeta functions
associated with \itemph{all} hypergraphs on a given vertex set.

\subsection{Two adjacency representations and modules associated with a graph}
\label{ss:two_adjacencies}

Let $\Gamma = (V,E,\abs\dtimes)$ be a graph.
We construct two module representations $\gamma_+$ and $\gamma_-$ associated
with $\Gamma$ which we call the \emph{adjacency representations} of $\Gamma$.
The first of these module representations is defined for all graphs $\Gamma$ 
and parameterises symmetric matrices with suitably constrained support.
The second is defined whenever $\Gamma$ is simple and
parameterises antisymmetric matrices with support constrained by 
$\Gamma$.
For our purposes, the antisymmetric case is usually more interesting.

Denote the exterior (resp.\ symmetric) square of a module $M$ by $M\wedge M$
(resp.~$M\odot M$).

\paragraph{Alternating case: construction of $\gamma_-$.}
Let $\Gamma$ be simple.
For an $R$-module $M$, let
\[
(M \wedge M)^* \xto{\So\!_M} \Hom(M,M^*)
\]
be the module representation which sends $\psi \in (M\wedge M)^*$ to the map
\[
  M \to M^*,\quad
  m \mapsto (n\mapsto (m\wedge n)\psi).
\]
If $M = R^n$, then the evident choices of bases furnish an isotopy 
between $\So_M$ and the inclusion $\So_n(R) \incl \Mat_n(R)$,
of alternating (= antisymmetric with zero diagonal) $n\times n$ matrices into $\Mat_n(R)$.
We define the \emph{negative adjacency representation} $\gamma^R_-$ of $\Gamma$
over $R$ to be the composite
\[
  \left(\frac{R V\wedge R V}{\NonEdges(\Gamma,-1;R)}\right)^* 
  \into (RV\wedge R V)^* \xto{\So_{R V}} \Hom(RV, (R V)^*),
\]
where $\NonEdges(\Gamma,-1;R)$ is the submodule of $R V \wedge R V$ generated by all $v\wedge
w$ such that $v,w\in V$ are \itemph{non-adjacent} in $\Gamma$ and the first map is the
dual of the quotient map.
For an explicit description of $\gamma^R_-$ in terms of matrices, let $V =
\{v_1,\dotsc,v_n\}$, where $n = \card V$.
Let
\[
  M^- = \Bigl\{
  [a_{ij}] \in \So_n(R) :
  a_{ij} = 0 \text{ whenever } v_i \not\sim v_j
  \Bigr\};
\]
cf.\ the definition of $M^-(\Gamma)$ in \S\ref{ss:intro/graphs}.
It is easy to see that $\gamma^R_-$ is isotopic to the inclusion
$M^- \incl \Mat_n(R)$;
a proof is implicitly given in the proof of
Proposition~\ref{prop:ask_adjacency_module} below.

We refer to $\gamma_- := \gamma^{\ZZ}_-$ as \itemph{the} (absolute) negative
adjacency representation of $\Gamma$.
As with incidence representations in \S\ref{ss:incidence}, this notation is
unambiguous: $(\gamma_-^R)^S = \gamma_-^S$ for each ring map $R\to S$.

\paragraph{Symmetric case: construction of $\gamma_+$.}
Let $\Gamma$ be not necessarily simple.
For an $R$-module~$M$, let
\[
(M \odot M)^* \xto{\Sym_M} \Hom(M,M^*)
\]
be the module representation which sends $\psi \in (M\odot M)^*$ to the map
\[
  M \to M^*,\quad
  m \mapsto (n\mapsto (m\odot n)\psi).
\]
$\Sym_{R^n}$ is isotopic to the inclusion $\Sym_n(R) \incl \Mat_n(R)$
of symmetric matrices.

We define the \emph{positive adjacency representation} $\gamma^R_+$ of
$\Gamma$ over $R$ to be the composite
\[
  \left(\frac{R V\odot R V}{\NonEdges(\Gamma,+1;R)}\right)^*
  \into
  (RV\odot R V)^* \xto{\Sym_{R V}} \Hom(RV, (R V)^*),
\]
where $\NonEdges(\Gamma,+1;R)$ is the submodule of $R V \odot R V$ generated
by all $v\odot w$ such that $v,w\in V$ are non-adjacent in $\Gamma$.
In terms of matrices, let $V = \{v_1,\dotsc,v_n\}$, where $n = \card V$.
Let
\[
  M^+ = \Bigl\{
  [a_{ij}] \in \Sym_n(R) :
  a_{ij} = 0 \text{ whenever } v_i \not\sim v_j
  \Bigr\};
\]
cf.\ \S\ref{ss:intro/rank} and the definition of $M^+(\Gamma)$ in \S\ref{ss:intro/graphs}.
Then $\gamma^R_+$ is isotopic to the inclusion $M^+ \incl \Mat_n(R)$.
As above, we call $\gamma_+ := \gamma^{\ZZ}_+$ \itemph{the} (absolute) positive
adjacency representation of $\Gamma$.

\paragraph{Adjacency modules.}
Let $X = (X_v)_{v\in V}$ as in \S\ref{ss:incidence}.
For $v,w\in V$, define
\[
  [v,w;\pm 1] :=
  \begin{cases}
    X_v w \pm X_w v, & \text{if }v \not= w,\\
    \pm X_v v, & \text{if }v = w,
  \end{cases}
\]
an element of $\ZZ[X]\, V$,
and
\[
  \adj(\Gamma,\pm 1;R) :=
  \Bigl\langle
  [v,w;\pm 1] : v, w \in V, \, v\sim w\Bigr\rangle
  \le R[X]\, V.
\]
(The relevance of our definition of $[v,v;-1]$ will become apparent in
\S\ref{s:uniformity}.)

The \emph{(positive resp.\ negative) adjacency module} of $\Gamma$ over $R$ is
\[
  \Adj(\Gamma,\pm 1;R) := \frac{R[X]\,V}{\adj(\Gamma,\pm 1;R)}.
\]

\begin{lemma}
  $\Adj(\Gamma,\pm 1;S) = \Adj(\Gamma,\pm 1;R)^{S[X]}$ for each ring map $R\to
  S$. \qed
\end{lemma}

We write $\Adj(\Gamma,\pm 1) := \Adj(\Gamma,\pm 1;\ZZ)$.

\paragraph{Adjacency modules and ask zeta functions of graphs.}
In the same way that incidence modules of hypergraphs determine the ask zeta
functions associated with incidence representations,
adjacency modules of $\Gamma$ are related to ask zeta functions
derived from~$\gamma_\pm$.
\begin{prop}
  \label{prop:ask_adjacency_module}
  For each compact \DVR{} $\fO$,
  \[
    \zeta_{\gamma_\pm^\fO}^\ak(s) =
    (1-q^{-1})^{-1} \, \zeta_{\Adj(\Gamma,\pm 1;\fO)}(s).
  \]
  (Here, we assume that $\Gamma$ is simple in the negative case.)
\end{prop}
\begin{proof}
  We only spell out the ``negative case''; the positive one can be established
  along similar lines.
  Let $\sqsubseteq$ be an arbitrary total order on $V$.
  Let
  \[
    \OrderedAdjPairs(\Gamma,\sqsubseteq) := \{ (v,w) \in V\times W : v \sim w \text{ and } v \sqsubseteq w\}.
  \]
  Let $V^* = \{ v^*: v\in V\}$ denote the dual basis associated with the basis $V$
  of $RV$.
  The images of the symbols $v\wedge w$ for
  $(v,w) \in \OrderedAdjPairs(\Gamma,\sqsubseteq)$ form a basis of
  $\frac{RV\wedge RV}{\NonEdges(\Gamma,-1;R)}$; 
  let $\Phi := \{ \phi_{vw} : (v,w) \in \OrderedAdjPairs(\Gamma,\sqsubseteq)\}$ denote the
  associated dual basis of
  $\left(\frac{RV \wedge RV}{\NonEdges(\Gamma,-1;R)}\right)^*$, indexed in the
  natural way.
  Define a module representation $R \OrderedAdjPairs(\Gamma,\sqsubseteq)
  \xto{\theta = \theta^R_\sqsubseteq} \Hom(RV,RV)$, where
  for $(v,w) \in \OrderedAdjPairs(\Gamma,\sqsubseteq)$ and $u \in V$,
  \[
    u ((v,w)\theta) = \begin{cases}
      +w, & \text{if }u = v,\\
      -v, & \text{if }u = w, \\
      \phantom+0, & \text{otherwise}.
    \end{cases}
  \]
  It follows that the diagram
  \[
    \begin{CD}
      RV @>{(v,w)\theta}>> RV \\
      @| @VV{\nu}V\\
      RV @>{\phi_{vw}{\gamma^R_-}}>> (RV)^*
    \end{CD}
  \]
  commutes, where $RV\xto{\nu}(RV)^*$ is the isomorphism $v\nu = v^*$ ($v\in V$).
  Hence, $\gamma^R_-$ and $\theta$ are isotopic (see \S\ref{ss:mreps}).
  Lemma~\ref{lem:ask_image_presentation} shows that
  \[
    \Img\Bigl(\sC^{\OrderedAdjPairs(\Gamma,\sqsubseteq),V,V}_\theta(X)\Bigr)
    = \Bigl\langle
    X_v w - X_w v : (v,w) \in \OrderedAdjPairs(\Gamma,\sqsubseteq)
    \Bigr\rangle
    =
    \adj(\Gamma,-1;R)
  \]
  whence
  \[
    \Coker\Bigl(\sC^{\Phi,V,V^*}_{\gamma^R_-}-(X)\Bigr)
    \approx_{R[X]}
    \Coker\Bigl(\sC^{\OrderedAdjPairs(\Gamma,\sqsubseteq),V,V}_\theta(X)\Bigr)
    = \Adj(\Gamma,-1;R).
  \]
  The claim now follows from Corollary~\ref{cor:dual_coker_int}.
\end{proof}

\subsection{Graphical groups and group schemes}
\label{ss:graphical_groups}

Let $\Gamma = (V,E,\abs\dtimes)$ be a simple graph.
Let $\sqsubseteq$ be an arbitrary total order on $V$.
Define an alternating bilinear map
\[
  \baer \colon \ZZ V\times \ZZ V\to \ZZ E
\]
by letting, for $v,w\in V$ with $v\sqsubset w$,
\[
  v \baer w :=
  \begin{cases}
    e, & \text{if there exists $e\in E$ with }\abs e = \{ v,w\},\\
    0, & \text{otherwise}.
  \end{cases}
\]

We leave it to the reader to verify that the isomorphism type of the Baer
group scheme~$\GG_\baer$ (see \S\ref{ss:baer}) associated with $\baer$ only
depends on $\Gamma$ and not on the chosen total order $\sqsubseteq$.
We call $\GG_\Gamma := \GG_\baer$ the \emph{graphical group scheme} associated
with $\Gamma$.
If $\Gamma$ is a cograph (see \S\S\ref{ss:intro/cographs}, \ref{ss:cographs}),
we talk about the \emph{cographical group scheme} associated with $\Gamma$.
By a (\emph{co})\emph{graphical group} (over $R$) we mean a group of rational
points of a ({co}){graphical group scheme}, i.e.\ a group of the
form $\GG_{\Gamma}(R)$ for some ring $R$ and (co)graph~$\Gamma$.

\begin{rem}
  \label{rem:raag}
    The group $\GG_\Gamma(\ZZ)$ of $\ZZ$-points of $\GG_\Gamma$ is a
    finitely generated torsion-free nilpotent group.
    It admits the presentation
    \begin{align*}
      \GG_\Gamma(\ZZ) \approx
      \Bigl\langle
      V \sqcup E \,\Bigm\vert\, &
                                \text{$[v,w] = e$  for $e\in E$ with $\abs e = \{
                                v,w\}$ and $v \sqsubset w$}, \\
                              & \text{$[v,w] = 1$ for non-adjacent $v,w\in V$, and} \\
                              & \text{$[v,e] = [e,f] = 1$ for all $v\in V$ and $e,f\in E$}
                                \Bigr\rangle.
    \end{align*}

    Equivalently, $\GG_\Gamma(\ZZ)$ is the maximal nilpotent quotient
    of nilpotency class at most $2$ of the right-angled Artin group
    \[
      \bigl\langle V \bigm\vert
      [v,w] = 1 \text{ for all non-adjacent } v,w \in V
      \bigr\rangle
    \]
    associated with the \itemph{complement} of $\Gamma$; see e.g.\
    \cite{Rut07}.
    It will prove advantageous for our graph-theoretic arguments in
    \S\S\ref{s:uniformity}--\ref{s:models} to work with $\Gamma$ rather than
    with its complement.
\end{rem}

The following variant of Proposition~\ref{prop:antisymmetric_cc} (which was
proved in \S\ref{ss:baer}) will be crucial in establishing
Corollary~\ref{cor:graphical_cc} in \S\ref{ss:wsm}.

\begin{prop}
  \label{prop:graphical_cc}
  Let $\Gamma$ be a simple graph with $m$ edges and
  let $\gamma_-$ denote its negative adjacency representation over $\ZZ$.
  Let $R$ be the ring of integers of a local or global field of arbitrary
  characteristic.
  Then $\zeta^\cc_{\GG_\Gamma \otimes R}(s) =
  \zeta^\ak_{\gamma_-^R}(s-m)$.
\end{prop}
\begin{proof}
  Let $\ZZ V \xto{\alpha}\Hom(\ZZ V,\ZZ E)$
  be the module representation with $v (w\alpha) = v \baer w$
  for all $v,w\in V$.
  By Proposition~\ref{prop:baer_cc},
  $\zeta^\cc_{\GG_\Gamma \otimes R}(s) =
  \zeta^\ak_{\alpha^R}(s-m)$.
  A straightforward calculation as in the proof of
  Proposition~\ref{prop:ask_adjacency_module} shows that the dual 
  $\alpha^\MW$ (see \S\ref{ss:mreps}) of $\alpha$ is isotopic to $\gamma_-$.
  Using Theorem~\ref{thm:ask_duality} and the final argument from the proof of
  Proposition~\ref{prop:antisymmetric_cc} (see \S\ref{ss:baer}),
  we conclude that
  $\zeta^\ak_{\alpha^R}(s) = \zeta^\ak_{\gamma_-^R}(s)$
  which completes the proof.
\end{proof}

\paragraph{Disjoint unions and joins.}
Let $\Gamma_1$ and $\Gamma_2$ be simple graphs.
Recall from \S\ref{ss:graphs} that $\Gamma_1\oplus \Gamma_2$ and
$\Gamma_1\join \Gamma_2$ denote the disjoint union and join of $\Gamma_1$ and
$\Gamma_2$, respectively.
Clearly, $\GG_{\Gamma_1 \oplus \Gamma_2}$ and
$\GG_{\Gamma_1}\times\GG_{\Gamma_2}$ are isomorphic group schemes
whence
$\GG_{\Gamma_1 \oplus \Gamma_2}(R) \approx \GG_{\Gamma_1}(R)
\times\GG_{\Gamma_2}(R)$ for each ring $R$.
Denote the lower central series of a group $G$ by
$G = \LCS_1(G) \ge \LCS_2(G) \ge \dotsb$.
For groups $G_1$ and $G_2$, let
\begin{equation}
  \label{eq:freep_of_groups}
  G_1\freep G_2 := (G_1 * G_2) / \LCS_3(G_1 * G_2)
\end{equation}
be their {free class-$2$-nilpotent free product},
i.e.\ the maximal nilpotent quotient of class at most $2$ of the free product
$G_1 * G_2$.
Note that $\GG_{\Gamma_1\join \Gamma_2}(R) \approx \GG_{\Gamma_1}(R) \freep
\GG_{\Gamma_2}(R)$ if $R = \ZZ$ or, more generally, $R = \ZZ/k\ZZ$
for $k \in \ZZ$.
This isomorphism fails for general rings.
For example, let $\bullet$ be a simple graph with one vertex.
Then $H := \GG_{\bullet\join\bullet}(\ZZ) \approx \Uni_3(\ZZ)$ is the discrete
Heisenberg group (see Example~\ref{exa:heisenberg})
so that $\GG_{\bullet\join\bullet}(\ZZ^2) \approx H\times H$ has Hirsch
length $6$, where $\ZZ^2 = \ZZ\times \ZZ$ (product ring).
On the other hand, $\GG_\bullet(\ZZ^2) \approx \ZZ^2$ as groups and
$\ZZ^2\freep \ZZ^2$ has Hirsch length~$8$.

Recall from \S\ref{ss:intro/cographs} that cographs form the smallest class of
graphs which contains an isolated vertex and which is closed under taking
disjoint unions and joins.
We conclude that the class of cographical groups over $\ZZ$ is precisely the
smallest class of torsion-free finitely generated groups which contains $\Z$
and which is closed under taking both direct and free class-$2$-nilpotent products.

\section{Modules over toric rings and associated zeta functions}
\label{s:toric_modules}

By Corollary~\ref{cor:dual_coker_int}, 
the functions $\zeta_M(s)$ attached to modules $M$ over polynomial rings
generalise ask zeta functions.
In this section, we introduce a further generalisation of these functions by
replacing polynomial rings by more general toric rings.
This more general setting will provide us with a sufficient criterion
(Proposition~\ref{prop:combinatorial_module_uniformity}) for
proving uniformity results such as Theorem~\ref{thm:graph_uniformity}.
Part (\ref{thm:graph_uniformity1}) of the latter will be proved here
while parts~(\ref{thm:graph_uniformity2})--(\ref{thm:graph_uniformity2}) will
be proved in \S\ref{s:uniformity}.

Throughout, as before, $V$ is a finite set and $R$ is a ring.

\subsection{Cones and fans}
\label{ss:cones}

We recall some standard notions from convex and toric geometry;
see \cite[Ch.\ 1--3]{CLS11}.
Unless otherwise indicated, by a \emph{cone} in $\RR V$ we mean a closed,
rational, and polyhedral cone---in other words, cones are finite intersections
of $\ZZ$-defined linear half-spaces in $\RR V$. 

A \emph{fan} in $\RR V$ is a non-empty finite set $\cF$ consisting of
cones in $\RR V$ such that
\begin{enumerate}
\item every face of every cone in $\cF$ belongs to $\cF$ and
\item the intersection of any two cones in $\cF$ is a common face of both.
\end{enumerate}
The \emph{support} of a fan $\cF$ is $\card{\cF} = \bigcup\cF$.
The fan $\cF$ is \emph{complete} if $\card{\cF} = \RR V$.
Let $\cF$ and $\cG$ be two fans in $\RR V$.
We say that $\cG$ \emph{refines} $\cF$
if every cone in $\cG$ is contained in some cone in $\cF$.
The \emph{coarsest common refinement} of $\cF$ and $\cG$ is the fan (!)
$\cF \wedge \cG := \{ \sigma \cap \tau : \sigma \in \cF,
\tau\in\cG \}$;
its support is $\abs{\cF\wedge\cG} = \abs{\cF}\cap \abs{\cG}$.

Let $\dtimes$ be the standard inner product
$\bil x y = \sum\limits_{v\in V} x_vy_v$
on $\RR V$.
If $\sigma\subset \RR V$ is a cone, then so is its \emph{dual}
$\sigma^* = \bigl\{ x \in \RR V : \forall y\in \sigma, \bil x y \ge 0\bigr\}$.

Let $\sigma \subset\Orth V$ be a cone.
Recall that a \emph{preorder} on a set is a reflexive and transitive
relation. If all elements are comparable, then the preorder is \emph{total}.
We note that ``total preorders'' and ``weak orders'' (cf.\
\S\S\ref{ss:intro/hypergraphs}, \ref{ss:master}) are equivalent concepts.
We define a preorder $\le_\sigma$ on $\ZZ V$ by letting $x \le_\sigma y$ if and
only if $y-x \in \sigma^*$.

\begin{lemma}
  \label{lem:total_preorder_refinement}
  For every fan $\cF$ in $\RR V$ 
  and finite set $\Phi \subset \ZZ V$,
  there exists a refinement $\cF'$ of $\cF$ with $\abs{\cF'} = \abs{\cF}$ and
  such that $\le_\sigma$ induces a total preorder on $\Phi$ for each $\sigma \in \cF'$.
\end{lemma}
\begin{proof}
  We may assume that $\Phi \not= \emptyset$.
  For $x\in \RR V$, let $x^\pm := \{ y\in \RR V : \pm \bil x y \ge 0\}$ be the
  associated linear half-space and
  $x^{=} := x^+ \cap x^- = x^\perp$.
  We obtain a complete fan
  $\cF_x := \{ x^+,x^{-}, x^=\}$ consisting of precisely three cones,
  except when $x = 0$ in which case $\cF_x = \{ \RR V\}$.
  Clearly, the refinement $\cF' := \cF \wedge \bigwedge\limits_{x,y\in \Phi} \cF_{x-y}$ has
  the desired property.
\end{proof}

\subsection{Affine toric schemes and their rational points over \DVR{}s}
\label{ss:toric_points}

\paragraph{Toric rings and affine toric schemes.}
Let $\sigma \subset \RR V$ be a cone.  By
Gordan's lemma (see \cite[Proposition~1.2.17]{CLS11}), the additive monoid
$\sigma^*\cap \ZZ V$ is finitely generated.
Let $X = (X_v)_{v\in V}$ consist of algebraically independent variables
over $R$.  For $\alpha \in \ZZ V$, write $X^\alpha =
\prod\limits_{v\in V}X_v^{\alpha_v}$.  In the same way, we define
$x^\alpha$, where $x = \sum\limits_{v\in V} x_v v$ and all the $x_v$
are units (in some ambient ring).  We let $$R_\sigma := R[X^\alpha :
  \alpha \in \sigma^*\cap \ZZ V]$$ be the \emph{toric ring} associated
with $\sigma$ and $R$.  We let $\cX_{\sigma,R} = \Spec(R_\sigma)$ be
the associated \emph{affine toric scheme} over $R$; we write
$\cX_\sigma := \cX_{\sigma,\ZZ}$.

\paragraph{Rational points over \DVR{}s.}
Let $\sigma \subset \Orth V$ be a cone
and let $\fO$ be a \DVR{};
we do not insist that $\fO$ be compact here.
Recall that $\nu$ denotes the normalised valuation on $\fO$.
For $x = \sum\limits_{v\in V} x_vv \in \fO V$ with $\prod\limits_{v\in V}x_v \not= 0$,
we write $\nu(x) := \sum\limits_{v\in V} \nu(x_v) v\in \ZZ V$.
Define
\[
\sigma(\fO) := \Bigl\{ x \in \fO V : \prod\limits_{v\in V} x_v \not= 0 \text{ and } \nu(x) \in \sigma\Bigr\}.
\]

Alternatively, $\sigma(\fO)$ admits the following dual description.
\begin{lemma}
  \label{lem:cone_points}
  $\displaystyle \sigma(\fO) = \bigl\{ x \in \fO V : \prod\limits_{v\in V} x_v \not=
  0 \text{ and } x^\alpha \in \fO \text{ for each } \alpha\in
  \sigma^*\cap \ZZ V\bigr\}$.
\end{lemma}
\begin{proof}
  Let $x \in \fO V$ with $\prod\limits_{v\in V} x_v \not= 0$.
  Then $x^\alpha \in \fO$ if and only if $\nu(x^\alpha) =
  \bil{\nu(x)}\alpha \ge 0$.
  The latter condition holds for all $\alpha\in \sigma^*\cap
  \ZZ V$ if and only if $\nu(x) \in \sigma^{**} = \sigma$.
\end{proof}

Recall that $\cX_\sigma = \Spec(\ZZ_\sigma)$.
As before, write $X = (X_v)_{v \in V}$.

\begin{lemma}
  \label{lem:toric_points_from_cone}
  Let $\varphi$ be the natural map $\cX_\sigma(\fO) \to \fO V$
  induced by the inclusion $\sigma \subset \Orth V$.
  Let $Z = \{ x\in \fO V: \prod\limits_{v\in V} x_v = 0\}$ and let $Z'$ be the
  preimage of $Z$ under $\varphi$.
  Then $\varphi$ induces a bijection $\cX_\sigma(\fO)\setminus Z' \to
  \sigma(\fO)$.
\end{lemma}
\begin{proof}
  Let $\bm x\in \cX_\sigma(\fO)\setminus Z'$.
  Then $\bm x$ corresponds to a ring map $\ZZ_\sigma \xto\lambda \fO$.
  Since $\sigma \subset \Orth V$, we have $\Orth V = (\Orth V)^* \subset
  \sigma^*$.
  In particular, $\ZZ[X]\subset \ZZ_\sigma$ and $X_v\lambda$ is defined for
  $v\in V$.
  Let $x := \bm x \varphi$ so that $x_v = X_v \lambda$.
  Since $\bm x\not\in Z'$, we have $\prod\limits_{v\in V}x_v\not= 0$.
  Let $\alpha \in \sigma^* \cap \ZZ V$ be arbitrary.
  Then there exists $\beta \in \ZZ_{\ge 0} V$ with $\alpha + \beta \in
  \ZZ_{\ge 0} V$; note that $X^\beta \in \ZZ_\sigma$.
  We conclude that
  $x^{\alpha + \beta} = (X^{\alpha+\beta})\lambda
  = (X^\alpha)\lambda \dtimes (X^\beta)\lambda
  = (X^\alpha)\lambda \dtimes x^\beta$
  and therefore $(X^\alpha)\lambda = x^\alpha \in \fO\setminus\{0\}$.
  By Lemma~\ref{lem:cone_points}, $x = \bm x\varphi \in \sigma(\fO)$.
  We have thus shown
  that $\lambda$ (hence $\bm x$) is uniquely
  determined by $x$ which implies that
  $\varphi$ injectively maps $\cX_\sigma(\fO)\setminus Z'$ onto a
  subset of $\sigma(\fO)$.
  It remains to show that the latter subset is all of $\sigma(\fO)$.
  Indeed, for each $y\in \sigma(\fO)$, by Lemma~\ref{lem:cone_points},
  we obtain a ring map
  \[
    \ZZ_\sigma \to \fO, \quad X^\alpha \mapsto y^\alpha
  \]
  whose corresponding point $\bm y \in \cX_\sigma(\fO)$ does not
  belong to $Z'$ and satisfies $\bm y \varphi = y$.
\end{proof}

We henceforth tacitly embed $\sigma(\fO) \subset \cX_\sigma(\fO)$ via
Lemma~\ref{lem:toric_points_from_cone}.

\subsection{Zeta functions associated with modules over toric rings}
\label{ss:toric_module_zetas}

We now generalise the definition of the zeta functions $\zeta_M$ (see
\S\ref{ss:polynomial_module_zeta}) attached to modules over polynomial rings
to those over toric rings.

Let $\sigma \subset \Orth V$ be a cone.
Recall that $\cX_\sigma = \Spec(\ZZ_\sigma)$ is the affine toric
scheme
(over~$\ZZ$) associated with $\sigma$.
Let $\fO$ be a compact \DVR{} and let $M$ be a finitely generated $\fO_\sigma$-module.
Generalising the definition of $M_x$ in \S\ref{ss:polynomial_module_zeta}
(cf.\ Lemma~\ref{lem:natural_specialisation}),
for each $x\in \cX_\sigma(\fO)$ ($= \cX_{\sigma,\fO}(\fO)$), let $M_x$
denote the $\fO$-module $M\otimes_{\fO_\sigma} \fO$, where the
$\fO_\sigma$-module structure on $\fO$ is induced by the ring map
$\fO_\sigma \to \fO$ corresponding to $x$.
When $\sigma = \Orth V$, we recover the definition of $M_x$ given
in \S\ref{ss:polynomial_module_zeta}.
Recall that we identify $\sigma(\fO) \subset \fO V$ with a subset of
$\cX_\sigma(\fO)$ via Lemma~\ref{lem:toric_points_from_cone}.

\begin{defn}
  \label{d:zeta_toric_module}  
  Define a zeta function
  \begin{equation*}
    \label{eq:zeta_toric_module}
    \zeta_M(s) := \int\limits_{\sigma(\fO) \times\fO}
    \abs{y}^{s-1}
    \dtimes \card{M_x\otimes \fO/y}
    \,\,
    \dd\mu_{\fO V\times \fO}(x,y).
  \end{equation*}
\end{defn}

\begin{rem}
  \quad
  \begin{enumerate}
  \item If $M$ is an $\fO[X]$-module ($=\fO_{\Orth V}$-module), then
    we recover Definition~\ref{d:zeta_polynomial_module}.
  \item
    The function $\zeta_M(s)$ only depends on the isomorphism
    type of $M$ as an $\fO_\sigma$-module.
  \item 
    Exactly as in Remark~\ref{rem:two_point_modules},
    we may identify $M_x \otimes_{\fO}\fO/y = M\otimes_{\fO_\sigma}
    (\fO/y)_x$.
  \end{enumerate}
\end{rem}

\begin{lemma}
  \label{lem:inclusion_exclusion_over_fan}
  Let $o \subset \Orth V$ be a cone.
  Let $M_o$ be a finitely generated $\fO_o$-module.
  Let $\cF$ be a fan with $\card{\cF} = o$.
  For $\sigma\in \cF$, let $M_\sigma$ denote the $\fO_\sigma$-module $M_o\otimes_{\fO_o}\fO_\sigma$.
  Then
  \[
    \zeta_{M_o}(s) =
    \sum_{\emptyset\not= \Sigma \subset \cF}
    (-1)^{\card{\Sigma}+1}
    \zeta_{M_{\sigma}}(s),
    \]
    where we wrote $\sigma := \bigcap \Sigma$.
\end{lemma}
\begin{proof}
  This follows by combining the inclusion-exclusion principle and the
  identification $(M_\sigma)_x = (M_o)_x$ for $\sigma\in \cF$ and
  $x\in \sigma(\fO) \subset o(\fO)$ (``transitivity of base change'').
\end{proof}

\paragraph{Global setting.}
We now provide a global setting for the functions $\zeta_M$.
Let $R$ be a Noetherian ring, let $o \subset\Orth V$ be a cone,
and let $M_o$ be a finitely generated $R_o$-module.
For each ring map $R \xto\lambda \fO$, let
\[
  \zeta_{M_o,\lambda}(s) := \zeta_{M_o\otimes_{R_o} \fO_o}(s),
\]
where the ring map $R_o\xto{\lambda_o} \fO_o$ is induced by $\lambda$;
when the reference to $\lambda$ is clear,
we also write $\zeta_{M_o,\fO}$ in place of $\zeta_{M_o,\lambda}$ in the
following.

This global setting is compatible with
Lemma~\ref{lem:inclusion_exclusion_over_fan} in the sense that
for a cone $\sigma\subset o$, by transitivity of base change, we
may identify
$(M_o \otimes_{R_o}\fO_o) \otimes_{\fO_o}\fO_\sigma
= (M_o \otimes_{R_o}R_\sigma)\otimes_{R_\sigma} \fO_\sigma$. 

\subsection{Combinatorial and torically combinatorial modules}
\label{ss:combinatorial_modules}

\paragraph{Toric properties.}
Let $R$ be a ring.
Let $\cP$ be a property of $R_\sigma$-modules $M$ or, more generally,
of families $(M_1,\dotsc,M_r)$ of $R_\sigma$-modules such that $\cP$ is
defined as $\sigma$ ranges over all cones contained in $\RR V$.
(For $r = 1$, we will not distinguish between an $R_\sigma$-module $M$ and the
family~$(M)$.)
We assume that $\cP$ is closed under isomorphism in the sense that if
a family $(M_1,\dotsc,M_r)$ of $R_\sigma$-modules has property $\cP$
and $M_i\approx M_i'$ for $i=1,\dotsc,r$ and $R_\sigma$-modules
$M_1',\dotsc,M_r'$, then $(M_1',\dotsc,M_r')$ has property $\cP$.

We say that a family $(M_1,\dotsc,M_r)$ of $R_\sigma$-modules has the property
$\cP$ \emph{torically} if there exists a fan $\cF$ of cones in $\RR V$ with
$\card \cF = \sigma$ such that the family of $R_\tau$-modules
$(M_1\otimes_{R_\sigma}R_\tau,\dotsc,M_r\otimes_{R_\sigma}R_\tau)$ has
property $\cP$ for all $\tau\in\cF$.

For example, let $\cI$ be the property defined as follows:
a pair $(M_1,M_2)$ of $R_\sigma$-modules has $\cI$ if and only if $M_1\approx
M_2$.
We say that $R_\sigma$-modules $M_1$ and $M_2$ are \emph{torically isomorphic}
if the pair $(M_1,M_2)$ has property $\cI$ torically.
(This notion will feature crucially in \S\ref{s:models}.)

\paragraph{Combinatorial modules.}
Let $\sigma\subset \RR V$ be a cone.
Let $R$ be a ring.
By a \emph{monomial ideal} $I$ of $R_\sigma$, we mean an ideal generated by
(finitely many) Laurent monomials $X^\alpha$ for $\alpha \in \sigma^* \cap \ZZ V$.
We say that an $R_\sigma$-module is \emph{combinatorial} if it is isomorphic to 
$R_\sigma/I_1\oplus\dotsb\oplus R_\sigma/I_m$, where each
$I_j$ is a monomial ideal of $R_\sigma$.

\begin{ex}[Incidence modules are combinatorial]
  \label{ex:hypergraph_combinatorial}
  Let $\Eta$ be a hypergraph with vertex set $V$.
  By \eqref{eq:Inc_sum}, the incidence module $\Inc(\Eta;R)$ is a
  combinatorial $R[X]$-module, where $X = (X_v)_{v\in V}$.
\end{ex}

\begin{prop}[Uniformity of zeta functions of torically combinatorial modules]
  \label{prop:combinatorial_module_uniformity}
  \hspace*{.1em}
  
  \noindent Let $\sigma \subset \Orth V$ be a cone
  and let $M$ be a torically combinatorial $R_\sigma$-module.
  Then there exists
  $W(X,T) \in \QQ(X,T)$ such that $\zeta_{M,\lambda}(s) = W(q,q^{-s})$
  for each compact \DVR{} $\fO$
  and ring map $R \xto\lambda\fO$.
\end{prop}
\begin{proof}
  Fix $R \xto\lambda\fO$.
  First, 
  let $A\subset \sigma^*\cap \ZZ V$ be a finite set.
  Let $I = \langle X^\alpha:\alpha \in A\rangle \normal R_{\sigma}$ and $N =
  R_\sigma/I$.
  Let $x\in \sigma(\fO)$ and $y\in \fO\setminus\{0\}$.
  The evident free presentation $R_\sigma A \to R_\sigma \to N \to 0$
  yields, by base change, a presentation of the $\fO/y$-module
  \[
    (N\otimes_{R_\sigma}\fO_\sigma)_x\otimes_{\fO} \fO/y
    \approx N\otimes_{R_\sigma} (\fO/y)_x
    \approx
    \fO/\langle x^\alpha \,(\alpha \in A); \,y\rangle =: N_{x,y};
  \]
  cf.\ Proposition~\ref{prop:hypergraph_int}.
  In particular, $\card{N_{x,y}} = \norm{x^\alpha\,(\alpha\in A); \,y}^{-1}$,
  independently of $\lambda$.

  Next, by Lemma~\ref{lem:inclusion_exclusion_over_fan}, after shrinking
  $\sigma$ if necessary, we may assume that $M$ is in fact combinatorial
  instead of merely torically combinatorial,
  say $M \approx R_\sigma/I_1\oplus\dotsb \oplus R_\sigma/I_m$, where $I_j = \langle
  X^\alpha:\alpha\in A_j\rangle \normal R_\sigma$ and each $A_j\subset \sigma^*\cap \ZZ V$ is finite.
  We thus obtain
  \[
    \zeta_{M,\lambda}(s) =
    \int\limits_{\sigma(\fO) \times\fO}
    \frac{\abs{y}^{s-1}}
    {
      \prod\limits_{j=1}^m \norm{x^\alpha \,(\alpha \in A_j); \,y}
    }
    \,\,
    \dd\mu_{\fO V\times \fO}(x,y).
  \]
  The claimed uniformity result for $p$-adic integrals defined by such
  monomial expressions is well-known, see e.g.\ \cite[Proposition\ 3.9]{topzeta}.
\end{proof}

\begin{rem}
  \quad
  \begin{enumerate}
  \item
    If $R$ admits \itemph{any} ring map to any compact \DVR{}, then the
    rational function $W(X,T)$ in
    Proposition~\ref{prop:combinatorial_module_uniformity} is uniquely determined.
    Indeed, if $R \to \fO$ is such a ring map, where $\fO$ has residue field
    size $q$, then we obtain ring maps from $R$ to a compact \DVR{} with
    residue field size $q^f$ for each $f\ge 1$.
    Uniqueness of $W(X,T)$ then essentially boils down to the fact that
    infinite subsets of $\CC$ are Zariski dense.
  \item
    The preceding condition is satisfied, in particular,
    if $R$ is finitely generated over $\ZZ$.
    To see that, let $\fm$ be an arbitrary maximal ideal of $R$.
    By the Nullstellensatz for Jacobson rings (see e.g.\
    \cite[Theorem~4.19]{Eis95}), $R/\fm$ is then a finite field.
    We then e.g.\ obtain a ring map $R \to (R/\fm)\llb z\rrb$.
  \end{enumerate}
\end{rem}
  
\begin{proof}[Proof of Theorem~\ref{thm:graph_uniformity}(\ref{thm:graph_uniformity1})]
  Combine Example~\ref{ex:hypergraph_combinatorial} and
  Proposition~\ref{prop:combinatorial_module_uniformity}.
\end{proof}

In \S\ref{s:uniformity}, we will show that negative adjacency modules of
graphs are always torically  combinatorial and that their positive
counterparts are torically combinatorial over any ground ring in which $2$ is
invertible.
These facts will imply the remaining parts
(\ref{thm:graph_uniformity2})--(\ref{thm:graph_uniformity3}) of
Theorem~\ref{thm:graph_uniformity}.

\section{Ask zeta functions of hypergraphs}
\label{s:master}

Let $\Eta = (V,E,\abs\dtimes)$ be a hypergraph.  We write $n=\card V$ and
and $m=\card E$.  As explained in \S\ref{ss:incidence}, this allows us to
think of the incidence representation $\eta$ of $\Eta$ in
terms of a generic $n\times m$ matrix with support constrained by the
hyperedge support function $\abs\dtimes$.
In \S\ref{ss:master}, we derive an explicit combinatorial formula
for the rational function $W_\Eta(X,T)$ in
Theorem~\ref{thm:graph_uniformity}(\ref{thm:graph_uniformity1})
and thus, for each compact $\DVR$ $\fO$, for the ask zeta function
$\zeta_{\eta^\fO}^\ak(s)$.
We then consider two natural operations of hypergraphs: disjoint unions (see
\S\ref{subsec:hyp.dis.uni}) and complete unions (see \S\ref{subsec:freep} and
\S\ref{ss:graphs}).
As special cases, we derive explicit formulae for ask zeta functions of
hypergraphs with pairwise disjoint (resp.~codisjoint) hyperedge
supports in \S\ref{subsubsec:disjoint}
(resp.~\S\ref{subsubsec:codisjoint}).
In \S\ref{subsec:hyp.ops}, we describe the effects of four further fundamental
hypergraph operations on the rational functions $W_\Eta(X,T)$.
In \S\ref{subsec:hyp.ana}, we use our explicit formulae to deduce crucial
analytic properties of local and global ask zeta functions associated with
hypergraphs.

Throughout this section and beyond, we use the notation
$(\undl{\sd})= 1-q^{-\sd}$.

\subsection{An explicit formula for the ask zeta function of a hypergraph}\label{ss:master}

The main result of this section, Corollary~\ref{cor:master},
provides an explicit formula for the rational function $W_{\Eta(\bfmu)}(X,T)$
(see Theorem~\ref{thm:graph_uniformity}(\ref{thm:graph_uniformity1}))
associated with an arbitrary hypergraph $\Eta(\bfmu)$ given by a vector
$\bfmu$ of hyperedge multiplicities; see Definition~\ref{def:hyp.mu}.
This formula will, in particular, imply Theorem~\ref{thm:master.intro}.

\paragraph{Socles.}
For applications later on, it will prove advantageous to study the rational
function $W_{\Eta(\bfmu)}(X,T)$ in (what appears to be) a slightly more general setup.

\begin{defn}\label{def:hyp.socle}
  Given a $\sd$-element set $\sD$ with $V \cap \sD = \emptyset$ and a vector
  $\bfmu$ of hyperedge multiplicities as in Definition~\ref{def:hyp.mu},
  define a hypergraph $\Eta(\bfmu,\sD)$ with vertex set $\sD \sqcup V$ and
  vector of hyperedge multiplicities $(\nu_{J})_{J \subset \sD \sqcup V}$
  given by
  \[
    \nu_J =
    \begin{cases}
      \mu_I, & \text{if }J = \sD \sqcup I \text{ for some }I\subset V,\\
      0, & \text{otherwise.}
    \end{cases}
  \]
\end{defn}

We write $m=\sum_{I\subset V}\mu_I$ for the common number of hyperedges of
$\Eta(\bfmu)$ and $\Eta(\bfmu,\sD)$.

Informally speaking, the hypergraph $\Eta(\bfmu,\sD)$ arises from
$\Eta(\bfmu)$ by inflating each hyperedge by the same fixed set (``socle'') $\sD$.
Thus, if $A\in\Mat_{n\times m}(\Z)$ is an incidence matrix of $\Eta(\bfmu)$, then
$$\begin{bmatrix}\bfo_{\sd\times m}\\A\end{bmatrix} \in\Mat_{(\sd +
n)\times m}(\Z)$$ is an incidence matrix of $\Eta(\bfmu,\sD)$.

We now derive an explicit formula for $W_{\Eta(\bfmu,\sD)}(X,T)$;
the shape of this formula will often allow us to reduce to the case~$\sD =
\varnothing$.

\paragraph{Setup and strategy.}
From now on, let $\fO$ be an arbitrary compact \DVR{} with residue
field cardinality~$q$.
Without loss of generality, suppose that $0\not\in V \sqcup \sD$
and write $D_0 := D \sqcup\{0\}$.
Recall from \S\ref{ss:intro/notation} that
for a non-trivial $\fO$-module $M$, we write $M^\times = M\setminus\fP M$
and $\{0\}^\times = \{0\}$.
For $J \subset V$, define $p$-adic integrals
\begin{align} \Zeta_{J,\sD}(\bfs) :=
  \Zeta_{J,\sD}\Bigl(s_0,\,\left(s_I)_{I\subset J}\right)\Bigr)
  &:=
    \int\limits_{\fO J\times \fO \sD_0} \abs{y_0}^{s_0}\prod_{I\subset
    J}\norm{x_I;y}^{s_I} \,\dd\mu_{\fO J\times \fO\sD_0}(x,y)
    \quad\text{and}
    \nonumber \\
  \mcI_{J,\sD}(\bfs)
  &:=
    \int\limits_{(\fO J)^\times\times\fP\sD_0} \abs{y_0}^{s_0}\prod_{I\subset J}\norm{x_I;y}^{s_I}
    \,\dd\mu_{\fO J\times\sD_0}(x,y),\label{def:IJdelta}
\end{align}
where $x_I := (x_i : i\in I) \in \fO I\subset \fO J$.
Depending on context,
we regard $\Zeta_{J,\sD}(\bfs)$ and $\mcI_{J,\sD}(\bfs)$ both as functions of
the $1+2^{\card J}$ variables $s_0$ and $(s_I)_{I\subset J}$ and also as
functions of the $1 + 2^{\card V}$ variables $s_0$ and $(s_I)_{I \subset V}$;
in any case, $s_0$ and $s_{\varnothing}$ are different variables.

Let $\eta_{\bfmu,\sD}$ be the incidence representation of
$\Eta(\bfmu,\sD)$; see \S\ref{ss:incidence}.
We seek to determine $W_{\Eta(\bfmu,\sD)}(X,T)$ (and hence also
$W_{\Eta(\bfmu)}(X,T)$) by expressing $\zeta^{\ak}_{\eta_{\bfmu,\sD}^\fO}(s)$ as 
a rational function in $q$ and $q^{-s}$.
By Proposition~\ref{prop:hypergraph_int}, 
\begin{align}
  \zeta^\ak_{\eta_{\bfmu,\sD}^{\fO}}(s)
  &= (1-q^{-1})^{-1} \int\limits_{\fO V\times\fO \sD_0}
    \abs{y_0}^{s-(n+\sd)+m-1}\prod_{I\subset V}\norm{x_I;y}^{-\mu_I}
    \,\dd\mu_{\fO V \times\fO \sD_0}(x,y)
    \nonumber \\
  &= (1-q^{-1})^{-1} \, \Zeta_{ V,\sD}\Bigl(s-(n+\sd)+m-1,\,
    \left(-\mu_I\right)_{I\subset V}\Bigr).
    \label{eq:Zask=Zsubs}
\end{align}
This allows us to study $W_\Eta(X,T)$ by analysing the functions
$\Zeta_{V,\sD}(\bfs)$.

\paragraph{A recursive formula.}
Our first goal is to derive a recursive formula for $\Zeta_{V,\sD}(\bfs)$; see
Proposition~\ref{prop:master.rec}. We start with a rewrite.  In the following,
we write $t_I = q^{-s_I}$ (where $I \subset V$ or $I = 0$) and $t = q^{-s}$.
We further write $\gp{x} = x/(1-x)$ and $\gpzero{x} = 1/(1-x)$.

\begin{lemma}\label{lem:ZVD}
\begin{equation*}
 \Zeta_{ V,\sD}(\bfs) = \mcI_{ V,\sD}(\bfs) +
q^{-n-1-\sd}\,t_0\,\left(\prod_{I\subset V}t_I\right) \Zeta_{
V,\sD}(s) + (\undl{1})\left(1 + (\undl{\sd})\gp{q^{-1}t_0}\right).
\end{equation*}
\end{lemma}

\begin{proof}
  We identify $\fO \sD_0 = \fO \times \fO D$ and we partition
  $\fO V\times \fO \sD_0 = \fO V\times\fO\times\fO \sD$, the domain of
  integration of~$\Zeta_{V,\sD}(\bfs)$, as
\begin{equation}\label{eq:decomp}
  \left((\fO
  V)^\times \times \fP \times \fP\sD \right) \sqcup \left(\fP
  V \times \fP \times \fP\sD \right) \sqcup \left(\fO
  V\times \fO^\times \times \fO \sD\right) \sqcup \left(\fO
  V \times \fP\times(\fO \sD)^\times\right).
\end{equation}

We now study the integrals obtained by restricting the domain of integration 
of the integral defining $\Zeta_{V,\sD}(\bfs)$ to each of the four parts of
the partition \eqref{eq:decomp}.
For the first part, observe that
$$\int\limits_{(\fO V)^\times \times \fP \times \fP\sD
} \abs{y_0}^{s_0}\prod_{I\subset J}\norm{x_I;y}^{s_I} \,\dd\mu_{\fO V\times\fO
  \times \fO D}(x,y) = \mcI_{ V,\sD}(\bfs)$$ by the definition
of $\mcI_{ V,\sD}(\bfs)$ in~\eqref{def:IJdelta}.
For the second part, a straightforward change of variables
(cf.\ \cite[Proposition 7.4.1]{Igu00}) identifying
$\fP V \times \fP \times \fP\sD $ with $\fO V \times \fO \times \fO\sD$ yields
$$\int\limits_{\fP
  V \times \fP \times \fP\sD } \abs{y_0}^{s_0}\prod_{I\subset
  J}\norm{x_I;y}^{s_I} \,\dd\mu_{\fO V\times\fO \times \fO D}(x,y) =
q^{-n-1-\sd}\,t_0\,\left(\prod_{I\subset V}t_I\right) \Zeta_{ V,\sD}(s).$$
For the third part, we trivially have
$$\int\limits_{\fO
  V \times \fO^\times \times \fP\sD } \abs{y_0}^{s_0}\prod_{I\subset
  J}\norm{x_I;y}^{s_I} \,\dd\mu_{\fO V\times\fO \times \fO D}(x,y) =
(\undl{1}) = \mu_{_{\fO V\times\fO \times \fO D}}\!\left( \fO V \times
  \fO^\times \times \fO D\right).$$
Finally, for the fourth part, we find that
\begin{align*}
  \int\limits_{\fO V \times \fP \times
  (\fO\sD)^\times} \abs{y_0}^{s_0}\prod_{I\subset J}\norm{x_I;y}^{s_I}
  \,\dd\mu_{\fO V\times\fO \times \fO D}(x,y)
  &
    = \mu_{\fO D}((\fO D)^\times) \int\limits_{\fP} \abs{y_0}^{s_0}
    \dd\mu_{\fO}(y_0)
  \\
  & = (\undl{1}) (\undl{d}) \, \gp{q^{-1}t_0}.
\end{align*}
The lemma follows by adding the four preceding integrals.
\end{proof}
    
For $J\subset V$, let
\begin{equation}\label{eq:ZJd.vs.IJd}
  \mcZ_{J,\sD}(\bfs) := \frac{1-q^{-\sd-1}t_0}{1-q^{-1}t_0} +
  (\undl{1})^{-1} \mcI_{J,\sD}(\bfs);
\end{equation}
here, and in the sequel, we often abbreviate $\mcZ_{J}(\bfs) :=
\mcZ_{J,\varnothing}(\bfs)$. 
Thus, using Lemma~\ref{lem:ZVD},
\begin{equation}\label{eq:Z.vs.mcZ}
  \Zeta_{ V,\sD}(\bfs) =
  \operatorname{gp_0}\Bigl(q^{-n-1-\sd}t_0\prod_{I\subset V}t_I\Bigr)
  \, (\undl{1}) \,
  \mcZ_{ V,\sD}(\bfs)
\end{equation}
and hence, by combining \eqref{eq:Zask=Zsubs} and \eqref{eq:Z.vs.mcZ}, 
\begin{equation}\label{eq:zeta.Z}
  \zeta^\ak_{\eta_{\bfmu,\sD}^{\fO}}(s) = \frac{1}{1-t}
  \, \mcZ_{V,\sD}\Bigl(s-(n+\sd)+m-1, (-\mu_I)_{I\subset V}\Bigr);
\end{equation}
recall that $m=\sum_{I\subset V}\mu_I$.

The function $\mcZ_{V,\sD}(\bfs)$ admits the following recursive expression.
\begin{prop}\label{prop:master.rec}
  \begin{equation}\label{eq:rec}
    \mcZ_{ V,\sD}(\bfs) = \frac{1-q^{-\sd-1}t_0}{1-q^{-1}t_0} +
    \sum_{J\subsetneq  V} (\underline{1})^{n-\card J}
    \,
    \operatorname{gp}\Bigl(
    q^{-\sd-1-\card J} t_0 \prod_{I \subset J}t_I
    \Bigr)
    \,
    \mcZ_{J,\sD}(\bfs).
    \end{equation}
\end{prop}

\begin{proof}
  We decompose the first factor of the domain of integration of
  $\mcI_{ V,\sD}(\bfs)$ defined in~\eqref{def:IJdelta} according to
  precisely which entries of $x\in (\fO V)^\times$ are $\fP$-adic
  units; no such entry affects the integrand.  By Fubini's theorem, we
  may then split off the relevant copies of~$\fO^\times$, each of Haar
  measure $(\undl{1})$, and
  write \begin{equation}\label{rewrite} \mcI_{ V,\sD}(\bfs)
  = \sum_{J \subsetneq V}(\underline{1})^{n-\card
  J} \underbrace{\int\limits_{\fP
  J\times\fP \sD_0} \abs{y_0}^{s_0}\prod_{I \subset
  J}\norm{x_I;y}^{s_I}\dd\mu_{\fO
  J\times \fO \sD_0}(x,y)}_{=:\mcI_{J,\sD}^\circ(\bfs)}.  \end{equation}
  By identifying $\fP J\times\fP \sD_0$ and
  $\fO J\times\fO \sD_0$ via a change of variables,
  we obtain \[ \mcI_{J,\sD}^\circ(\bfs) =
  q^{-\sd-1-\card J} \, t_0\Bigl(\prod_{I \subset
  J}t_I\Bigr) \int\limits_{\fO
  J\times\fO \sD_0} \abs{y_0}^{s_0}\prod_{I \subset
  J}\norm{x_I;y}^{s_I} \, \dd\mu_{\fO J\times\fO\sD_0}(x,y).  \]

  Using a
  decomposition of $\fO J \times \fO D_0 = \fO J\times \fO \times \fO \sD$ analogous to
  \eqref{eq:decomp}, we obtain
  \[
    \frac{\mcI_{J,\sD}^\circ(\bfs)}{q^{-\sd-1-\card J}\,t_0
      \Bigl(\prod\limits_{I
        \subset J}t_I\Bigr)}
    = \mcI_{J,\sD}(\bfs) +
    \mcI_{J,\sD}^\circ(\bfs)
    +(\undl{1})\left(\frac{1-q^{-1-\sd}t_0}{1-q^{-1}t_0}\right)
\]
and
hence
\begin{align}\label{eq:rec.rel}
\mcI^\circ_{J,\sD}(\bfs) &= \gp{q^{-\sd-1-\card J}t_0\prod_{I
    \subset J}t_I}
\left((\underline{1})\left(\frac{1-q^{-1-\sd}\,t_0}{1-q^{-1}\,t_0}\right)
+ \mcI_{J,\sD}(\bfs)\right) \nonumber\\&= (\underline{1})
\,\gp{q^{-\sd-1-\card J}\,t_0\prod_{I \subset J}t_I} \mcZ_{J,\sD}(\bfs).
\end{align}

By combining~\eqref{eq:ZJd.vs.IJd} and
\eqref{rewrite}--\eqref{eq:rec.rel}, we finally obtain
\begin{align*}
  \mcZ_{ V,\sD}(\bfs)
  & =
    \frac{1-q^{-1-\sd}t_0}{1-q^{-1}t_0} + (\underline{1})^{-1}
    \mcI_{V,\sD}(\bfs)
  \\
  & =
    \frac{1-q^{-1-\sd}t_0}{1-q^{-1}t_0}
    +
    \sum_{J\subsetneq  V} (\underline{1})^{n-\card J-1}
    \, \mcI^\circ_{J,\sD}(\bfs)
  \\
  & =
    \frac{1-q^{-1-\sd}t_0}{1-q^{-1}t_0}
    +
    \sum_{J\subsetneq  V}
    (\underline{1})^{n-\card J}
    \,
    \operatorname{gp}\Bigl(q^{-\sd-1-\card J}\,
    t_0
    \prod_{I \subset J} t_I
    \Bigr)
    \,
    \mcZ_{J,\sD}(\bfs).
    \qedhere
\end{align*}
\end{proof}

\paragraph{An explicit formula in terms of weak orders.}
Our next goal is to translate the recursive formula in
Proposition~\ref{prop:master.rec} into an explicit form given by
a sum over a suitable combinatorial object.

\begin{definition}\label{def:weak.order}
  Let $\WOhat(V)$ be the poset of flags of subsets of
  $V$.
  That is, $\WOhat(V)$ consists of elements of the form
  \[
    y = \Bigl(I_1 \subsetneq I_2 \subsetneq \dots \subsetneq I_{\ell}\Bigr),
  \]
  where $\ell \ge 0$ and $I_i \subset V$ for $i = 1,\dotsc,\ell$.
  Note that we allow both $I_1 = \emptyset$ and $I_\ell = V$ but do not require
  either condition to be satisfied.
  We define the \emph{rank} of $y\in \WOhat(V)$ to be
  \[
    \rk(y) = \abs{\sup(y)} =\sup\{\abs I : I \in y\} \in \NN_0;
  \]
  empty flags have rank $0$.
  We denote by $\WO(V)$ the subposet of $\WOhat(V)$ consisting of all flags 
  of \itemph{non-empty} subsets of $V$ only.
  We often write $\WOhat_n$ and $\WO_n$ instead of
  $\WOhat([n])$ and $\WO([n])$, respectively.
\end{definition}

\begin{remark}\label{rem:fubini}\quad
  \begin{enumerate}
  \item Clearly, $\rk(y)=0$
    if and only if $y$ is either the empty flag or the singleton flag $(\emptyset)$.
    At the other extreme, $\rk(y) = \card V = n$ if and only if $V\in y$.
    The latter condition is satisfied for precisely half of the elements of
    $\WO(V)$.
    The fact that $V\in y$ is permitted for elements $y$ of $\WO(V)$
    marks the difference between the latter and the poset $\textup{WO}_n$ of
    weak orders of $n$ objects; cf.\ e.g.\ \cite[\S 2.3]{SV1/15}.
    In particular, 
    $\frac{1}{2}\card{\WOhat(V)} = \card{\WO(V)} = 2 \card{\textup{WO}_n} =
    2f_n$,
    where $f_n$ denotes the $n$th Fubini number as in
    \S\ref{ss:intro/hypergraphs};
    cf.\ \cite[\href{http://oeis.org/A000670}{A000670}]{OEIS_fubini}, \cite[p.~228]{Comtet/74}, and \eqref{eq:fubini_growth}.
  \item
    The poset $\WO(V)$ is isomorphic to $\textup{WO}_{\bfo^{(n)}}$ in \cite[\S 3.1]{CSV/19}.
   \end{enumerate}
\end{remark}

We obtain the following explicit formula for $\mcZ_{V,\sD}(\bfs)$.
\begin{theorem}\label{thm:master}
  \[
    \mcZ_{ V,\sD}\left(\bfs \right)
    = \frac{1-q^{-\sd-1}t_0}{1-q^{-1}t_0}
    \sum_{y\in \WO(V)} (\underline{1})^{\rk(y)}
    \prod_{J\in y}
    \operatorname{gp}\Bigl(q^{-1-n+\card J-\sd}t_0\prod_{I \subset V
      \setminus J}t_I\Bigr).
  \]
\end{theorem}

\begin{proof}
  Recursively apply Proposition~\ref{prop:master.rec} to the 
  terms $\mcZ_{J,\sD}(\bfs)$ on the right-hand side of~\eqref{eq:rec}.
\end{proof}

In particular, using \eqref{eq:zeta.Z}, we obtain  the following explicit
formulae for the rational function $W_{\Eta(\bfmu,\sD)}(X,T)$ associated with
the hypergraph $\Eta(\bfmu,\sD)$.

\begin{corollary}\label{cor:master}
  \begin{align}\label{eq:master.rat}
    \lefteqn{W_{\Eta(\bfmu,\sD)}(X,T)} \nonumber\\
    &= 
      \frac{1-X^{n-m}T}{(1-X^{\sd+n-m}T)(1-T)}\sum_{y\in \WO(V)}
      (1-X^{-1})^{\rk(y)}
      \prod_{J\in y}
      \operatorname{gp}\Biggl(X^{\card J-\sum_{I \cap J \neq
      \varnothing}\mu_I}T\Biggr)
    \\ &=
         \frac{1-X^{n-m}T}{1-X^{\sd+n-m}T}\sum_{y\in \WOhat(V)}
         (1-X^{-1})^{\rk(y)}
         \prod_{J\in y}
         \operatorname{gp}\Biggl(X^{\card J-\sum_{I \cap J \neq
         \varnothing}\mu_I}T\Biggr)
         \label{eq:master.WOhat}
    \\ &=
         \frac{1-X^{n-m}T}{1-X^{\sd+n-m}T}
         \,
         W_{\Eta(\bfmu)}(X,T).\nonumber
         \pushQED{\qed}
         \qedhere
         \popQED
  \end{align}
\end{corollary}

\begin{proof}[Proof of Theorem~\ref{thm:master.intro}]
  Apply Corollary~\ref{cor:master} with $D = \emptyset$.
\end{proof}

\begin{remark}\quad
  \begin{enumerate}
  \item For $n=0$, we recover the formula for the ask zeta function
    associated with the block hypergraph $\BE_{d,m}$
    in \eqref{def:block.hyp}; see
    Example~\ref{exa:staircase}\eqref{exa:BEnm}.
  \item The rational functions in \eqref{eq:master.rat} and
    \eqref{eq:master.WOhat} are reminiscent of the ``generalised Igusa
    function'' $I^{\textup{wo}}_{\bfo^{(n)}}(\bfX)$ associated with the 
    all-one-vector $\bfo^{(n)}=(1,\dots,1)$ in
    \cite[Definition~3.5]{CSV/19}.
    The curious factors $(1-X^{-1})^{\rk(y)}$, however,
    set these two types of combinatorially defined functions apart.
   \end{enumerate}
\end{remark} 

\begin{example}
  We write out the formulae for the functions $W_{\Eta(\bfmu)}(X,T) =
  W_{\Eta(\bfmu,\varnothing)}(X,T)$ given in~\eqref{eq:master.WOhat} for
  $n\in\{2,3\}$.
  We identify $V = [n]$ and set, for
  $J \subset [n]$,
  $\gps{J}
  := \gps{J}(\bfmu)
  := \operatorname{gp}\Bigl(
  X^{\card J - \sum_{I \cap J \neq \varnothing} \mu_I} T \Bigr)$.
  Write $\gps{i} := \gps{\{i\}}$ and similarly $\gps{ij\dotsb} := \gp{\{i,j, \dotsc\}}$.
  \begin{enumerate}
    \item ($n=2$) The ranks of the six flags in $\WO_2$ are given as
      follows.
      \begin{center}
  \begin{tabular}{r|cccccc}
    $y\in \WO_2$ & $(\,)$ & $(\{1,2\})$ & $(\{1\} \subsetneq
    \{1,2\})$ & $(\{2\} \subsetneq \{1,2\})$ & $(\{ 1\})$ & $(\{2\})$
    \\\hline $\rk(y)$&0&2&2&2&1&1
  \end{tabular}
 \end{center}
Thus
      \begin{multline}\label{eq:master.2}
W_{\Eta(\bfmu)}(X,T)   = \\\frac{1}{1-T}\left(1 +
  (1-X^{-1})^2 \gps{12} \left( 1 + \gps{1}+\gps{2}\right)+
  (1-X^{-1})\left(\gps{1}+\gps{2}\right) \right),
\end{multline}
where the relevant substitutions are given by the numerical data
\[
  X^{\card J-\sum_{I \cap J \neq \varnothing}\mu_I}T = \begin{cases}
  X^{2-\mu_1-\mu_2-\mu_{12}}T, & \textup{ for }J=\{1,2\},
  \\ X^{1-\mu_2-\mu_{12}}T, & \textup{ for
  }J=\{1\},\\ X^{1-\mu_1-\mu_{12}}T,& \textup{ for
  }J=\{2\}. \end{cases}
\]
\item ($n=3$) Here, $\card{\WO_3}=26$ and

\begin{dmath*}
  W_{\Eta(\bfmu)}(X,T) = \frac{1}{1-T}\left( 1 +
  (1-X^{-1})^3\gps{123}\left( 1 + \gps{1} + \gps{2} +\gps{3}
  +\gps{12}\left( 1 + \gps{1} + \gps{2}\right) + \gps{13}\left( 1 +
  \gps{1} + \gps{3}\right)+\gps{23}\left( 1 + \gps{2} +
  \gps{3}\right)\right) +(1-X^{-1})^2\left(\gps{12}\left( 1 + \gps{1}
  + \gps{2}\right) + \gps{13}\left( 1 + \gps{1} +
  \gps{3}\right)+\gps{23}\left( 1 + \gps{2} + \gps{3}\right)\right)+
  (1-X^{-1})\left(\gps{1} +\gps{2} +\gps{3}\right)\right);
\end{dmath*}
we omit the lengthy substitutions.
  \end{enumerate}
\end{example}

It seems remarkable how slight the dependence of $W_{\Eta(\bfmu,\sD)}(X,T)$
on the ``socle'' $\sD$~is.
The final equality in Corollary~\ref{cor:master} often allows us to reduce
to the case $\sD = \varnothing$ or, equivalently, to assume that no vertex of
our hypergraph is incident to every hyperedge.

\subsubsection{A special case: staircase hypergraphs}\label{subsec:staircase}

Let $\bm m = (m_0,\dotsc,m_n)\in\N_0^{n+1}$ and
write $m = m_0 + \dotsb + m_n$. 
Recall the definition of the staircase hypergraph $\Stair_{\bm m}$ from
\eqref{def:staircase.graph}.
 The upper block-triagonal ``staircase
matrix''
\[
  M_{\bm m} = 
  \left[ \delta_{j>\sum\limits_{\iota<i}m_\iota}\right]_{\substack{i=1,\dotsc,n\\j=1,\dotsc,m}}
  \in\Mat_{n\times m}(\Z)
\]
is the incidence matrix of $\Stair_{\bfm}$ with respect to the
natural order on $[n] = \verts(\Stair_{\bm m})$ and the lexicographic order
on $\edges(\Stair_{\bm m})$ (where the first components are ordered by
inclusion); see Example~\ref{ex:3111.3} for an illustration.

The rational function $W_{\Stair_{\bfm}}(X,T)$ associated 
with $\Stair_{\bfm}$ admits the following concise description.

\begin{proposition}
  \label{prop:staircase}
  \begin{equation}\label{equ:staircase}
    W_{\Stair_{\bm m}}(X,T) = \frac{1}{1-T}
    \prod_{j=0}^{n-1} \frac{1-X^{-1+n-j-\sum\limits_{\iota>j}
        m_\iota}T}{1-X^{n-j-\sum\limits_{\iota>j} m_\iota}T}.
  \end{equation}
\end{proposition}

\begin{proof}
  Combine Proposition~\ref{prop:hypergraph_int} and \cite[Lemma~5.8]{ask}.
\end{proof}

Proposition~\ref{prop:staircase} generalises several previously known results.

\begin{example}\label{exa:staircase}\quad
  \begin{enumerate}
  \item\label{exa:BEnm} If $m_0 = \dotso = m_{n-1} = 0$ and $m_n = m$, then
    $\Stair_{\bfm} = \BE_{n,m}$.  Proposition~\ref{prop:staircase}
    yields, in accordance with \cite[Proposition~1.5]{ask},
    \begin{align*}
      W_{\Stair_{\bm m}}(X,T) = W_{\BE_{n,m}}(X,T)
      & = \frac{1}{1-T} \prod_{j=0}^{n-1}\frac{1-X^{-1+n-j-m}T}{1-X^{n-j-m}T}
      \\
      & = \frac{1-X^{-m}T}{(1-T)(1-X^{n-m}T)}.
    \end{align*}
  \item
    \label{exa:trn}
    If $m_0 = 0$ and $m_1=\dotso = m_n = 1$, then
  $M_{\bm m}=[\delta_{i\leq
    j}]\in\Mat_n(\Z)$. Proposition~\ref{prop:staircase} yields, in
  accordance with \cite[Proposition~5.13(ii)]{ask},
  $$W_{\Stair_{\bm m}}(X,T) = \frac{1}{1-T}
  \prod_{j=0}^{n-1} \frac{1-X^{-1+n-j-(n-j)}T}{1-X^{n-j-(n-j)}T} =
  \frac{(1-X^{-1}T)^n}{(1-T)^{n+1}}.$$
  \end{enumerate}
\end{example}

\subsection{Ask zeta functions of disjoint unions of
  hypergraphs}\label{subsec:hyp.dis.uni}

In this section, we consider ask zeta functions associated with disjoint
unions of hypergraphs.
As our main result, in \S\ref{subsubsec:disjoint}, we record an explicit
formula for ask zeta functions attached to hypergraphs with pairwise disjoint
(hyperedge) supports.

\paragraph{Hadamard products.}
Recall that the \emph{Hadamard product}
of two generating functions $F(T) = \sum\limits_{k=0}^\infty a_kT^k$
and $G(T) = \sum\limits_{k=0}^\infty b_k T^k$ with coefficients in some common
field is
\[
  F(T) \hada G(T) := \sum_{k=0}^\infty a_kb_kT^k.
\]

If $F(T)$ and $G(T)$ are both rational, then so is $F(T) \hada G(T)$;
see \cite[Proposition~4.2.5]{Sta12}.
We will use the symbol $\hada$ to denote Hadamard products of rational
generating functions in a variable $T$, with coefficients in the field
$\QQ(X)$, such as the bivariate rational functions~$W_\Eta(X,T)$.

\paragraph{Disjoint unions.}
Let $\Eta_1,\dotsc,\Eta_\nb$ be hypergraphs
with pairwise disjoint vertex sets $V_1,\dotsc,V_\nb$.
Let $\Eta := \Eta_1\oplus \dotsb \oplus \Eta_\nb$ be the disjoint union of
$\Eta_1,\dotsc,\Eta_\nb$  as in \S\ref{ss:graphs}.

\begin{prop}
  \label{prop:W_Hadamard}
  $\displaystyle W_{\Eta_1\oplus\dotsb\oplus\Eta_\nb}(X,T) =  W_{\Eta_1}(X,T) \hada
  \dotsb \hada W_{\Eta_\nb}(X,T)$.
\end{prop}
\begin{proof}
  Let $\eta_i$ and $\eta$ be the incidence representations of $\Eta_i$ and
  $\Eta$, respectively.
  We may identify $\eta = \eta_1\oplus\dotsb\oplus \eta_\nb$ (see~\S\ref{ss:mreps}).
  Now apply \cite[Lemma~3.1]{ask2}; cf.\ \cite[Corollary~3.6]{ask}.
\end{proof}

We conclude that the set of rational functions $W_\Eta(X,T)$ associated with
hypergraphs (or, equivalently, the class of rational functions given by the
right-hand side of \eqref{eq:master.rat}) is closed under taking Hadamard
products.

It is natural to seek to exploit this closure property.
Let $\Eta_i \approx \Eta(\bfmu^{(i)})$ for a vector $\bfmu^{(i)}$ of
hyperedge multiplicities as in Definition~\ref{def:hyp.mu}.
By combining Corollary~\ref{cor:master} and Proposition~\ref{prop:W_Hadamard},
we obtain
\begin{equation}
  \label{eq:hadamard_of_master}
  W_{\Eta_1\oplus\dotsb\oplus\Eta_\nb}(X,T)
  = \bighada_{i=1}^\nb \sum_{y_i\in \WOhat(V_i)}
  (1-X^{-1})^{\rk(y_i)}
  \prod_{J\in y_i}
  \operatorname{gp}
  \Bigl(
  X^{\card{J} - \sum_{I \cap J \neq \varnothing} \mu^{(i)}_I} T
  \Bigr).
\end{equation}

The right-hand side of \eqref{eq:hadamard_of_master} falls short of being
truly explicit due to the rather mysterious nature of Hadamard products.
On the other hand, Corollary~\ref{cor:master} provides an explicit formula for
$W_{\Eta}(X,T)$ in terms of the hyperedge
multiplicity vector $\bfmu \in \NN_0^{\Pow(V)}$, where
$V := V_1\sqcup \dotsb \sqcup V_\nb$ and
$\mu_I := \sum_{i=1}^\nb \delta_{I\subset V_i}\, \mu^{(i)}_{I \cap V_i}$
for $I\subset V$.
This approach, however, takes no advantage of the fact that $\Eta$ is a
disjoint union.
As we will now see, it turns out that we can do much better at least when each
$\Eta_i$ is a block hypergraph as in \eqref{def:block.hyp}.

\subsubsection{A special case: hypergraphs with disjoint supports}
\label{subsubsec:disjoint} 

Let $\bfn = (n_1,\dots,n_\nb)\in\N^{\nb}$ and
$\bfm=(m_1,\dots,m_\nb)\in\N^{\nb}$;
write $n = n_1 + \dotsb + n_\nb$ and $m = m_1 + \dotsb + m_\nb$.
Let $\Eta := \BE_{\bfn,\bfm}$ be
the disjoint union of the block hypergraphs $\Eta_i := \BE_{n_i,m_i}$;
see \eqref{def:block.hyp.multi}.
Let $V_i$ be the set of vertices of $\Eta_i$ and $V = V_1 \sqcup \dotsb
\sqcup V_\nb$ be that of $\Eta$.
Note that $\bfo_{n_i\times m_i}\in\Mat_{n_i\times m_i}(\Z)$ is the (unique!) incidence
matrix of $\Eta_i$.
It follows that
\begin{equation}\label{eq:inc.mat.dis}
  \diag\left(\bfo_{n_1\times m_1},\dots,\bfo_{n_\nb\times m_\nb}\right) \in
  \Mat_{n \times m}(\Z)
\end{equation}
is an incidence matrix of $\BE_{\bfn,\bfm}$.
Note that, up to reordering of rows and columns where necessary, this
is the general form of incidence matrices of hypergraphs with disjoint
supports (i.e.~whenever $\mu_I\mu_J \neq0$,
then $I=J$ or $I\cap J=\varnothing$) and which also
satisfy~$\bigcup_{\mu_I > 0} I= V$.
(We will see that the latter condition imposes no real restrictions, nor would
allowing some $m_i = 0$ offer anything new; see Remark~\ref{rem:hyp.red}.)

In Corollary~\ref{cor:master}, we obtained an expression for $W_{\Eta}(X,T)$ as
a sum over $\WOhat(V) \approx \WOhat_n$.
Our main result (Corollary~\ref{cor:master.dis.WO}) of this section
provides an expression for $W_\Eta(X,T)$ as a sum over $\WOhat_\nb$.
Apart from better reflecting the structure of the hypergraph $\Eta$,
in light of the rapid growth of Fubini numbers in~\eqref{eq:fubini_growth},
our formula has a more favourable complexity if $r \ll n$;
see also Remark~\ref{rem:three_complexities}.

\paragraph{Auxiliary functions.}
We consider the specialisation
\begin{equation}
  \mcZ^{\oplus}_{\bfn}(\bfs)
  := 
  \mcZ^{\oplus}_{\bfn}
  \bigl(
  s_0; \, s_{V_1},\dotsc,s_{V_\nb}
  \bigr)
  := \mcZ_{V}\Bigl(s_0, \Bigl(\delta_{\exists i \in [r]\colon I=V_i}\,s_{I}\Bigr)_{I\subset V}\Bigr)
  \label{eq:spec.dis}
\end{equation}
of the function $\mcZ_{V}(\bfs) =\mcZ_{V,\varnothing}(\bfs)$ from
\eqref{eq:ZJd.vs.IJd}.
In other words, $\mcZ^{\oplus}_{\bfn}(\bfs)$ is obtained from $\mcZ_V(\bfs)$
by setting all variables $s_I$ corresponding to subsets  $I\subset V$ to zero,
except for those subsets equal to one of the pairwise disjoint sets~$V_i$.
From \eqref{eq:zeta.Z} (with $\sD = \varnothing$), we obtain
\[
  \zeta^\ak_{\eta_1^\fO \oplus \dotsb \oplus \eta_\nb^\fO}(s)
  =
  \frac{1}{1-t} \,
  \mcZ^{\oplus}_{\bfn}\left(s-n+m-1; -m_1,\dotsc,-m_\nb\right).
\]
Given $J\subset [\nb]$, write $n_J=\sum_{j\in J}n_j$ and
$\bfn_{J}=(n_j)_{j\in J}$.
Generalising~\eqref{eq:spec.dis}, we define
\[
  \mcZ^{\oplus}_{\bfn_J}(\bfs)
  :=
  \mcZ_{V}\left(s_0, \left(\delta_{\exists j\in J\colon I=V_j} \, s_{I}\right)_{I\subset V}\right).
\]

\paragraph{A recursive formula.}
We obtain the following recursive formula for $\mcZ^\oplus_{\bfn}(\bfs)$;
the proof is similar to that of Proposition~\ref{prop:master.rec} and hence
omitted.
\begin{prop}\label{prop:master.dis.rec}
  \[
    \mcZ^{\oplus}_{\bfn}(\bfs)
    =
    1
    +
    \sum_{J \subsetneq [r]}
    \Bigl(
      \prod_{k \not\in J}
      (\undl{n_k})
    \Bigr)
    \,
    \operatorname{gp}
    \Biggl(
    q^{-1-n_J} t_0
    \prod_{j\in J} t_{V_j}
    \Biggr)
    \,
    \mcZ^{\oplus}_{\bfn_J}(\bfs).
    \pushQED{\qed}
    \qedhere
    \popQED
  \]
\end{prop}

\paragraph{An explicit formula.}
Just as Proposition~\ref{prop:master.rec} implies
Theorem~\ref{thm:master}, we obtain the following by unravelling the
recursive formula in Proposition~\ref{prop:master.dis.rec}.

\begin{theorem}\label{thm:master.dis}
  \[
    \mcZ_{\bfn}^{\oplus}(\bfs)
    =
    \sum_{y \in \WO_r}
    \Bigl(
    \prod_{i\in\sup(y)}
    (\undl{n_i})
    \Bigr)
    \prod_{J\in y}
    \operatorname{gp}
    \Biggl(q^{-1-n+n_J}t_0\prod_{j\in [\nb]\setminus J}t_{V_j}\Biggr).
    \pushQED{\qed}
    \qedhere
    \popQED
  \]
\end{theorem}

In particular, Theorem~\ref{thm:master.dis} allows us to produce the following
explicit formula for the rational function   $W_{\BE_{\bfn,\bfm}}(X,T)$
associated with the disjoint union $\BE_{\bfn,\bfm}=\bigoplus\limits_{i=1}^\nb
\BE_{n_i,m_i}$ .

\begin{corollary}\label{cor:master.dis.WO}
  \begin{equation}\label{eq:master.dis.rat}
    W_{\BE_{\bfn,\bfm}}(X,T)
    =
    \sum_{y \in \WOhat_r}
    \Bigl(
    \prod_{i\in \sup(y)}
    \bigl(1-X^{-n_i}\bigr)
    \Bigr)
    \prod_{J\in y}
    \operatorname{gp}\Biggl(
    X^{\sum\limits_{j \in J} n_j-m_j} T
    \Biggr).
    \qed
    \end{equation}
\end{corollary}

\begin{example}\label{exa:r=2}
  For $\nb=2$, formula \eqref{eq:master.dis.rat} for the ask zeta
  function associated with the disjoint union of two block hypergraphs
  $\BE_{n_i,m_i}$ reads
  \begin{multline}W_{\BE_{(n_1,n_2),(m_1,m_2)}}(X,T) =\\
    \frac{1}{1-T}\left( 1 + (1-X^{-n_1}) \gp{X^{n_1-m_1}T} +
    (1-X^{-n_2})\gp{X^{n_2-m_2}T} +
    \right.\\ \left. (1-X^{-n_1})(1-X^{-n_2}) \gp{X^{n_1+n_2 -
        m_1-m_2}T} \left( 1 + \gp{X^{n_1-m_1}T}+ \gp{X^{n_2-m_2}T}
    \right)\right).\label{equ:dis.r=2}
\end{multline}
It is instructive to compare \eqref{equ:dis.r=2} and the general formula
\eqref{eq:master.2} for $W_\Eta(X,T)$ in the special case that $\Eta$ is a
hypergraph on two vertices.
\end{example}

\subsection{Ask zeta functions of complete unions of hypergraphs}
\label{subsec:freep}

Let $\Eta_1$ and $\Eta_2$ be hypergraphs on disjoint sets $V_1$ and $V_2$ of
vertices.
Recall from \S\ref{ss:graphs} the definition of the complete union
$\Eta_1\freep\Eta_2$ of $\Eta_1$ and $\Eta_2$,
a hypergraph with vertex set $V_1\sqcup V_2$.
In the main result of this section, Corollary~\ref{cor:master.freep},
we express $W_{\Eta_1\freep  \Eta_2}(X,T)$ in terms of $W_{\Eta_1}(X,T)$ and
$W_{\Eta_2}(X,T)$.
In \S\ref{subsubsec:codisjoint},
we also record an explicit formula for the rational function $W_\Eta(X,T)$
whenever $\Eta$ has pairwise codisjoint hyperedge supports;
such hypergraphs are precisely the reflections (see \S\ref{ss:graphs}) of
those considered in \S\ref{subsubsec:disjoint}.

Let $\Eta_i$ have $n_i$ vertices and $m_i$ hyperedges;
write $n = n_1 + n_2$, and $m = m_1 + m_2$.
Let $\Eta_1$, $\Eta_2$, and $\Eta := \Eta_1 \freep \Eta_2$ be given by the
multiplicity vectors $\bfmu^{(1)}$, $\bfmu^{(2)}$, and $\bfmu$, respectively;
cf.\ Definition~\ref{def:hyp.mu}.
For $I\subset V := V_1\sqcup V_2$, let $I_i := I \cap V_i$
so that
\[
  \mu_{I}
  =
  \delta_{I_2 = V_2}\, \mu_{I_1}^{(1)}
  \,+\,
  \delta_{I_1 = V_1}\, \mu_{I_2}^{(2)}.
\]

\paragraph{An auxiliary function.}
Recall the definition of $\mcZ_V(\bfs) = \mcZ_{V,\emptyset}(\bfs)$ from
\eqref{eq:ZJd.vs.IJd} and consider the specialisation
\begin{equation*}
  \mcZ_{(V_1,V_2)}^\freep(\bfs)
  :=
  \mcZ_{(V_1,V_2)}^\freep\bigl(s_0, \, \bfs^{(1)}, \bfs^{(2)}\bigr)
  :=
  \mcZ_{V}\Bigl(
  s_0,\,
  \Bigl(
  \delta_{I_2 = V_2} \, s^{(1)}_{I_1} 
  \,+\,
  \delta_{I_1 = V_1} \, s^{(2)}_{I_2} 
  \Bigr)_{I\subset V}\Bigr).
\end{equation*}

In other words,
$\mcZ_{(V_1,V_2)}^\freep(\bfs)$ is obtained from $\mcZ_V(\bfs)$ 
by setting all variables $s_I$ for $I \subset V$ to zero, except
for those $I$ that \itemph{contain} one of the disjoint sets $V_1$ and $V_2$;
note that the variable $s_{V}$ is substituted by
$s^{(1)}_{V_1}+s^{(2)}_{V_2}$.
Thus, $\mcZ_{(V_1,V_2)}^\freep(\bfs)$ is a function of $1+2^{n_1}+2^{n_2}$
complex variables $s_0$, $\bfs^{(1)} = \Bigl(s^{(1)}_{I_1}\Bigr)_{I_1\subset V_1}$, and
$\bfs^{(2)} = \Bigl(s^{(2)}_{I_2}\Bigr)_{I_2\subset V_2}$.
In particular, $s_0$, $s^{(1)}_\varnothing$, and $s^{(2)}_\varnothing$ are
three distinct variables.

Let $\eta_1$, $\eta_2$, and $\eta_1\freep \eta_2$ be the incidence
representations of $\Eta_1$, $\Eta_2$, and $\Eta_1\freep \Eta_2$, respectively.
From \eqref{eq:zeta.Z} (with $\sD=\varnothing$), we obtain
\[
  \zeta^\ak_{(\eta_1 \freep \eta_2)^\fO}(s)
  =
  \frac{1}{1-t}\,
  \mcZ_{(V_1,V_2)}^\freep
  \Bigl(
  s-n+m-1, \,
  \Bigl(-\mu^{(1)}_{I_1}\Bigr)_{I_1\subset V_1}, \,
  \Bigl(-\mu^{(2)}_{I_2}\Bigr)_{I_2\subset V_2}
  \Bigr).
  \]

\paragraph{Recursive formulae.}
This identity allows us to relate the rational functions $W_\Eta(X,T)$,
$W_{\Eta_1}(X,T)$, and $W_{\Eta_2}(X,T)$.
We first express 
$\mcZ^\freep_{(n_1,n_2)}(\bfs)$ in terms of (translates of) the
functions $\mcZ_{V_i}(s_0,\bfs^{(i)})$;
cf.\ \eqref{eq:ZJd.vs.IJd}.
Let $t_0 := q^{-s_0}$.

\begin{prop}
\begin{align}\label{eq:master.freep}
  \mcZ_{(V_1,V_2)}^\freep(\bfs)   =
      \Bigl(q^{-n-1} t_0 - 1 
      &+ \mcZ_{V_1}\bigl(s_0+n_2,\bfs^{(1)}\bigr) (1-q^{-n_2-1}t_0)\nonumber\\
      &+ \mcZ_{V_2}\bigl(s_0+n_1,\bfs^{(2)}\bigr) (1-q^{-n_1-1}t_0)\Bigr)
      / (1-q^{-1}t_0)
\end{align}
\end{prop}

\begin{proof}
  It suffices to analyse the function
  \begin{multline*}
    \mcI^\freep_{(V_1,V_2)}(\bfs) := \\
    \int\limits_{(\fO V)^\times\times\fP}\card{y}^{s_0}
    \left(\prod_{I_1\subset V_1}\norm{x^{(1)}_{I_1},x^{(2)},y}^{s_{I_1}}
    \right)\left(\prod_{I_2\subset V_2}\norm{x^{(2)}_{I_2},x^{(1)},y}^{s_{I_2}}
    \right)\dd\mu_{\fO V \times \fO}(x,y),
  \end{multline*}
  where
  $x^{(i)}_{I_i} = \bigl(x^{(i)}_j : j \in I_i\bigr)$ and
  $x = \bigl(x^{(1)},x^{(2)}\bigr) =
  \bigl(x_1^{(1)},\dots,x_{n_1}^{(1)},x^{(2)}_1,\dots,x^{(2)}_{n_2}\bigr)$.
  Indeed,
  $\mcZ_{(V_1,V_2)}^\freep(\bfs) = 1 +
  (\undl{1})^{-1}\mcI^\freep_{(V_1,V_2)}(\bfs)$;
  see~\eqref{eq:ZJd.vs.IJd}. We proceed by decomposing the domain of
  integration of this function. On the set $S\times \fP$ for
  \[
    S := \left\{
      \bigl(x^{(1)},x^{(2)}\bigr)\in (\fO V)^\times :
      x^{(1)}\not\equiv 0 \not\equiv x^{(2)} \pmod{\fP}
    \right\},
  \]
  the integral is very simple.
  Indeed, $\mu\bigl((\fO V)^\times \setminus  S\bigr)
  = (\undl{n_1})q^{-n_2} + (\undl{n_2})q^{-n_1}$
  whence
  \begin{multline*}
    (\undl{1})^{-1} \int\limits_{S \times \fP}
    \abs{y}^{s_0}
    \left(\prod_{I_1\subset V_1}\norm{x^{(1)}_{I_1},x^{(2)},y}^{s_{I_1}}
    \right)
    \left(\prod_{I_2\subset V_2}\norm{x^{(2)}_{I_2},x^{(1)},y}^{s_{I_2}}
    \right)
    \,
    \dd\mu_{\fO V \times \fO}(x,y)
    = \\ (\undl{1})^{-1}\left(1-q^{-n} - (\undl{n_1})q^{-n_2} -
      (\undl{n_2})q^{-n_1}\right) \int\limits_{\fP}\abs{y}^{s_0} \dd\mu_{\fO}(y_0) =  \\\left(1-q^{-n} - (\undl{n_1})q^{-n_2} -
      (\undl{n_2})q^{-n_1}\right)\gp{q^{-1}t_0}.
  \end{multline*}
  It remains to deal with
  \begin{alignat*}{2}
    \int\limits_{(\fO V)^\times \setminus S \,\times\, \fP} \abs{y}^{s_0}&
    \left(\prod_{I_1\subset V_1}\norm{x^{(1)}_{I_1},x^{(2)},y}^{s_{I_1}}
    \right)&\left(\prod_{I_2\subset
      V_2}\norm{x^{(2)}_{I_2},x^{(1)},y}^{s_{I_2}} \right) \,&\dd\mu_{\fO V
    \times \fO}(x,y)\\
  =\int\limits_{\fP V_1 \times (\fO V_2)^\times\times \fP}
    \abs{y}^{s_0}& &\left(\prod_{I_2\subset
      V_2}\norm{x^{(2)}_{I_2},x^{(1)},y}^{s_{I_2}} \right)&\dd\mu_{\fO V
    \times \fO}(x,y)\\
  + \int\limits_{(\fO V_1)^\times \times \fP V_1 \times
      \fP} \abs{y}^{s_0}& \left(\prod_{I_1\subset
      V_1}\norm{x^{(1)}_{I_1},x^{(2)},y}^{s_{I_1}} \right)&&\dd\mu_{\fO V
    \times \fO}(x,y) \\=\,\mcI_{ V_2,V_1}\bigl(s_0,\bfs^{(2)}\bigr) \,+\,&
  \mcI_{V_1,V_2}\bigl(s_0,\bfs^{(1)}\bigr);&&
  \end{alignat*}
  cf.\ \eqref{def:IJdelta}.
  For $i=1$,
  by invoking \eqref{eq:ZJd.vs.IJd} again and also Theorem~\ref{thm:master},
  we obtain
  \begin{align*}
    \mcI_{ V_1,n_2}(s_0,\bfs^{(1)})
    &= (\undl{1})\, \left(\mcZ_{V_1,V_2}(s_0,\bfs^{(1)}) -
      \frac{1-q^{-n_2-1}t_0}{1-q^{-1}t_0} \right)
    \\
    &=(\undl{1})\,\frac{1-q^{-n_2-1}t_0}{1-q^{-1}t_0}
      \left(\mcZ_{V_1}(s_0+n_2,\bfs^{(1)}) - 1 \right);
  \end{align*}
  the argument for $i=2$ is analogous.
\end{proof}

We now obtain the following expression for the rational function
$W_{\Eta_1\freep\Eta_2}(X,T)$ associated with 
the complete union $\Eta_1\freep\Eta_2$ of the hypergraphs $\Eta_1$
and $\Eta_2$.

\begin{cor}\label{cor:master.freep}
    \begin{align}
W_{\Eta_1\freep\Eta_2}(X,T)    & =
      \bigl(X^{-m} T - 1
    \nonumber\\
    & \qquad
      + W_{\Eta_1}(X,X^{-m_2}T)(1-X^{-m_2}T)(1-X^{n_1-m}T)
    \nonumber\\
    & \qquad
      + W_{\Eta_2}(X,X^{-m_1}T)(1-X^{-m_1}T)(1-X^{n_2-m}T)\bigr)
    \nonumber\\
    & \qquad
      /\bigl((1-T)(1-X^{n-m}T)\bigr).\label{eq:master.freep.rat}
  \end{align}
  In particular, if $\Eta_1$ is the block hypergraph $\BE_{n_1,m_1}$,
  then
  \begin{equation}\label{equ:freep.with.complete}
    W_{\Eta_1\freep\Eta_2}(X,T) = W_{\Eta_2}(X,X^{-m_1}T)
    \, \frac{(1-X^{-m_1}T)(1-X^{n_2-m}T)}{(1-T)(1-X^{n-m}T)}.
  \end{equation}
\end{cor}

\begin{proof}
  Write $t = q^{-s}$.
  As
  $$\zeta^\ak_{(\eta_1\freep\eta_2)^\fO}(s) =
  W_{\Eta_1\freep\Eta_2}(q,t) = \frac{1}{1-t}
  \mcZ_{(V_1,V_2)}^\freep \!\left(s-n+m-1, \bigl(-\mu^{(1)}_{I_1}\bigr)_{I_1\subset
      V_1},\bigl(-\mu^{(2)}_{I_2}\bigr)_{I_2\subset V_2}\right)\!,$$
  we seek to describe the effect of replacing $s_0$ by $s-n+m-1$ and each
  $s^{(i)}_{I_i}$ by $-\mu^{(i)}_{I_i}$ in each of the two functions
  $\mcZ_{V_i}(s_0+n_{3-i},\bfs^{(i)})$ in \eqref{eq:master.freep}.
  Since
  $$\zeta^\ak_{\eta_i^\fO}(s) = W_{\Eta_i}(q,t) =
  \frac{1}{1-t}\mcZ_{V_i}\Bigl(s-n_i+m_i-1,
    \Bigl(-\mu^{(i)}_{I_i}\Bigr)_{I_i\subset V_i}\Bigr)$$
  by \eqref{eq:zeta.Z},
  we obtain
  \[
    \mcZ_{V_i}\Bigl(s-n+m-1+n_{3-i},\Bigl(-\mu^{(i)}_{I_i}\Bigr)_{I_i\subset V_i}\Bigr)
    = W_{\Eta_i}(q,q^{-m_{3-i}}t) \, (1-q^{-m_{3-i}}t).
  \]
  This establishes the first claim.
  The special case follows from a simple computation using
  $W_{\BE_{n_1,m_1}}(X,T) = (1-X^{-m_1}T)/\bigl( (1-T)(1-X^{n_1-m_1}T)\bigr)$;
  see Example~\ref{exa:staircase}\eqref{exa:BEnm}.
\end{proof}

Given hypergraphs $\Eta_1,\dotsc,\Eta_r$,
repeated application of Corollary~\ref{cor:master.freep} yields explicit
formulae for $W_{\Eta_1\oplus\dotsb \oplus \Eta_r}(X,T)$ in terms of
$W_{\Eta_1}(X,T),\dotsc,W_{\Eta_r}(X,T)$.

\subsubsection{A special case: hypergraphs with codisjoint supports}
\label{subsubsec:codisjoint}

Let $\bfn,\bfm\in\N^\nb$.
The main result of this section, namely Corollary~\ref{cor:master.codis.rat}, 
provides an explicit formula for the rational function $W_\Eta(X,T)$
associated with 
$\Eta := \RE_{\bfn,\bfm}= \RE_{n_1,m_1} \freep \dotsb \freep
\RE_{n_\nb,m_\nb}$, the reflection of $\BE_{\bfn,\bfm}$; see
\eqref{def:block.hyp.multi}--\eqref{def:block.codis.multi}.
Note that
$$\bfo_{n\times m} - \diag\left(\bfo_{n_1\times m_1},\dots,\bfo_{n_\nb\times m_\nb}\right) \in
\Mat_{n \times m}(\Z)$$
is an incidence matrix of $\Eta$;
up to reordering rows and columns, this is the
general form of incidence matrices of hypergraphs with codisjoint supports
(i.e.~whenever $\mu_I\mu_J\neq 0$ for $I,J\subset V$, then $I=J$ or
$I^\comp \cap J^\comp = \varnothing$) and which also satisfy
$\bigcap_{\mu_I > 0} I = \varnothing$.

Let $V_i$ be the set of vertices of $\RE_{n_i,m_i}$
and let $V = V_1\sqcup \dotsb \sqcup V_\nb$ be that of $\Eta =
\RE_{\bfn,\bfm}$.

\paragraph{An auxiliary function.}
Consider the specialisation
\begin{equation*}
  \mcZ^{\freep, \codis}_{\bfV}(\bfs)
  :=
  \mcZ^{\freep,\codis}_{\bfV}
  \Bigl(
  s_0;\,
  s_{V_1^\comp}, \dotsc, s_{V_\nb^\comp}
  \Bigr)
  :=
  \mcZ_{V}\Bigl(s_0,
    \Bigl(
    \delta_{\exists i \in [r]\colon I = V_i^\comp} \, s_I
    \Bigr)_{I \subset  V}
  \Bigr)
\end{equation*}
of $\mcZ_{V}(\bfs)$.
In other words,
$\mcZ^{\freep, \codis}_{\bfV}(\bfs)$ is obtained from
$\mcZ_V(\bfs)$ by setting all variables $s_I$ to zero, except 
for those with $I = V_i^\comp$ for some $i$.

Let $n = n_1 + \dotsb + n_\nb$ and $m = m_1 + \dotsb + m_\nb$.
Let $\eta^\freep_{\bfn,\bfm}$ be the incidence
representation of $\Eta$.
From \eqref{eq:zeta.Z} (with $\sD = \varnothing$), we obtain the identity 

\begin{equation}\label{eq:codis}
  \zeta^\ak_{\left(\eta^\freep_{\bfn,\bfm}\right)^\fO}(s) =
  \frac{1}{1-t}\,
  \mcZ^{\freep,\codis}_{\bfV}\Bigl(s-n+m-1; \,
  -m_1,\dotsc,-m_\nb\Bigr).
\end{equation}
As before, we write $t_I := q^{-s_I}$.

\paragraph{An explicit formula and its consequences}

\begin{theorem}\label{thm:master.codisjoint}
  \begin{equation*}
    \mcZ^{\freep, \codis}_{\bfV}(\bfs) = \frac{1}{1-q^{-1}t_0} \left(
    1 - q^{-n-1}t_0 \left( 1 - \sum_{i=1}^r
    \frac{(\undl{n_i})\,q^{n_i}\Bigl(t_{V_i^\comp}-1\Bigr)}{1-q^{-n-1+n_i}\,t_0
      \,t_{V_i^\comp}}\right)\right).
\end{equation*}
\end{theorem}

Prior to proving Theorem~\ref{thm:master.codisjoint}, we record
our main result here, namely the following immediate consequence of
Theorem~\ref{thm:master.codisjoint} and \eqref{eq:codis}.

\begin{corollary}\label{cor:master.codis.rat}
  \begin{equation*}
    W_{\RE_{\bfn,\bfm}}(X,T) =
    \frac{1}{(1-T)(1-X^{n-m}T)}\left( 1-X^{-m}T\left(1- \sum_{i=1}^\nb
        \frac{(X^{n_i}-1)(X^{m_i}-1)}
        {1-X^{n_i+m_i-m}T}\right)\right).
    \pushQED{\qed}
    \qedhere
    \popQED
  \end{equation*}
\end{corollary}

Note that the formulae in Corollaries \ref{cor:master.freep} and
\ref{cor:master.codis.rat} indeed coincide where they overlap.

\begin{remark}
  \label{rem:three_complexities}
  Using ``Big Theta Notation'',
  the estimate from \cite{Bar80} for the $n$th Fubini number $f_n$ cited in
  \eqref{eq:fubini_growth} implies that
  $f_n = \Theta\Bigl(\frac{n!}{(\log 2)^{n+1}}\Bigr)$. 
  We have thus produced explicit formulae of three (generally strictly decreasing)
  complexities:
  \begin{enumerate}
  \item
    \label{rem:three_complexities1}
    For a general hypergraph $\Eta$ on $n$ vertices,
    Corollary~\ref{cor:master} expresses $W_\Eta(X,T)$ as a sum of
    $\card{\WOhat_n} = \Theta\Bigl(\frac{n!}{(\log 2)^{n+1}}\Bigr)$ rational functions.
  \item
    \label{rem:three_complexities2}
    If $\Eta = \BE_{\bfn,\bfm}$ for $\bfn,\bfm\in \NN^\nb$, then
    Corollary~\ref{cor:master.dis.WO} expresses $W_\Eta(X,T)$ as a sum
    of $\Theta\Bigl(\frac{r!} {(\log 2)^{r+1}}\Bigr)$ rational functions.
  \item
    \label{rem:three_complexities3}
    Finally, if $\Eta = \RE_{\bfn,\bfm}$ is the reflection (see
    \S\ref{ss:graphs}) of $\BE_{\bfn,\bfm}$, then
    Corollary~\ref{cor:master.codis.rat} expresses $W_\Eta(X,T)$ as a sum of
    $\Theta(r)$ rational functions.
  \end{enumerate}
  Writing $m = \card{\edges(\Eta)}$,
  the rational functions appearing in these sums
  are products of $\mathcal O(n)$ factors of the form $\pm X^a$, $\pm X^a T$, $1 - X^a$,
  and $(1- X^aT)^{\pm 1}$ for $a\in \ZZ$ with $\abs a = \mathcal O(n+m)$.

  For another point view, write each of the above formulae over a common
  denominator.
  Then we saw that the denominators can (essentially) be written as products of
  $\mathcal O(2^n)$ factors of the  form $1-X^AT$ in case~(\ref{rem:three_complexities1}),
  products of $\mathcal O(2^r)$ such factors in
  case~(\ref{rem:three_complexities2}), and as products of $\mathcal O(r)$
  factors in case~(\ref{rem:three_complexities3}).
  While cancellations may reduce the actual number of factors for any
  given hypergraph, experiments suggest that our bounds generally indicate the
  correct order of magnitude.
\end{remark}

\paragraph{Proof of Theorem~\ref{thm:master.codisjoint}.}
Informally speaking, the relatively low complexity of the formula
provided in Theorem~\ref{thm:master.codisjoint} reflects the fact that
the domain of integration of the integral in question may be broken up
into a comparatively small number of pieces, on each of which the
integrand is very simple. The following observation will be helpful.

\begin{lemma}\label{lem:p}
  For all $N\in\N$,
  \begin{multline*}
    \int\limits_{\fP^N\times\fP}\!\!\!
    \abs{y}^{a_0}\left\lVert x_1,\dots,x_{n-1},y\right\Vert^{a_{N-1}} \dd\mu_{\fO^{N}
    \times\fO}(x,y) =
  \frac{q^{-a_0-a_{N-1}-(N+1)}(\undl{1})\,(1-q^{-a_0-N})}{(1-q^{-a_0-1})(1-q^{-a_0-a_{N-1}-N})}.
\end{multline*}
\end{lemma}

\begin{proof}
  This is a straightforward corollary of \cite[Lemma~5.8]{ask} which implies that
\begin{multline*}
  F_N(a_0,0,\dots,0,a_{N-1},0) :=
  \int\limits_{\fO^N\times\fO}\abs{y}^{a_0}\left\lVert x_1,\dots,x_{N-1},y\right\rVert^{a_{N-1}}\dd\mu_{\fO^{N} \times\fO}(x,y) =\\
  \frac{(\undl{1})(1-q^{-a_0-N})}{(1-q^{-a_0-1})(1-q^{-a_0-a_{N-1}-N})}.\qedhere
\end{multline*}
\end{proof}

\begin{proof}[Proof of Theorem~\ref{thm:master.codisjoint}]
  Consider
  \[
    \mcZ^{\freep,\codis}_{\bfV}(s_0;\, s_{V_1^\comp},\dotsc,s_{V_\nb^\comp})
    = 1 + (\undl{1})^{-1}
    \int\limits_{(\fO V)^\times \times \fP}
    \abs{y_0}^{s_0} \prod_{i=1}^r \norm{x_{V_i^\comp};y_0}^{s_{V_i^\comp}}
    \dd\mu_{\fO V \times\fO}(x,y_0).
  \]
  Note that the product in the integrand is trivial unless
  $x\in (\fO V)^\times$ has its $\fP$-adic 
  units concentrated in exactly one of the sets~$V_i$. We therefore
  split up the first factor of the  domain of integration in the form $(\fO V)^\times =
  S \sqcup ((\fO V)^\times \!\setminus\! S)$, where
  \[
    S := \Bigl\{ x \in (\fO V)^\times : \# \bigl\{ j\in [\nb] :
    \exists v\in V_j\colon x_v\in \fO^\times
    \bigr\} > 1\Bigr\}.
  \]
  
  Clearly $\mu((\fO V)^\times \setminus S) = \sum_{i=1}^\nb(\undl{n_i})q^{n_i-n}$
  whence
  \begin{multline}
    (\underline{1})^{-1} \int\limits_{S \times \fP} \abs{y_0}^{s_0}
    \prod_{i=1}^\nb \norm{x_{V_i^\comp};y}^{s_{V_i^\comp}} \dd\mu_{\fO V
      \times\fO}(x,y) = \Bigl(1-q^{-n} - \sum_{i=1}^\nb
    (\undl{n_i})q^{n_i-n}\Bigr)\, \gp{q^{-1}t_0}.\label{Snew}
  \end{multline}
  
  By applying Lemma~\ref{lem:p} for each $j\in[\nb]$
  (with $N = n -n_i+1$, $a_0 = s_0$, $a_{N-1} =
  s_{V_i^\comp}$), we obtain
\begin{align}
  \lefteqn{
  (\underline{1})^{-1}
  \int\limits_{(\fO V)^\times \setminus S \,\times\, \fP}
  \abs{y}^{s_{0}}
  \prod_{i=1}^\nb
  \norm{x_{V_i^\comp};y}^{s_{V_i^\comp}}
  \,
  \dd\mu_{\fO V \times\fO}(x,y)
  }
  \nonumber\\
  &= (\underline{1})^{-1}\, \sum_{i=1}^\nb
    (\undl{n_i})q \int\limits_{\fP^{n-n_i+1} \times \fP} \abs{y_0}^{s_0}
    \norm{x_{[n-n_i]};y}^{s_{V_i^\comp}}
    \,
    \dd\mu_{\fO^{n-n_i+1} \times\fO}(x,y)
    \nonumber\\
  &=
    \sum_{i=1}^\nb(\undl{n_i})
    \frac{(1-q^{-n-1+n_i}t_0)}{(1-q^{-1}t_0)}
    \,
    \gp{q^{-n-1+n_i}t_{0}t_{V_i^\comp}}.
    \label{notSnew}
\end{align}
Combining \eqref{Snew} and \eqref{notSnew} yields, after some trivial
simplifications, that indeed
\begin{align*}
  \mcZ^{\freep,\codis}_{\bfV}\bigl(s_0;\, s_{V_1^\comp},\dotsc, s_{V_\nb^\comp}\bigr)
  &= 1 +
    \Bigl(
    1-q^{-n} - \sum_{i=1}^\nb (\undl{n_i}) q^{n_i-n}
    \Bigr)
    \gp{q^{-1}t_0} + \\
  &\quad \quad 
  \sum_{i=1}^\nb(\undl{n_i})
  \frac{(1-q^{-n-1+n_i}t_0)}{(1-q^{-1}t_0)}\gp{q^{-n-1+n_i}t_{0}t_{V_i^\comp}}
  \\
  &=
    \frac{1}{1-q^{-1}t_0} \left( 1- q^{-n-1}t_0 \left(1 - \sum_{i=1}^\nb
    \frac{(\undl{n_i})q^{n_i}\Bigl(t_{V_i^\comp}-1\Bigr)}
    {1-q^{-n-1+n_i}t_0 \, t_{V_i^\comp}} \right) \right).\qedhere
 \end{align*} 
\end{proof}

\subsection{Four basic operations on
  hypergraphs}\label{subsec:hyp.ops}

In this section, we study four fundamental operations for hypergraphs:
insert either a row or a column of either all $0$s or all $1$s into any
incidence matrix.
In Proposition~\ref{prop:hyp.red}, we record the effects of these operations
on associated ask zeta functions.
For group-theoretic applications of these results, see \S\ref{s:cographical}.

Throughout, let $\bfmu=(\mu_I)_{I\subset V}$ be the vector of
hyperedge multiplicities of a hypergraph~$\Eta$ on the vertex set $V$;
see Definition~\ref{def:hyp.mu} and the comments that follow it.
Let $n = \card V$ and $m = \card{\edges(\Eta)} = \sum_{I\subset V}\mu_I$.
Let $\bullet$ be a singleton set disjoint from~$V$.

\begin{defn}\label{def:hyp.inflation} We define
  \begin{enumerate}
  \item $\bfmu_\bfo = (\nu_J)_{J\subset V \sqcup\bullet}$ by
    $\nu_J =
    \begin{cases}
      \mu_I, &
      \text{if } J = I \sqcup \bullet \text{ for } I \subset V,
      \\
      0, & \text{otherwise,}
    \end{cases}$
  \item
    $\bfmu_\bfz= (\nu_J)_{J\subset V \sqcup\bullet}$ by
    $\nu_J =
    \begin{cases}
      \mu_J,
      & \text{if } J \subset V,
      \\
      0, & \text{otherwise,}
    \end{cases}$
  \item
    \label{hyp.inf.1.col}
    $\bfmu^\bfo = \Bigl( \mu_I+\delta_{I= V}\Bigr)_{I\subset V}$, and
  \item
    \label{hyp.inf.0.col}
    $\bfmu^{\bfz} = \Bigl(\mu_I + \delta_{I = \varnothing}\Bigr)_{I\subset V}$.
   \end{enumerate}
\end{defn}

In other words, beginning with an arbitrary incidence matrix of $\Eta$, we
obtain associated hypergraphs
\begin{center}
\begin{tabular}{llll}
$\Eta_{\bfo} := \Eta(\bfmu_\bfo)$ & by inserting a $\bfo$-row,&
  $\Eta_{\bfz} := \Eta(\bfmu_\bfz)$ & by inserting a
    $\bfz$-row,\\ $\Eta^{\bfo} := \Eta(\bfmu^{\bfo})$ & by inserting a
    $\bfo$-column,& $\Eta^{\bfz} := \Eta(\bfmu^{\bfz})$ & by inserting a
    $\bfz$-column.
\end{tabular}
\end{center}
Note that $\Eta(\bfmu_\bfo) = \Eta(\bfmu,\bullet)$ in the sense of
Definition~\ref{def:hyp.socle}.  We write $\bfmu_{\bfo^{(0)}}=\bfmu$
and, for $r\in\N$, $\bfmu_{\bfo^{(r)}} =
\left(\bfmu_{\bfo^{(r-1)}}\right)_{\bfo}$. Likewise, we write
$\Eta^{\bfo^{(0)}}=\Eta$ and $\Eta^{\bfo^{(r)}} =
\left(\Eta^{(\bfo^{(r-1)})}\right)$. We use analogous notation for the
other three operations.  All four operations turn out to have tame
effects on the ask zeta functions associated with $\Eta$.

\begin{prop}\label{prop:hyp.red}
  \begin{alignat}{2}
    W_{\Eta_{\bfo}}(X,T) &=&
    \,\frac{1-X^{n-m}T}{1-X^{1+n-m}T}\,&W_{\Eta}(X,T),\label{equ:insert.1.row}\\ W_{\Eta_{\bfz}}(X,T)
    &= &&W_{\Eta}(X,XT),\label{equ:insert.0.row}\\ W_{\Eta^\bfo}(X,T) &=
    &\frac{1-X^{-1}T}{1-T}
    &W_{\Eta}(X,X^{-1}T),\label{equ:insert.1.col}\\ W_{\Eta^\bfz}(X,T) &=
    &&W_{\Eta}(X,T).\label{equ:insert.0.col}
  \end{alignat}
\end{prop}

\begin{proof}
  The statement about $W_{\Eta_{\bfz}}(X,T)$
  and $W_{\Eta^{\bfz}}(X,T)$ follow from \cite[\S 3.4]{ask},
  the others by inspection of \eqref{eq:master.rat} (with $\sd=1$ for
  $W_{\Eta_{\bfo}}(X,T)$). 
\end{proof}

\begin{remark}\label{rem:hyp.red}
  For the purpose of determining $W_\Eta(X,T)$ for a hypergraph $\Eta \approx
  \Eta(\bfmu)$, Proposition~\ref{prop:hyp.red} allows us to assume that $\bfmu$
  satisfies $\mu_{V} = \mu_\varnothing = 0$, $\bigcap_{\mu_I > 0}I =
  \varnothing$, and $\bigcup_{\mu_I > 0}I = V$.
  In other words, we may assume that no incidence matrix of $\Eta(\bfmu)$ has
  rows or columns comprised exclusively of $0$s or $1$s.
  Conversely, by adding suitable rows or columns of $\bfz$s,
  Proposition~\ref{prop:hyp.red} also allows us to e.g.\ assume that incidence
  matrices of hypergraphs are squares; cf.\ \cite[Cor.\ 3.7]{ask}.
\end{remark}

We may now continue the story that began in Example~\ref{exa:K3_K3_K2}.

\def\exaK3K3K2pt2{Example~\ref{exa:K3_K3_K2}, part II}
\begin{example}[\exaK3K3K2pt2]
  \label{exa:K3_K3_K2_pt2}
  Let $\Eta$ be the hypergraph on $8$ vertices
  with incidence matrix \eqref{eq:M11}
  in Example~\ref{exa:K3_K3_K2}.
  We are now in a position to compute the rational function $W_\Eta(X,T)$.
  Indeed, $\Eta$ is isomorphic to $\left(\BE_{3,2} \oplus \BE_{3,2}\right)^\bfz\freep \BE_{2,2}$.
  Using equations~\eqref{equ:dis.r=2} (with $n_1 = n_2 = 3$ and $m_1 = m_2
  =2$) and \eqref{equ:insert.0.col}, we obtain
  $$W_{\left(\BE_{3,2} \oplus \BE_{3,2}\right)^\bfz}(X,T) = \frac{1
    + X^{-4}T - 2 X^{-2}T - 2X^{-1}T + XT +
    X^{-3}T^2}{(1-T)(1-XT)(1-X^2T)}.$$
  Therefore, by \eqref{equ:freep.with.complete},
  \begin{align}
    W_\Eta(X,T)
    & = W_{\left(\BE_{3,2} \oplus \BE_{3,2}\right)^\bfz\freep \BE_{2,2}}(X,T)
      \nonumber \\
    & = W_{\left(\Eta^{(3)} \oplus \Eta^{(3)}\right)^\bfz}(X,X^{-2}T)
      \, \frac{(1-X^{-2}T)(1-X^{-1}T)}{(1-T)(1-XT)}
      \nonumber \\
    & = \frac{1 + X^{-6}T - 2 X^{-4}T - 2 X^{-3} T + X^{-1} T + X^{-7}
      T^2}{(1-T)^2(1-XT)}.
      \label{eq:K3_K3_k2_pt2}
  \end{align}
  
  Note that $W_\Eta(X,T)$ coincides with the formula for
  $W_\Gamma^-(X,T)$ given in \eqref{eq:H332}.
  We will be able to explain this following our proof of
  the Cograph Modelling Theorem (Theorem~\ref{thm:cograph}) in
  \S\ref{s:models}; see Example~\ref{exa:K3_K3_K2_pt3}.
\end{example}

\subsection{Analytic properties of ask zeta functions of
  hypergraphs}\label{subsec:hyp.ana}

Let $K$ be a number field with ring of integers $\mcO = \mcO_K$.
Let $\zeta_K(s)$ be the Dedekind zeta function of~$K$.
As in \S\ref{ss:intro/zeta},
let $\mcV_K$ be the set of non-Archimedean places of $K$
and, for $v\in\mcV$, let $\mcO_v$ be the valuation ring of the $v$-adic
completion of $K$.
Let $q_v$ be the residue field size of $\mcO_v$.

Let $\eta$ be the incidence representation (see \S\ref{ss:incidence}) of a
hypergraph $\Eta=\Eta(\bfmu)$ on a set~$V$ of cardinality $n \ge 1$,
where $\bfmu = (\mu_I)_{I\subset V}\in \N_0^{\Pow( V)}$ is a vector of
hyperedge multiplicities.
By \cite[Proposition~3.4]{ask},
\[
  \zeta^\ak_{\eta^\mcO}(s)
  = \prod_{v\in \mcV_K}
  \zeta^\ak_{\eta^{\mcO_v}}(s)
  = \prod_{v\in \mcV_K}
  W_\Eta(q_v,q_v^{-s}).
\]
The explicit formula for $W_\Eta(X,T)$ in Corollary~\ref{cor:master} allows us
to deduce the following.

\begin{theorem}
  \label{thm:ana}
  Let $m' := \sum_{\varnothing \neq I \subset V}\mu_I$
  be the number of non-empty hyperedges of $\Eta$.
  \begin{enumerate}
  \item
    \label{thm:ana1}
    For each compact \DVR{} $\fO$,
    the real parts of the poles of
    $\zeta^\ak_{\eta^\fO}(s)$ are contained in
    \[
      \mcP_{\Eta}
      :=
      \Bigl\{
      \card J - \sum_{I\cap J \neq \varnothing} \mu_I
      :
      J \subset V \Bigr\}
      \subset
      \bigl\{ 1-m', 2 - m', \dotsc, n-1, n \bigr\},
    \]
    a set of integers (!) of cardinality at most $\min\{ 2^n, n+m' \}$.
  \item
    \label{thm:ana2}
    The abscissa of convergence $\alpha(\Eta)$ of
    $\zeta^\ak_{\eta^{\mcO_K}}(s)$ is a positive integer.
    It satisfies $\alpha(\Eta) \le n + 1$ and is independent of $K$.
\end{enumerate}
\end{theorem}
\begin{proof}
  First note that if $m' = 0$, then
  $\zeta^\ak_{\eta^\fO}(s) = 1/(1-q^{n-s})$ and $\zeta^\ak_{\eta^{\mcO_K}}(s) =
  \zeta_K(s-n)$; cf.~\cite[p.\ 577]{ask}.
  As $n \ge 1$, both claims then follow immediately.
  Henceforth, suppose that $m' > 0$ 
  so that $\card{J} - \sum_{I\cap J\not= \emptyset}\mu_I \in
  \{1-m',\dotsc,n\}$ for each $J\subset V$.
  Since $\zeta^\ak_{\eta^\fO}(s) = W_\Eta(q,q^{-s})$,
  part~(\ref{thm:ana1}) thus follows from
  Corollary~\ref{cor:master} (with $\sD = \varnothing$).
  
  For part~(\ref{thm:ana2}), we first paraphrase \eqref{eq:master.WOhat} in the
  form 
  \begin{equation}\label{equ:W.sum}
    W_{\Eta}(X,T) = \frac{1}{1-T}
    \,
    \left(
      1 + \sum_{i=1}^N f_i(X^{-1})\prod_{j\in I_i}\gp{X^{A_{ij}}T}
    \right)
  \end{equation}
  for some $N\in\N_0$, non-empty subsets $I_i\subset \N$, non-constant
  polynomials $f_i(Y)\in\Z[Y]$ with constant term $f_i(0)=1$,
  and~$A_{ij}\in\mcP_{\Eta}$. We may assume that $N>0$ and
  write~\eqref{equ:W.sum} over a common denominator
 $$W_\Eta(X,T) = \frac{1 +
   \sum\limits_{k=1}^\infty\left(\sum\limits_{l=-\infty}^\infty
   a_{lk}X^l\right)T^k}{(1-T)\prod\limits_{i,j}(1-X^{A_{ij}}T)}.$$ As a
 product of finitely many translates of $\zeta_K(s)$, the Euler
 product
 \[
   \prod_{v\in\mcV_K}
   \left.
     \frac{1}{(1-T)
       \prod\limits_{i,j}(1-X^{A_{ij}}T)}
     \right\vert_{X=q_v, T=q_v^{-s}}
   = \zeta_K(s)\,\prod_{i,j}\zeta_K(s-A_{ij})
 \]
has abscissa
 of convergence $\alpha:= \max\{1, A_{ij}+1 : i\in[N], j\in I_i\} \le n + 1$,
 where the estimate follows as in (\ref{thm:ana1}).
 Moreover, this product may be analytically continued to a meromorphic
 function on the whole complex plane.
 It thus suffices to show that the abscissa of convergence, $\alpha'$ say, of
 the Euler product
 \begin{equation}\label{eq:euler.num}
 N_\Eta(s) := \prod_{v\in\mcV_K}\left.\left( 1 +
 \sum_{k=1}^\infty\left(\sum_{l=-\infty}^\infty a_{lk}X^l\right)T^k
 \right)\right|_{X=q_v, T=q_v^{-s}}
 \end{equation}
 is strictly less than $\alpha$.
 By \cite[Lemma~5.4]{duSWoodward/08},
 $$\alpha' \leq \max\left\{ \frac{l+1}{k} : l\in\Z, k\in\N,
 a_{lk}\neq 0\right\} = \max\left\{\alpha'_1, \alpha'_{\geq
   2}\right\},$$ where
 \begin{align*}
   \alpha'_1 = \sup\{l+1 : l\in\Z, a_{l1}\neq 0\} \quad \textup{
     and } \quad \alpha'_{\geq2} = \sup\left\{ \frac{l+1}{k} :
   l\in\Z, k \in\N_{\geq 2}, a_{lk}\neq 0\right\}.
 \end{align*}

 The coefficient of $T$ in each Euler factor on the right-hand side of
 \eqref{eq:euler.num} is
 \[
   \sum_{l = -\infty}^\infty a_{l1}X^l = \sum_{i,j}
   \Bigl(f_i(X^{-1})-1\Bigr)X^{A_{ij}}.
 \]
 Hence, by the aforementioned properties of the polynomials $f_i(Y)$,
 we conclude that $\alpha'_1< \alpha$.
 Next, for each subset $S\subset I_1 \times \dotsb \times I_N$ with $\card S\geq 2$,
 \[
   \frac{1 + \sum_{(i,j)\in S}A_{ij}}{\card S}
   <
   \frac{\sum_{(i,j)\in S}(1 + A_{ij})}{\card S} \leq \max\{1 + A_{ij}
   : (i,j)\in S\}
   \leq \alpha
 \]
 whence $\alpha'_{\geq2} < \alpha$.
 The independence of $\alpha(\Eta)$ from $K$ has been established, in greater
 generality, in~\cite[Theorem~4.20]{ask}.
\end{proof}
\begin{remark}
  \quad
  \begin{enumerate}
  \item
    As we mentioned in Remark~\ref{rem:three_complexities},
    for specific $\bfmu$ the formula~\eqref{eq:master.WOhat}
    may simplify due to cancellations, possibly leading to a much smaller set
    of real parts of poles than~$\mcP_{\Eta}$.
    Based on experimental evidence, however, for suitably ``generic'' $\bfmu$,
    we expect most of these at most $2^{n}$ candidate real poles in
    $\mcP_{\Eta}$ to survive cancellation.
  \item
    Every integer from $1$ up to (and including) $n+1$ arises as the abscissa
    of convergence of the ask zeta function of a hypergraph on $n$
    vertices.
    Indeed, $\alpha(\BE_{n,m}) =\max\{1,n-m+1\}$; see
    Example~\ref{exa:staircase}(\eqref{exa:BEnm}) and cf.\
    \cite[Example~3.5]{ask}.
  \item
    The fact that both $\alpha(\Eta)$ and all elements of $\mcP_{\Eta}$ are
    \itemph{integers} seems noteworthy.
    Indeed, the abscissae of convergence of Dirichlet generating functions
    arising from related counting problems in subgroup or representation
    growth tend to be \itemph{rational} but typically non-integral numbers;
    cf.\ \cite[Theorem~1.3 and \S 6]{duSG_ICM/06} and
    \cite[Theorem~A(ii)]{Rossmann/17} for (non-)integrality results in
    the area of subgroup and submodule zeta functions and, for
    instance, \cite[Theorem~1.2]{Avn11}, \cite[Corollary~B]{DV/17}, and
    \cite[Theorem~4.22]{Snocken/14} 
    in the context of representation zeta functions.
  \item
    Example~\ref{exa:K3_K3_K2} shows that, in general, integrality statements
    such as those in Theorem~\ref{thm:ana} hold neither for the (Euler
    products of instances of the) functions $W^+_\Gamma(X,T)$ featuring in
    Theorem~\ref{thm:graph_uniformity}\eqref{thm:graph_uniformity3}
    nor for the functions $W^-_\Gamma(X,T)$ in
    Theorem~\ref{thm:graph_uniformity}\eqref{thm:graph_uniformity2},
    unless~$\Gamma$ is a cograph; cf.\ Theorem~\ref{thm:cograph}
    and see Question~\ref{qu:half.int}.
  \end{enumerate}
\end{remark}

For staircase hypergraphs (\S\ref{subsec:staircase}), we can
considerably strengthen Theorem~\ref{thm:ana}(\ref{thm:ana2}).
Indeed, inspection of~\eqref{equ:staircase} yields the following result.

\begin{proposition}
  Let $\bfm =(m_0,\dots,m_n)\in\N_0^{n+1}$.
  Let ${\sigma\eta}_{\bfm}$ denote the incidence representation of the staircase
  hypergraph $\Stair_{\bfm}$; see \eqref{def:staircase.graph}.
  Then for each number field $K$ with ring of
  integers $\mcO_K$, the abscissa of convergence
  of $\zeta^\ak_{\sigma\eta_{\bfm}}(s)$ is given by
  \[
    \alpha(\Stair_\bfm) =
    \max\Bigl\{
    1,\,1+n-j-\sum_{\iota > j}m_\iota: j=0,\dotsc,n-1\Bigr\}.
\]
  Moreover, the function $\zeta^\ak_{\sigma\eta_{\bfm}^{\mcO_K}}(s)$ may be
  meromorphically continued to the whole of~$\C$.
  \qed
\end{proposition}

\section{Uniformity for ask zeta functions of graphs}
\label{s:uniformity}

The main result of this section,
Theorem~\ref{thm:torically_combinatorial}, establishes that, subject
to very mild assumptions, a simultaneous generalisation of the two
types of adjacency modules from~\S\ref{ss:two_adjacencies} is
torically combinatorial (see \S\ref{ss:combinatorial_modules}).
This will, in particular, provide a constructive proof of
Theorem~\ref{thm:graph_uniformity}(\ref{thm:graph_uniformity2})--(\ref{thm:graph_uniformity3}).

In \S\ref{ss:wsm}, we develop the general setup for the class of adjacency
modules that appear in Theorem~\ref{thm:torically_combinatorial}.
In \S\ref{ss:surgery}, we describe several graph-theoretic operations with
tame effects on adjacency modules.
In \S\ref{ss:torically_torically}, we show that ``torically torically combinatorial''
and ``torically combinatorial'' are equivalent properties.
Both \S\S\ref{ss:surgery}--\ref{ss:torically_torically} are then employed
in the proof of Theorem~\ref{thm:torically_combinatorial} in
\S\ref{ss:proof_torically_combinatorial}.

Throughout, let $R$ be a ring.

\subsection{Weighted signed multigraphs and their adjacency modules}
\label{ss:wsm}
\begin{defn}
  \label{d:wsm}
  A \emph{weighted signed multigraph (\WSM)} (over $\sigma$) is a quadruple
  \[
    \bm\Gamma =(\Gamma,\sigma,\wt,\sgn),
  \]
  where
  \begin{enumerate}[label={(\textsf{W\arabic*})}]
\item
  \label{d:wsm1}
  $\Gamma = (V,E,\abs\dtimes)$ is a multigraph (see \S\ref{ss:graphs}),
\item
  \label{d:wsm2}
  $\sigma\subset \Orth\, V$ is a cone,
\item
  \label{d:wsm3}
  $\wt$ is a function $E \to \ZZ V$ with $u + \wt(e) \in \sigma^*$ for
  all $e\in E$ and $u\in \abs e$, and
\item
  \label{d:wsm4}
  $\sgn$ is a function $E\to \{\pm 1\}$.
\end{enumerate}
\end{defn}

Henceforth, let $\bm\Gamma$ be a \WSM{} as above.
Let $X = (X_v)_{v\in V}$ as before.
For $u,v\in V$ and $\omega \in \ZZ V$, define
\[
\xcomm u \omega {\pm 1} v_R := 
\begin{cases}
  X^{u + \omega} v \pm  X^{v + \omega} u, & \text{if } u \not= v,\\
  \pm X^{u + \omega} u, & \text{if }u = v,
\end{cases}
\]
an element of $R[X^{\pm 1}]\, V$; we usually drop the subscript $R$ in the following.
Note that if $u\not= v$, then
$\xcomm v \omega {\pm 1} u = \pm \xcomm u \omega {\pm 1} v$.
Further note that $\xcomm u 0 {\pm 1} v = [u,v;\pm 1]$, where the right-hand
side is defined as in \S\ref{ss:two_adjacencies}.

By \ref{d:wsm3} in Definition~\ref{d:wsm}, for $e\in E$ and $u\in \abs e$, we
have $X^{u + \wt(e)}\in R_\sigma$; see \S\ref{ss:toric_points} for a definition
of $R_\sigma$.
Let
\[
  \adj(\bm\Gamma;R)
  := 
  \Bigl\langle
  \xcomm u {\wt(e)}{\sgn(e)} v_R
  : e \in E \text{ with } \abs e = \{u,v\} \Bigr\rangle
  \le R_\sigma V.
\]
The \emph{adjacency module} of $\bm\Gamma$ over $R$ is the $R_\sigma$-module
\[
  \Adj(\bm\Gamma;R) := \frac{R_\sigma\, V}{\adj(\bm\Gamma;R)}.
\]

\begin{rem}
  \label{rem:general_adjacency_module}
  \quad
  \begin{enumerate}
  \item
    \label{rem:general_adjacency_module1}
    If $\Gamma$ is a graph, then $\Adj(\Gamma,\Orth V, 0, \pm 1;R)$ coincides
    with the adjacency module $\Adj(\Gamma,\pm 1;R)$ of $\Gamma$ over $R$ as
    defined in \S\ref{ss:two_adjacencies}.
    (Here, we assume that $\Gamma$ is simple in the negative case.)
  \item
    \label{rem:general_adjacency_module2}
    Of course, the signs of loops have no effect on adjacency modules.
    They are included for notational convenience only.
  \end{enumerate}
\end{rem}

\begin{lemma}
  \label{lem:WSM_Adj_base_change}
  \quad
  \begin{enumerate}
  \item
    $\Adj(\bm\Gamma;S) = \Adj(\bm\Gamma;R)^{S_\sigma}$ for each ring map $R\to
    S$.
  \item
    Let $\bm\Gamma'$ be the \WSM{} obtained from $\bm\Gamma$ by replacing 
    $\sigma$ by a cone $\tau \subset \sigma$.
    Then
    $\Adj(\bm\Gamma';R) \approx_{R_{\tau}} \Adj(\bm\Gamma;R)
    \otimes_{R_\sigma} R_{\tau}$.
    \qed
  \end{enumerate}
\end{lemma}

The following result and its constructive proof constitute the main contribution of
the present section;
a proof will be given in \S\ref{ss:proof_torically_combinatorial}.

\begin{thm}
  \label{thm:torically_combinatorial}
  Let $\bm\Gamma = (\Gamma,\sigma,\wt,\sgn)$ be a weighted signed multigraph.
  Let $R$ be a ring.
  Suppose that one of the following conditions is satisfied:
  \begin{enumerate*}
  \item $2 \in R^\times$.
  \item $2 = 0$ in $R$.
  \item
    \label{thm:torically_combinatorial3}
    $\sgn \equiv -1$ (irrespective of $R$).
  \end{enumerate*}
  Then $\Adj(\bm\Gamma;R)$ is torically combinatorial
  (see \S\ref{ss:combinatorial_modules}).
\end{thm}

Theorem~\ref{thm:torically_combinatorial} easily implies
Theorem~\ref{thm:graph_uniformity}:

\begin{proof}[Proof of Theorem~\ref{thm:graph_uniformity}]
  Part~(\ref{thm:graph_uniformity1}) was already proved in
  \S\ref{ss:combinatorial_modules}.
  For (\ref{thm:graph_uniformity2})--(\ref{thm:graph_uniformity3}),
  combine Proposition~\ref{prop:ask_adjacency_module},
  Proposition~\ref{prop:combinatorial_module_uniformity},
  Remark~\ref{rem:general_adjacency_module}(\ref{rem:general_adjacency_module1}),
  and Theorem~\ref{thm:torically_combinatorial} (with $R = \ZZ$ or $R = \ZZ[1/2]$).
\end{proof}

\begin{proof}[Proof of Corollary~\ref{cor:graphical_cc}]
  Combine Theorem~\ref{thm:graph_uniformity} and
  Proposition~\ref{prop:graphical_cc}.
\end{proof}

While the preceding two proofs only applied
Theorem~\ref{thm:torically_combinatorial} in the special case that $\bm\Gamma$
arises as in Remark~\ref{rem:general_adjacency_module}, our recursive proof of
Theorem~\ref{thm:torically_combinatorial} heavily relies on the greater
generality developed here.

A particularly easy special case of Theorem~\ref{thm:torically_combinatorial}
deserves to be spelled out at this point.
We say that a graph is \emph{solitary} if each of its edges is a loop.

\begin{prop}
  \label{prop:Adj_solitary}
  Let $\sigma \subset \Orth V$ be a cone.
  Let $\bm\Gamma$ be a \WSM{} over $\sigma$ (with underlying vertex set $V$)
  such that the underlying graph of $\bm\Gamma$ is solitary.
  Then $\Adj(\bm\Gamma;R)$ is a combinatorial $R_\sigma$-module.
  \qed
\end{prop}

Informally, our proof of Theorem~\ref{thm:torically_combinatorial}
given in \S\ref{ss:proof_torically_combinatorial}
proceeds by induction on an invariant
which measures to what extent a graph fails to be solitary;
the base case of our induction will be provided by
Proposition~\ref{prop:Adj_solitary}.

\subsection{Multigraph surgery}
\label{ss:surgery}

As we will see in this subsection,
subject to various assumptions, we may modify the edges (as well as their
weights and signs) of weighted signed multigraphs without
affecting the isomorphism type of the associated adjacency module.
This will constitute the heart of our proof of
Theorem~\ref{thm:torically_combinatorial} in \S\ref{ss:proof_torically_combinatorial}.

Throughout let $\bm\Gamma =(\Gamma,\sigma,\wt,\sgn)$ be a 
\WSM{}, where $\Gamma = (V,E,\abs\dtimes)$.
An \emph{incident pair} of $\Gamma$ is a pair (``formal product'') $u.e$, where
$e\in E$ and $u \in \abs e$.
Recall that $\le_\sigma$ denotes the preorder on $\ZZ V$ such that $u
\le_\sigma v$ if and only if $v-u \in \sigma^*$.
Given incident pairs $u.e$ and $w.f$ of  $\Gamma$,
we say that $u.e$ \emph{dominates} $w.f$ (in $\bm\Gamma$)
if $u + \wt(e) \le_\sigma w + \wt(f)$.

The next three lemmas and their proofs (nearly) follow an identical pattern:
subject to dominance conditions, suitable edges of $\Gamma$ can be
transplanted to produce a new \WSM{} $\bm\Gamma'$ such that
$\adj(\bm\Gamma;R) = \adj(\bm\Gamma';R)$ for every ring $R$.  The
effects of the operations $\bm\Gamma\leadsto \bm\Gamma'$ on the
underlying multigraphs are indicated in
Figure~\ref{fig:graph_effects}.  Note that these figures only depict
those parts of the respective multigraphs that are relevant for the
result in question.

\begin{lemma}[Dominant loop vs non-loop]
  \label{lem:dominant_loop}
  Let $u,v\in V$ be distinct.
  Let $\ell,h \in E$ with $\abs \ell = \{u\}$ and $\abs h = \{u,v\}$.
  Suppose that $u.\ell$ dominates~$v.h$.
  Define a multigraph $\Gamma' := (V,E,\norm\dtimes)$, where
  \[
  \norm e := \begin{cases}
    \,\,\abs e, & \text{if } e\not= h, \\
    \{ v\}, & \text{if } e = h.
    \end{cases}
  \]
   Define $\wt'\colon E\to \ZZ V$ via
  \[
  \wt'(e) := \begin{cases}
    \wt(e), & \text{if }  e \not= h, \\
    u - v + \wt(h), & \text{if } e = h.
    \end{cases}
  \]
  Then the following hold:
  \begin{enumerate}
  \item
    \label{lem:dominant_loop1}
    $\bm\Gamma' := (\Gamma',\sigma,\wt',\sgn)$ is a \WSM{}.
  \item
    \label{lem:dominant_loop2}
    $\adj(\bm\Gamma;R) = \adj(\bm\Gamma';R)$ for every ring $R$; in
    particular, $\Adj(\bm\Gamma;R) = \Adj(\bm\Gamma';R)$.
  \end{enumerate}
\end{lemma}
\begin{proof}
  \quad
  \begin{enumerate}
  \item
    We need to check that
    $x + \wt'(e) \in \sigma^*$ for all $e\in E$ and $x \in \norm e$.
    For $e \not= h$, this clearly follows since $\bm\Gamma$ is a \WSM{}.
    It also follows in the remaining case $e = h$ since
    $v + \wt'(h) = u + \wt(h) \in \sigma^*$,
    again since $\bm\Gamma$ is a \WSM{}.
  \item
    Let
    \[
      I := 
      \Bigl\langle
      \xcomm x {\wt(e)}{\sgn(e)} y
      : e \in E\setminus\{h\} \text{ with } \abs e = \{x,y\} \Bigr\rangle
      \subset \adj(\bm\Gamma;R) \cap \adj(\bm\Gamma';R).
    \]
    Further define
    \begin{align*}
      a\phantom' & := \xcomm u {\wt(h)}{\sgn(h)} v \in \adj(\bm\Gamma;R),\\
      a' & := \xcomm v {\wt'(h)}{\sgn(h)} v \in \adj(\bm\Gamma';R),
           \text{ and}\\
      b\phantom' & := \xcomm u{\wt(\ell)}{\sgn(\ell)} u \in I
    \end{align*}
    and note that
    $\adj(\bm\Gamma;R)  = \langle a \rangle + I$ and
    $\adj(\bm\Gamma';R) = \langle a' \rangle + I$.
    Since $u.\ell$ dominates~$v.h$, we have
    $t := X^{v + \wt(h) - u - \wt(\ell)}\in R_\sigma$.
    As
    \[
      \sgn(h)a' = X^{u+\wt(h)} v = a - \sgn(h)\sgn(\ell) tb,
    \]
    we conclude that $a \equiv \pm a' \pmod I$.
    \qedhere
  \end{enumerate}
\end{proof}

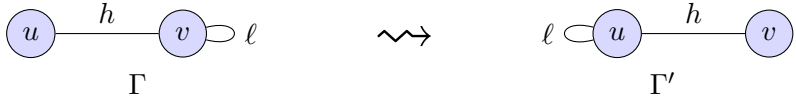
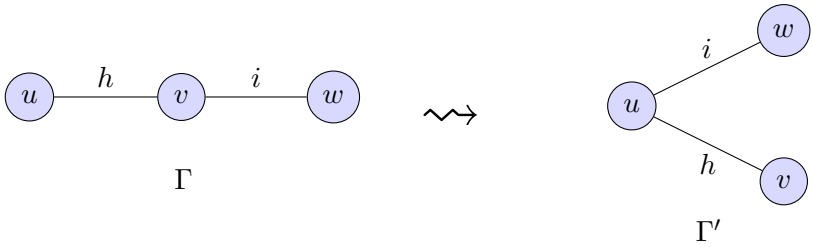
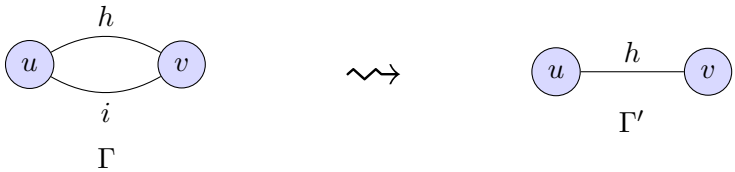
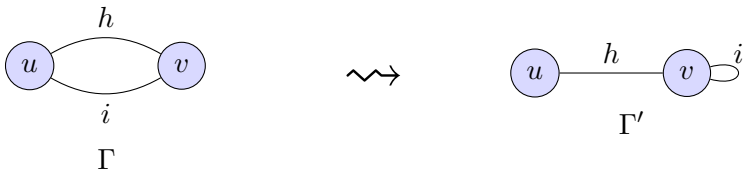
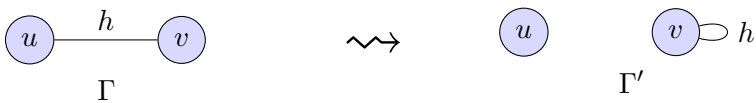
\begin{figure}
  \thisfloatpagestyle{empty}
  \centering
  \begin{subfigure}{\textwidth}
    \begin{minipage}{6cm}
      \begin{center}
        \begin{tikzpicture}
          \tikzset{every loop/.style={}}
          \tikzstyle{Black Vertex}=[draw=black, fill=blue!15, shape=circle]
          \node [style=Black Vertex] (0) at (0, 0) {$u$};
          \node [style=Black Vertex] (1) at (2, 0) {$v$};
          \path (0) edge[loop left] node {$\ell$} (0);
          \path (0) edge node[above] {$h$} (1);
        \end{tikzpicture}

        $\Gamma$
      \end{center}
    \end{minipage}
    \LEADSTO
    \begin{minipage}{6cm}
      \begin{center}
        \begin{tikzpicture}[]
          \tikzset{every loop/.style={}}
          \tikzstyle{Black Vertex}=[draw=black, fill=blue!15, shape=circle]
          \node [style=Black Vertex] (0) at (0, 0) {$u$};
          \node [style=Black Vertex] (1) at (1.5, 0) {$v$};
          \path (0) edge[loop left] node {$\ell$} (0);
          \path (1) edge[loop right] node {$h$} (1);
        \end{tikzpicture}

        $\Gamma'$
      \end{center}
    \end{minipage}
    \caption{Lemma~\ref{lem:dominant_loop}: dominant loop vs non-loop}
  \end{subfigure}

  \vspace*{3.5em}
  
  \begin{subfigure}{\textwidth}
    \begin{minipage}{6cm}
      \begin{center}
        \begin{tikzpicture}[]
          \tikzset{every loop/.style={}}
          \tikzstyle{Black Vertex}=[draw=black, fill=blue!15, shape=circle]
          \node [style=Black Vertex] (0) at (0, 0) {$u$};
          \node [style=Black Vertex] (1) at (2, 0) {$v$};
          \path (0) edge node[above] {$h$} (1);
          \path (1) edge[loop right] node {$\ell$} (1);
        \end{tikzpicture}

        $\Gamma$
      \end{center}
    \end{minipage}
    \LEADSTO
    \begin{minipage}{6cm}
      \begin{center}
        \begin{tikzpicture}[]
          \tikzset{every loop/.style={}}
          \tikzstyle{Black Vertex}=[draw=black, fill=blue!15, shape=circle]
          \node [style=Black Vertex] (0) at (0, 0) {$u$};
          \node [style=Black Vertex] (1) at (2, 0) {$v$};
          \path (0) edge[loop left] node {$\ell$} (0);
          \path (0) edge node[above] {$h$} (1);
        \end{tikzpicture}

        $\Gamma'$
      \end{center}
    \end{minipage}
    \caption{Lemma~\ref{lem:dominant_nonloop_loop}: dominant non-loop vs loop}
  \end{subfigure}

  \vspace*{2.5em}

  \begin{subfigure}{\textwidth}
    \begin{minipage}{6cm}
      \begin{center}
        \begin{tikzpicture}[]
          \tikzstyle{Black Vertex}=[draw=black, fill=blue!15, shape=circle]
          \node [style=Black Vertex] (0) at (0, 0) {$u$};
          \node [style=Black Vertex] (1) at (2, 0) {$v$};
          \node [style=Black Vertex] (2) at (4, 0) {$w$};
          \path (0) edge node[above] {$h$} (1);
          \path (1) edge node[above] {$i$} (2);
        \end{tikzpicture}

        \vspace*{1em}
        $\Gamma$
      \end{center}
    \end{minipage}
    \LEADSTO
    \begin{minipage}{6cm}
      \begin{center}
        \begin{tikzpicture}[]
          \tikzstyle{Black Vertex}=[draw=black, fill=blue!15, shape=circle]
          \node [style=Black Vertex] (0) at (0, 0) {$u$};
          \node [style=Black Vertex] (1) at (2, -1) {$v$};
          \node [style=Black Vertex] (2) at (2, +1) {$w$};
          \path (0) edge node[below] {$h$} (1);
          \path (0) edge node[above] {$i$} (2);
        \end{tikzpicture}

        $\Gamma'$
      \end{center}
    \end{minipage}
    \caption{Lemma~\ref{lem:dominant_nonloop_nonloop}: dominant non-loop vs non-loop}
  \end{subfigure}

  \vspace*{3.5em}

  \begin{subfigure}{\textwidth}
    \begin{minipage}{6cm}
      \begin{center}
        \begin{tikzpicture}[]
          \tikzstyle{Black Vertex}=[draw=black, fill=blue!15, shape=circle]
          \node [style=Black Vertex] (0) at (0, 0) {$u$};
          \node [style=Black Vertex] (1) at (2, 0) {$v$};
          \path (0) edge[out=+30,in=+150] node[above] {$h$} (1);
          \path (0) edge[out=-30,in=-150] node[below] {$i$} (1);
        \end{tikzpicture}

        $\Gamma$
      \end{center}
    \end{minipage}
    \LEADSTO
    \begin{minipage}{6cm}
      \begin{center}
        \begin{tikzpicture}[]
          \tikzstyle{Black Vertex}=[draw=black, fill=blue!15, shape=circle]
          \node [style=Black Vertex] (0) at (0, 0) {$u$};
          \node [style=Black Vertex] (1) at (2, 0) {$v$};
          \path (0) edge node[above] {$h$} (1);
        \end{tikzpicture}

        $\Gamma'$
      \end{center}
    \end{minipage}
    \caption{Lemma~\ref{lem:parallel_edges_1}: parallel edges I}
  \end{subfigure}

  \vspace*{3.5em}

  \begin{subfigure}{\textwidth}
    \begin{minipage}{6cm}
      \begin{center}
        \begin{tikzpicture}[]
          \tikzstyle{Black Vertex}=[draw=black, fill=blue!15, shape=circle]
          \node [style=Black Vertex] (0) at (0, 0) {$u$};
          \node [style=Black Vertex] (1) at (2, 0) {$v$};
          \path (0) edge[out=+30,in=+150] node[above] {$h$} (1);
          \path (0) edge[out=-30,in=-150] node[below] {$i$} (1);
        \end{tikzpicture}

        $\Gamma$
      \end{center}
    \end{minipage}
    \LEADSTO
    \begin{minipage}{6cm}
      \begin{center}
        \begin{tikzpicture}[]
          \tikzset{every loop/.style={}}
          \tikzstyle{Black Vertex}=[draw=black, fill=blue!15, shape=circle]
          \node [style=Black Vertex] (0) at (0, 0) {$u$};
          \node [style=Black Vertex] (1) at (2, 0) {$v$};
          \path (0) edge node[above] {$h$} (1);
          \path (1) edge[loop right] node[above] {$i$} (1);
        \end{tikzpicture}

        $\Gamma'$
      \end{center}
    \end{minipage}
    \caption{Lemma~\ref{lem:parallel_edges_2}: parallel edges II}
  \end{subfigure}

  \vspace*{3.5em}

  \begin{subfigure}{\textwidth}
    \begin{minipage}{6cm}
      \begin{center}
        \begin{tikzpicture}[]
          \tikzstyle{Black Vertex}=[draw=black, fill=blue!15, shape=circle]
          \node [style=Black Vertex] (0) at (0, 0) {$u$};
          \node [style=Black Vertex] (1) at (2, 0) {$v$};
          \path (0) edge node[above] {$h$} (1);
        \end{tikzpicture}

        $\Gamma$
      \end{center}
    \end{minipage}
    \LEADSTO
    \begin{minipage}{6cm}
      \begin{center}
        \begin{tikzpicture}[]
          \tikzset{every loop/.style={}}
          \tikzstyle{Black Vertex}=[draw=black, fill=blue!15, shape=circle]
          \node [style=Black Vertex] (0) at (0, 0) {$u$};
          \node [style=Black Vertex] (1) at (2, 0) {$v$};
          \path (1) edge[loop right] node[right] {$h$} (1);
        \end{tikzpicture}

        $\Gamma'$
      \end{center}
    \end{minipage}
    \caption{Lemma~\ref{lem:trim_spikes}: trimming spikes}
  \end{subfigure}
  \caption[Illustrations]{Illustrations of Lemma~\ref{lem:dominant_loop}--\ref{lem:trim_spikes}}
  \label{fig:graph_effects}
\end{figure}

\begin{lemma}[Dominant non-loop vs loop]
  \label{lem:dominant_nonloop_loop}
  Let $u,v\in V$ be distinct.
  Let $\ell,h \in E$ with $\abs \ell = \{v\}$ and $\abs h = \{u,v\}$.
  Suppose that $u.h$ dominates~$v.\ell$.
  Define a multigraph $\Gamma' := (V,E,\norm\dtimes)$, where
  \[
  \norm e := \begin{cases}
    \,\,\abs e, & \text{if } e\not= \ell, \\
    \{ u\}, & \text{if } e = \ell.
    \end{cases}
  \]
  Define $\wt'\colon E\to \ZZ V$ via
  \[
  \wt'(e) := \begin{cases}
    \wt(e), & \text{if }  e \not= \ell, \\
    2v - 2u + \wt(\ell), & \text{if } e = \ell
    \end{cases}
  \]
  and $\sgn'\colon E \to \{\pm 1\}$ via
  \[
  \sgn'(e) := \begin{cases}
    \phantom+\sgn(e), & \text{if }  e \not= \ell, \\
    -\sgn(h)\sgn(\ell), & \text{if } e = \ell.
    \end{cases}
  \]
  Then the following hold:
  \begin{enumerate}
  \item
    $\bm\Gamma' := (\Gamma',\sigma,\wt',\sgn')$ is a \WSM{}.
  \item $\adj(\bm\Gamma;R) = \adj(\bm\Gamma';R)$ for every ring $R$; in
    particular, $\Adj(\bm\Gamma;R) = \Adj(\bm\Gamma';R)$.
  \end{enumerate}
\end{lemma}
\begin{proof}
  \quad
  \begin{enumerate}
  \item
    We need to check that
    $x + \wt'(e) \in \sigma^*$ for all $e\in E$ and $x \in \norm e$.
    For $e \not= \ell$ and $x\in \norm e = \abs e$, we have
    $x + \wt'(e) = x + \wt(e) \in \sigma^*$ since $\bm\Gamma$ is a \WSM{}.
    Moreover,
    $v + \wt(h) \in \sigma^*$ since $\Gamma$ is a \WSM{} and
    $v + \wt(\ell) - u - \wt(h) \in \sigma^*$ since $u.h$ dominates~$v.\ell$.
    Hence, $u + \wt'(\ell) = 2v - u + \wt(\ell) \in \sigma^*$.
  \item
    Let
    \[
      I := 
      \Bigl\langle
      \xcomm x {\wt(e)}{\sgn(e)} y
      : e \in E\setminus\{\ell\} \text{ with } \abs e = \{x,y\} \Bigr\rangle
      \subset \adj(\bm\Gamma;R) \cap \adj(\bm\Gamma';R).
    \]
    Further define
    \begin{align*}
      a\phantom' & := \xcomm u {\wt(h)}{\sgn(h)} v \in I, \\
      b\phantom' & := \xcomm v {\wt(\ell)}{\sgn(\ell)} v
                   \in \adj(\bm\Gamma;R), \text{ and}\\
      b' & := \xcomm u {\wt'(\ell)}{\sgn'(\ell)} u \in \adj(\bm\Gamma';R)
    \end{align*}
    and note that
    $\adj(\bm\Gamma;R)  = \langle b \rangle + I$ and
    $\adj(\bm\Gamma';R) = \langle b' \rangle + I$.
    Since $u.h$ dominates~$v.\ell$, we have
    $t := X^{v + \wt(\ell) - u - \wt(h)}\in R_\sigma$.
    As $b - \sgn(\ell) ta = b'$, we conclude that
    $b \equiv b' \pmod I$. \qedhere 
  \end{enumerate}
\end{proof}

\begin{lemma}[Dominant non-loop vs non-loop]
  \label{lem:dominant_nonloop_nonloop}
  Let $u,v,w \in V$ be distinct.
  Let $h\in E$ with $\abs h = \{u,v\}$ and $i\in E$ with $\abs i = \{v,w\}$.
  Suppose that $u.h$ dominates~$w.i$.
  Define a multigraph $\Gamma' := (V,E,\norm\dtimes)$, where
  \[
  \norm e := \begin{cases}
    \,\,\abs e, & \text{if } e\not= i, \\
    \{ u, w\}, & \text{if } e = i.
    \end{cases}
  \]
  Define $\wt'\colon E\to \ZZ V$ via
  \[
  \wt'(e) := \begin{cases}
    \wt(e), & \text{if }  e \not= i, \\
    v - u + \wt(i), & \text{if } e = i
    \end{cases}
  \]
  and $\sgn'\colon E \to \{\pm 1\}$ via
  \[
  \sgn'(e) := \begin{cases}
    \phantom+\sgn(e), & \text{if }  e \not= i, \\
    -\sgn(h)\sgn(i), & \text{if } e = i.
    \end{cases}
  \]
  Then the following hold:
  \begin{enumerate}
  \item
    $\bm\Gamma' := (\Gamma',\sigma,\wt',\sgn')$ is a \WSM{}.
  \item
    $\adj(\bm\Gamma;R) = \adj(\bm\Gamma';R)$ for every ring $R$; in
    particular, $\Adj(\bm\Gamma;R) = \Adj(\bm\Gamma';R)$.
  \end{enumerate}
\end{lemma}

\clearpage

\begin{proof}
  \quad
  \begin{enumerate}
  \item
    We need to check that
    $x + \wt'(e) \in \sigma^*$ for all $e\in E$ and $x \in \norm e$.
    For $e \not= i$ and $x\in \norm e = \abs e$,
    $x + \wt'(e) = x + \wt(e) \in \sigma^*$.
    As $\bm\Gamma$ is a \WSM{}, each of $v + \wt(h)$, $v + \wt(i)$, and $w +
    \wt(i)$ belongs to $\sigma^*$.
    Since  $u.h$ dominates~$w.i$, we conclude that
    $u + \wt'(i) = v + \wt(i) \in \sigma^*$ and 
    \begin{align*}
      w + \wt'(i) & = w + v - u + \wt(i) = (v + \wt(h)) + (w + \wt(i)) -
                                              (u + \wt(h)) \in \sigma^*.
    \end{align*}
  \item
    Let
    \[
      I := 
      \Bigl\langle
      \xcomm x {\wt(e)}{\sgn(e)} y
      : e \in E\setminus\{i\} \text{ with } \abs e = \{x,y\} \Bigr\rangle
      \subset \adj(\bm\Gamma;R) \cap \adj(\bm\Gamma';R).
    \]
    Further define
    \begin{align*}
      a\phantom' & := \xcomm u {\wt(h)}{\sgn(h)} v \in I, \\
      b\phantom' & := \xcomm v {\wt(i)}{\sgn(i)} w
                   \in \adj(\bm\Gamma;R), \text{ and}\\
      b' & := \xcomm u {\wt'(i)}{\sgn'(i)} w \in \adj(\bm\Gamma';R)
    \end{align*}
    and note that
    $\adj(\bm\Gamma;R)  = \langle b \rangle + I$ and
    $\adj(\bm\Gamma';R) = \langle b' \rangle + I$.
    
    Since $u.h$ dominates~$w.i$, we have
    $t := X^{w + \wt(i) - u - \wt(h)}\in R_\sigma$.
    As $b - \sgn(i) ta = b'$, we conclude that
    $b \equiv b' \pmod I$. \qedhere
  \end{enumerate}
\end{proof}

\begin{lemma}[Parallel edges I]
  \label{lem:parallel_edges_1}
  Let $h,i \in E$ be distinct with $\abs h = \abs i$.
  Suppose that $\wt(h) \le_\sigma \wt(i)$.
  Define a multigraph $\Gamma' := (V,E\setminus\{i\},\norm\dtimes)$
  and a \WSM{} $$\bm\Gamma' := (\Gamma',\sigma,\wt',\sgn'),$$ where
  $\norm\dtimes$, $\wt'$, and $\sgn'$ are the restrictions of $\abs\dtimes$,
  $\wt$, and $\sgn$ to $E\setminus\{i\}$, respectively.
  Suppose that one of the following conditions is satisfied:
  \begin{enumerate*}
  \item $h$ (hence $i$) is a loop.
  \item $\sgn(h) = \sgn(i)$.
  \end{enumerate*}
  Then $\adj(\bm\Gamma;R) = \adj(\bm\Gamma';R)$ for every ring
  $R$; in particular, $\Adj(\bm\Gamma;R) = \Adj(\bm\Gamma';R)$.
\end{lemma}
\begin{proof}
  Write $\abs h = \{u,v\}$. Since $\wt(h) \leq_\sigma \wt(i)$, we have $t:=X^{\wt(i) - \wt(h)}\in R_\sigma$.
  The claim follows since, in each one of the two cases listed,
  \[
    t \dtimes \xcomm u{\wt(h)}{\sgn(h)} v
    = \pm \xcomm u{\wt(i)}{\sgn(i)} v.\qedhere
  \]
\end{proof}

\begin{lemma}[Parallel edges II]
  \label{lem:parallel_edges_2}
  Let $u,v\in V$ be distinct.
  Let $h,i\in E$ with $\abs h = \abs i = \{ u,v\}$,
  $\sgn(h) = - \sgn(i)$, and $\wt(h) \le_{\sigma} \wt(i)$.
  Define a multigraph $\Gamma' := (V,E,\norm\dtimes)$, where
  \[
  \norm e := \begin{cases}
    \,\,\abs e, & \text{if } e\not= i, \\
    \{ v\}, & \text{if } e = i.
    \end{cases}
  \]
  Define $\wt'\colon E\to \ZZ V$ via
  \[
  \wt'(e) := \begin{cases}
    \wt(e), & \text{if }  e \not= i, \\
    u - v + \wt(i), & \text{if } e = i
    \end{cases}
  \]
  Then the following hold:
  \begin{enumerate}
  \item
    $\bm\Gamma' := (\Gamma',\sigma,\wt',\sgn)$ is a \WSM{}.
  \item
    $\adj(\bm\Gamma;R) = \adj(\bm\Gamma';R)$ for every ring $R$
    \underline{in which $2$ is invertible};
    in particular, $\Adj(\bm\Gamma;R) = \Adj(\bm\Gamma';R)$ for such rings
    $R$.
  \end{enumerate}
\end{lemma}
\begin{proof}
\quad
\begin{enumerate}
\item
    For $e \not= i$ and $x\in \norm e = \abs e$,
    we have $x + \wt'(e) = x + \wt(e) \in \sigma^*$.
    As $\bm\Gamma$ is a \WSM{}, $u + \wt(i)$ belongs to $\sigma^*$
    whence $v + \wt'(i) = u + \wt(i)\in \sigma^*$.
  \item
    Let
    \[
      I := 
      \Bigl\langle
      \xcomm x {\wt(e)}{\sgn(e)} y
      : e \in E\setminus\{i\} \text{ with } \abs e = \{x,y\} \Bigr\rangle
      \subset \adj(\bm\Gamma;R) \cap \adj(\bm\Gamma';R).
    \]
    Further define
    \begin{align*}
      a\phantom' & := \xcomm u {\wt(h)}{\sgn(h)} v \in I, \\
      b\phantom' & := \xcomm u {\wt(i)}{\sgn(i)} v
                   \in \adj(\bm\Gamma;R), \text{ and}\\
      b' & := \xcomm v {\wt'(i)}{\sgn(i)} v \in \adj(\bm\Gamma';R)
    \end{align*}
    and note that
    $\adj(\bm\Gamma;R)  = \langle b \rangle + I$ and
    $\adj(\bm\Gamma';R) = \langle b' \rangle + I$.

    Since $\wt(h) \le_\sigma \wt(i)$, we have
    $t := X^{\wt(i)-\wt(h)}\in R_\sigma$.
    As $b + ta = \pm 2b'$, we conclude that
    $\adj(\bm\Gamma;R) = \adj(\bm\Gamma';R)$ whenever $2$ is invertible in
    $R$.
    \qedhere
\end{enumerate}
\end{proof}

By a \emph{spike} of $\bm\Gamma$, we mean a pair $(u,v)$
of distinct vertices of $\Gamma$ such that
\begin{enumerate*}
\item $u$ is the only neighbour of $v$,
\item there is only one edge $e\in E$ with $\abs e = \{u,v\}$, and
\item $u \le_\sigma v$.
\end{enumerate*}

\begin{lemma}[Trimming spikes]
  \label{lem:trim_spikes}
  Let $(u,v)$ be a spike of $\bm\Gamma$.
  Let $h\in E$ be the unique edge with $\abs h = \{u,v\}$.
  Define a multigraph $\Gamma' := (V,E,\norm\dtimes)$, where
  \[
  \norm e := \begin{cases}
    \,\,\abs e, & \text{if } e \not= h, \\
    \{ v\}, & \text{if } e = h.
    \end{cases}
  \]
  Define $\wt'\colon E\to \ZZ V$ via
  \[
  \wt'(e) := \begin{cases}
    \wt(e), & \text{if } e \not= h,\\
    u - v + \wt(h), & \text{if }  e = h.
    \end{cases}
  \]
  \begin{enumerate}
  \item
    $\bm\Gamma' := (\Gamma',\sigma,\wt',\sgn)$ is a \WSM{}.
  \item
    $\Adj(\bm\Gamma;R) \approx_{R_\sigma} \Adj(\bm\Gamma';R)$ for every
    ring $R$. 
    (However, in contrast to the preceding lemmas, $\adj(\bm\Gamma;R)$ and
    $\adj(\bm\Gamma';R)$ may differ.)
  \end{enumerate}
\end{lemma}
\begin{proof}
  \quad
  \begin{enumerate}
  \item
    For $e\in E\setminus\{h\}$ and $x\in \norm e = \abs e$, we have
    $x + \wt'(e) = x + \wt(e) \in \sigma^*$.
    Moreover,
    $v + \wt'(h) = u + \wt(h) \in \sigma^*$ since $\bm\Gamma$ is a \WSM{}.
  \item
    Since $u \le_{\sigma} v$,
    we obtain an $R_\sigma$-module automorphism $\theta$ of $R_\sigma\,V$ 
    given by
    \[
    x \theta = \begin{cases}
      x, & \text{if } x \not= v, \\
      v -\sgn(h)  X^{v-u} u, & \text{if } x = v.
      \end{cases}
    \]
    We now show that $\adj(\bm\Gamma;R)\theta = \adj(\bm\Gamma';R)$; the claim
    then follows immediately.

    Let $e\in E\setminus\{h\}$ with $\abs e = \{x,y\}$.
    Since $(u,v)$ is a spike but $e\not= h$, we have $v\not\in \abs e$.
    Hence, $\theta$ fixes $\xcomm x {\wt(e)}{\sgn(e)} y$ ($=
    \xcomm x {\wt'(e)}{\sgn(e)} y$). The claim follows since $\xcomm u{\wt(h)}{\sgn(h)} v \,\theta
       =
        X^{u+\wt(h)}v
        = \pm \xcomm v{\wt'(h)}{\sgn(h)} v$. \qedhere
  \end{enumerate}
\end{proof}

In order to make use of condition~(\ref{thm:torically_combinatorial3}) in
Theorem~\ref{thm:torically_combinatorial} in our proof of the latter, we will
rely on the following observation.

\begin{lemma}
  \label{lem:negativity_preserved}
  Let the \WSM{} $\bm\Gamma' = (\Gamma',\sigma,\wt',\sgn')$ be derived from
  $\bm\Gamma$ using any one of
  Lemmas~\ref{lem:dominant_loop}--Lemma~\ref{lem:parallel_edges_1} or Lemma~\ref{lem:trim_spikes}.
  If $\sgn \equiv -1$, then $\sgn'\equiv -1$.
  \qed
\end{lemma}

\begin{rem}
  Note that even if the underlying multigraph of a \WSM{} admits
  no parallel edges, each of
  Lemmas~\ref{lem:dominant_loop}--Lemma~\ref{lem:dominant_nonloop_nonloop},
  Lemma~\ref{lem:parallel_edges_2}, or Lemma~\ref{lem:trim_spikes}
  might introduce parallel edges.
\end{rem}

\subsection{Torically torically combinatorial modules are torically combinatorial}
\label{ss:torically_torically}

This section establishes the (intuitively evident)
fact that a torically \{torically combinatorial\} module over a toric ring is
itself torically combinatorial; see
Corollary~\ref{cor:torically_torically_combinatorial}.

\begin{lemma}
  \label{lem:fan_from_cones}
  Let $T$ be a non-empty finite set of cones in $\RR V$.
  Then there exists a fan~$\cF'$ in $\RR V$ such that the following conditions
  are satisfied:
  \begin{enumerate}
  \item
    \label{lem:fan_from_cones1}
    For each $\tau\in T$, there exists $\Sigma \subset \cF'$ with $\tau =
    \bigcup \Sigma$.
  \item
    \label{lem:fan_from_cones2}
    For each $\sigma \in \cF'$, there exists $\tau\in T$ with $\sigma\subset \tau$.
  \item
    \label{lem:fan_from_cones3}
    $\abs{\cF'} = \bigcup T$.
  \end{enumerate}
\end{lemma}
\begin{proof}
  For $x\in \RR V$, define $x^\pm$ and $x^=$ as in the proof of
  Lemma~\ref{lem:total_preorder_refinement}.
  For each $\tau \in T$,
  there exists a non-empty finite set $H_\tau \subset \ZZ V$
  such that $\tau = \bigcap\limits_{h\in H_\tau}h^+$.
  For $h \in H := \bigcup\limits_{\tau\in T}H_\tau$,
  define a complete fan  $\cF_h := \{ h^+,h^{-}, h^=\}$.
  Let $\cF := \bigwedge\limits_{h\in H} \cF_h$
  and 
  $\cF' := \{ \sigma \in \cF : \exists \tau\in T,\, \sigma\subset \tau\}$.
  Since $\cF$ is a fan, so is $\cF'$.
  We claim that $\cF'$ has the desired properties, (\ref{lem:fan_from_cones2})
  being satisfied by construction.
  
  For~(\ref{lem:fan_from_cones1}), let $\tau \in T$.
  Recall that $\tau = \bigcap\limits_{h \in H_\tau}h^+$.
  Let $x\in \tau$ be arbitrary.
  For each $h\in H$, define $\sigma_x(h) \in \cF_h$ via
  \[
    \sigma_x(h) = \begin{cases}
      h^\pm, & \text{ if } x\in h^\pm\setminus h^=,\\
      h^=, & \text{ if } x\in h^=;
    \end{cases}
  \]
  in other words, $\sigma_x(h)$ is the unique cone in $\cF_h$ which contains
  $x$ in its relative interior.
  Let $\sigma_x = \bigcap\limits_{h\in H}\sigma_x(h) \in \cF$ and note that
  $x\in \sigma_x$.
  Since $x\in \tau = \bigcap\limits_{h\in H_\tau}h^+$, 
  for each $h\in H_\tau$, we have
  $\sigma_x(h) \in \{ h^+, h^=\}$ and hence $\sigma_x(h) \subset h^+$.
  Thus,
  \[
    \sigma_x =
    \bigcap_{h\in H} \sigma_x(h)
    \subset
    \bigcap_{h\in H_\tau} \sigma_x(h)
    \subset
    \bigcap_{h\in H_\tau} h^+ = \tau;
  \]
  in particular, $\sigma_x \in \cF'$.  We may thus take
  $\Sigma$ to be the finite (!) set $\{ \sigma_x :x\in
  \tau\}$.  Finally, by
  (\ref{lem:fan_from_cones1})--(\ref{lem:fan_from_cones2}), we have
  $\bigcup T \subset \abs{\cF'}\subset \bigcup T$.
\end{proof}

In particular, we can construct ``fans of fans'' as follows.
\begin{cor}
  \label{cor:fan_of_fans}
  Let $\cF$ be a fan of cones in $\RR V$.
  For each $\sigma\in \cF$, let $\cF_\sigma$ be a fan of cones in $\RR V$ with
  $\card{\cF_\sigma} = \sigma$.
  Then there exists a fan $\cF''$ of cones in $\RR V$ with the following properties:
  \begin{enumerate}
  \item
    \label{cor:fan_of_fans1}
    $\cF''$ refines $\cF$.
  \item
    \label{cor:fan_of_fans2}
    $\abs{\cF''} = \abs{\cF}$.
  \item
    \label{cor:fan_of_fans3}
    For each $\sigma\in \cF$ and $\sigma' \in \cF_\sigma$, there exists
    $\Sigma''\subset \cF''$ with
    $\sigma' = \bigcup \Sigma''$.
  \item
    \label{cor:fan_of_fans4}
    For each $\sigma''\in \cF''$, there exist $\sigma\in \cF$ and $\sigma'\in
    \cF_\sigma$ with $\sigma''\subset \sigma'$.
  \end{enumerate}
\end{cor}
\begin{proof}
  Let $\cF'$ be as in Lemma~\ref{lem:fan_from_cones} with $T :=
  \bigcup\limits_{\sigma\in\cF}\cF_\sigma$.
  Let $\cF'' := \cF \wedge \cF'$.
  The first property holds by definition and the second one since
  $\abs{\cF'} = \bigcup T = \abs{\cF}$.
  For (\ref{cor:fan_of_fans3}), let $\sigma\in \cF$ and $\sigma'\in \cF_\sigma \subset T$.
  By Lemma~\ref{lem:fan_from_cones}(\ref{lem:fan_from_cones1}), there exists
  $\Sigma\subset\cF'$ with $\sigma' = \bigcup \Sigma$.
  As $\sigma' \subset \sigma$,
  we have $\sigma' = \bigcup\Sigma'$, where
  $\Sigma' := \{ \varrho \cap \sigma : \varrho\in \Sigma\} \subset \cF''$.
  For (\ref{cor:fan_of_fans4}), every cone in $\cF''$ is contained in a cone from $\cF'$
  and each cone in $\cF'$ is contained in an element of~$T$.
\end{proof}

\begin{cor}
  \label{cor:torically_torically_combinatorial}
  Let $o \subset \Orth \,V$ be a cone.
  Let $M$ be an $R_o$-module.
  Suppose that $\cF$ is a fan in $\Orth\, V$ with $\card{\cF}
  = o$ such that $M\otimes_{R_o}R_\sigma$ is torically combinatorial
  over $R_\sigma$ for each $\sigma \in \cF$.
  Then $M$ is torically combinatorial as an $R_o$-module.
\end{cor}
\begin{proof}
  By assumption,
  for each $\sigma \in \cF$, there exists a fan $\cF_\sigma$
  with support $\sigma$ such that
  $(M \otimes_{R_o}R_\sigma) \otimes_{R_\sigma} R_\tau \approx_{R_\tau}
  M\otimes_{R_o}R_\tau$ is combinatorial over $R_\tau$ for each $\tau \in \cF_\sigma$.
  Now apply the preceding corollary and note that a change of scalars
  of a combinatorial module along a natural ring map
  $R_\tau\incl R_{\tau'}$ (coming from an inclusion $\tau'\subset \tau$ of cones)
  preserves the property of being combinatorial.
\end{proof}

\subsection[Solitary induction]{Proof of Theorem~\ref{thm:torically_combinatorial}: ``solitary induction''}
\label{ss:proof_torically_combinatorial}

Let $\bm\Gamma = (\Gamma,\sigma,\wt,\sgn)$ be a \WSM{},
where $\Gamma = (V,E,\abs\dtimes)$ and $\sigma \subset \Orth \,V$.

Define $s(\Gamma)$ (the ``social degree'' of $\Gamma$) to be 
the number of non-loops of $\Gamma$ in $E$.
Note that $s(\Gamma) = 0$ if and only if $\Gamma$ is solitary (see
\S\ref{ss:wsm}).
Let $R$ be a ring.
If $2 = 0$ in $R$, then $\Adj(\bm\Gamma;R)$ does not depend on $\sgn(\dtimes)$ at all
so we may assume that $\sgn \equiv -1$ in this case.
To prove Theorem~\ref{thm:torically_combinatorial}, it thus suffices
to show that $\Adj(\bm\Gamma;R)$ is torically combinatorial whenever $\sgn\equiv -1$
or $2$ is invertible in $R$.
We proceed by induction on $s(\Gamma)$.

\paragraph{Base case.}
If $s(\Gamma) = 0$, then $\Gamma$ is solitary and $\Adj(\bm\Gamma;R)$ is
combinatorial (not merely \itemph{torically} combinatorial) by
Proposition~\ref{prop:Adj_solitary}.\\

\noindent Henceforth, suppose that $s(\Gamma) > 0$.

\paragraph{General assumptions and reductions.}
We first carry out a number of general reductions; none of these increases
$s(\Gamma)$.

The following operations amount to
\begin{enumerate*}
\item
  constructing a fan $\cF$ with support $\sigma$ and
\item
  considering the cases obtained by replacing $\sigma$ by a cone
  in $\cF$;
\end{enumerate*}
this strategy is justified by
Lemma~\ref{lem:WSM_Adj_base_change} and
Corollary~\ref{cor:torically_torically_combinatorial}.

Thus, by shrinking $\sigma$ via Lemma~\ref{lem:total_preorder_refinement} and using
Lemma~\ref{lem:parallel_edges_1}, we may assume that $\Gamma$ has no parallel
edges except possibly parallel non-loops with different signs.
By using Lemma~\ref{lem:total_preorder_refinement} to shrink~$\sigma$ yet
further, we may also assume that the preorder $\le_\sigma$ from
\S\ref{ss:cones} is \itemph{total} on elements $x +
\wt(e) \in \ZZ V$ for $e\in E$ and $x\in \abs e$.
This is equivalent to any two incident pairs of $\Gamma$ being comparable
under domination; see \S\ref{ss:surgery} for the latter notion.

\paragraph{Parallel non-loops with opposite signs.}
Suppose that $\Gamma$ has parallel non-loops with opposite signs.
In particular, $\sgn\not\equiv -1$ and we may assume that $2$ is invertible in
$R$.
Let $u,v\in V$ be distinct and let $h,i\in E$ with
$\abs h = \abs i = \{u,v\}$ but $\sgn(h)\not= \sgn(i)$.
Our assumption on dominance of incidence pairs implies that
$\wt(h)\le_\sigma \wt(i)$ or $\wt(h)\ge_{\sigma} \wt(i)$.
Without loss of generality, suppose that we are in the former case.
Let $\bm\Gamma'$ be the \WSM{} obtained from $\bm\Gamma$ using
Lemma~\ref{lem:parallel_edges_2}.
By construction, $s(\Gamma') < s(\Gamma)$.
Indeed, the non-loop $i$ of $\Gamma$ is a loop of $\Gamma'$ and other
edges coincide in the sense that they have the same
support in each multigraph. 
By induction, $\Adj(\bm\Gamma';R) = \Adj(\bm\Gamma;R)$ is torically
combinatorial.\\

\noindent
We may therefore assume that $\Gamma$ has no parallel edges at all.
Recall that we also assume that given any two incident pairs of $\Gamma$, one
of them dominates the other.
Since $s(\Gamma) > 0$, we may choose a connected component, $\Xi$ say, of
$\Gamma$ with $s(\Xi) > 0$.

\paragraph{Using a dominant loop.}
Suppose that $\ell\in E$ is a loop at $u \in V$ in $\Xi$ such that $u.\ell$
dominates each incident pair of $\Xi$. 
Since $\Xi$ is not solitary but connected, it contains a non-loop $h\in E$
with $u\in \abs h$.
Let $\bm\Gamma'$ be the \WSM{} obtained from $\bm\Gamma$ using
Lemma~\ref{lem:dominant_loop}.
Since $h$ is a loop of $\Gamma'$ and all other edges are unchanged as above,
$s(\Gamma') < s(\Gamma)$.
Hence, $\Adj(\bm\Gamma';R) = \Adj(\bm\Gamma;R)$ is torically combinatorial by
induction.

\paragraph{Using a dominant non-loop.}
We may thus assume that $u.h$ is an incident pair of $\Xi$ which
dominates all incident pairs of $\Xi$ and that $h$ is not a loop.
Write $\abs h = \{u,v\}$ and note that $u\le_\sigma v$ since $u.h$ dominates~$v.h$.

For each edge $i\in E\setminus\{h\}$ with $v \in \abs i$,
we then obtain a \WSM{} $\bm\Gamma'$ as in Lemma~\ref{lem:dominant_nonloop_loop} or
Lemma~\ref{lem:dominant_nonloop_nonloop} with $v \not\sim_{\Gamma'} i$ and
such that all other edges of $\Gamma'$ have the same support in $\Gamma$ and $\Gamma'$.
We may repeatedly apply these lemmas to all such edges $i$, one after the
other, to derive a \WSM{} $\tilde{\bm\Gamma}$ with $\Adj(\bm\Gamma;R) = 
\Adj(\tilde{\bm\Gamma};R)$ and such that the underlying graph $\tilde\Gamma$
of $\tilde{\bm\Gamma}$ satisfies
$s(\tilde\Gamma) = s(\Gamma)$.
By construction, $(u,v)$ is then a spike of $\tilde{\bm\Gamma}$.
By deriving $\tilde{\bm\Gamma}'$ from $\tilde{\bm\Gamma}$ via
Lemma~\ref{lem:trim_spikes},
we obtain $s(\tilde\Gamma') < s(\tilde\Gamma) = s(\Gamma)$ whence
$\Adj(\tilde{\bm\Gamma}';R) \approx_{R_\sigma} \Adj(\tilde{\bm\Gamma};R) =
\Adj(\bm\Gamma;R)$ is torically combinatorial by induction. 

\paragraph{Restrictions on $R$.}
We only made use of the assumption that $2$ be invertible in $R$ when we
considered parallel edges with opposite signs.
If all edge signs of a \WSM{} $\bm\Gamma$ are $-1$, then by
Lemma~\ref{lem:negativity_preserved}, the same is true for all the graphs
derived from $\bm\Gamma$ as part of our inductive proof above.
Hence, no restrictions on $R$ are needed in this case and this completes the
proof of Theorem~\ref{thm:torically_combinatorial}. \qed

\begin{rem}
  \label{rem:explicit_computations}
  Given a \WSM{} $\bm\Gamma = (\Gamma,\sigma,\wt,\sgn)$ as above,
  our inductive proof of Theorem~\ref{thm:torically_combinatorial}
  gives rise to a recursive algorithm for constructing
  a fan $\cF$ with support $\sigma$ and
  for each $\tau\in\cF$ a \WSM{} $\bm\Gamma_\tau$ with solitary underlying
  graph such that $\Adj(\bm\Gamma;R) \otimes_{R_\sigma} R_\tau
  \approx_{R_\tau} \Adj(\bm\Gamma_\tau;R)$ for each $\tau \in \cF$ (and
  subject to the assumptions on $R$ from above).
  Together with Proposition~\ref{prop:combinatorial_module_uniformity}
  and the techniques for computing monomial integrals
  from \cite{topzeta,padzeta},  we thus obtain an algorithm for explicitly
  computing the rational functions in Theorem~\ref{thm:graph_uniformity}.
  This algorithm turns out to be quite practical; see \S\ref{s:examples}.
\end{rem}

\begin{rem}\label{rem:char.2}
    In the setting of
    Theorem~\ref{thm:graph_uniformity}(\ref{thm:graph_uniformity3}),
    the arguments developed in this section do not apply to compact \DVR{}s of
    characteristic $2$ due to the factors $\pm 2$ in the penultimate line of the
    proof of Lemma~\ref{lem:parallel_edges_2}.
    Indeed, the conclusion of
    Theorem~\ref{thm:graph_uniformity}(\ref{thm:graph_uniformity3})
    does not generally hold for compact \DVR{}s $\fO$ with residue
    characteristic $2$.
    For example, using either the method from \cite[\S 9.1]{ask} or
    the one developed here (see \S\ref{ss:computer}), we find that
    \begin{dmath*}
      W^+_{\CG_3}(X,T) = 
      \bigl(T^{2} + T + 1 - 3 X^{-1} T^{2} - 6 X^{-1} T + 6 X^{-2} T^{2} + 3
      X^{-2} T - X^{-3} T^{3} - X^{-3} T^{2} - X^{-3} T\bigr)/ (1 - T)^4
      \\
      = 1 + (\underbrace{5 - 6 X^{-1} + 3 X^{-2} - X^{-3}}_{=: g(X)}) T + \mathcal O(T^2).
    \end{dmath*}
    On the other hand,
    a simple calculation shows that  the average size of the kernel of a
    matrix of the form
    \[
      \begin{bmatrix}
        0 & x & y \\
        x & 0 & z \\
        y & z & 0 
        \end{bmatrix}
      \]
      over $\FF_{2^f}$ is given by $h(2^f)$, where
      $h(X) = 1 + X + X^{-2}$;
      note that $g(x) \not= h(x)$ for all real $x > 1$.
      In particular, for each compact \DVR{} $\fO$ with residue field size $q =
      2^f$, the function $W_{\CG_3}^+(q,q^{-s})$ differs from the ask zeta
      function of the positive adjacency representation associated with
      $\CG_3$ over $\fO$.
\end{rem}

\section{Graph operations and ask zeta functions of cographs}
\label{s:models}

In this section, we deduce the Cograph Modelling Theorem
(Theorem~\ref{thm:cograph})
from a structural result (Theorem~\ref{thm:cograph_model})
which relates incidence modules of hypergraphs and adjacency modules of
cographs. 
After collecting some facts about cographs in \S\ref{ss:cographs},
we formally state Theorem~\ref{thm:cograph_model} in
\S\ref{ss:module_comparison}.
We give an outline of the latter theorem's proof in
\S\ref{ss:model_overview} which is then fleshed out in
\S\S\ref{ss:orientation}--\S\ref{ss:proof_model_join}.

Since we will focus exclusively on negative adjacency
representations of simple graphs, in this section,
we frequently omit references to the ``negative'' part.
Throughout, $R$ is a ring, $V$ is a finite set, and
$X = (X_v)_{v\in V}$ consists of algebraically independent variables over
$R$.
All graphs are assumed to be simple in this section.

\subsection{Background on cographs}
\label{ss:cographs}

A \emph{cograph} is a graph which belongs to the smallest class of graphs
which contains isolated vertices and which is closed under both disjoint
unions and joins of graphs.
In this definition, ``joins'' can be replaced by ``taking complements''.
Cographs have appeared in various contexts and under various names
such as ``complement reducible graphs'' and ``$P_4$-free graphs'';
see \cite{CLB81}. 
They admit numerous equivalent characterisations; see \cite[Theorem~2]{CLB81}.
For instance, cographs are precisely those graphs all of whose connected
induced subgraphs have diameter at most $2$.
Moreover, cographs are also precisely those graphs that  do not contain a path
on four vertices as an induced subgraph.

As explained in \cite{CLB81}, each cograph can be
represented by a \emph{cotree}: a rooted tree whose internal vertices
are labelled using one of the symbols $\oplus$ and $\join$
(corresponding to disjoint unions and joins, respectively) and whose
leaves correspond to the vertices of the cograph.
This representation is unique up to isomorphism of rooted
trees (with labelled internal vertices) provided that 
\begin{enumerate*}
\item each internal vertex has at least two descendants and
\item adjacent internal vertices are labelled differently.
\end{enumerate*}

\subsection{Comparing adjacency and incidence modules}
\label{ss:module_comparison}

Generalising the definition of incidence modules $\Inc(\Eta;R)$ in
\S\ref{ss:incidence}, for a hypergraph $\Eta = (V,E,\abs\dtimes)$ and
cone $\sigma \subset \Orth V$,
we let
\[
  \inc(\Eta,\sigma;R) :=
  \Bigl\langle X_v e :  v \sim_\Eta e \,\, (v\in V, \, e\in E) \Bigr\rangle \le R_\sigma E
\]
and we define
the \emph{incidence module} of $\Eta$ with respect to $\sigma$ over
$R$ to be
\[
  \Inc(\Eta,\sigma;R) := \frac{R_\sigma  E}{\inc(\Eta,\sigma;R)}.
\]
Clearly,
\begin{equation}
  \label{eq:Inc_relative}
  \Inc(\Eta,\sigma;R) \approx_{R_\sigma} \Inc(\Eta;R) \otimes_R R_\sigma.
\end{equation}

For a simple graph $\Gamma$ with vertex set $V$ and a cone $\sigma\subset
\Orth V$, we obtain a weighted signed multigraph (see \S\ref{ss:wsm})
$\bm\Gamma := (\Gamma,\sigma,0,-1)$. We set
\[
  \adj(\Gamma,\sigma;R) :=
  \adj(\bm\Gamma;R) = 
  \Bigl\langle
  X_vw - X_wv : v, w \in V, \, v\sim_\Gamma w
  \Bigr\rangle
  \le R_\sigma V
\]
and define the \emph{adjacency module} of $\Gamma$ with respect to
$\sigma$ over $R$ to be
\[
  \Adj(\Gamma,\sigma;R) :=
  \Adj(\bm\Gamma;R) = 
  \frac{R_\sigma V}{\adj(\Gamma,\sigma;R)}.
\]
To further simplify our notation, we let
$\Adj(\Gamma;R) := \Adj(\Gamma,\Orth V;R)$; note that this notation is
consistent with \S\ref{ss:two_adjacencies}.
Observe that
\begin{equation}
  \label{eq:Adj_relative}
  \Adj(\Gamma,\sigma;R) \approx_{R_\sigma} \Adj(\Gamma;R) \otimes_{R[X]} R_\sigma;
\end{equation}
cf.\ Lemma~\ref{lem:WSM_Adj_base_change}.
In the following, we often omit $R$ from our notation in case $R = \ZZ$.

The following is the main result of the present section.
\begin{thm}[Cograph Modelling Theorem: structural form]
  \label{thm:cograph_model}
  Let $\Gamma$ be a cograph.
  Let $C$ be the set of connected components of $\Gamma$.
  Then there exists a hypergraph $\Eta$ with $\verts(\Gamma) = \verts(\Eta)$,
  $\card{\edges(\Eta)}  = \card{\verts(\Gamma)} - \card{C}$,
  and such that $\Adj(\Gamma)$ and $\Inc(\Eta) \oplus \ZZ[X] C$ are torically
  isomorphic $\ZZ[X]$-modules, where $X = (X_v)_{v \in V}$.
\end{thm}

Our proof of this theorem in \S\ref{ss:models} below is based on a
number of algebraic and graph-theoretic techniques developed in the
following.  Our proof is effective: given a cograph $\Gamma$,
we can write down an explicit hypergraph $\Eta$ (a ``model'' of $\Gamma$ in a
sense to be formalised in \S\ref{ss:models}) as in
Theorem~\ref{thm:cograph_model}.
The final piece towards a proof of Theorem~\ref{thm:cograph} is the 
following comparison result for adjacency and incidence
representations.

\begin{lemma}
  \label{lem:global_model_from_union}
  Let $\Gamma$ be a graph and let $\Eta$ be a hypergraph, both
  with common vertex set~$V$.
  Let $c \ge 0$ 
  and suppose that $\card{\edges(\Eta)} = \card{V} - c$.
  Let $\Sigma$ be a
  set of cones with $\bigcup\Sigma = \Orth V$ and such that
  $\Adj(\Gamma,\sigma) \approx_{\ZZ_\sigma} \Inc(\Eta,\sigma) \oplus
  \ZZ_\sigma^c$ for each $\sigma\in \Sigma$.
  Then $W^-_\Gamma(X,T) = W_\Eta(X,T)$.
\end{lemma}
\begin{proof}
  Let $\gamma$ ($=\gamma_-$) and $\eta$ be the adjacency and incidence
  representation of $\Gamma$ and $\Eta$ over $\ZZ$, respectively; see
  \S\S\ref{ss:incidence}--\ref{ss:two_adjacencies}.    
  As always, let $\fO$ be a compact \DVR{}.
  First, for each cone $\sigma\subset \Orth V$ and finitely generated
  $\fO_\sigma$-module $M$, we clearly have
  $\zeta_{M\oplus \fO_\sigma}(s) = \zeta_M(s-1)$.
  Using Lemma~\ref{lem:fan_from_cones}
  and \eqref{eq:Inc_relative}--\eqref{eq:Adj_relative},
  we may assume that $\Sigma$ is a fan
  of cones with support $\Orth V$.
  Now combine Proposition~\ref{prop:Inc_zeta},
  Proposition~\ref{prop:ask_adjacency_module}, and
  Lemma~\ref{lem:inclusion_exclusion_over_fan} to obtain
    \begin{align*}
    \zeta^{\ak}_{\gamma^\fO}(s) =
    (1-q^{-1})^{-1}\zeta_{\Adj(\Gamma)\otimes \fO[X]}(s) =
    (1-q^{-1})^{-1} \zeta_{\Inc(\Eta)\otimes\fO[X]}(s-c) =
      \zeta^{\ak}_{\eta^{\fO}}(s).\quad\qedhere
  \end{align*}
\end{proof}

We may now deduce the version of the Cograph Modelling
Theorem from the introduction.

\begin{proof}[Proof of Theorem~\ref{thm:cograph}]
  Combine Theorem~\ref{thm:cograph_model} and
  Lemma~\ref{lem:global_model_from_union} with $c = \card C$.
\end{proof}

\begin{rem}\label{rem:near_square}
 Assuming the validity of Theorem~\ref{thm:cograph_model}, we actually
 proved a slightly stronger result than Theorem~\ref{thm:cograph}.
 Namely, for each cograph $\Gamma$, there exists a hypergraph~$\Eta$
 on the same set of vertices with $\card{\edges(\Eta)} <
 \card{\verts(\Gamma)}$ and $W^-_\Gamma(X,T) = W_\Eta(X,T)$.  By
 repeated application of \eqref{equ:insert.0.col}, we may assume
 that $\card{\edges(\Eta)} = \card{\verts(\Gamma)} - 1$. The incidence
 matrices of $\Eta$ are then ``near squares'' in the sense that only
 one column is missing from a square.
\end{rem}

We note that the hypothesis of Theorem~\ref{thm:cograph} itself is not
optimal:
\begin{ex}
  The rational function $W^-_{\Path 4}(X,T)$ associated with a path on
  four vertices coincides with $W_\Eta(X,T)$, where $\Eta$ is a
  hypergraph with incidence matrix
  \[
    \begin{bmatrix}
      1 & 1 & 1 \\
      0 & 1 & 0 \\
      0 & 0 & 1 \\
      0 & 0 & 1
    \end{bmatrix}.
  \]
  This can be verified by direct computations; see
  \S\ref{ss:graphs.leq.4}.  We regard examples such as the above as
  evidence that the existence of a toric isomorphism in
  Theorem~\ref{thm:cograph_model} is perhaps a more natural question
  to investigate than coincidence of rational functions.
\end{ex}

Likewise, the conclusion of Theorem~\ref{thm:cograph} does not hold
for arbitrary graphs:

\begin{ex}
  \label{ex:ninja}
  Let $\Gamma$ be the graph
  
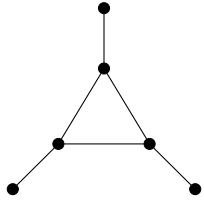
\begin{center}
    \begin{tikzpicture}[scale=0.2]
      \tikzstyle{Black Vertex}=[fill=black, draw=black, shape=circle, scale=0.4]
      \tikzstyle{Solid Edge}=[-]
        \node [style=Black Vertex] (0) at (0, 0) {};
        \node [style=Black Vertex] (1) at (6, 0) {};
        \node [style=Black Vertex] (2) at (3, 5) {};
        \node [style=Black Vertex] (3) at (3, 9) {};
        \node [style=Black Vertex] (4) at (9, -3) {};
        \node [style=Black Vertex] (5) at (-3, -3) {};
        \draw (3) to (2);
        \draw (2) to (0);
        \draw (0) to (1);
        \draw (1) to (2);
        \draw (0) to (5);
        \draw (1) to (4);
      \end{tikzpicture}
    \end{center}
    By an explicit computation using \S\ref{s:uniformity} (see
    \S\ref{ss:computer}), we find that
    \begin{dmath}
      W^-_\Gamma(X,T) = 
      \bigl(-X^{3} T^{4} + 5 X^{2} T^{4} - 6 X^{2} T^{3} + 4 X^{2} T^{2} - 6 X
      T^{4} + 14 X T^{3} - 15 X T^{2} + 4 X T + T^{5} - 5 T^{4} + 5 T^{3} + 5
      T^{2} - 5 T + 1 + 4 X^{-1} T^{4} - 15 X^{-1} T^{3} + 14 X^{-1} T^{2} - 6
      X^{-1} T + 4 X^{-2} T^{3} - 6 X^{-2} T^{2} + 5 X^{-2} T - X^{-3} T\bigr) 
      /\bigl((1 - T)(1 - XT)^3(1 - X^3T^2)\bigr),
      \label{eq:ninja}
    \end{dmath}
    where the numerator and denominator are both factored into irreducibles in
    $\QQ(X)[T]$.
    In view of the quadratic irreducible factor $1 - X^3T^2$  in \eqref{eq:ninja},
    Theorem~\ref{thm:master.intro} shows that $W^-_\Gamma(X,T)$ is not of the
    form $W_\Eta(X,T)$ for \itemph{any} hypergraph $\Eta$.
\end{ex}

\subsection[Informal overview]{Informal overview of the proof of
  Theorem~\ref{thm:cograph_model}}
\label{ss:model_overview}

Let $\Gamma$ be a cograph with vertex set $V$. At the heart of our
constructive proof of Theorem~\ref{thm:cograph_model} lies the notion of
a \itemph{scaffold} on $V$ over a cone $\sigma \subset \R_{\geq 0}V$;
see Definition~\ref{d:scaffold}. Informally, scaffolds 
are forests (i.e.\ disjoint unions of trees) on $V$ with the
same connected components as $\Gamma$. These forests all come with
\itemph{outgoing orientations} given by specifying a root in each of the
forest's trees and letting all edges point away from their
associated root.
Crucially, these orientations are required to be compatible with the preorder
$\leq_\sigma$ on $\Z V$ induced by the cone~$\sigma$; see~\S\ref{ss:cones}.
By shrinking $\sigma$, we may further assume that the restriction of
this preorder to $V$ is total, i.e.\ a weak order.
In addition to the above, the edges of a scaffold carry weights in the form of
subsets of~$V$.
In this way, scaffolds give rise to hypergraphs
and also to weighted signed multigraphs (\WSM{}s; see \S\ref{ss:wsm}) and
adjacency modules.

We say that a scaffold \itemph{encloses} a (co)graph $\Gamma$ over $\sigma$
if the adjacency module of the \WSM{} associated with the scaffold and the
adjacency module of $\Gamma$ (with respect to $\sigma$) coincide in a strong 
sense. In this case, we call the scaffold's hypergraph a \itemph{local model} of
$\Gamma$ over $\sigma$; see Definition~\ref{d:model}(\ref{d:model1}).

A fundamental idea behind our proof of Theorem~\ref{thm:cograph_model} is to
approximate the graph $\Gamma$ by scaffolds attached to various cones. These
cones cover the positive orthant $\R_{\geq 0}V$. Crucially, all scaffolds
realise the same (!) hypergraph $\Eta$ (up to suitable identifications) as
local models of $\Gamma$.  In this case, we call $\Eta$ a \itemph{global
  model} of $\Gamma$; see Definition~\ref{d:model}(\ref{d:model2}).

As cographs (save for singletons) arise as either disjoint unions or joins
of smaller cographs, we are looking to recursively construct (global) models
of disjoint unions and joins of cographs.
The case of disjoint unions is comparatively simple:
Proposition~\ref{prop:model_disjoint_union} establishes that if
$\Gamma_1$ and $\Gamma_2$ are cographs with modelling hypergraphs
$\Eta_1$ and $\Eta_2$, then the disjoint union $\Eta_1\oplus\Eta_2$ is a model
of $\Gamma_1\oplus\Gamma_2$.

The case of joins of (co)graphs, which is settled in
Theorem~\ref{thm:model_join}, is much more involved. In
\S\ref{ss:proof_model_join}, we construct a model for the join of
$\Gamma_1 \vee \Gamma_2$ of $\Gamma_1$ and $\Gamma_2$ by implementing
the following strategy.

We fix a cone $\sigma\subset\R_{\geq 0}V$ and scaffolds
$\cS_i(\sigma)$ enclosing $\Gamma_i$ for $i=1,2$.
In \itemph{Phase 1},
using a process governed by removing, one at a time, suitably chosen
connecting edges between $\Gamma_1$ and $\Gamma_2$, we modify the disjoint
union of the scaffolds $\cS_i(\sigma)$ to obtain a scaffold $\cS^{(N)}$
which ``almost encloses'' the join $\Gamma_1 \join \Gamma_2$;
more precisely, it encloses said join up to factoring out a particular
submodule.
It then remains to consider this ``error term''.

The scaffold $\cS^{(N)}$ differs from the disjoint union of the
scaffolds $\cS_1(\sigma)$ and $\cS_2(\sigma)$ only in the weights borne by its
edges. 
The disjoint union of two scaffolds is, in particular, a disjoint union of two
forests.
In order to obtain a scaffold enclosing the connected~(!) graph $\Gamma_1 \vee
\Gamma_2$, we grow, in \itemph{Phase~2} of our construction, a single oriented
tree out of the two oriented forests comprising $\cS^{(N)}$.
In order to ensure that the resulting scaffold $\cS^{(\infty)}$ has the
desired property of giving rise to a local model of $\Gamma_1\vee\Gamma_2$
over $\sigma$, we graft judiciously chosen (directed and weighted) edges
between pairs of roots from both forests.
A final analysis shows that the given procedure is sufficiently independent of
the many choices made along the way and, crucially, the chosen cone $\sigma$.
In particular, the hypergraph associated with $\cS^{(\infty)}$ essentially
only depends on the graphs $\Gamma_1$ and $\Gamma_2$ and the hypergraphs
associated with the scaffolds $\cS_1(\sigma)$ and $\cS_2(\sigma)$.
This allows us to combine global models of each of $\Gamma_1$ and $\Gamma_2$ into a
global model of the join $\Gamma_1 \join \Gamma_2$.

\subsection{Outgoing orientations of forests}
\label{ss:orientation}

By an \emph{orientation} of a graph $\Gamma = (V,E,\abs\dtimes)$,
we mean a function $\ori\colon E \to V\times V$
which assigns
an ordered pair $(u,v) = \ori(e)$ to each edge $e\in E$ with $\abs e = \{u,v\}$.
We call $u$ and $v$ the \emph{source} and \emph{target} of $e$,
respectively.
We use the notation $u\xto e v$ for an oriented edge $e$ with $\ori(e)
= (u,v)$.

The \emph{indegree} (resp.\ \emph{outdegree}) $\indeg(u)$ (resp.\
$\outdeg(v)$) of $u\in V$ with respect
to an orientation is the number of edges with target (resp.\ source)
$u$. An orientation of $\Gamma$ is \emph{outgoing} if each vertex has
indegree at most one. 

If $\Tau$ is a tree and $u$ is a vertex of $\Tau$, then the
\emph{rooted orientation} of $\Tau$ with root $u$ has all edges
pointing away from $u$. 
More formally, let $e$ be any edge with $\abs{e} = \{ v,w\}$, where $v$
precedes $w$ on the unique simple path from $u$ to $w$.
We then define $v$ to be the source of $e$.
This orientation of $\Tau$ is clearly outgoing.
Trees endowed with such orientations are often referred to as
\emph{arborescences} or \emph{out-trees} in the literature.
Outgoing and rooted orientations of trees are identical concepts:

\begin{prop}[{Cf.\ \cite[\S 3.5]{Gal11}}]
  \label{prop:tree_outgoing_orientation}
  Let $\Tau$ be a tree endowed with an outgoing orientation $\ori$.
  Then $\Tau$ contains a vertex $u$ such that $\ori$ is the rooted orientation
  of $\Tau$ with root~$u$.
\end{prop}
\begin{proof}
  Let $n$ be the number of vertices of $\Tau$.
  Then $\Tau$ contains precisely $n-1$ edges.
  Hence, the sum of the indegrees of all vertices is $n-1$.
  Since $\ori$ is an outgoing orientation,
  we conclude that a unique vertex $u$ has indegree zero;
  see \cite[Theorem~16.4]{Har69}.
  Let $v_1,\dotsc,v_m$ be the distinct neighbours of $u$.
  Let $\Tau_1,\dotsc,\Tau_m$ be the different trees that constitute the
  forest obtained from $\Tau$ by deleting $u$; we assume that $\Tau_i$
  contains $v_i$.
  Then each $\Tau_i$ inherits an outgoing orientation from $\Tau$.
  Moreover, $v_i$ is the unique vertex in $\Tau_i$ with indegree zero.
  By induction, the induced orientation of each $\Tau_i$ is therefore the
  rooted orientation with respect to $v_i$.
  The claim for $\Tau$ then follows immediately.
\end{proof}

In particular, each outgoing orientation of a forest $\Phi$ naturally induces
a partial order $\prec$ on $\verts(\Phi)$.
In detail, vertices $u,v\in \verts(\Phi)$ are comparable if and only if they
belong to the same connected component, $C$ say, and
in that case, $u \prec v$ if and only if $u$ precedes~$v$ on the unique simple
path from the root of $C$ to $v$.
The $\prec$-minimal elements of $\verts(\Phi)$ are exactly the
roots of its connected components.

\subsection{Scaffolds}
\label{ss:scaffolds}

\begin{defn}
  \label{d:scaffold}
  A \emph{scaffold} $\cS =(\Phi,\sigma,\ori,\norm\dtimes)$ on the vertex set
  $V$ over a cone $\sigma\subset \Orth V$  consists of a forest $\Phi =
  (V,E,\abs\dtimes)$ endowed with an 
  outgoing orientation $\ori\colon E \to V\times V$ and a support function
  $\norm\dtimes\colon E \to \Pow(V)$ such that the following
  conditions are satisfied: 
  \begin{enumerate}[label={(\textsf{S\arabic*})}]
  \item
    \label{d:scaffold1}
    For each oriented edge $u \xto e v$ in $\Phi$, we have $u
    \le_\sigma v$ (see \S\ref{ss:cones}).
  \item
    \label{d:scaffold2}
    $\norm e \not= \emptyset$ for each $e\in E$.
  \end{enumerate}
\end{defn}

Given a scaffold $\cS$ as in Definition~\ref{d:scaffold}, 
we obtain a hypergraph $\Eta(\cS) := (V,E,\norm\dtimes)$;
note that $\edges(\Eta(\cS)) = \edges(\Phi) = E$.
Apart from the outgoing orientation,
a scaffold consists of a forest $\Phi$
and a hypergraph $\Eta(\cS)$ related by a common set of
(hyper)edges.
A scaffold~$\cS$ as above also gives rise to a weighted signed multigraph (see
Definition~\ref{d:wsm})
\[
  \bm\Gamma(\cS) := (\Gamma(\cS), \sigma, \wt_{\cS},-1)
\]
constructed as follows:
\begin{enumerate}
\item
  $\Gamma(\cS) := (V, \edges(\cS), \abs\dtimes_{\cS})$,
  where
  \[
    \edges(\cS) := \left\{ (e,w) : e\in \edges(\Phi), w\in \norm e \right\}
  \]
  and $\abs{(e,w)}_{\cS} := \abs  e$ for each $(e,w) \in \edges(\cS)$.

  In other words, $\Gamma(\cS)$ is obtained from the forest $\Phi$ by
  replacing each edge $e$ in $\Phi$ by a set of parallel edges with the same
  support as $e$, one for each element of $\norm e$.
  By condition \ref{d:scaffold2}, the pair $(\Phi,\norm\dtimes)$ and the
  multigraph $\Gamma(\cS)$ determine one another.
\item
  The weight of an edge $(e,w)$ of $\Gamma(\cS)$ for an oriented (!) edge
  $u \xto e v$ of $\Phi$ and $w\in \norm e$ is given by
  $\wt_{\cS}(e,w) := w - u$.
  
  Note that $u + \wt_{\cS}(e,w) = w$ and $v + \wt_{\cS}(e,w) = v + w - u$ both
  belong to $\sigma^*$ (the latter since $u\le_\sigma v$ by \ref{d:scaffold1})
  so that condition \ref{d:wsm3} in Definition~\ref{d:wsm} is satisfied.
\end{enumerate}

For a scaffold $\cS$, we write $\adj(\cS;R) := \adj(\bm\Gamma(\cS);R)$
and $\Adj(\cS;R) := \Adj(\bm\Gamma(\cS);R)$;
as before, we often drop $R$ when $R = \ZZ$.
By definition, for each ring $R$,
\begin{equation}
  \label{eq:adj_scaffold}
  \adj(\cS;R) =
  \Bigl\langle
  X^w v - X^{v + w - u} u
  : u \xto e v \text{ in $\Phi$ and } w \in \norm e
  \Bigr\rangle \le R_\sigma V.
\end{equation}

Lemma~\ref{lem:global_model_from_union} provides a sufficient condition for
equality of ask zeta functions of adjacency and incidence representations.
In order to use this lemma, we need to be able to establish ``toric isomorphisms''
between suitable adjacency and incidence modules.
Scaffolds provide natural examples of such isomorphisms:

\begin{prop}
  \label{prop:Adj_of_scaffold}
  Let $\cS$ be a scaffold as in Definition~\ref{d:scaffold}.
  Let $C$ be the set of connected components of $\Phi$.
  Then $\Adj(\cS;R) \approx_{R_\sigma} \Inc(\Eta(\cS),\sigma;R) \oplus R_\sigma C$.
\end{prop}
\begin{proof}
  By a \itemph{weak scaffold} on $V$, we mean a quadruple $\cS$ as in
  Definition~\ref{d:scaffold} except that $\Phi$ is allowed to be a
  forest on a \itemph{subset} $V_0 := \verts(\Phi)$ of
  $V$.  (This notion will not be used elsewhere.)  The hypergraph
  $\Eta(\cS)$ associated with a weak scaffold has vertex set $V_0$.
  We define $\adj(\cS;R) \le R_\sigma V_0$ as in
  \eqref{eq:adj_scaffold} and $\Adj(\cS;R) := R_\sigma
  V_0/\adj(\cS;R)$.

  We will establish Proposition~\ref{prop:Adj_of_scaffold} for weak scaffolds
  on $V$ by induction on the number of edges of $\Phi$;
  our proof uses  the same idea as in Lemma~\ref{lem:trim_spikes}.
  First, if $\Phi$ contains no edges, then $\card C = \card{V_0}$,
  $\Adj(\cS;R) = R_\sigma V_0 \approx_{R_\sigma} R_\sigma C$, and
  $\Inc(\Eta(\cS),\sigma;R) = 0$.
  
  Now suppose that $\Phi$ contains at least one edge.
  Consider any connected component $\Tau$ (a tree!) of $\Phi$
  consisting of more than one vertex.
  Then $\Tau$ contains a vertex $v$ with indegree $1$ but outdegree $0$
  (i.e.\ a leaf distinct from the root).
  Let $u \xto e v$ be the unique oriented edge with target $v$ in $\Tau$ (and
  in $\Phi$).
  The change of coordinates on $R_\sigma V_0$ given by
  \[
    x \mapsto \begin{cases}
      v + X^{v-u} u, & \text{ if } x = v,\\
      x, & \text{ if } x\in V_0\setminus\{v\}
    \end{cases}
  \]
  maps the defining generator of $\adj(\cS;R)$ corresponding to $e$ and
  $w\in \norm e$ on the right-hand side of \eqref{eq:adj_scaffold}
  to $X^w v$ while preserving generators arising from
  the other edges (as all these have trivial $v$-coordinate).
  Let $\cS'$ be the weak scaffold on $V$ obtained by deleting $v$
  and $e$  from $\Phi$;
  the underlying forest $\Phi'$ of $\cS'$ satisfies $\verts(\Phi') = V_0\setminus\{v\}$.
  Note that
  \[
    \Adj(\cS;R) \approx \Adj(\cS';R) \oplus R_\sigma/\langle X^w :
    w\in \norm e\rangle.
  \]
  Since $v$ is a leaf of $\Tau$ but not a root,
  deleting it did not increase the number of connected components;
  hence, $\Phi'$ has precisely $\card C$ connected components.
  Therefore, 
  $\Adj(\cS';R) \approx_{R_\sigma} \Inc(\Eta(\cS');R) \oplus R_\sigma C$
  by induction and the claim follows using~\eqref{eq:Inc_sum}.
\end{proof}

Proposition~\ref{prop:Adj_of_scaffold} allows us to compare
adjacency and incidence modules associated with scaffolds.
Our next step is to relate the latter to adjacency modules of general graphs.

\begin{defn}
  \label{d:enclosure}
  Let $\Gamma$ be a simple graph and $\cS = (\Phi,\sigma,\ori,\norm\dtimes)$
  be a scaffold as in Definition~\ref{d:scaffold}, both on the same vertex set
  $V = \verts(\Gamma) = \verts(\Phi)$.
  We say that the scaffold $\cS$ \emph{encloses} $\Gamma$ (over the
  underlying cone $\sigma$ of $\cS$) if the following conditions are
  satisfied:
  \begin{enumerate}
  \item
    \label{d:enclosure1}
    $\Gamma$ and the underlying forest $\Phi$ of $\cS$ have the same (vertex
    sets of) connected components.
  \item
    \label{d:enclosure2}
    $\adj(\Gamma,\sigma) = \adj(\cS)$.
  \end{enumerate}
\end{defn}

\begin{ex}[Scaffolds and discrete graphs]
  \label{ex:discrete_enclosure}
  Recall that $\DG_n$ is the discrete graph  on the vertex set $V = \{1,\dotsc,n\}$; see \eqref{def:disc.hyp}.  Trivially, for each cone
  $\sigma \subset \Orth V$, there is a unique scaffold on $V$ with
  underlying forest $\DG_n$ and this scaffold encloses $\DG_n$.
\end{ex}

\begin{ex}[Scaffolds and complete graphs]
  \label{ex:complete_enclosure}
  Consider the complete graph $\CG_n$ on $n$ vertices; see
  \eqref{def:comp.hyp}.  To avoid notational confusion, we also denote
  its vertices by $v_1,\dotsc,v_n$ (where $v_i = i$).  Let
  \[
    \sigma = \{ x \in \Orth^n : x_1 \le x_2,\dotsc,x_n \}.
  \]
  Recall from \eqref{def:star.graph} that $\Star_{n}$ denotes the star
  graph on $\{1,\dotsc,n\}$ with centre $v_1 = 1$.  Let
  $\cS = (\Star_{n},\sigma,\ori,\norm\dtimes)$ be the scaffold on
  $\{1,\dotsc,n\}$, where
  \begin{enumerate*}
  \item the edges of $\Star_{n}$ are oriented in the form $v_1 \to v_i$ and
  \item $\norm e = \{ v_1\}$ for each edge $e$ of $\Star_{n}$.
  \end{enumerate*}
  Note that $\cS$ is indeed a scaffold by our definition of $\sigma$.
  In other words, our orientation of $\Star_n$ is compatible with the preorder
  on vertices induced by $\sigma$.
  Figure~\ref{fig:CG3} depicts the scaffold $\cS$ and the graph $\CG_n$ for $n
  = 3$; edges of the former are labelled by their supports in the associated
  hypergraph.
  
  \begin{figure}[h]
    \centering
    \begin{tikzpicture}
      \tikzstyle{Black Vertex}=[draw=black, fill=blue!15, shape=circle]
      \tikzstyle{Scaffold}=[color=red, line width=.5mm]

      \node [style=Black Vertex] (0) at (0, 0) {$v_1$};
      \node [style=Black Vertex] (1) at (2, 1) {$v_2$};
      \node [style=Black Vertex] (2) at (2, -1) {$v_3$};
      \draw (0) to (1) to (2) to (0);

      \draw[->] [style=Scaffold] (0) edge[bend left,above] node{$\{v_1\}$} (1);
      \draw[->] [style=Scaffold] (0) edge[bend right,below] node{$\{v_1\}$} (2);
    \end{tikzpicture}
    \caption[A scaffold]{A scaffold enclosing $\CG_3$ over the cone $x_1 \le x_2,x_3$}
    \label{fig:CG3}
  \end{figure}

  We claim that $\cS$ encloses $\CG_n$ over $\sigma$;
  our proof of this fact contains ideas
  that will feature in our proof of Theorem~\ref{thm:cograph_model}.

  First note that we may identify
  \[
    \ZZ_\sigma
    = \ZZ[X_1,\dotsc,X_n, X_1^{-1} X_2,\dotsc,X_1^{-1}X_n]
    = \ZZ[X_1, X_1^{-1} X_2,\dotsc,X_1^{-1}X_n].
  \]
  Next,
  \begin{align*}
    \adj(\CG_n,\sigma) &= \left\langle
      X_i v_j - X_j v_i : 1 \le i < j \le n
    \right\rangle \quad\text{and} \\
    \adj(\cS) &= \left\langle
      X_1 v_j - X_j v_1 : 2\le j \le n\right\rangle.
  \end{align*}

  In particular, $\adj(\cS) \subset \adj(\CG_n,\sigma)$.  To show the
  reverse inclusion, let $1 < i < j \le n$.  Then, over $\ZZ_\sigma$
  and modulo $\adj(\cS)$,
  \begin{align*}
    X_i v_j - X_j v_i & \equiv
                        (X_iv_j - X_jv_i) + X_1^{-1}X_j(X_1v_i - X_i v_1)
    \\ &
     = X_i v_j - X_1^{-1} X_i X_j v_1\\
                      & = X_1^{-1} X_i \dtimes (X_1 v_j - X_j v_1)\\
    &\equiv 0 \pmod{\adj(\cS)}.
  \end{align*}
  Thus, $\adj(\cS) = \adj(\CG_n,\sigma)$ and $\cS$ encloses $\CG_n$
  over $\sigma$.
\end{ex}

\begin{ex}[Scaffolds and $\Path 3$]
  \label{ex:path3_enclosure}
  Consider the path $\Gamma = \Path 3$ on three vertices:
  
  
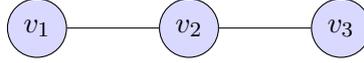
\begin{center}
    \begin{tikzpicture}
      \tikzstyle{Black Vertex}=[draw=black, fill=blue!15, shape=circle]
      \node [style=Black Vertex] (0) at (0, 0) {$v_1$};
      \node [style=Black Vertex] (1) at (2, 0) {$v_2$};
      \node [style=Black Vertex] (2) at (4, 0) {$v_3$};
      \draw (0) to (1) to (2);
    \end{tikzpicture}
  \end{center}

  In the following, we identify $v_i = i$ as in
  Example~\ref{ex:complete_enclosure}.
  For $i = 1,2,3$, let
  \[
    \sigma_i = \left\{ x\in \Orth^3 : x_i \le x_j \text{ for }j = 1,2,3\right\}.
  \]

  We construct scaffolds enclosing $\Path 3$ over $\sigma_i$ for each
  $i = 1,2,3$.  For $i = 2$, let $\cS$ be the scaffold on
  $\{v_1,v_2,v_3\}$ with underlying forest $\Path 3$ oriented in the
  form $v_1 \leftarrow v_2 \rightarrow v_3$ and with supports
  $\norm e = \{v_2\}$ for both edges $e$, as depicted in
  Figure~\ref{fig:P3_middle}.

  \begin{figure}[h]
    \centering
    \begin{tikzpicture}
      \tikzstyle{Black Vertex}=[draw=black, fill=blue!15, shape=circle]
      \tikzstyle{Scaffold}=[color=red, line width=.5mm]
      \node [style=Black Vertex] (0) at (0, 0) {$v_1$};
      \node [style=Black Vertex] (1) at (2, 0) {$v_2$};
      \node [style=Black Vertex] (2) at (4, 0) {$v_3$};
      \draw (0) to (1) to (2);

      \draw[->] [style=Scaffold] (1) edge[bend right,above] node{$\{v_2\}$} (0);
      \draw[->] [style=Scaffold] (1) edge[bend left,above] node{$\{v_2\}$} (2);
    \end{tikzpicture}
    \caption[A scaffold]{A scaffold enclosing $\Path 3$ over the cone $x_2\le x_1,x_3$}
    \label{fig:P3_middle}
  \end{figure}
  Arguments similar to those in
  Example~\ref{ex:complete_enclosure} then show that $\cS$ encloses
  $\Gamma$ over $\sigma_2$.
  Next, by symmetry, the cases $i = 1$ and $i = 3$ are interchangeable; we
  only consider the former.
  Let $\cS$ be the scaffold depicted in Figure~\ref{fig:P3_left}.
  \begin{figure}[h]
    \centering
    \begin{tikzpicture}
      \tikzstyle{Black Vertex}=[draw=black, fill=blue!15, shape=circle]
      \tikzstyle{Scaffold}=[color=red, line width=.5mm]      \tikzstyle{Red Vertex}=[draw=red, shape=circle]

      \node [style=Black Vertex] (0) at (0, 0) {$v_1$};
      \node [style=Black Vertex] (1) at (2, 0) {$v_2$};
      \node [style=Black Vertex] (2) at (4, 0) {$v_3$};
      \draw (0) to (1) to (2);
      \draw[->] [style=Scaffold] (0) edge[bend left,above] node{$\{v_1\}$} (1);
      \draw[->] [style=Scaffold] (0) edge[bend right,below] node{$\{v_2\}$} (2);
    \end{tikzpicture}
    \caption[A scaffold]{A scaffold enclosing $\Path 3$ over the cone $x_1\le x_2,x_3$}
    \label{fig:P3_left}    
  \end{figure}
  Since
  \begin{align*}
    \adj(\Path 3,\sigma)
    & =
      \langle \underbrace{X_1 v_2 - X_2 v_1}_{=:f_1}, \underbrace{X_2 v_3 - X_3 v_2}_{=:f_2}\rangle
    \\
    & =
      \langle f_1, f_2 + X_1^{-1}X_3 f_1\rangle
    \\
    & =\left\langle f_1, X_2\Bigl(v_3 - X_1^{-1}{X_3} v_1\Bigr)\right\rangle
    \\
    & = \adj(\cS),
  \end{align*}
  $\cS$ encloses $\Path 3$ over $\sigma_1$.
\end{ex}

\begin{rem}
  \label{rem:shrink_scaffold}
  Clearly, if $\cS$ encloses $\Gamma$ over $\sigma$ and if $\sigma'\subset
  \sigma$ is a cone,
  then by shrinking the cone of $\cS$, we obtain a scaffold $\cS'$ which
  encloses $\Gamma$ over $\sigma'$.
\end{rem}

\subsection{Models}
\label{ss:models}

\begin{defn}
  \label{d:model}
  \quad
  Let $\Gamma$ and $\Eta$ be a (simple) graph and hypergraph,
  respectively, both on the same vertex set $V$.
  \begin{enumerate}
  \item
    \label{d:model1}
    We say that $\Eta$ is a \emph{local model} of $\Gamma$ over a cone
    $\sigma\subset \Orth V$ if there exists
    a scaffold $\cS$ on $V$ over $\sigma$ which encloses $\Gamma$
    (see Definition~\ref{d:enclosure})
    together with a bijection $\edges(\Eta(\cS)) \xto\phi  \edges(\Eta)$
    such that $\norm e  = \abs{e\phi}_{\Eta}$ for all $e\in
    \edges(\Eta(\cS))$.
  \item
    \label{d:model2}
    We say that $\Eta$ is a \emph{(global) model} of $\Gamma$ is there
    exists a finite set $\Sigma$ of cones with $\bigcup \Sigma = \Orth
    V$ such that $\Eta$ is a local model of $\Gamma$ over each
    $\sigma\in \Sigma$.
  \end{enumerate}
\end{defn}

In other words, $\Eta$ is a local model of $\Gamma$ over $\sigma$ if,
up to relabelling of its hyperedges, $\Eta$
``is'' the hypergraph $\Eta(\cS)$ (see \S\ref{ss:scaffolds}) of some scaffold
$\cS$  enclosing $\Gamma$ over $\sigma$.
In the case of a global model, the particular scaffold and the relabelling of
hyperedges may vary with the particular cone but the hypergraph remains fixed.

\begin{rem}
  \label{rem:may_assume_fan}
  By Lemma~\ref{lem:fan_from_cones} and Remark~\ref{rem:shrink_scaffold}, we
  may equivalently require the set $\Sigma$ in
  Definition~\ref{d:model}(\ref{d:model2}) to be a \itemph{fan} of cones.
\end{rem}

\begin{ex}
  \label{ex:model_discrete_graph}
  Example~\ref{ex:discrete_enclosure} shows that the discrete graph $\DG_n$ is
  a global model of itself.
\end{ex}

\begin{ex}\label{ex:model_complete_graph}
  The block hypergraph $\BE_{n,n-1}$ (see \eqref{def:block.hyp}) is a global
  model of the complete graph $\CG_n$.
  To see that, consider the cover $\Orth^n = \bigcup\limits_{i=1}^n\sigma_i$, where
  \[
    \sigma_i = \{ x\in \Orth^n : x_i \le x_j\text{ for } j =1,\dotsc,n\};
  \]
  here and in the following, we use the notation from
  Example~\ref{ex:complete_enclosure}.
  Let $\cS_i$ be the scaffold over $\sigma_i$ whose underlying
  graph is the star graph on $1,\dotsc,n$ with centre $i$,
  oriented edges of the form $i \to j$ for $i\not= j$, and 
  all hyperedge supports of $\Eta(\cS_i)$ equal to $\{i\}$.
  By Example~\ref{ex:complete_enclosure},
  $\cS_i$ encloses $\CG_n$ over $\sigma_i$.

  Let $\Eta_i$ be a hypergraph with vertices $1,\dotsc,n$ and $n-1$
  hyperedges, each with support~$\{ i\}$.
  Then, up to relabelling of its hyperedges, $\Eta(\cS_i)$ coincides with
  $\Eta_i$ which is therefore a local model of $\CG_n$ over $\sigma_i$.

  Next, by construction, $X_i$ divides each $X_j$ in $\ZZ_{\sigma_i}$.
  In particular, if we redefine all hyperedge supports of $\Eta(\cS_i)$
  to be $\{1,\dotsc,n\}$ instead of $\{i\}$, the resulting scaffold
  still encloses $\Gamma$ over $\sigma_i$.
  Therefore, up to relabelling of its hyperedges, $\Eta(\cS_i)$ coincides
  with~$\BE_{n,n-1}$ for each $i=1,\dotsc,n$.
  We conclude that $\BE_{n,n-1}$ is a global model of $\CG_n$.
\end{ex}

\begin{ex}[A model of $\Path 3$]
  We now construct a global model of $\Path 3$.
  We continue to use the notation from Example~\ref{ex:path3_enclosure}.
  For $i=1,2,3$, let
  \[
    \sigma_i = \{ x\in \Orth^n : x_i \le x_j\text{ for } j =1,\dotsc,3\}.
  \]
  Let $\Eta_i$ be a hypergraph with vertices $1,2,3$ and incidence matrix
  $A_i$, where the rows are ordered naturally and $A_i$ is given by
  \[
    A_1 = \begin{bmatrix} \textcolor{blue}1 & 0\\0&\textcolor{magenta}1\\0&0\end{bmatrix},\qquad
    A_2 = \begin{bmatrix} 0 & 0\\\textcolor{blue}1&\textcolor{magenta}1\\0&0\end{bmatrix},\qquad
    A_3 = \begin{bmatrix} 0 & 0\\0&\textcolor{magenta}1\\\textcolor{blue}1&0\end{bmatrix}.
  \]
  We showed in Example~\ref{ex:path3_enclosure} that $\Eta_i$ is a
  local model of $\Path 3$ over $\sigma_i$ for $i=1,2,3$.  Let $\Eta$
  be a hypergraph with vertices $1,2,3$
  and with incidence matrix
  \[
    A = \begin{bmatrix}\textcolor{blue}1 & 0\\\textcolor{blue}1&\textcolor{magenta}1\\\textcolor{blue}1&0\end{bmatrix}.
  \]
  By redefining hyperedge supports of scaffolds as in
  Example~\ref{ex:model_complete_graph} and using that $X_i$ divides each
  $X_j$ in $\ZZ_{\sigma_i}$, we conclude that $\Eta$ is a global model of
  $\Path 3$.
\end{ex}

\begin{lemma}
  \label{lem:model_no_hyperedges}
  Let $\Eta$ be a local model of $\Gamma$ over some cone $\sigma$.
  Let $\nc$ be the number of connected components of $\Gamma$.  Then
  $\card{\edges(\Eta)} = \card{\verts(\Gamma)} - c$.
\end{lemma}
\begin{proof}
  There exists a scaffold $\cS$ which encloses $\Gamma$ over $\sigma$.
  As the underlying forest, $\Phi$ say, of $\cS$ and $\Gamma$ have the same
  connected components and since a tree on $n$ vertices contains $n-1$ edges,
  the number of hyperedges of $\Eta$ (= number of edges of $\Phi$) is as
  stated.
\end{proof}

\begin{prop}
  \label{prop:global_model_comparison}
  Let $\Eta$ be a global model of $\Gamma$.
  Let $C$ be the set of connected components of $\Gamma$.
  Then $\Adj(\Gamma)$ and $\Inc(\Eta) \oplus \ZZ C$ are torically isomorphic.
 \end{prop}
\begin{proof}
  Let $\Sigma$ be a finite set of cones in $\Orth V$ with $\bigcup
  \Sigma = \Orth V$ and such that $\Eta$ is a local model of $\Gamma$
  over each $\sigma\in \Sigma$.
  By Remark~\ref{rem:may_assume_fan}, we may assume that $\Sigma$ is a fan.
  Fix $\sigma \in \Sigma$.
  Then, up to relabelling of $\edges(\Eta)$,
  $\Eta$ is the hypergraph $\Eta(\cS)$ associated with a scaffold $\cS$ which
  encloses $\Gamma$ over $\sigma$.
  Hence,
  by Proposition~\ref{prop:Adj_of_scaffold},
  $\Adj(\Gamma,\sigma) =\Adj(\cS) \approx_{\ZZ_\sigma}\Inc(\Eta,\sigma) \oplus
  \ZZ_\sigma C$.
\end{proof}

Recall that by a model of a graph, we mean a global one.
We will use this terminology from now on.

\begin{cor}
  \label{cor:global_model_W}
  If $\Eta$ is a model of $\Gamma$, then $W^-_\Gamma(X,T) = W_\Eta(X,T)$.
\end{cor}
\begin{proof}
  Combine Lemma~\ref{lem:global_model_from_union} and
  Proposition~\ref{prop:global_model_comparison}.
\end{proof}

\begin{rem}
  \quad
  \begin{enumerate}
  \item
    Example~\ref{ex:model_complete_graph} and
    Corollary~\ref{cor:global_model_W}
    provide a new proof of the identity
    $\Zeta^{\ak}_{\So_n(\fO)}(T) = \Zeta^{\ak}_{\Mat_{n\times(n-1)}(\fO)}(T)$
    from \cite[Proposition\ 5.11]{ask}.
  \item
    Corollary~\ref{cor:global_model_W} shows that the graph in
    Example~\ref{ex:ninja} does \itemph{not} admit a model.
  \end{enumerate}
\end{rem}

Our definition of models is specifically chosen to allow us to prove
Proposition~\ref{prop:global_model_comparison} and its consequence
Corollary~\ref{cor:global_model_W} as well as
Proposition~\ref{prop:model_disjoint_union} and Theorem~\ref{thm:model_join}. 

\begin{prop}[Models of disjoint unions of graphs]
  \label{prop:model_disjoint_union}
  \hspace*{1em}
  
  \noindent
  Let $\Gamma_i = (V_i,E_i,\abs\dtimes_i)$ be a graph for $i = 1,2$.
  Let $\Eta_i = (V_i,H_i,\norm\dtimes_i)$ be a model of $\Gamma_i$.
  Then $\Eta_1 \oplus \Eta_2$ (see \S\ref{ss:graphs}) is a model of
  $\Gamma_1 \oplus \Gamma_2$.
\end{prop}
\begin{proof}
  We may assume that $V_1\cap V_2 = \emptyset$.
  For $i = 1,2$, there exist a collection of cones $\Sigma_i$ with
  $\bigcup \Sigma_i = \Orth V_i$ and a collection of scaffolds
  $(\cS_i(\sigma_i))_{\sigma_i\in\Sigma_i}$ on $V_i$ such that
  each $\cS_i(\sigma_i)$ encloses $\Gamma_i$ over $\sigma_i$
  and such that each $\Eta(\cS_i(\sigma_i))$ coincides with $\Eta_i$ up to
  relabelling of hyperedges.
  Let $V := V_1\sqcup V_2$
  and $\Sigma := \{\sigma_1\times\sigma_2 : \sigma_i\in\Sigma_i;\,i=1,2\}$
  so that $\Orth V = \bigcup \Sigma$.
  For $\sigma_i\in\Sigma_i$ ($i=1,2$), let $\cS(\sigma_1,\sigma_2)$ be the
  scaffold on $V$ over $\sigma_1\times\sigma_2$ whose
  underlying forest is the disjoint union of the underlying forests of
  $\cS_1(\sigma_1)$ and $\cS_2(\sigma_2)$, whose associated hypergraph is
  the disjoint union of $\Eta(\cS_1(\sigma_1))$ and
  $\Eta(\cS_2(\sigma_2))$, and whose outgoing orientation is induced by
  those of the scaffolds $\cS_i(\sigma_i)$; note that conditions
  \ref{d:scaffold1}--\ref{d:scaffold2} are satisfied here so that
  $\cS(\sigma_1,\sigma_2)$ is a scaffold.
  Next, note that $\adj(\Gamma_1\oplus\Gamma_2,\sigma_1\times\sigma_2)$ is
  generated by (the images of) $\adj(\Gamma_1,\sigma_1)$ and
  $\adj(\Gamma_2,\sigma_2)$.
  By construction, $\cS(\sigma_1,\sigma_2)$ thus encloses
  $\Gamma_1\oplus\Gamma_2$ over $\sigma_1\times\sigma_2$ whence the claim
  follows.
\end{proof}

While formally similar to the preceding proposition,
our next result requires considerably more work;
a proof of the following theorem will be given in~\S\ref{ss:proof_model_join}.

\begin{thm}[Models of joins of graphs]
  \label{thm:model_join}
  \hspace*{1em}

  \noindent
  Let $\Gamma_i$ be a non-empty graph for $i = 1,2$.
  Write $V_i = \verts(\Gamma_i)$.
  Let $\Eta_i = (V_i,H_i,\norm\dtimes_i)$ be a model of $\Gamma_i$.
  Suppose that $V_1\cap V_2 = \emptyset = H_1\cap H_2$.
  Let $\nc_i$ be the number of connected components of $\Gamma_i$.
  Let $f_{ij}$ ($i=1,2$; $j=1,\dotsc,\nc_i-1$) and $g$ be distinct symbols, none
  of which belongs to $H_1 \sqcup H_2$.
  Define a hypergraph $\Eta = (V,H,\norm\dtimes)$, where
  $V := V_1 \sqcup V_2$,
  $H := H_1 \sqcup H_2 \sqcup \{ f_{ij} : i=1,2; j=1,\dotsc,\nc_i-1\} \sqcup \{g\}$,
  and $H\xto{\norm\dtimes} \Pow(V)$ is defined by 
  \[
    \norm h :=
    \begin{cases}
      \norm{h}_1 \sqcup V_2, & \text{if } h\in H_1,\\
       V_1 \sqcup \norm{h}_2, & \text{if } h\in H_2,\\
      V_2, & \text{if } h = f_{1j} \text{ for } j = 1,\dotsc,\nc_1-1, \\
      V_1, & \text{if } h = f_{2j}\text{ for } j = 1,\dotsc,\nc_2-1,  \\
      V_1\sqcup V_2, & \text{if } h = g.
    \end{cases}
  \]
  Then $\Eta$ is a model of the join $\Gamma_1 \join \Gamma_2$ (see
  \S\ref{ss:graphs}) of $\Gamma_1$ and $\Gamma_2$.
\end{thm}

\begin{remark}\label{rem:models.joins.freep}\
  \begin{enumerate}
  \item\label{rem:models.joins.freep.i} The hypergraph
    $\Eta$ in Theorem~\ref{thm:model_join} may be expressed in terms
    of complete unions of hypergraphs as follows. For each
    $i\in\{1,2\}$, let $\Eta^\square_i =
    (\Eta_i)^{\bfz^{(\nc_i-1)}}$; cf.\
    Definition~\ref{def:hyp.inflation}(\ref{hyp.inf.0.col}). Informally
    speaking, $\Eta_i$ and
    $\Eta_i^\square$ coincide except for the multiplicity of the empty
    hyperedge; an incidence matrix of
    $\Eta_i^\square$ may be obtained from an incidence matrix of
    $\Eta_i$ by inserting
    $\nc_i-1$ zero columns. One proves inductively that these are
    $\card{V_i} \times
    (\card{V_i}-1)$-matrices, i.e.\ ``near squares'' (cf.\
    Remark~\ref{rem:near_square}). Then $\Eta=(\Eta^\square_1 \freep
    \Eta_2^\square)^\bfo$; cf.\
    Definition~\ref{def:hyp.inflation}\eqref{hyp.inf.1.col}.

  \item\label{rem:models.joins.freep.ii}
    If $A_i\in\Mat_{n_i\times(n_i-\nc_i)}(\Z)$ are incidence matrices of
    $\Eta_i$, then the following (with $n=n_1+n_2$) is an incidence matrix
    of~$\Eta$:
    \[
        \begin{bmatrix}
          A_1 & \bfo_{n_1\times (n_2-\nc_2)}
          & \bfz_{n_1 \times (\nc_1-1)} & \bfo_{n_1\times (c_2-1)}
          & \bfo_{n_1\times 1} \\
          \bfo_{n_2\times (n_1-\nc_1)} & A_2
          & \bfo_{n_2\times (c_1-1)} & \bfz_{n_2 \times (\nc_2-1)}
          & \bfo_{n_2\times 1}
        \end{bmatrix}
      \in\Mat_{n\times (n-1)}(\Z).
    \]
\end{enumerate}
\end{remark}

\begin{cor}
  \label{cor:cographs_have_models}
  Every cograph admits a model.
\end{cor}
\begin{proof}
  Combine the description of cographs in
  terms of disjoint unions and joins in \S\ref{ss:cographs},
  Example~\ref{ex:model_discrete_graph} (for $n=1$),
  Proposition~\ref{prop:model_disjoint_union},
  and Theorem~\ref{thm:model_join}.
\end{proof}

\begin{proof}[Proof of Theorem~\ref{thm:cograph_model}]
  Combine Corollary~\ref{cor:cographs_have_models} and
  Proposition~\ref{prop:global_model_comparison}.
\end{proof}

\begin{rem}[Canonical models]
  \label{rem:canonical_models}
  Let $\Gamma$ be a cograph represented by a cotree as in~\S\ref{ss:cographs}.
  The uniqueness of cotrees shows that the model, $\Eta(\Gamma)$ say, of
  $\Gamma$ constructed in the proof of
  Corollary~\ref{cor:cographs_have_models} is uniquely determined by 
  $\Gamma$ up to isomorphism of hypergraphs fixing all vertices of the set 
  $\verts(\Gamma) = \verts(\Eta(\Gamma))$.
  We note a cotree representing a given cograph can be computed in linear
  time.
  This was first proved in \cite{CPS85}; see \cite{HP05,BCHP08} for more
  recent and simpler algorithms.
  We conclude that $\Eta(\Gamma)$ can be efficiently constructed from $\Gamma$.

  We note that the hypergraph $\Eta(\Gamma)$ is in fact a graph if and only if
  $\Gamma$ is a disjoint union of copies of $\CG_1$ and $\CG_2$.
  (To see this, suppose that $\Gamma$ is a connected cograph on $n>1$ vertices.
  Then $\Gamma$ is a join of smaller cographs.
  By Theorem~\ref{thm:model_join}, $\Eta(\Gamma)$ then contains a
  hyperedge whose support contains all $n$ vertices.
  Hence, if $\Eta(\Gamma)$ is a graph, then $n = 2$.)
  In particular, if $\Eta(\Gamma)$ is a graph, then $\Eta(\Gamma)\approx
  \Gamma$; cf.\ Example~\ref{ex:model_discrete_graph}.
  
  The hypergraphs of the form $\Eta(\Gamma)$ for cographs $\Gamma$ are rather
  special.
  Recall that, by Corollary~\ref{cor:global_model_W},
  $W^-_\Gamma(X,T) = W_{\Eta(\Gamma)}(X,T)$.
  Let $n = \card{\verts(\Gamma)}$.
  By Lemma~\ref{lem:model_no_hyperedges}, $\card{\edges(\Eta(\Gamma))}
  = n - c$, where $c$ is the number of connected components of $\Gamma$.
 In view of the formula~\eqref{equ:master.intro} for $W_{\Eta(\Gamma)}(X,T)$,
 we note that this implies that
 $\# \{ \card J -\sum_{I\cap J \neq \varnothing} \mu_I : J \subset V\} \leq
 2n-c + 1$.
 (Observe that $c - (n-1) = 1-m \le \card J -\sum_{I\cap J \neq \varnothing}
 \mu_I\le n$ for $\emptyset \not= J \subset V$.)
 As every $y\in \WOhat(V)$ has at most $n+1$ terms, the function
 $W_{\Eta(\Gamma)}(X,T)$ can thus be written over a denominator with
 $\mathcal O(n^2)$ factors of the form $1 - X^A T$;
 for general hypergraphs on $n$ vertices, we obtain
 an upper bound of $2^n$ such factors.
 For another restriction,
 \[
   \sum_{e\in \edges(\Eta(\Gamma))}\# \norm{e}_{\Eta(\Gamma)}
   = 2 \card{\edges(\Gamma)}; 
 \]
 in particular, the number of non-zero entries in any incidence matrix of
 $\Eta(\Gamma)$ is even.
\end{rem}

\def\exaK3K3K2pt3{Example~\ref{exa:K3_K3_K2}, part III}
\begin{example}[\exaK3K3K2pt3]
  \label{exa:K3_K3_K2_pt3}
  We resume the story begun in Example~\ref{exa:K3_K3_K2}.
  Let the graph $\Gamma \approx (\CG_3 \oplus \CG_3) \join \CG_2$ and
  hypergraph $\Eta$ be as defined there.  As we observed in
  Example~\ref{exa:K3_K3_K2_pt2},
  $\Eta \approx \left(\BE_{3,2} \oplus \BE_{3,2}\right)^\bfz\freep
  \BE_{2,2}$.  By Example~\ref{ex:model_complete_graph}, the block
  hypergraph $\BE_{3,2}$ (resp.\ $\BE_{2,1}$) is a model of the
  complete graph $\CG_3$ (resp.\ $\CG_2$).  Therefore, by
  Proposition~\ref{prop:model_disjoint_union}, the disjoint union
  $\BE_{3,2} \oplus \BE_{3,2}$ is a model of $\CG_3\oplus \CG_3$.  By
  Theorem~\ref{thm:model_join} (see also
  Remark~\ref{rem:models.joins.freep}(\ref{rem:models.joins.freep.i})),
  \[
    ((\BE_{3,2} \oplus \BE_{3,2})^{\bfz} \freep \BE_{2,1})^{\bfo} \approx
    (\BE_{3,2} \oplus \BE_{3,2})^{\bfz} \freep \BE_{2,2}
    \approx \Eta
  \]
  is a model of $\Gamma$.
  By Corollary~\ref{cor:global_model_W}, $W^-_\Gamma(X,T) =
  W_\Eta(X,T)$ is thus given by \eqref{eq:K3_K3_k2_pt2}.
\end{example}

\subsubsection{Models and the addition of one vertex}

Before we turn to the proof of Theorem~\ref{thm:model_join} we record,
for later use, the effects on models of taking disjoint unions and
joins with a simple graph on a single vertex.

We denote, more precisely, by $\bullet = \CG_1 = \Delta_1$ a fixed
(simple) graph on one vertex and study the effects of the operations
$\Gamma \leadsto \Gamma \oplus \Point$ and
$\Gamma \leadsto \Gamma \join \Point$. Given a hypergraph $\Eta$, set
$\Eta^\bfz_\bfz = \left(\Eta^\bfz\right)_{\bfz} =
\left(\Eta_\bfz\right)^{\bfz}$ and likewise
$\Eta^\bfo_\bfo = \left(\Eta^\bfo\right)_{\bfo} =
\left(\Eta_\bfo\right)^{\bfo}$; cf.\
Definition~\ref{def:hyp.inflation}.

\begin{proposition}\label{prop:models.bullet}
  Let $\Eta$ be a model of $\Gamma$.
  \begin{enumerate}
  \item $\Eta^\bfz_\bfz$ is a model of $\Gamma\oplus \bullet$,
  \item\label{eq:prop.join}
    Let $c$ be the number of connected components of $\Gamma$.
    Then $(\Eta^{\bfz^{(c-1)}})^\bfo_\bfo$ 
    is a model of $\Gamma\vee \bullet$.
  \end{enumerate}
\end{proposition}

\begin{proof}
  This follows from Proposition~\ref{prop:model_disjoint_union} (for $\oplus$)
  and Theorem~\ref{thm:model_join} (for $\join$).
\end{proof}

\begin{cor}\label{cor:Gamma.point}
  \quad
  \begin{enumerate}
  \item\label{eq:sum}
    For each graph $\Gamma$, we have
    $W^-_{\Gamma\oplus \bullet}(X,T) =
    W^-_\Gamma(X,XT)$,
  \item\label{eq:join}
    For each cograph $\Gamma$, we have
    $W^-_{\Gamma\vee\bullet}(X,T)
    = \frac{1-X^{-1}T}{1-XT} \, W^-_\Gamma(X,X^{-1}T)$.
  \end{enumerate}
\end{cor}

\begin{proof}
  Let $\cX = (X_u)_{u\in \verts(\Gamma\oplus\bullet)}$ consist of
  algebraically independent elements over $\ZZ$.
  (These elements are unrelated to the ``$X$'' in
  Corollary~\ref{cor:Gamma.point}.)
  It follows immediately from the definition of adjacency modules
  (see \S\ref{ss:module_comparison} and cf.\ \S\ref{ss:two_adjacencies}) that
  \[
    \Adj(\Gamma\oplus\bullet) \approx_{\ZZ[\cX]}
    \Adj(\Gamma)^{\ZZ[\cX]} \oplus \ZZ[\cX].
  \]
  Part \eqref{eq:sum} now follows easily from
  Proposition~\ref{prop:ask_adjacency_module}---we note that we used a 
  similar argument in our proof of Lemma~\ref{lem:global_model_from_union}.
  For \eqref{eq:join}, combine
  Proposition~\ref{prop:models.bullet}\eqref{eq:prop.join} with
  \eqref{equ:insert.1.row} and~\eqref{equ:insert.1.col}.
\end{proof}
  
\begin{remark}
  The question whether the assumption in~\eqref{eq:join} that $\Gamma$ be a
  cograph is unnecessary is generalised in Question~\ref{que:join}.
\end{remark}
  
\subsection[Proof]{Proof of Theorem~\ref{thm:model_join}}
\label{ss:proof_model_join}

At this point, there is but one missing piece towards our proof
of Theorem~\ref{thm:cograph_model} (and Theorem~\ref{thm:cograph}),
namely Theorem~\ref{thm:model_join}, whose notation we now adopt.
Write $\Gamma := \Gamma_1 \join \Gamma_2$.

\subsubsection{Phase 0: setup}
\label{sss:phase0}

\paragraph{Suitable collections of cones.}
Similar to the proof of Proposition~\ref{prop:model_disjoint_union},
we obtain a collection $\Sigma$ of cones with
$\bigcup \Sigma = \Orth V$ and, for each $\sigma\in \Sigma$ and
$i = 1, 2$, a scaffold $\cS_i = \cS_i(\sigma)$ on $V_i$ which
encloses $\Gamma_i$ over the image $\sigma_i$ of $\sigma$ under the
projection $\Orth V \onto \Orth V_i$ such that $\Eta(\cS_i(\sigma))$
coincides with $\Eta_i$ up to relabelling of hyperedges.
Using Lemma~\ref{lem:total_preorder_refinement},
Lemma~\ref{lem:fan_from_cones}, and Remark~\ref{rem:shrink_scaffold} to modify
$\Sigma$ if necessary, we may further assume that $\le_\sigma$ (see
\S\ref{ss:cones}) induces a total preorder on $V$ for each $\sigma \in
\Sigma$.

Note that in contrast to our proof of
Proposition~\ref{prop:model_disjoint_union},
we do not assume that each $\sigma\in \Sigma$ is of the form $\sigma =
\sigma_1\times \sigma_2$.
However, $\sigma \subset \sigma_1\times \sigma_2$ which allows us to
identify $\ZZ_{\sigma_i} \subset \ZZ_{\sigma_1\times\sigma_2} \subset
\ZZ_\sigma$ and also e.g.\ $\ZZ_{\sigma_i}V_i \subset \ZZ_\sigma V$.

\paragraph{A fixed cone.}
Henceforth, let $\sigma \in \Sigma$ be fixed but arbitrary.
It suffices to construct a scaffold $\cS$ which encloses $\Gamma$ over
$\sigma$ and whose associated hypergraph $\Eta(\cS)$ coincides with~$\Eta$ (as
defined in the statement of Theorem~\ref{thm:model_join}) up to relabelling hyperedges.

Write $\cS_i := \cS_i(\sigma) = (\Phi_i,\sigma_i,\ori_i,\norm\dtimes_i)$
and $\Phi_i = (V_i,E_i,\abs\dtimes_i)$.
We may assume that $E_1\cap E_2=\emptyset$.
For $u_i\in V_i$ and $v_j \in V_j$, we use the suggestive notation $u_iv_j$
both for $(u_i,v_j)\in V_i\times V_j$ and, if $i = j$ and $u_i\sim
v_i$ in $\Phi_i$, for the oriented edge $u_i\to v_i$ of $\Phi_i$.

\paragraph{Strategy.}
For $u,v\in V$, let $[uv] := X^u v - X^v u\in \ZZ_\sigma V$.
(Hence, using the notation from~\S\ref{ss:two_adjacencies}, if $u\not= v$,
then $[uv] = [u,v;-1]$.)
For $A \subset V_1\times V_2$, write $[A] := \{ [v_1v_2] : v_1v_2\in A\}$.
Let $\cS^{(0)} = (\Phi^{(0)}, \sigma,\ori^{(0)}, \norm\dtimes^{(0)})$ be the
``disjoint union'' of $\cS_1$ and $\cS_2$ as constructed in the proof of
Proposition~\ref{prop:model_disjoint_union}.
This is a scaffold enclosing the disjoint union $\Gamma_1\oplus\Gamma_2$ over
$\sigma$.
The underlying forest $\Phi := \Phi^{(0)} = \Phi_1\oplus \Phi_2$ satisfies
$\edges(\Phi^{(0)}) = E_1 \sqcup E_2$ with the evident support function.

Let $M^{(0)} := V_1\times V_2$.
Recall that $\sigma_i$ denotes the image of $\sigma$ under the projection
$\Orth V\onto \Orth V_i$.
Since $\cS_i$ encloses $\Gamma_i$ over $\sigma_i$ and we identify
$\ZZ_{\sigma_i} \subset \ZZ_\sigma$ as above,
\begin{align*}
  \adj(\Gamma,\sigma)
  & =
    \langle \adj(\Gamma_1,\sigma_1)\rangle
    + \langle \adj(\Gamma_2,\sigma_2)\rangle
    + \langle [ M^{(0)}]\rangle \\
  & =
    \langle \adj(\cS_1)\rangle
    + \langle \adj(\cS_2)\rangle
    + \langle [ M^{(0)}]\rangle \\
  & =
    \adj(\cS^{(0)}) + \langle [ M^{(0)}]\rangle 
\end{align*}

Beginning with $\cS^{(0)}$ and $M^{(0)}$, in the following,
we use graph-theoretic operations to construct
a finite sequence of scaffolds $\cS^{(n)}$ over $\sigma$ and sets
$M^{(n)}\subset V_1\times V_2$ such that 
\[
  \adj(\Gamma,\sigma) = \adj(\cS^{(n)}) + \langle [ M^{(n)}]\rangle 
\]
for each $n\ge 0$.
The very last of these will satisfy $M^{(\infty)} = \emptyset$ and
$\cS^{(\infty)}$ will enclose $\Gamma$ over~$\sigma$.
Moreover, by construction the hypergraph $\Eta(\cS^{(\infty)})$ will coincide with
$\Eta$ as defined in Theorem~\ref{thm:model_join} up to relabelling of
hyperedges.

\subsubsection{Phase 1: eliminating non-radical joining edges}

In order to construct $\cS^{(n+1)}$ from $\cS^{(n)}$, we will employ the
following observation based on the same idea as
Lemma~\ref{lem:dominant_nonloop_nonloop}. Recall that we assume throughout that $\le_\sigma$ induces a total preorder on $V$.

\begin{lemma}[``Triangle reduction'']
  \label{lem:triangle}
    Let $\cS = (\Phi,\sigma,\ori,\norm\dtimes)$ be a scaffold on $V$.
  Let $u\to v$ be an oriented edge in $\Phi$.
  Let $M\subset V \times V$,
  let $z\in V$, and suppose that $uz, vz\in M$.
  Let $M' := M\setminus\{vz\}$.
  Define a scaffold $\cS' = (\Phi,\sigma,\ori,\norm\dtimes')$ via
  \[
    \norm{h}' = \begin{cases}
      \norm h, & \text{if }h \not= uv, \\
      \norm{uv} \cup \{ z\}, & \text{if }h = uv.
    \end{cases}
  \]
  Then
  $\adj(\cS) + \langle [M] \rangle
    =
    \adj(\cS') + \langle [M']\rangle$.
\end{lemma}
\begin{proof}
  By condition \ref{d:scaffold1} in Definition~\ref{d:scaffold},
  we have $u \le_\sigma v$. 
  Condition \ref{d:scaffold2} allows us to choose a vertex $w\in \norm{uv}$.
  Define
  \begin{align*}
    g &:= X^{w} v - X^{v + w - u} u \in \adj(\cS), \\
    g' &:= X^{z} v - X^{v + z - u} u \in \adj(\cS'),
  \end{align*}
  and note that $\adj(\cS') = \adj(\cS) + \langle g'\rangle$.
  Further observe that
  \begin{equation}
    \label{eq:vzg'}
    g' + [vz] = X^{v-u} [uz].
  \end{equation}
  We claim that $\langle g, [uz], [vz]\rangle = \langle g, g', [uz]\rangle$
  over $\ZZ_\sigma$.
  To prove that, we consider two cases.
  First suppose that $w \le_\sigma z$.
  Then $g' = X^{z-w} g$ over $\ZZ_\sigma$ 
  and \eqref{eq:vzg'} implies that
  $\langle g, [uz], [vz] \rangle = \langle g, [uz]\rangle = \langle g, g', [uz] \rangle$.
  Next, suppose that $z \le_\sigma w$.
  Then $$g + X^{w-z} [vz] = X^{(v-u) + (w-z)} [uz]$$ over $\ZZ_\sigma$ whence
  $\langle g,[uz],[vz]\rangle = \langle [uz],[vz]\rangle$.
  Moreover, by \eqref{eq:vzg'} and since $g = X^{w-z} g'$ over $\ZZ_\sigma$,
  we have
  $\langle [uz], [vz]\rangle = \langle g', [uz] \rangle = \langle g, g', [uz]\rangle$.
  The claim now follows since
  \begin{align*}
    \adj(\cS) + \langle [M]\rangle
    & = \adj(\cS) + \langle g, [uz],[vz]\rangle + \langle [M']\rangle \\
    & = \adj(\cS) + \langle g, g', [uz]\rangle + \langle [M']\rangle \\
    & = \adj(\cS') + \langle[M']\rangle. \qedhere
  \end{align*}
\end{proof}
\begin{rem}
  Since $[zu] = -[uz]$, the preceding lemma remains true if $uz$ is
  replaced by $zu$ or $vz$ is replaced by $zv$.
\end{rem}

\paragraph{Invariants.}
Let $\prec_i$ be the natural partial order induced on $V_i$ by the
given orientation $\ori_i$ on $\Phi_i$; see \S\ref{ss:orientation}.
Let $\prec \,\,:=\, \prec_1\!\times\!\prec_2$ be the product order on
$V_1\times V_2$.  Recall that $E_i$ denotes the edge set of $\Phi_i$.
Suppose that
\begin{itemize}[label={$\circ$}]
\item scaffolds $\cS^{(0)}, \dotsc, \cS^{(n)}$ on $V$,
\item subsets
  $V_1 \times V_2 = M^{(0)} \supset M^{(1)} \supset \dotsb \supset
  M^{(n)}$, and
\item elements $v_i^{(0)}, \dotsc,v_i^{(n-1)}\in V_i$ for
  $i=1,2$
\end{itemize} have been constructed and that the following conditions
are satisfied for $\ell = 1,\dotsc,n$:
\begin{enumerate}[label={(\textsf{M\arabic*})}]
\item
  \label{modelseq_1}
  $\cS^{(\ell)} = (\Phi, \sigma,\ori,\norm\dtimes^{(\ell)})$.
\item
  \label{modelseq_2}
  $M^{(\ell)}$ is downward closed with respect to $\prec$.
\item
  \label{modelseq_3}
  $M^{(\ell)}$ contains all $\prec$-minimal elements of $V_1\times V_2$.
\item
  \label{modelseq_4}
  $v_1^{(\ell-1)}v_2^{(\ell-1)} \in M^{(\ell-1)}$
  and $M^{(\ell)} = M^{(\ell-1)} \setminus \{ v_1^{(\ell-1)}v_2^{(\ell-1)}\}$.
\item
  \label{modelseq_5}
  For $i+j = 3$, if $u_i v_i \in E_i$, then
  $\norm{u_iv_i}^{(\ell-1)} \subset \norm{u_iv_i}^{(\ell)} \subset
  \norm{u_iv_i}^{(\ell-1)} \cup \{ v_j^{(\ell-1)}\}$.
\item
  \label{modelseq_6}
  There exist $i\in \{1,2\}$
  and an edge $u_i \to v_i^{(\ell-1)}$ in $\Phi_i$
  with $v_j^{(\ell-1)} \in \norm{u_iv_i^{(\ell-1)}}^{(\ell)}$ for $i+j = 3$.
\item
  \label{modelseq_7}
  $\adj(\Gamma,\sigma) = \adj(\cS^{(\ell)}) + \langle [ M^{(\ell)}]\rangle$.
\end{enumerate}

Some comments on these conditions are in order.
Formalising the strategy in \S\ref{sss:phase0}, condition \ref{modelseq_7}
asserts that $\adj(\Gamma,\sigma)$, the module of primary interest to us,
coincides with $\adj(\cS^{(\ell)})$, except for an error measured by
$M^{(\ell)}$.
As outlined earlier, the objective of our construction is to eventually
eliminate this error term.

Condition \ref{modelseq_4} states that $M^{(\ell)}$ is obtained from
$M^{(\ell-1)}$ by removing a single distinguished pair
$v_1^{(\ell-1)}v_2^{(\ell-1)}$.
In our construction, these distinguished pairs will be chosen among
$\prec$-maximal elements of $M^{(\ell-1)}$;
when working with the latter, \ref{modelseq_2} will be crucial.

Condition \ref{modelseq_1} asserts that $\cS^{(\ell)}$ only differs from
$\cS^{(0)}$ in its support function.
This is made more precise by
\ref{modelseq_5} which asserts that if $e$ is any edge in $\Phi_i$, then
the $\norm\dtimes^{(\ell)}$-support of $e$ coincides with its
$\norm\dtimes^{(\ell-1)}$-support except possibly for the addition of the
distinguished vertex $v_j^{(\ell-1)}$---here and in \ref{modelseq_6},
$j\in\{1,2\}$ is the ``index distinct from $i$'', captured succinctly by the
identity ``$i + j = 3$''.
Together with \ref{modelseq_5}, condition~\ref{modelseq_6} asserts that there
is an oriented edge $e$ with target $v_i^{(\ell-1)}$ in some~$\Phi_i$ such that
$\norm e^{(\ell)} = \norm e^{(\ell-1)} \cup \{v_j^{(\ell-1)}\}$. 
Finally,~\ref{modelseq_3} will guarantee that after finitely many steps,
$M^{(\ell)}$ stabilises at the set of minimal elements of $M^{(0)} = V_1\times
V_2$; this will conclude Phase~1.

\paragraph{Removing non-minimal pairs.}

Let $R_i$ denote the set of $\prec_i$-minimal elements of $V_i$;
note that this is precisely the set of roots of the connected components of $\Phi_i$.
Clearly, $R_1\times R_2$ is the set of $\prec$-minimal elements of $V_1\times V_2$.

Suppose that $M^{(n)}\supsetneq R_1\times R_2$ and
choose a non-minimal pair $v'_1v'_2 \in M^{(n)}\setminus (R_1\times R_2)$.
Let $v_1v_2$ be any $\prec$-maximal element of $M^{(n)}$ with $v_1'v_2' \prec
v_1v_2$; note that $v_1v_2\not\in R_1\times R_2$.
Define $v_1^{(n)}v_2^{(n)} := v_1v_2$ and $M^{(n+1)} := M^{(n)} \setminus
\{ v_1v_2\}$; clearly, $M^{(n+1)}$ is downward closed and $M^{(n+1)}\supset
R_1\times R_2$.

Next, we construct $\cS^{(n+1)}$.
Without loss of generality, suppose that $v_1\not\in R_1$.
(If both $v_1\not\in R_1$ and $v_2\not\in R_2$, we proceed as in the
following.) 
Let $u_1 \in V_1$ be the (unique) $\prec_1$-predecessor of $v_1$.
Define $\norm\dtimes^{(n+1)}\colon E_1\sqcup E_2 \to \Pow(V)$ via
\[
  \norm h^{(n+1)} :=
  \begin{cases}
    \norm h^{(n)}, & \text{if } h \not= u_1 v_1, \\
    \norm {u_1v_1}^{(n)} \cup \{ v_2\}, & \text{if } h = u_1v_1.
  \end{cases}
\]
Let $\cS^{(n+1)} := (\Phi,\sigma,\ori,\norm\dtimes^{(n+1)})$.
Then \ref{modelseq_1}--\ref{modelseq_6} are clearly satisfied for $\ell =
n+1$.

Since $v_1 v_2 \in M^{(n)}$ and $M^{(n)}$ is downward closed, $u_1 v_2$ belongs to
$M^{(n)}$ and also to $M^{(n+1)}$. 
It thus follows from Lemma~\ref{lem:triangle} that \ref{modelseq_7} is
satisfied for $\ell = n+1$.

\paragraph{Changing support.}
Since each $M^{(n+1)}$ is a proper subset of $M^{(n)}$ and both of these are
supersets of $R_1\times R_2$, the above construction terminates after finitely
many steps when $M^{(N)}= R_1\times R_2$ for some $N \ge 0$.
A key property of $\cS^{(N)}$ is the following:

\begin{lemma}
  \label{lem:norm_contains_roots}
  Let $i + j = 3$.
  Then for each oriented edge $u_i \to v_i$ in $\Phi_i$,
  \[  \norm{u_iv_i}^{(0)} \cup R_j \subset \norm{u_i v_i}^{(N)} \subset
    \norm{u_iv_i}^{(0)} \cup V_j.
  \]
\end{lemma}
\begin{proof}
  The second inclusion is immediate from \ref{modelseq_5}.
  To prove the first inclusion, we assume, without loss of generality, that
  $i = 1$ and $j = 2$. (When $i = 2$ and $j = 1$, we only need to suitably
  reverse ordered pairs in the following.)
  Let $r_2 \in R_2$ be arbitrary.
  It suffices to show that $r_2\in \norm{u_1v_1}^{(N)}$.
  Since $v_1\not\in R_1$ (because $u_1\prec_1 v_1$)
  and $M^{(N)} = R_1\times R_2$,
  we have $v_1 r_2 \in M^{(0)}\setminus M^{(N)}$.
  Hence, by \ref{modelseq_4}, for some $\ell \in \{1,\dotsc,N\}$,
  we have $v_1 r_2 = v_1^{(\ell-1)} v_2^{(\ell-1)}$.
  As $r_2$ is a root (= $\prec_2$-minimal element) of one of the connected
  components of $\Phi_2$, the edge $u_i\to v_i^{(\ell-1)}$ in \ref{modelseq_6}
  has to be the given edge $u_1\to v_1$.
  In particular, \ref{modelseq_5}--\ref{modelseq_6} imply that
  $r_2 = v_2^{(\ell-1)} \in \norm{u_1v_1}^{(\ell)} \subset \norm{u_1v_1}^{(N)}$.
\end{proof}

To proceed further, we need another lemma.

\begin{lemma}[``Support addition'']
  \label{lem:support_addition}
  Let $\cS = (\Phi,\sigma,\ori,\norm\dtimes)$ be a scaffold on $V$.
  Let $u\to v$ be an oriented edge of $\Phi$.
  Let $w\in \norm{uv}$, let $z \in V$, and suppose that $w\le_\sigma z$.
  Define a scaffold $\cS' = (\Phi,\sigma,\ori,\norm\dtimes')$ via
  \[
    \norm{h}' = \begin{cases}
      \norm h, & h \not= uv, \\
      \norm{uv} \cup \{z\}, & h = uv.
    \end{cases}
  \]
  Then $\adj(\cS) = \adj(\cS')$.
\end{lemma}
\begin{proof}
  Define
  \begin{align*}
    g &:= X^{w} v - X^{v + w - u} u \in \adj(\cS), \\
    g' &:= X^{z} v - X^{v + z - u} u \in \adj(\cS')
  \end{align*}
  so that $\adj(\cS') = \adj(\cS) + \langle g'\rangle$.
  By assumption, $g' = X^{z-w} g \in \adj(\cS)$ over $\ZZ_\sigma$.
\end{proof}

Define $\norm\dtimes^{(N+1)}\colon E_1\sqcup E_2 \to \Pow(V)$ via
\[
  \norm h^{(N+1)} :=
  \begin{cases}
    \norm h_1 \sqcup V_2, & \text{if } h \in E_1,\\
    V_1 \sqcup \norm h_2, & \text{if } h \in E_2.
  \end{cases}
\]
Let $\cS^{(N+1)} := (\Phi,\sigma,\ori,\norm\dtimes^{(N+1)})$ and
$M^{(N+1)} := M^{(N)} = R_1\times R_2$.

\begin{cor}
  \label{cor:SN+1}
  $\adj(\Gamma,\sigma) = \adj\!\left(\cS^{(N+1)}\right) + \langle [M^{(N+1)}]\rangle$.
\end{cor}
\begin{proof}
  By \ref{modelseq_7},
  it suffices to show that $\adj(\cS^{(N)}) = \adj(\cS^{(N+1)})$.
  Let $i + j = 3$ and let $h$ be any oriented edge of $\Phi_i$.
  Let $z_j \in V_j$ be arbitrary.
  Let $r_j \in R_j$ be the root of the connected component of $\Phi_j$ which
  contains $z_j$.
  By condition \ref{d:scaffold1} in Definition~\ref{d:scaffold},
  $r_j \le_\sigma z_j$.
  By Lemmas~\ref{lem:norm_contains_roots}--\ref{lem:support_addition},
  $\adj(\cS^{(N)})$ remains unchanged after adding $z_j$ to
  $\norm{h}^{(N)}$.
  Repeated application gives the desired result.
\end{proof}

\subsubsection{Phase 2: growing a tree from two forests}

\paragraph[Constructing a scaffold.]{Constructing a scaffold enclosing $\Gamma_1 \join \Gamma_2$.}
By assumption, there exists $v \in V = V_1\sqcup V_2$ such that $v \le_\sigma
u$ for all $u\in V$.
Without loss of generality, suppose that $v\in V_1$.
Let $a_1$ be the root of the connected component of $\Phi_1$ which contains~$v$.
By condition \ref{d:scaffold1} in Definition~\ref{d:scaffold}, $a_1 \le_\sigma
v$ so that $a_1$ too is a $\le_\sigma$-minimum of $V$.
Choose $b_2 \in R_2$ among the $\le_\sigma$-minima of $R_2$.
Define $\cS^{(\infty)} :=
(\Phi^{(\infty)},\sigma,\ori^{(\infty)},\norm\dtimes^{(\infty)})$ as follows:

\begin{itemize}
\item $\Phi^{(\infty)}$ is the tree (!) with orientation
  $\ori^{(\infty)}$ obtained from $\Phi = \Phi_1 \oplus \Phi_2$ by
  inserting a directed edge $a_1 \to r_i$ for each
  $r_i \in (R_1\setminus\{a_1\}) \cup R_2$.  Note that this
  orientation is outgoing with $a_1$ as the root of $\Phi^{(\infty)}$.
\item
  $\norm\dtimes^{(\infty)}\colon\edges(\Phi^{(\infty)}) \to \Pow(V)$ is defined via
  \[
    \norm h^{(\infty)} = \begin{cases}
      V_1, & \text{if } h = a_1 r_2 \text{ for } r_2\in R_2\setminus\{b_2\},\\
      V_1\sqcup V_2, & \text{if } h = a_1 b_2, \\
      V_2, & \text{if } h = a_1 r_1 \text{ for } r_1 \in R_1\setminus\{a_1\},\\
      \norm{h}^{(N+1)}, & \text{otherwise}.
    \end{cases}
  \]
\end{itemize}

By our choice of $a_1$ as a $\le_\sigma$-minimal element of $V$, we see that
$\cS^{(\infty)}$ is a scaffold on $V$.

\begin{lemma}
  \label{lem:Sinfinity_enclosure}
  $\cS^{(\infty)}$ encloses $\Gamma = \Gamma_1\join \Gamma_2$ over $\sigma$. 
\end{lemma}
\begin{proof}
  First note that $\Phi^{(\infty)}$ and $\Gamma$ are both connected:
  the former by construction and the latter since it is a join of non-empty
  graphs.
  It thus only remains to show that
  $\adj(\Gamma,\sigma) = \adj(\cS^{(\infty)})$.
  Let
  \[
    F := \Bigl\langle X^{w-a_1} [a_1,r] :
    r \in (R_1\setminus\{a_1\}) \sqcup R_2, \, w \in \norm{a_1 r}^{(\infty)}
    \Bigr\rangle \le \ZZ_\sigma V
  \]
  and note that, by \eqref{eq:adj_scaffold},
  $\adj(\cS^{(\infty)}) = \adj(\cS^{(N+1)}) + F$.

  By condition \ref{d:scaffold1} in Definition~\ref{d:scaffold} and
  since $R_i$ consists of the roots of $\Phi_i$, for each $v_i \in
  V_i$, there exists $r_i\in R_i$ with $r_i \le_\sigma v_i$.
  Moreover, $a_1 \le_\sigma v$ for each $v\in V$ and $b_2 \le_\sigma
  v_2$ for each $v_2 \in V_2$ by our choices of $a_1$ and $b_2$.
  Hence, by the definition of $\norm\dtimes^{(\infty)}$,
  \[
    F =
    \Bigl\langle X^{b_2-a_1} [a_1,r_1] : r_1 \in R_1\setminus\{a_1\}\Bigr\rangle
    + {\Bigl\langle [a_1,r_2] : r_2\in R_2\Bigr\rangle}.
  \]
  
  On the other hand,
  setting $G := \langle [R_1\times R_2]\rangle \le \ZZ_\sigma V$,
  by Corollary~\ref{cor:SN+1},
  $\adj(\Gamma,\sigma) = \adj(\cS^{(N+1)}) + G$.
  It thus suffices to show that $F = G$.
  Write $H := \langle [a_1,r_2] : r_2\in R_2 \rangle \subset F\cap G$.
  For $r_1\in R_1\setminus\{a_1\}$ and $r_2\in R_2$,
  since $a_1\le_\sigma r_1$ and $a_1 \le_\sigma r_2$,
  we obtain the ``triangle identity'' (cf.\ Lemma~\ref{lem:dominant_nonloop_nonloop})
  \begin{equation}
    \label{eq:r1r2}
    [r_1,r_2] = X^{r_1-a_1}[a_1,r_2] - X^{r_2-a_1}[a_1,r_1].
  \end{equation}
  As $[a_1,r_2] \in H\subset F$ and
  $X^{r_2-a_1}[a_1,r_1] = X^{r_2-b_2} \dtimes X^{b_2-a_1}[a_1,r_1] \in F$,
  we obtain $G\subset F$.
  Conversely, by taking $r_2 = b_2$ in \eqref{eq:r1r2},
  we see that $X^{b_2-a_1}[a_1,r_1] \in G$ whence $F\subset G$.
\end{proof}

\begin{rem}
  The proof of Lemma~\ref{lem:Sinfinity_enclosure} rested on the validity of
  the following conditions:
  \begin{itemize}
  \item $a_1 \in \norm{a_1 r_2}^{(\infty)}$ for all $r_2\in R_2$.
  \item $b_2 \in \norm{a_1 r_1}^{(\infty)} \subset V_2$
    for all $r_1\in R_1\setminus\{a_1\}$. 
  \end{itemize}
  In particular, numerous alternative definitions of
  $\norm\dtimes^{(\infty)}$ are possible while maintaining the
  validity of Lemma~\ref{lem:Sinfinity_enclosure}.  The crucial point
  of the definition that we chose---to be exploited in the upcoming final step
  of our proof of Theorem~\ref{thm:model_join}---is that, up to
  relabelling hyperedges, the specific choice that we made works
  uniformly in all possible cases.  That is to say, it works uniformly
  for all possible choices of $a_1$ and $b_2$ and also in the case
  that all $\le_\sigma$-minimal elements of $V$ belong to $V_2$ (in which case
  we choose $a_2\in V_2$ and $b_1\in V_1$ and proceed analogously to
  what we did above).
\end{rem}

\paragraph{Finale.}
Recall that $\nc_i$ denotes the number of connected components of
$\Gamma_i$.  Note that $\nc_i = \card{R_i}$ by condition
\ref{d:scaffold2} in Definition~\ref{d:scaffold}.  Hence, by
unravelling the definition of $\norm\dtimes^{(\infty)}$ from above, we
see that, up to relabelling of hyperedges, the hypergraph
$\Eta(\cS^{(\infty)})$ coincides with $\Eta$ in
Theorem~\ref{thm:model_join}.  In particular, $\Eta$ is a local model
of $\Gamma$ over our fixed but arbitrary cone $\sigma$ from the
beginning of this section.  This completes the proof of
Theorem~\ref{thm:model_join}.  \qed

\section{Cographs, hypergraphs, and cographical groups}
\label{s:cographical}

As in \S\ref{s:models}, all graphs in this section are assumed to be simple.

\paragraph{The story so far.}
In \S\ref{ss:graphical_groups} we attached 
a unipotent group scheme (``graphical group scheme'') $\GG_{\Gamma}$
to each graph $\Gamma$.
For each compact \DVR{}~$\fO$ we
expressed, in Corollary~\ref{cor:graphical_cc}, the class counting zeta functions of the group scheme
$\GG_\Gamma\otimes \fO$ in terms of the rational function
$W_\Gamma^-(X,T)$ from Theorem~\ref{thm:graph_uniformity}(\ref{thm:graph_uniformity2}):
\[
  \zeta^\cc_{\GG_\Gamma\otimes \fO}(s) = W^-_{\Gamma}\Bigl(q,\,q^{\card{\edges(\Gamma)}-s}\Bigr).
\]

For a \itemph{cograph} $\Gamma$, the Cograph Modelling Theorem (Theorem~\ref{thm:cograph})
established that there exists an explicit \itemph{modelling
hypergraph} $\Eta = \Eta(\Gamma)$ for~$\Gamma$.
This is a specific hypergraph on the same vertex set as $\Gamma$ which satisfies
\[
  W^-_\Gamma(X,T) = W_\Eta(X,T),
\]
where $W_\Eta(X,T)$ is the rational function associated with $\Eta$ 
in Theorem~\ref{thm:graph_uniformity}(\ref{thm:graph_uniformity1}).

Our proof of the Cograph Modelling Theorem was constructive.
Indeed, cographs (save for isolated vertices) are disjoint unions or joins of
smaller cographs.
Given modelling hypergraphs of two cographs 
we constructed, in Proposition~\ref{prop:model_disjoint_union} and 
Theorem~\ref{thm:model_join}, modelling
hypergraphs of their disjoint union and join, respectively;
cf.\ Remark~\ref{rem:canonical_models}.

In \S\ref{s:master}, we carried out
an extensive analysis of the rational functions $W_\Eta(X,T)$ associated with
hypergraphs $\Eta$ resulting, in particular, in an explicit formula, viz.\ Theorem~\ref{thm:master.intro}.
We also investigated the effects of taking disjoint unions and complete unions
of hypergraphs. 
This ties in well with our constructive proof of the Cograph Modelling
Theorem.
Namely, by Proposition~\ref{prop:model_disjoint_union},
the disjoint union $\Eta_1\oplus \Eta_2$ of
modelling hypergraphs $\Eta_1$ and $\Eta_2$ of cographs $\Gamma_1$ and
$\Gamma_2$ is a model of the cograph $\Gamma_1\oplus\Gamma_2$.
Moreover, the modelling hypergraph of the join $\Gamma_1 \join \Gamma_2$
can be described in terms of the complete union $\Eta_1 \freep \Eta_2$ and
the operations from \S\ref{subsec:hyp.ops};
cf.~Remark~\ref{rem:models.joins.freep}\eqref{rem:models.joins.freep.i}.

In the present section, we apply the results 
from \S\ref{s:master} to class counting zeta functions of 
cographical group schemes.

\subsection{Proofs of Theorems~\ref{thm:nonneg}--\ref{thm:ana.cograph}}
\label{ss:proofs_iv}

\begin{proof}[Proof of Theorem~\ref{thm:nonneg}]
  Let $V$ be the set of vertices of the cograph~$\Gamma$.  Let
  $\Eta=\Eta(\Gamma)$ be a modelling hypergraph for $\Gamma$ with hyperedge
  multiplicities $(\mu_I)_{I\subset V}$ as in
  Theorems~\ref{thm:master.intro}--\ref{thm:cograph}.  Our proof of
  Theorem~\ref{thm:cograph} in \S\ref{s:models} shows that we may assume that
  $\sum_I \mu_I = n-c$, where $n$ and $c$ are the numbers of vertices and
  connected components of $\Gamma$, respectively; cf.\
  Lemma~\ref{lem:model_no_hyperedges}.  Let $m $ be the number of edges of
  $\Gamma$.  The trivial bound $m \ge n - c$ then implies that, for each
  summand of $W_\Eta(X,X^mT)$ in \eqref{equ:master.intro}, the coefficients of
  $T^k$ in $X-1$ are non-negative.
  Indeed, observe that
  for each $J\subset V$,
  \[
    m - \sum_{I\cap J\neq \varnothing}\mu_I \ge m - \sum_I \mu_I = m-(n-c) \ge 0
  \]
  whence
  \[
    (1-X^{-1})^{|J|}\left(X^{|J|-\sum_{I\cap J\neq \varnothing}\mu_I}\cdot
      X^m\right)
    = (X-1)^{\card J} X^{m-\sum_{I\cap J\neq \varnothing}\mu_I}
  \]
  is a polynomial with non-negative coefficients in $X-1$.
\end{proof}

\begin{proof}[Proof of Theorem~\ref{thm:ana.cograph}]
  For the first part and the integrality of local poles,
  combine Corollary~\ref{cor:graphical_cc}, Theorem~\ref{thm:cograph}, and
  Theorem~\ref{thm:ana}.
  It remains to prove that the real parts of the poles of
  $\zeta^\cc_{\GG_\Gamma\otimes\fO_K}(s)$ are positive.
  Let $\Gamma$, $V$, $\Eta$, $(\mu_I)_{I\subset V}$, $m$, and $n$ be as in
  the proof of Theorem~\ref{thm:nonneg} above.
  As we argued there, $m - \sum_{I\cap J \not= \emptyset} \mu_I \ge 0$
  for each $J\subset V$.
  In particular, $f(J) := \card J + m - \sum\limits_{I\cap J \not= \emptyset} \mu_I > 0$
  whenever $J\not= \emptyset$.
  Unless $\Gamma$ is discrete, $f(\emptyset) = m > 0$.
  If $\Gamma = \Delta_n$ is discrete, then the real parts of the
  poles of $\zeta^\cc_{\GG_{\Delta_n}\otimes\fO}(s) = 1/(1-q^{n-s})$ are equal
  to $n$ which is positive since cographs are non-empty.
\end{proof}

\subsection{Disjoint unions of hypergraphs and direct products of cographical groups}
\label{ss:disjoint_union_graphs}

Much as for hypergraphs in \S\ref{subsec:hyp.dis.uni},
for arbitrary graphs $\Gamma_1$ and $\Gamma_2$,
the rational function $W^\pm_{\Gamma_1\oplus \Gamma_2}(X,T)$ is the Hadamard
product of $W^\pm_{\Gamma_1}(X,T)$ and $W^\pm_{\Gamma_2}(X,T)$.
In particular, if $R$ is the ring of integers of a number field or a
compact \DVR{}, then the class counting zeta function
$\zeta^\cc_{\GG_{\Gamma_1\oplus\Gamma_2} \otimes R}(s)$ is the Hadamard
product of the Dirichlet series $\zeta^\cc_{\GG_{\Gamma_1}\otimes R}(s)$
and $\zeta^\cc_{\GG_{\Gamma_2}\otimes R}(s)$;
this simply reflects the fact that class numbers of finite groups are
multiplicative: $\concnt(H_1\times H_2) = \concnt(H_1) \times \concnt(H_2)$
for finite groups $H_1$ and $H_2$.

\subsubsection{A special case: hypergraphs with disjoint supports and
  direct products of free class-$2$-nilpotent groups}
\label{sss:product_F2n}

We now apply \S\ref{subsubsec:disjoint} to study class counting zeta functions
of cographical groups modelled by hypergraphs with disjoint supports.

Let $\bfn=(n_1,\dots,n_\nb)\in\N^\nb$.
We write $n=\sum_{i=1}^\nb n_i$ and
$\binom{\bfn}{2}=\sum_{i=1}^\nb \binom{n_i}{2}$.
We consider the cographical group scheme associated with the cograph
\[
  \CG_{\bfn} = \CG_{n_1} \oplus \dots \oplus \CG_{n_\nb};
\]
see \eqref{def:comp.hyp} and note that $\CG_{\bfn}$ has $m = \binom{\bfn}{2}$
edges.
These cographs are of specific group-theoretic interest since
$\GG_{\CG_{\bfn}}(\Z) = F_{2,n_1}\times \dots \times F_{2,n_{\nb}}$ is the
direct product of the free class-$2$-nilpotent groups on $n_i$ generators;
in particular, $\GG_{\CG_n}(\Z) = F_{2,n}$.

\begin{remark}
  Conflicting notation in the literature notwithstanding, the cograph
  $\CG_{\bfn}$ is not to be confused with the complete multipartite graph on
  disjoint, independent sets of cardinalities $n_1,\dots,n_\nb$; the latter
  graph will feature as $\DG_{\bfn}$ in~\S\ref{ss:codis.app}.
\end{remark}

Combining Proposition~\ref{prop:model_disjoint_union} and
Example~\ref{ex:model_complete_graph}, we see that $\CG_{\bfn}$ is
modelled by the hypergraph $\Eta(\CG_{\bfn}) =
\BE_{\bfn,\bfn-\bfo} = \bigoplus_{i=1}^\nb
\BE_{n_i,n_i-1}$.
By Corollary~\ref{cor:master.dis.WO} (noting that
$n_i-m_i= n_i - (n_i-1) = 1$ for all $i\in[\nb]$),
\begin{dmath*}
  W^-_{\CG_{\bfn}}(X,T) = W_{\Eta(\CG_{\bfn})}(X,T) =
  W_{\BE_{\bfn,\bfn-\bfo}}(X,T) = \sum_{y\in
    \WOhat_\nb}\left(
  \prod_{i\in\sup(y)}(1-X^{-n_i})\right)\prod_{J\in y} \gp{X^{\card{J}}T}.
\end{dmath*}

\begin{cor}
  Let $\bfn=(n_1,\dots,n_\nb)\in\N^\nb$. For each compact \DVR{} $\fO$,
  \begin{align*}
    \zeta^\cc_{\GG_{\CG_{\bfn}}\otimes \fO}(s) &=
    W^-_{\CG_{\bfn}}(q,q^{\binom{\bfn}{2}-s}) =
    \sum_{y\in
      \WOhat_\nb}\left(
    \prod_{i\in\sup(y)}(\undl{n_i})\right)\prod_{J\in y}
    \gp{q^{\binom{\bfn}{2}+\card{J}-s}}. 
    \tag*{\qedsymbol}                                                 
  \end{align*}
\end{cor}

\begin{example}\label{exa:F2N}\
  \begin{enumerate}
  \item\label{exa:comp.hyp} If $\nb=1$ and $\bfn=(n)$, then
    $\binom{\bfn}{2} = \binom{n}{2}$, whence
    \begin{equation*}
      \zeta^\cc_{\GG_{\CG_n}\otimes \fO}(s) =
      \frac{1-q^{\binom{n-1}{2}-s}}{\left(1-q^{\binom{n}{2}-s}\right)\left(1-q^{\binom{n}{2}+1-s}\right)},
    \end{equation*}
    in accordance with \cite[Corollary~1.5]{Lins2/20}. There $\GG_{\CG_{n}}$
    goes by the name~$F_{\textup{n},\delta}$, where $n=2\textup{n}+\delta$
    with $\delta\in\{0,1\}$. See Example~\ref{ex:lins.biv.F} for a bivariate
    version of this formula.
\item If $\nb=2$ and $\bfn=(n_1,n_2)$, then $n=n_1+n_2$ and
  $\binom{\bfn}{2} = \binom{n_1}{2} + \binom{n_2}{2}$, whence
  \begin{multline}
    \zeta^\cc_{\GG_{\CG_{(n_1,n_2)}}\otimes \fO}(s) =
    \zeta^\cc_{\GG_{\CG_{n_1}}\otimes \fO}(s) \hada
    \zeta^\cc_{\GG_{\CG_{n_2}}\otimes \fO}(s)=\\ \frac{1 +
      q^{\binom{\bfn}{2}+1-s}\left(1 - q^{-n_1} - q^{-n_2} -q^{-n_1+1}
      - q^{-n_2+1} + q^{-n+1}\right)+
      q^{2\binom{\bfn}{2}-n+3-s}}{\left(1-q^{\binom{\bfn}{2}-s}\right)\left(1-q^{\binom{\bfn}{2}+1-s}\right)\left(1-q^{\binom{\bfn}{2}+2-s}\right)}.\label{equ:F2n1*F2n2}
  \end{multline}
  \end{enumerate}
\end{example}

\subsection[Complete unions of hypergraphs and free class-$2$-nilpotent
products]{Complete unions of hypergraphs and free class-$2$-nilpotent products
  of cographical groups}

The results of \S\ref{subsec:freep} on complete unions of hypergraphs
have direct corollaries pertaining to class counting zeta functions of
joins of graphs.
Recall from \S\ref{ss:graphical_groups} that for graphs $\Gamma_1$ and $\Gamma_2$, 
the graphical group $\GG_{\Gamma_1\join \Gamma_2}(\ZZ)$ is the free
class-$2$-nilpotent product of $\GG_{\Gamma_1}(\ZZ)$ and
$\GG_{\Gamma_2}(\ZZ)$.

\begin{prop}
  \label{prop:zeta_join_of_cographs}
  Let $\Gamma_1$ and $\Gamma_2$ be cographs on $n_1$ and $n_2$
  vertices, respectively.  Then
  \begin{align}
    W^-_{\Gamma_1\join \Gamma_2}(X,T)
    & =
      \bigl(X^{1-n_1-n_2} T - 1
    \nonumber\\
    & \qquad
      + W^-_{\Gamma_1}(X,X^{-n_2}T)(1-X^{-n_2}T)(1-X^{1-n_2}T)
    \nonumber\\
    & \qquad
      + W^-_{\Gamma_2}(X,X^{-n_1}T)(1-X^{-n_1}T)(1-X^{1-n_1}T)\bigr)
    \nonumber\\
    & \qquad
      /\bigl((1-T)(1-XT)\bigr).\label{equ:zeta_join_of_cographs}
  \end{align}

  In particular, if $\Gamma$ is a cograph, then $W^-_{\Gamma\join
    \Point}(X,T) = \frac{1 - X^{-1}T}{1-XT{\phantom .}} \dtimes
  W^-_\Gamma(X,X^{-1}T) $.
\end{prop}
\begin{proof}
  We may assume that
  $\verts(\Gamma_1) \cap \verts(\Gamma_2) = \emptyset$.  By
  Corollary~\ref{cor:cographs_have_models}, each $\Gamma_i$ admits a
  model, $\Eta_i$ say.  In particular,
  $W^-_{\Gamma_i}(X,T) = W_{\Eta_i}(X,T)$ for $i = 1,2$ by
  Corollary~\ref{cor:global_model_W}.

  Let $\Eta = \left( \Eta_1^\square \freep \Eta_2^\square\right)^\bfo$ (see
  Definition~\ref{def:hyp.inflation}),
  where $\Eta_i^\square = (\Eta_i)^{\bfz^{(\nc_i-1)}}$ and $\nc_i$ is the
  number of connected components of $\Gamma_i$.
  By Theorem~\ref{thm:model_join} and Remark~\ref{rem:models.joins.freep},
  $\Eta$ is a model of $\Gamma_1 \join \Gamma_2$.
  Hence, by Corollary~\ref{cor:global_model_W},
  $W^-_{\Gamma_1\join\Gamma_2}(X,T) = W_{\Eta}(X,T)$.
  By Proposition~\ref{prop:hyp.red} (applying \eqref{equ:insert.0.col}
  $\nc_1$ resp.\ $\nc_2$ times and \eqref{equ:insert.1.col} once),
  $$W_\Eta(X,T) = \frac{1-X^{-1}T}{1-T} \dtimes
  W_{\Eta_1\freep \Eta_2}(X,X^{-1}T).$$ The claim now
  follows from Corollary~\ref{cor:master.freep} by substituting
  $n_i-1$ for $m_i$ in \eqref{eq:master.freep.rat}.
  This reflects the fact that the hypergraphs $\Eta_i^\square$ are ``near
  squares'': they have $n_i$ vertices and a total number of $n_i-1$
  hyperedges; see Remark~\ref{rem:near_square}.
\end{proof}

Let $\GG_{\Gamma_1}$ and $\GG_{\Gamma_2}$ be the cographical group
schemes associated with the cographs $\Gamma_1$ and $\Gamma_2$.
Let $\Gamma_i$ have $n_i$ vertices and $m_i$ edges.
For each compact \DVR{} $\fO$, Proposition~\ref{prop:zeta_join_of_cographs}
now allows us to express
$$\zeta^\cc_{\GG_{\Gamma_1\join \Gamma_2}\otimes \fO} =
W^-_{\Gamma_1\vee \Gamma_2}(q,q^{m_1 + m_2 + n_1 n_2-s})$$
in terms of
$\zeta^\cc_{\GG_{\Gamma_1}\otimes\fO}(s)$ and
$\zeta^\cc_{\GG_{\Gamma_2}\otimes\fO}(s)$.
As a special case, we record the following.

\begin{cor}
  Let $\Gamma$ be a cograph with $n$ vertices and $m$ edges.
  Write $\Point = \CG_1 = \DG_1$.
  Then for each compact \DVR{} $\fO$,
  \[
    \pushQED{\qed}
    \zeta^\cc_{\GG_{\Gamma\join \Point}\otimes \fO}(s) =
    \frac{1-q^{m+n-1-s}}{1-q^{m+n+1-s}}
    \zeta^\cc_{\GG_{\Gamma}\otimes \fO}(s+1-m-n).
    \qedhere
    \popQED
  \]
\end{cor}

\begin{rem}\label{rem:funeq}
  Via the functional equations
  $W^-_\Gamma(X^{-1},T^{-1}) = -X^nT \, W^-_{\Gamma}(X,T)$ in
  Corollary~\ref{cor:feqn},
  the numbers $n_1$ and $n_2$ in
  Proposition~\ref{prop:zeta_join_of_cographs}---and hence the left-hand
  side of \eqref{equ:zeta_join_of_cographs}---are already determined
  by the rational functions $W^-_{\Gamma_i}(X,T)$.
\end{rem}

\subsubsection{A special case: hypergraphs with codisjoint supports
  and free class-$2$-nilpotent products of abelian groups}\label{ss:codis.app}

We now apply the results and formulae developed in
\S\ref{subsubsec:codisjoint} to cographical groups modelled by
hypergraphs with codisjoint supports.

As before, let $\bfn=(n_1,\dots,n_\nb)\in\N^\nb$,
$n=\sum_{i=1}^\nb n_i$, and
$\binom{\bfn}{2}=\sum_{i=1}^\nb \binom{n_i}{2}$.
We consider the cographical group scheme associated
with the cograph
$$\Delta_{\bfn} := \Delta_{n_1} \vee \dots \vee \Delta_{n_\nb},$$
viz.\ the complete multipartite graph on disjoint, independent sets of
cardinalities $n_1,\dots,n_\nb$.
Note that $\Delta_{\bfn}$ has $m=\binom{n}{2}-\binom{\bfn}{2}$ edges.
These cographs are of specific group-theoretic interest since
$\GG_{\Delta_{\bfn}}(\Z) = \Z^{n_1} \freep \dotsb \freep \Z^{n_\nb}$
(see \eqref{eq:freep_of_groups}) is the free
class-$2$-nilpotent product of free abelian groups of
ranks~$n_1,\dotsc,n_\nb$.
In particular, $\GG_{\Delta_{\bfo^{(r)}}}(\Z)= F_{2,r}$ is the free
class-$2$-nilpotent group on $\nb$ generators. 

By Remark~\ref{rem:models.joins.freep}\eqref{rem:models.joins.freep.ii}
and using the notation in Definition~\ref{def:hyp.inflation} and
\eqref{def:block.codis.multi},
$\Delta_{\bfn}$ is modelled by
$\Eta(\Delta_{\bfn}) = \left(\RE_{\bfn, \bfn-\bfo}\right)^{\bfo^{(\nb-1)}}$.
By Proposition~\ref{prop:hyp.red} (applying \eqref{equ:insert.1.col} $\nb-1$ times),
$$W_{\Eta(\Delta_{\bfn})}(X,T) = \frac{1-X^{1-\nb}T}{1-T}
W_{\RE_{\bfn, \bfn-\bfo}}(X, X^{1-r}T).$$ Combining
Corollary~\ref{cor:graphical_cc} with the explicit formula for
$W_{\RE_{\bfn, \bfn-\bfo}}(X,T)$ in Corollary~\ref{cor:master.codis.rat}
(substituting $n -\nb$ for $m$ there), we obtain the following.

\begin{proposition}\label{prop:gen.freep.abel}
  Let $\bfn=(n_1,\dots,n_\nb)\in\N^\nb$ and
  $m=\binom{n}{2}-\binom{\bfn}{2}$. For each compact~\DVR{}~$\fO$,
  \begin{align}\zeta^\cc_{\GG_{\Delta_{\bfn}}\otimes \fO}(s) &=
  W^-_{\Delta_{\bfn}}(q,q^{m-s}) =
  W_{\Eta(\Delta_{\bfn})}(q,q^{m-s})
  =\frac{1-q^{1-\nb+m-s}}{1-q^{m-s}}
  W_{\RE_{\bfn, \bfn-\bfo}}(q,
  q^{1-\nb+m-s})\nonumber\\ &=
  \frac{1}{\left(1-q^{m-s}\right)\left(1-q^{1+m-s}\right)}\times \label{equ:gen.freep.abel}\\ &\quad
  \quad \left(1 -
  q^{1-n+m-s}\left(1-\sum_{i=1}^\nb
  \frac{(q^{n_i}-1)(q^{n_i-1}-1)}{1-q^{2n_i-n+m-s}}
  \right) \right).\tag*{\qedsymbol}
\end{align}
\end{proposition}

\begin{example}
  Proposition~\ref{prop:gen.freep.abel} unifies and generalises a
  number of known formulae.
\begin{enumerate}
\item If $\nb = 1$ and $\bfn=(n)$, then
  $m=\binom{n}{2}-\binom{\bfn}{2}=0$, confirming the trivial formula
  $$\zeta^\cc_{\GG_{\Delta_n}\otimes \fO}(s) = \zeta^\cc_{\GG_a^n\otimes\fO}(s) =
  \frac{1}{1-q^{n-s}},$$
  where $\GG_a$ denotes the additive group scheme.
\item If $\bfn=\bfo^{(r)} \in \NN^r$, then $\binom{\bfn}{2}=0 =
  q^{n_i-1}-1$. Proposition~\ref{prop:gen.freep.abel} thus reconfirms
  
$$\zeta^\cc_{\GG_{\Delta_{\bfo^{(\nb)}}}\otimes \fO}(s) =
\zeta^\cc_{\GG_{\CG_{\nb}}\otimes \fO}(s) =
\frac{1-q^{\binom{\nb-1}{2}-s}}{\left(1-q^{\binom{\nb}{2}-s}\right)\left(1-q^{\binom{\nb}{2}+1-s}\right)};$$
see Example~\ref{exa:F2N}\eqref{exa:comp.hyp}. 
\item If $\nb = 2$ and $\bfn = (N,N)$, then
  $m=\binom{n}{2}-\binom{\bfn}{2}=N^2$ and $2n_i=n=2N$, whence
    \begin{multline*}
      \zeta^\cc_{\GG_{\Delta_N \vee \Delta_N}\otimes \fO}(s) =\\
      \frac{(1-q^{N(N-1)-s})(1-q^{N(N-1)+1-s}) +
        q^{N^2-s}(1-q^{-N})(1-q^{-N+1})}{(1-q^{N^2-s})^2(1-q^{1+N^2-s})},
    \end{multline*}
    in accordance with \cite[Corollary~1.5]{Lins2/20}, where
    $\GG_{\Delta_N \vee \Delta_N}$ goes by the name~$G_N$. 
\end{enumerate}
\end{example}

\subsection{Threshold graphs}
\label{ss:kites}

Threshold graphs are a particularly well-behaved class of cographs that were
introduced in \cite{CH75}.
They have been been independently (re-)discovered by various authors, and they
have found applications in numerous fields; see \cite[\S 1.1]{MP95}.
The main result of this section shows that ask zeta functions associated with
threshold graphs are of ``Riemann-type''.

We first recall the definition (or rather one of several equivalent
definitions) of threshold graphs.
Throughout, $\Point = \CG_1 = \DG_1$ denotes a fixed simple graph on one vertex.

\begin{defn}\label{def:kite.graph}
  A \emph{threshold graph} is any graph belonging to the class $\Kites$ which is 
  recursively defined to be minimal subject to the following conditions:
  \begin{enumerate}
  \item
    $\Point \in \Kites$.
  \item
    If $\Gamma \in \Kites$, then $\Point \join \Gamma \in \Kites$ and $\Gamma\oplus
    \Point \in \Kites$.
  \item If $\Gamma \in \Kites$ and $\Gamma$ is isomorphic to a graph $\Gamma'$, then
    $\Gamma'\in \Kites$.
  \end{enumerate}
\end{defn}

Clearly, threshold graphs are cographs.
For further characterisations of threshold graphs, see \cite[Theorem~1.2.4]{MP95}.
For example, a graph is a threshold graph if and only if it does not admit any
of the graphs $\CG_2 \oplus \CG_2$, $\Path 4$, or $\Cycle 4$ (see
\S\ref{ss:graphs} for definitions) as an induced subgraph.
(Recall that cographs are precisely those graphs that do not admit $\Path 4$
as an induced subgraph.)

As we will see in Theorem~\ref{thm:kite_zeta} below, ask zeta functions
associated with (negative adjacency representations of) threshold graphs admit a
particularly nice and explicit description.

Note that the $\Z$-points of cographical group 
schemes associated with threshold graphs form exactly the class of those
torsion-free finitely generated groups of nilpotency class at most $2$ which
contains $\Z$ and which is closed under taking direct and free class-$2$-nilpotent
products with~$\Z$.

\begin{ex}
  \label{ex:3111.1}
  The following is an example of a connected threshold graph:
  
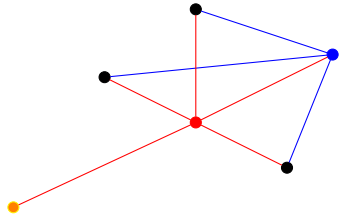
\begin{center}
    \begin{tikzpicture}[scale=0.3]
      \tikzstyle{Black Vertex}=[fill=black, draw=black, shape=circle, scale=0.4]
      \tikzstyle{Red Vertex}=[fill=red, draw=red, shape=circle, scale=0.4]
      \tikzstyle{Blue Vertex}=[fill=blue, draw=blue, shape=circle, scale=0.4]
      \tikzstyle{Orange Vertex}=[fill=orange, draw=yellow, shape=circle, scale=0.4]
      \tikzstyle{Solid Edge}=[-]
      \tikzstyle{Red Edge}=[-,draw=red]
      \tikzstyle{Blue Edge}=[-,draw=blue]
      \node [style=Orange Vertex] (0) at (-8, -3.75) {};
      \node [style=Red Vertex] (1) at (0, 0) {};
      \node [style=Blue Vertex] (2) at (6, 3) {};
      \node [style=Black Vertex] (3) at (4, -2) {};
      \node [style=Black Vertex] (4) at (-4, 2) {};
      \node [style=Black Vertex] (5) at (0, 5) {};
      \draw [style=Red Edge] (0) to (1);
      \draw [style=Red Edge] (1) to (2);
      \draw [style=Red Edge] (4) to (1);
      \draw [style=Red Edge] (1) to (3);
      \draw [style=Blue Edge] (2) to (3);
      \draw [style=Blue Edge] (4) to (2);
      \draw [style=Red Edge] (5) to (1);
      \draw [style=Blue Edge] (5) to (2);
    \end{tikzpicture}
  \end{center}
  Note that the central vertex is connected to all other vertices.
  Its removal results in a disconnected graph consisting of one isolated
  vertex and another component which is a star graph on four vertices.
  The above graph is therefore isomorphic
  to
  \[
    \Bigl(
    \bigl(
    (\Point \,\oplus\, \Point \,\oplus\, \Point)
    \,\join\, \textcolor{blue}{\Point}\bigr)
    \,\oplus\, \textcolor{orange}{\Point}\Bigr)
    \,\join\,
    \textcolor{red}{\Point}
  \]
  and is thus a threshold graph.
\end{ex}

There are precisely $2^{n-1}$ non-isomorphic threshold graphs on $n$ vertices.
In fact, these graphs correspond bijectively to bit strings of length
$n-1$; see the proof of \cite[Lemma~17.2.1]{MP95}.
For our purposes, a different but related parameterisation turns out to be
convenient.
Let $k_1,k_2,\dotsc$ be a sequence of non-negative integers.
Define $\KiteGraph()$ to be the empty graph and recursively define
\[
  \KiteGraph(k_1,\dotsc,k_{\lc+1}) :=
  \begin{cases}
    \KiteGraph(k_1,\dotsc,k_\lc) \oplus \DG_{k_{\lc+1}}, & \text{if $\lc$ is even}, \\
    \KiteGraph(k_1,\dotsc,k_\lc) \join \CG_{k_{\lc+1}},  & \text{if $\lc$ is odd}.
  \end{cases}
\]
Clearly, $\KiteGraph(k_1,\dotsc,k_\lc)$ is a threshold graph for each $\lc \ge 1$
and choice of $k_1,\dotsc,k_\lc$, provided that at least one $k_i$ is
positive.
(The empty graph is neither a threshold graph nor a cograph.)

\begin{ex}
  $\KiteGraph(n) = \DG_n$ and $\KiteGraph(1,n-1) = \CG_n$.
\end{ex}

Recall that a \emph{composition} of a non-negative integer $n$ is a sequence
$k = (k_1,\dotsc,k_\lc)$ of positive integers with $n = k_1 + \dotsb + k_\lc$.
We tacitly identify compositions and infinite sequences $k = (k_1,k_2,\dotsc)$
such that $k_i = 0$ for some $i$ and, in addition, $k_j = 0$ whenever
$i < j$ and $k_i = 0$.

\begin{prop}
  \label{prop:DK_repn}
  \quad
  \begin{enumerate}
  \item
    Every threshold graph on $n$ vertices is isomorphic to $\KiteGraph(k_1,\dotsc,k_\lc)$
    for some composition $(k_1,\dotsc,k_\lc)$ of $n$.
  \item
    Let $k$ and $k'$ be compositions of positive integers. Then $\KiteGraph(k)$ and
    $\KiteGraph(k')$ are isomorphic if and only if $k = k'$.
  \end{enumerate}
\end{prop}

\begin{proof}
  Given a label $\ell \in \{\oplus,\join\}$ and rooted labelled trees
  $\Tau_1,\dotsc,\Tau_u$, let
  $(\ell,\Tau_1,\dotsc,\Tau_u)$ denote the rooted tree whose root, $v$ say,
  has label $\ell$ and such that the descendant 
  trees of $v$ are precisely the trees $\Tau_1,\dotsc,\Tau_u$.
  For notational convenience, we identify a (numerical) label $\ell$ and the
  rooted labelled tree $(\ell)$.
  We see that threshold graphs with vertex set $\{1,\dotsc,n\}$ are precisely those cographs
  with cotrees (see \S\ref{ss:cographs}) of the form
  \[
    \Biggl\langle \dotso \biggr\{\oplus,
    \bigl[\join,
    (\oplus, 1,2,\dotsc,k_1),
    k_1+1,\dotsc,k_1+k_2\bigr],
    k_1+k_2+1,\dotsc,k_1+k_2+k_3\biggr\},\dotsc\Biggr\rangle,
  \]
  where $(k_1,k_2,\dotsc)$ is a composition of $n$ and we used
  different types of parentheses for clarity.  The uniqueness of
  cotrees of cographs (see \S\ref{ss:cographs}) now implies both
  claims.
\end{proof}
\begin{rem}
  As we mentioned above, the proof of \cite[Lemma~17.2.1]{MP95} establishes an
  explicit bijection between (isomorphism classes of) threshold graphs on $n$
  vertices and bit strings of length $n-1$.
  One can also deduce Proposition~\ref{prop:DK_repn} from this using
  the bijection between compositions of $n$ and bit strings of length $n-1$
  that sends a composition
  $(k_1)$ to $0\dotsb 0$ ($k_1-1$ zeros)
  and a composition $(k_1,\dotsc,k_\lc)$ with $c > 1$ to
  \[
    \underbrace{0 \dotsb 0}_{k_1-1}\, \underbrace{1\dotsb 1}_{k_2}
    \,\underbrace{0\dotsb 0}_{k_3} \,\dotsb
    \underbrace{b\dotsb b}_{k_\lc},
  \]
  where $b = 0$ if $\lc$ is odd and $b = 1$ if $c$ is even.
\end{rem}

\begin{ex}[Example~\ref{ex:3111.1}, part~II]\label{ex:3111.2}
  The threshold graph in Example~\ref{ex:3111.1} is isomorphic to
  \[
    \bigl((\DG_3 \join \CG_1)\oplus \DG_1\bigr)\join \CG_1
    \approx
    \KiteGraph(3,1,1,1).
  \]
\end{ex}

Let $k = (k_1,k_2,\dotsc)$ be a composition of a positive integer.
For $t \ge 1$, let $k(t) := \sum\limits_{i=1}^t k_i$ and $k[t] :=
\sum\limits_{i=t}^\infty (-1)^{i+1} k_i$; note that $k[t] = 0 = k_t$
for $t \gg 0$.  The following is easily proved by induction.
\begin{lemma}
  \label{lem:kite_num_verts_edges}
  \[
    \displaystyle \card{\verts(\KiteGraph(k))} = \sum_{i=1}^\infty k_i, \hfill\quad
    \card{\edges(\KiteGraph(k))} = \sum_{i=1}^\infty
    \binom{k(2i)} 2 - \binom{k(2i-1)} 2.
    \pushQED{\qed}
    \qedhere
    \popQED
\]
\end{lemma}

Recall from \eqref{def:staircase.graph} the notion of the staircase
hypergraph $\Stair_{\bfm}$ associated with a vector $\bfm =
(m_0,\dots,m_n)\in\N_0^{n+1}$.

\begin{proposition}\label{prop:kite.model.stair}
  Every threshold graph admits a staircase hypergraph as a model.
\end{proposition}
\begin{proof}
  We proceed by induction on the length $\lc$ of the composition $k =
  (k_1,\dotsc,k_\lc)$ representing a threshold graph.  For $\lc = 1$,
  note that $\Stair_{(0,\dotsc,0)}$ is a model of $\KiteGraph(k_1) =
  \DG_{k_1}$.  Next, supposing that $\KiteGraph(k_1,\dotsc,k_{r-1})$
  admits a model of the form $\Stair_{\bfm}$, we obtain a model
  $\Stair_{\bfm'}$ of $\KiteGraph(k_1,\dotsc,k_r)$ by repeated application
  of Proposition~\ref{prop:models.bullet}.
\end{proof}

\begin{ex}[Example~\ref{ex:3111.1}, part~III]\label{ex:3111.3}
  The staircase hypergraph $\Eta := \Stair_{(0,1,2,0,0,1,1)}$ with incidence
  matrix
  \[
    \begin{bmatrix}
      1 & 1 & 1 & 1 & 1\\
      0 & 1 & 1 & 1 & 1\\
      0 & 0 & 0 & 1 & 1\\
      0 & 0 & 0 & 1 & 1\\
      0 & 0 & 0 & 1 & 1\\
      0 & 0 & 0 & 0 & 1
    \end{bmatrix}
  \]
  is a model of the threshold graph $\KiteGraph(3,1,1,1)$ in
  Example~\ref{ex:3111.1}.
\end{ex}

Combining Proposition~\ref{prop:kite.model.stair} with
Proposition~\ref{prop:staircase}, we see that the rational function
$W^-_{\KiteGraph(k)}(X,T)$ associated with a threshold graph $\KiteGraph(k)$
is of a particularly simple form. The following theorem, which is the
main result of the present section, spells this out.

\begin{thm}
  \label{thm:kite_zeta}
  Let $k = (k_1,k_2,\dots)$ be a composition of a positive integer.
  Then
  \begin{equation}\label{equ:kite}
    W^-_{\KiteGraph(k)}(X,T) = \frac 1 {1 - X^{k[1]}T} \,
    \prod_{i=1}^\infty \frac {\bigl(1 - X^{k[2i+1] - k_{2i} + 1}
      T\bigr) \bigl(1 - X^{k[2i+1] - k_{2i}} T\bigr)} {\bigl(1
      - X^{k[2i+1] + 1} T\bigr)\bigl(1 - X^{k[2i+1]}T\bigr)}.
    \end{equation}
\end{thm}

\begin{proof}
  Straightforward induction along blocks $$(k_1,k_2,\dotsc,k_{2\rho-1},k_{2\rho})
  \leadsto (k_1,k_2,\dotsc,k_{2\rho+1},k_{2\rho+2})$$ using
  Corollary~\ref{cor:Gamma.point}.
\end{proof}

\begin{ex}[Example~\ref{ex:3111.1}, part~IV]\label{ex:3111.4}
  Consider the graph $\KiteGraph(k)$ for $k = (3,1,1,1)$ from
  Example~\ref{ex:3111.1}. Here, $k[1] = 2$ and $k[3] = k[5] = 0$.
  Theorem~\ref{thm:kite_zeta} thus asserts that
  \begin{align*}
    W^-_{\KiteGraph(3,1,1,1)}(X,T) & =
    \frac{1} {1-X^2T}
    \dtimes
    \frac{(1-T)(1-X^{-1}T)}{(1-XT)(1-T)}
    \dtimes
    \frac{(1-T)(1-X^{-1}T)}{(1-XT)(1-T)}\\&
    = \frac{(1-X^{-1}T)^2}{(1-XT)^2(1-X^2T)}.
  \end{align*}
\end{ex}

For threshold graphs, we can strengthen Theorem~\ref{thm:ana.cograph}. Recall that
$\card{\edges(\KiteGraph(k))}$ is given by
Lemma~\ref{lem:kite_num_verts_edges}.

\begin{thm}\label{thm:ana.kite}
 Let $k = (k_1,k_2,\dots)$ be a composition of a positive
 integer. Then for every number field $K$ with
 ring of integers $\mcO_K$, the abscissa of convergence of the class counting
 zeta function
 $\zeta^\cc_{\GG_{\KiteGraph(k)}\otimes \mcO_K}(s)$ is equal to
 $$\alpha(\KiteGraph(k)) = \card{\edges(\KiteGraph(k))} +
 \max\left\{k[1]+1, k[2i+1]+2 : i\in\N\right\}.$$ The function
 $\zeta^\cc_{\GG_{\KiteGraph(k)}\otimes\mcO_K}(s)$ may be meromorphically
 continued to all of~$\C$.
\end{thm}

\begin{proof}
  By Theorem~\ref{thm:kite_zeta}, the Euler product
  $$\zeta^\cc_{\GG_{\KiteGraph(k)}\otimes \mcO_K}(s) = \prod_{v \in \mcV_K}
  W^-_{\KiteGraph(k)}(q_v,q_v^{\card{\edges(\KiteGraph(k))}-s})$$  is a product of
  finitely many translates of the Dedekind zeta function $\zeta_K(s)$ and
  inverses of such translates. The abscissa of convergence is then
  readily read off from \eqref{equ:kite}.
\end{proof}

\begin{example}[Example~\ref{ex:3111.1}, part~V]\label{ex:3111.5}
  The graphical group scheme $\GG_{\KiteGraph(3,1,1,1)}$ has the
  property that, for each number field $K$,
  $$\zeta^{\cc}_{\GG_{\KiteGraph(3,1,1,1)}\otimes \mcO_K}(s) =
\frac{\zeta_K(s-9)^2\zeta_K(s-10)}{\zeta_K(s-7)^2},$$ with global
abscissa of convergence $\alpha(\KiteGraph(3,1,1,1))=11 = 8 +
\max\{2+1,0+2\}$, in accordance with Theorem~\ref{thm:ana.kite}.
\end{example}

\subsection[Bivariate conjugacy class zeta functions and
  cographical group schemes]{Bivariate conjugacy class zeta functions associated with
  cographical group schemes}
\label{ss:bivariate}

For a finite group $G$, 
let $\cc_n(G)$ denote its number of conjugacy classes of size $n$ and let
$\xi^\ccs_{G}(s) := \sum\limits_{n=1}^\infty \cc_n(G) n^{-s}$ be the associated
Dirichlet polynomial; note that $\concnt(G) = \xi^\ccs_G(0)$.

Let $\GG$ be a unipotent group scheme over the ring of integers $\mcO=\mcO_K$
of a number field~$K$; see \cite[\S 2.1]{SV14}.
For a place $v\in\Places_K$, let $\fP_v\in \Spec(\mcO)$ be
the associated prime ideal with residue field size $q_v = \card{\mcO/\fP_v}$
and let $\mcO_v = \varprojlim_k \mcO/\fP_v^k$.

In \cite[Definition\ 1.2]{Lins1/19}, Lins defined the \emph{bivariate conjugacy class zeta
  function}
\begin{equation}\label{def:global.biv.cc}
  \mcZ^{\cc}_{\GG\otimes\mcO}(s_1,s_2)
  =
  \sum_{0 \neq I \normal\, \mcO}\xi^\ccs_{\GG(\mcO/I)}(s_1) \dtimes \abs{\mcO/I}^{-s_2}
  \\ =
  \prod_{v \in \Places_K} \mcZ^{\cc}_{\GG \otimes \mcO_v}(s_1,s_2)
\end{equation}
associated with $\GG$, where the Euler factors are given by

\begin{equation}\label{def:local.biv.cc}
  \mcZ^{\cc}_{\GG \otimes \mcO_v}(s_1,s_2)
  = \sum_{i=0}^\infty \xi^{\ccs}_{\GG(\mcO/\fP_v^i)}(s_1) (q_v^{-s_2})^i.
\end{equation}

For all but (possibly) finitely many places $v\in\Places_K$, the Euler factors
\eqref{def:local.biv.cc} are rational functions in $q^{-s_1}$ and
$q^{-s_2}$; see \cite[Theorem~1.2]{Lins1/19}.
Both these local and
the global zeta functions~\eqref{def:global.biv.cc} refine the class
counting zeta functions defined in \S\ref{ss:intro/zeta}.
Indeed, as observed in \cite[\S 1.2]{Lins1/19},
\begin{equation}\label{eq:biv.cc.subs}
  \mcZ^{\cc}_{\GG\otimes\mcO}(0,s) = \zeta^{\cc}_{\GG\otimes\mcO}(s).
\end{equation}
(Lins used slightly different notation for these functions; see
Remark~\ref{rem:lins.not}.) Just as {univariate} ask zeta
functions may be expressed in terms of carefully designed \itemph{univariate}
$p$-adic integrals, Lins expressed bivariate conjugacy zeta
functions in terms of suitably defined \itemph{bivariate} $p$-adic
integrals; see \cite[\S 4]{Lins1/19}.

We record here, in all brevity, that expressing class counting zeta functions
$\zeta^{\cc}_{\GG_\Gamma \otimes \fO}(s) =
W_\Eta(q,q^{\card{\edges(\Gamma)}-s})$ associated with a cographical group
scheme $\GG_\Gamma$ in terms of the ask zeta function of a modelling
hypergraph $\Eta$ is compatible with Lins's bivariate refinement
of class counting zeta functions.
The reason for this is the common \itemph{multivariate} origin of all the
$p$-adic integrals involved.

To be more precise, let $\Eta = (V,E,\abs\dtimes)$ be a hypergraph with incidence representation~$\eta$. For each compact \DVR{} $\fO$,
we define the \emph{bivariate ask zeta function}
\begin{equation*}
 \zeta_{\eta^\fO}^{\ak}(s_1,s_2) = (1 - q^{-1})^{-1} \int\limits_{\fO V \times
   \fO} \abs{y}^{(s_1+1) \abs E +s_2-\abs V-1} \prod_{e\in E}
 \norm{x_e;y}^{-s_1-1} \,\dd \mu_{\fO V \times \fO}(x,y);
\end{equation*}
note that $\zeta_{\eta^\fO}^{\ak}(s) = \zeta_{\eta^\fO}^{\ak}(0,s)$; cf.\ Proposition~\ref{prop:hypergraph_int}.
One may define bivariate ask zeta functions in greater generality but we
shall not need this here.

Generalising \eqref{eq:Zask=Zsubs} (for $D = \varnothing$), we may express the bivariate
ask zeta functions in terms of the \itemph{multivariate} function
$\Zeta_{V,\emptyset}(\bfs)$ (see \eqref{def:IJdelta}):
$$\zeta_{\eta^{\fO}}^{\ak}(s_1,s_2) = ({1 - q^{-1}})^{-1}\,
\Zeta_{V,\emptyset}\left((s_1+1)\card{E} + s_2-\card{V}-1; \left(\mu_I(-s_1-1)\right)_{I\subset
    V}\right).$$
As in \S\ref{ss:combinatorial_modules},
there exists a rational function $W_\Eta(X,T_1,T_2)\in\QQ(X,T_1,T_2)$
such that
$$\zeta_{\eta^\fO}^{\ak}(s_1,s_2) = W_\Eta(X,T_1,T_2);$$
of course, $W_\Eta(X,T) = W_\Eta(X,1,T)$.
We may then use the multivariate nature
of \eqref{eq:rec} to deduce a ``trivariate'' analogue of the formula
for $W_\Eta(X,T)$ given in Corollary~\ref{cor:master}.

If $\Eta$ is a model of a cograph $\Gamma$,
then Lins's bivariate conjugacy class zeta function is recovered via the
formula
$$\mcZ^{\cc}_{\GG_{\Gamma} \otimes \fO}(s_1,s_2) = W_{\Eta}(q,
q^{-s_1}, q^{\card{\edges(\Gamma)}-s_2});$$
this is based on a trivariate form of Theorem~\ref{thm:cograph}.

\begin{ex}\label{ex:lins.biv.F}
  For the block hypergraph $\BE_{n,m}$, we readily obtain
\begin{equation}\label{eq:W.tri}
  W_{\BE_{n,m}}(X,T_1,T_2) =
  \frac{1-X^{-m}T_1^mT_2}{(1-T_2)(1-X^{n-m}T_1^mT_2)},
\end{equation}
generalising the bivariate formula given in
Example~\ref{exa:staircase}\eqref{exa:BEnm}. Using the fact that
$\BE_{n,n-1}$ is a model of the complete graph $\CG_{n}$ (see
Example~\ref{ex:model_complete_graph}) we may use this formula (with
$n = 2\textup{n}+\delta$ and $m=n-1$) to recover 
Lins's formula (cf.\ \cite[Theorem~1.4]{Lins2/20}) 
$$\mcZ^{\cc}_{F_{\textrm{n},\delta}\otimes \fO}(s_1,s_2) = \mcZ^{\cc}_{\GG_{\CG_n}\otimes \fO}(s_1,s_2) =
\frac{1-q^{\binom{n-1}{2}-(n-1)s_1-s_2}}{\Bigl(1-q^{\binom{n}{2}-s_2}\Bigr)\Bigl(1-q^{\binom{n}{2}+1-(n-1)
    s_1-s_2}\Bigr)}.$$
As predicted by \eqref{eq:biv.cc.subs}, setting
$(s_1,s_2)=(0,s)$, we recover the formula in
Example~\ref{exa:F2N}\eqref{exa:comp.hyp}.
\end{ex}

\begin{rem}\label{rem:lins.not}\
\begin{enumerate}
\item
  Lins's notation \cite{Lins1/19,Lins2/20} differs slightly from ours.
  Her $\mcZ^{\cc}_{\GG(R)}(s_1,s_2)$ is what we called
  $\mcZ^{\cc}_{\GG \otimes R}(s_1,s_2)$, for various rings~$R$.
  Our \itemph{class counting zeta function} $\zeta^{\cc}_{\GG\otimes R}(s)$ 
  goes by the name \itemph{class number zeta function} $\zeta^{\textup{k}}_{\GG(R)}(s)$
  in Lins's work.
  We note that our $\zeta^{\cc}_{\GG \otimes \mcO}(s)$ may not only be
  obtained by suitably specialising Lins's bivariate conjugacy class zeta
  function as in \eqref{eq:biv.cc.subs} but also by specialising her
  \emph{bivariate representation zeta function} 
  $\mcZ^{\textup{irr}}_{\GG(\mcO)}(s_1,s_2)$.
  The latter is defined analogously to~\eqref{def:global.biv.cc} by
  enumerating the ordinary irreducible characters of the finite groups $\GG(\mcO/I)$
  by degree (rather than conjugacy classes by cardinality);
  see \cite[(1.2)]{Lins1/19}.

  As is apparent from our discussion here,
  the techniques developed and employed
  in our study of class counting zeta functions of (co)graphical group
  schemes are slanted towards counting conjugacy classes rather than
  irreducible characters.  
\item Many of our results about univariate local and global class
  counting zeta functions associated with (co)graphical group schemes
  have bivariate analogues. The bivariate version of
  Theorem~\ref{thm:ana}, for instance, describes the domain of convergence
  of the bivariate ask zeta functions from above.
  General analytic properties of bivariate conjugacy class and representation
  zeta functions associated with unipotent group schemes over number fields 
  are studied in \cite{Lins3/18}.
\item\label{rem:paula.3}
  Beyond cographs, using suitable bivariate versions of
  Theorem~\ref{thm:graph_uniformity}(\ref{thm:graph_uniformity2}) and
  Corollary~\ref{cor:graphical_cc}, we may strengthen
  Corollary~\ref{cor:higman} as follows:
  for each simple graph $\Gamma$ and $k \ge 1$, the number of
  conjugacy classes of $\GG_\Gamma(\FF_q)$ of size $q^k$ is given by a
  polynomial in $q$.
 \end{enumerate}
\end{rem}

\section{Further examples}
\label{s:examples}

In this section, we collect a number of further examples of the
function $W_\Gamma^\pm(X,T)$ for graphs $\Gamma$ beyond the infinite
families covered in \S\ref{s:cographical}.

\subsection{Computer calculations: \textsf{Zeta}}
\label{ss:computer}

Our constructive proof of Theorem~\ref{thm:graph_uniformity} (see
\S\S\ref{ss:combinatorial_modules}, \ref{ss:proof_torically_combinatorial}) leads to
algorithms for computing the rational functions $W_\Eta(X,T)$ and
$W_\Gamma^\pm(X,T)$ associated with a hypergraph and graph, respectively,
complementing the formulae derived in \S\ref{s:master}.
In detail, given a
hypergraph~$\Eta$, Proposition~\ref{prop:hypergraph_int} expresses
$W_\Eta(X,T)$ in terms of a combinatorially defined $p$-adic integral.
The latter can be expressed as a univariate specialisation of the
integrals studied in~\cite[\S 3]{topzeta}.  In particular,
\cite[Proposition~3.9]{topzeta} and \cite[\S 6]{padzeta} together provide
practical means for computing $W_\Eta(X,T)$. Behind the scenes, these
techniques rely on algorithms due to Barvinok and Woods~\cite{BW03} for computing
with rational generating functions enumerating lattice points in polyhedra.

Regarding the case of a graph $\Gamma$,
the inductive proof of Theorem~\ref{thm:torically_combinatorial}
in \S\ref{ss:proof_torically_combinatorial} readily translates into a recursive
algorithm for computing $W_\Gamma^\pm(X,T)$.
For the base case, combine
Proposition~\ref{prop:Adj_solitary}
and Proposition~\ref{prop:combinatorial_module_uniformity} 
with \cite[\S 3]{topzeta} and \cite[\S 6]{padzeta} as above.

Based on the steps just outlined, the first author's software package 
\textsf{Zeta}~\cite{Zeta} for the computer algebra system
SageMath~\cite{SageMath} includes implementations of algorithms for computing
the rational functions $W_\Eta(X,T)$ and $W_\Gamma^\pm(X,T)$ in 
Theorem~\ref{thm:graph_uniformity}.
In practice, these algorithms often substantially outperform the previously
existing functionality for computing ask zeta functions based on \cite{ask} 
that is available in \textsf{Zeta}; note, however, that the present
algorithms are only applicable in the context of ask zeta functions
associated with graphs and hypergraphs.

In the remainder of this section, we record a number of explicit examples
of the rational functions $W_\Gamma^\pm(X,T)$ computed with the help of
\textsf{Zeta}.
For a list covering all 1252 (non-empty) simple graphs on at most seven
vertices (up to isomorphism), see \cite{cico-db}. 

\subsection{Graphs on at most four vertices}\label{ss:graphs.leq.4}

Table~\ref{tab:graphs4} lists both types of rational functions
$W_\Gamma^\pm(X,T)$ for all $18$ (isomorphism classes of non-empty) simple
graphs on at most four vertices; any entry ``\%'' in the column
$W_\Gamma^+(X,T)$ indicates that $W^-_\Gamma(X,T) = W^+_\Gamma(X,T)$
for the specific graph $\Gamma$ in question.  All graphs in
Table~\ref{tab:graphs4}, save for the path $\Path 4$, are cographs.  In
particular, $17$ of the formulae in Table~\ref{tab:graphs4} could, in
principle, be derived from
Theorems~\ref{thm:master.intro}--\ref{thm:cograph}.  We further note
that all but the following graphs in Table~\ref{tab:graphs4} are threshold
graphs: $\CG_2\oplus \CG_2$, $\Path 4$, and $\Cycle 4$; see
Question~\ref{qu:riemann}.

\subsection{Graphs on five vertices}

In this subsection, we list $W_\Gamma^-(X,T)$ for all $34$ simple
graphs on five vertices.
By Corollary~\ref{cor:Gamma.point}(\ref{eq:sum}) and using
Table~\ref{tab:graphs4}, it suffices to consider simple graphs on five
vertices without isolated vertices; there are precisely $23$ of these and
their associated rational functions $W_\Gamma^-(X,T)$ are listed in
Table~\ref{tab:graphs5}.
We chose not to include the (often bulky) corresponding rational functions
$W_\Gamma^+(X,T)$.
In Table~\ref{tab:graphs5}, cographs are flagged and threshold graphs are
labelled as such.

\subsection{Paths and cycles on at most nine vertices}
\label{ss:paths}

Recall that $\Path n$ and $\Cycle n$ denote the path and cycle graph
on $n$ vertices, respectively;
see \eqref{def:path.graph}--\eqref{def:cycle.graph}.
The rational functions $W_{\Path n}^-(X,T)$ and $W_{\Cycle n}^-(X,T)$ for
$n \le 9$ are given
in Tables~\ref{tab:paths} and~\ref{tab:cycles}.
For a group-theoretic interpretation in the case of paths, as
in~\S\ref{ss:intro/class}, let $\Uni_n$ denote the group scheme of upper
unitriangular $n\times n$ matrices.
It is easy to see that for each ring $R$,
the graphical group $\GG_{\Path n}(R)$ is isomorphic to the maximal
quotient $\Uni_{n+1,3}(R) := \Uni_{n+1}(R)/\LCS_3(\Uni_{n+1}(R))$ of
$\Uni_{n+1}(R)$ of class at most $2$; cf.~\S\ref{ss:intro/graphs}.
The class numbers of the finite groups $\Uni_{n+1,3}(\FF_q)$ were determined
by Marjoram~\cite[Theorem~7]{Mar99}.
For $n \le 9$, in accordance with Corollary~\ref{cor:graphical_cc}, his
general formulae agree with the coefficients of $T$ in the expansions of
the rational functions $W^-_{\Path n}(X,X^{n-1}T)$ recorded in
Table~\ref{tab:paths}.

\subsection[A numerator]{The numerator of $W^+_{(\CG_3\oplus\CG_3)\join\CG_2}$}\label{subsec:numer.332}

We may now finish Example~\ref{exa:K3_K3_K2}:
Table~\ref{tab:CG3_CG3_CG2} records the numerator of $W_\Gamma^+(X,T)$ in
\eqref{eq:K3_K3_K2}.

\begin{landscape}
  \begin{table}
    \centering
    \small
    \begin{tabular}{m{2.4cm}|lll}
      $\Gamma$ & $W^-_\Gamma(X,T)$ & $W^+_\Gamma(X,T)$\\
      \hline
      \begin{tikzpicture}[scale=0.2]
        \tikzstyle{Black Vertex}=[draw=black, fill=black, shape=circle, scale=0.2]
        \tikzstyle{Solid Edge}=[-]
        \node [style=Black Vertex] (0) at (0, 0) {0};
      \end{tikzpicture}
      & $1/(1-XT)$
      & \% \\
      
      \begin{tikzpicture}[scale=0.2]
        \tikzstyle{Black Vertex}=[draw=black, fill=black, shape=circle, scale=0.2]
        \tikzstyle{Solid Edge}=[-]
        \node [style=Black Vertex] (0) at (0, 0) {0};
        \node [style=Black Vertex] (1) at (3, 0) {1};
      \end{tikzpicture}
      & $1/(1-X^2T)$
      & \% \\

      \begin{tikzpicture}[scale=0.2]
        \tikzstyle{Black Vertex}=[draw=black, fill=black, shape=circle, scale=0.2]
        \tikzstyle{Solid Edge}=[-]
        \node [style=Black Vertex] (0) at (0, 0) {0};
        \node [style=Black Vertex] (1) at (3, 0) {1};
        \draw (0) to (1);
      \end{tikzpicture}
      & $\frac{1-X^{-1}T}{(1-T)(1-XT)}$
      & \% \\
      
      \begin{tikzpicture}[scale=0.2]
        \tikzstyle{Black Vertex}=[draw=black, fill=black, shape=circle, scale=0.2]
        \tikzstyle{Solid Edge}=[-]
        \node [style=Black Vertex] (0) at (0, 0) {0};
        \node [style=Black Vertex] (1) at (3, 0) {1};
        \node [style=Black Vertex] (2) at (6, 0) {2};
      \end{tikzpicture}
      & $1/(1-X^3T)$
      & \% \\

      \begin{tikzpicture}[scale=0.2]
        \tikzstyle{Black Vertex}=[draw=black, fill=black, shape=circle, scale=0.2]
        \tikzstyle{Solid Edge}=[-]
        \node [style=Black Vertex] (0) at (0, 0) {0};
        \node [style=Black Vertex] (1) at (3, 0) {1};
        \node [style=Black Vertex] (2) at (6, 0) {2};
        \draw (0) to (1);
      \end{tikzpicture}
      & $\frac{1-T}{(1-XT)(1-X^2T)}$
      & \% \\

      \begin{tikzpicture}[scale=0.2]
        \tikzstyle{Black Vertex}=[draw=black, fill=black, shape=circle, scale=0.2]
        \tikzstyle{Solid Edge}=[-]
        \node [style=Black Vertex] (0) at (0, 0) {0};
        \node [style=Black Vertex] (1) at (3, 0) {1};
        \node [style=Black Vertex] (2) at (6, 0) {2};
        \draw (0) to (1);
        \draw (1) to (2);
      \end{tikzpicture}
      & $\frac{1-X^{-1}T}{(1-XT)^2}$
      & \% \\

      \begin{tikzpicture}[scale=0.2]
        \tikzstyle{Black Vertex}=[draw=black, fill=black, shape=circle, scale=0.2]
        \tikzstyle{Solid Edge}=[-]
        \node [style=Black Vertex] (0) at (0, 0) {0};
        \node [style=Black Vertex] (1) at (-3, -3) {2};
        \node [style=Black Vertex] (2) at (3, -3) {3};
        \draw (0) to (1);
        \draw (0) to (2);
        \draw (1) to (2);
      \end{tikzpicture}
      & $\frac{1-X^{-2}T}{(1-T)(1-XT)}$
                          &
                            \begin{minipage}{10cm}$\displaystyle \bigl(T^{2}
                              + T + 1 - 3 X^{-1} T^{2} - 6 X^{-1} T + 6 X^{-2}
                              T^{2} + 3 X^{-2} T - X^{-3} T^{3} - X^{-3} T^{2}
                              - X^{-3} T\bigr)/(1 - T)^4$
                            \end{minipage}
      \\
      
      \begin{tikzpicture}[scale=0.2]
        \tikzstyle{Black Vertex}=[draw=black, fill=black, shape=circle, scale=0.2]
        \tikzstyle{Solid Edge}=[-]
        \node [style=Black Vertex] (0) at (0, 0) {0};
        \node [style=Black Vertex] (1) at (3, 0) {1};
        \node [style=Black Vertex] (2) at (6, 0) {2};
        \node [style=Black Vertex] (3) at (9, 0) {3};
      \end{tikzpicture}
      & $1/(1-X^4T)$        
      & \% \\

      \begin{tikzpicture}[scale=0.2]
        \tikzstyle{Black Vertex}=[draw=black, fill=black, shape=circle, scale=0.2]
        \tikzstyle{Solid Edge}=[-]
        \node [style=Black Vertex] (0) at (0, 0) {0};
        \node [style=Black Vertex] (1) at (3, 0) {1};
        \node [style=Black Vertex] (2) at (6, 0) {2};
        \node [style=Black Vertex] (3) at (9, 0) {3};
        \draw (0) to (1);
      \end{tikzpicture}
      & $\frac{1-XT}{(1-X^2T)(1-X^3T)}$
      \\

      \begin{tikzpicture}[scale=0.2]
        \tikzstyle{Black Vertex}=[draw=black, fill=black, shape=circle, scale=0.2]
        \tikzstyle{Solid Edge}=[-]
        \node [style=Black Vertex] (0) at (0, 0) {0};
        \node [style=Black Vertex] (1) at (3, 0) {1};
        \node [style=Black Vertex] (2) at (6, 0) {2};
        \node [style=Black Vertex] (3) at (9, 0) {3};
        \draw (2) to (3);
        \draw (0) to (1);
      \end{tikzpicture}
      & 
        $\frac{X T - 2 T + 1 + X^{-1} T^{2} - 2 X^{-1} T + X^{-2} T}{(1 -
        X^2T)(1 - XT)(1 - T)}$
      & \%
      \\

      \begin{tikzpicture}[scale=0.2]
        \tikzstyle{Black Vertex}=[draw=black, fill=black, shape=circle, scale=0.2]
        \tikzstyle{Solid Edge}=[-]
        \node [style=Black Vertex] (0) at (0, 0) {0};
        \node [style=Black Vertex] (1) at (3, 0) {1};
        \node [style=Black Vertex] (2) at (6, 0) {2};
        \node [style=Black Vertex] (3) at (9, 0) {3};
        \draw (0) to (1);
        \draw (1) to (2);
      \end{tikzpicture}
      & $\frac{1-T}{(1-X^2T)^2}$
      \\
      
      \begin{tikzpicture}[scale=0.2]
        \tikzstyle{Black Vertex}=[draw=black, fill=black, shape=circle, scale=0.2]
        \tikzstyle{Solid Edge}=[-]
        \node [style=Black Vertex] (0) at (0, 0) {0};
        \node [style=Black Vertex] (1) at (3, 0) {1};
        \node [style=Black Vertex] (2) at (6, 0) {2};
        \node [style=Black Vertex] (3) at (9, 0) {3};
        \draw (0) to (1);
        \draw (1) to (2);
        \draw (2) to (3);
      \end{tikzpicture}
      &
        $\frac{-X T^{2} + 3 T^{2} - T + 1 - X^{-1} T^{3} + X^{-1} T^{2} - 3
        X^{-1} T + X^{-2} T}{(1 - T)(1 - XT)^3}$
      & \%
      \\
      
      \begin{tikzpicture}[scale=0.2]
        \tikzstyle{Black Vertex}=[draw=black, fill=black, shape=circle, scale=0.2]
        \tikzstyle{Solid Edge}=[-]
        \node [style=Black Vertex] (0) at (0, -2) {0};
        \node [style=Black Vertex] (1) at (4, -2) {1};
        \node [style=Black Vertex] (2) at (4, 2) {2};
        \node [style=Black Vertex] (3) at (0, 2) {3};
        \draw (0) to (1);
        \draw (1) to (2);
        \draw (2) to (3);
        \draw (3) to (0);
      \end{tikzpicture} &
                          $\frac{T + 1 - 2 X^{-1} T - 2 X^{-2} T + X^{-3}
                          T^{2} + X^{-3} T}{(1 - T)^2(1 - XT)}$
      & \%
      \\
      
      \begin{tikzpicture}[scale=0.2]
        
        \tikzstyle{Black Vertex}=[draw=black, fill=black, shape=circle, scale=0.2]
        \tikzstyle{Solid Edge}=[-]
        \node [style=Black Vertex] (0) at (0, 0) {0};
        \node [style=Black Vertex] (1) at (4, 0) {1};
        \node [style=Black Vertex] (2) at (4, 4) {2};
        \node [style=Black Vertex] (3) at (0, 4) {3};
        \draw (0) to (1);
        \draw (0) to (2);
        \draw (1) to (2);
        \draw (2) to (3);
        \draw (3) to (0);
      \end{tikzpicture}
      & $\frac{(1-X^{-1}T)(1-X^{-2}T)}{(1-T)^2(1-XT)}$
                          & \begin{minipage}{10cm}
                            $\displaystyle \bigl(T^{3} + 2 T^{2} + 3 T + 1 - 3
                            X^{-1} T^{3} - 7 X^{-1} T^{2} - 4 X^{-1} T + 3
                            X^{-2} T^{3} - 3 X^{-2} T + 4 X^{-3} T^{3} + 7
                            X^{-3} T^{2} + 3 X^{-3} T - X^{-4} T^{4} - 3
                            X^{-4} T^{3} - 2 X^{-4} T^{2} - X^{-4}
                            T\bigr)/\bigl((1 - X^{-1} T)(1 - T)^2(1 - X T^2)\bigr)$
                          \end{minipage}
      \\
      
      \begin{tikzpicture}[scale=0.2]
        \tikzstyle{Black Vertex}=[draw=black, fill=black, shape=circle, scale=0.2]
        \tikzstyle{Solid Edge}=[-]
        \node [style=Black Vertex] (0) at (0, 0) {0};
        \node [style=Black Vertex] (1) at (0, 6) {1};
        \node [style=Black Vertex] (2) at (4, 3) {2};
        \node [style=Black Vertex] (3) at (8, 3) {3};
        \draw (0) to (1);
        \draw (0) to (2);
        \draw (1) to (2);
      \end{tikzpicture}
        & $\frac{ (1-X^{-1}T) }{(1-XT)(1-X^2T)}$
      &
        \begin{minipage}{10cm}
          $\displaystyle \bigl(X^{2} T^{2} - 3 X T^{2} + X T - T^{3} + 6 T^{2}
          - 6 T + 1 - X^{-1} T^{2} + 3 X^{-1} T - X^{-2} T\bigr)/(1 - XT)^4$ 
        \end{minipage}
      \\
      
      \begin{tikzpicture}[scale=0.2]
        \tikzstyle{Black Vertex}=[draw=black, fill=black, shape=circle, scale=0.2]
        \tikzstyle{Solid Edge}=[-]
        \node [style=Black Vertex] (0) at (0, 0) {0};
        \node [style=Black Vertex] (1) at (0, 6) {1};
        \node [style=Black Vertex] (2) at (4, 3) {2};
        \node [style=Black Vertex] (3) at (8, 3) {3};
        \draw (0) to (1);
        \draw (0) to (2);
        \draw (1) to (2);
        \draw (2) to (3);
      \end{tikzpicture}
        & $\frac{(1-X^{-1}T)^2}{(1-T)(1-XT)^2}$
      &
        \begin{minipage}{10cm}
          $\displaystyle \bigl(-X T^{4} - X T^{3} + 4 T^{4} + 4 T^{3} - 3
          T^{2} + 2 T + 1 - 8 X^{-1} T^{4} + 2 X^{-1} T^{3} + 2 X^{-1} T^{2} -
          8 X^{-1} T + X^{-2} T^{5} + 2 X^{-2} T^{4} - 3 X^{-2} T^{3} + 4
          X^{-2} T^{2} + 4 X^{-2} T - X^{-3} T^{2} - X^{-3} T\bigr)/
          \bigl((1 - T)^3 (1 - XT)(1 - X T^2) \bigr)$
        \end{minipage}
      \\
      
      \begin{tikzpicture}[scale=0.2]
        \tikzstyle{Black Vertex}=[draw=black, fill=black, shape=circle, scale=0.2]
        \tikzstyle{Solid Edge}=[-]
        \node [style=Black Vertex] (0) at (0, 0) {0};
        \node [style=Black Vertex] (1) at (0, 4) {1};
        \node [style=Black Vertex] (2) at (-3, -3) {2};
        \node [style=Black Vertex] (3) at (3, -3) {3};
        \draw (0) to (1);
        \draw (0) to (2);
        \draw (0) to (3);
      \end{tikzpicture}
      & $\frac{1-X^{-1}T}{(1-XT)(1-X^2T)}$
      & \%
      \\
      
      \begin{tikzpicture}[scale=0.2]
        \tikzstyle{Black Vertex}=[draw=black, fill=black, shape=circle, scale=0.2]
        \tikzstyle{Solid Edge}=[-]
        \node [style=Black Vertex] (0) at (-1, 0) {0};
        \node [style=Black Vertex] (1) at (-1, 5) {1};
        \node [style=Black Vertex] (2) at (-5, -4) {2};
        \node [style=Black Vertex] (3) at (3, -4) {3};
        \draw (0) to (1);
        \draw (0) to (2);
        \draw (0) to (3);
        \draw (1) to (2);
        \draw (1) to (3);
        \draw (2) to (3);
      \end{tikzpicture}
        &
          $\frac{1-X^{-3} T}{(1-T)(1-XT)}$ &
                                             \begin{minipage}{10cm}
                                               $\displaystyle\bigl(1 + 3
                                               X^{-1} T^{2} + 5 X^{-1} T + 3
                                               X^{-2} T^{3} - 8 X^{-2} T^{2} -
                                               14 X^{-2} T - 10 X^{-3} T^{3} +
                                               10 X^{-3} T + 14 X^{-4} T^{3} +
                                               8 X^{-4} T^{2} - 3 X^{-4} T - 5
                                               X^{-5} T^{3} - 3 X^{-5} T^{2} -
                                               X^{-6} T^{4}\bigr)/\bigl((1 - X^{-1}T^2)(1 -
                                               X^{-1}T)(1 - T)^2\bigr)$
                                             \end{minipage}
    \end{tabular}
    \caption{Graphs on at most four vertices and their ask zeta functions}
    \label{tab:graphs4}
  \end{table}
\end{landscape}

\begin{longtable}{m{4cm}cl}
  \caption{Graphs without isolated vertices on at most five vertices and their
    negative ask zeta functions}\label{tab:graphs5}\\
  \endfirsthead
  \endhead
  \small
  \qquad$\Gamma$
  & comment
  & $W^-_\Gamma(X,T)$
  \\
  \hline
  \phantom .
  \\
  \begin{tikzpicture}[scale=0.2]
    \tikzstyle{Black Vertex}=[draw=black, fill=black, shape=circle, scale=0.2]
    \tikzstyle{Solid Edge}=[-]
    \node [style=Black Vertex] (0) at (0, 0) {0};
    \node [style=Black Vertex] (1) at (4, 0) {1};
    \node [style=Black Vertex] (2) at (-4, 0) {2};
    \node [style=Black Vertex] (3) at (0, 4) {3};
    \node [style=Black Vertex] (4) at (0, -4) {4};
    \draw (0) to (1);
    \draw (0) to (2);
    \draw (0) to (3);
    \draw (0) to (4);
  \end{tikzpicture}
  & $\KiteGraph(4,1)$
  & $\frac{(1-X^{-1}T)}{(1-XT)(1-X^3T)}$
  \\
  \begin{tikzpicture}[scale=0.2]
    \tikzstyle{Black Vertex}=[draw=black, fill=black, shape=circle, scale=0.2]
    \tikzstyle{Solid Edge}=[-]
    \node [style=Black Vertex] (0) at (0, 0) {0};
    \node [style=Black Vertex] (1) at (-3, 3) {1};
    \node [style=Black Vertex] (2) at (-3, -3) {2};
    \node [style=Black Vertex] (3) at (3, 0) {3};
    \node [style=Black Vertex] (4) at (6, 0) {4};
    \draw (0) to (1);
    \draw (0) to (2);
    \draw (0) to (3);
    \draw (3) to (4);
  \end{tikzpicture}
  & no cograph
  &
  $\frac{X T - 2 T + 1 + X^{-1} T^{2} - 2 X^{-1} T + X^{-2} T}{(1 - XT)^2(1 - X^2T)}$
  \\
  \begin{tikzpicture}[scale=0.2]
    \tikzstyle{Black Vertex}=[draw=black, fill=black, shape=circle, scale=0.2]
    \tikzstyle{Solid Edge}=[-]
    \node [style=Black Vertex] (0) at (0, 0) {0};
    \node [style=Black Vertex] (1) at (-3, 3) {1};
    \node [style=Black Vertex] (2) at (-3, -3) {2};
    \node [style=Black Vertex] (3) at (3, 3) {3};
    \node [style=Black Vertex] (4) at (3, -3) {4};
    \draw (0) to (1);
    \draw (1) to (2);
    \draw (0) to (2);
    \draw (0) to (3);
    \draw (0) to (4);
  \end{tikzpicture}
  & $\KiteGraph(1,1,2,1)$
  &
  $\frac{(1-X^{-1}T)(1-T)}{(1-XT)^2(1-X^2T)}$
  \\
  \begin{tikzpicture}[scale=0.2]
    \tikzstyle{Black Vertex}=[draw=black, fill=black, shape=circle, scale=0.2]
    \tikzstyle{Solid Edge}=[-]
    \node [style=Black Vertex] (0) at (0, 0) {0};
    \node [style=Black Vertex] (1) at (3, 0) {1};
    \node [style=Black Vertex] (2) at (6, 0) {2};
    \node [style=Black Vertex] (3) at (9, 0) {3};
    \node [style=Black Vertex] (4) at (12, 0) {3};
    \draw (0) to (1);
    \draw (1) to (2);
    \draw (2) to (3);
    \draw (3) to (4);
  \end{tikzpicture}
  & no cograph ($\Path 5$)
  &\begin{minipage}{7.5cm}
    $\displaystyle \bigl(
    -2 X T^{2} + X T + 4 T^{2} - 2 T + 1 - X^{-1} T^{3} + 2 X^{-1} T^{2}
    - 4 X^{-1} T - X^{-2} T^{2} + 2 X^{-2} T\bigr)/(1 - XT)^4$
  \end{minipage}
  \\
  \begin{tikzpicture}[scale=0.2]
    \tikzstyle{Black Vertex}=[draw=black, fill=black, shape=circle, scale=0.2]
    \tikzstyle{Solid Edge}=[-]
    \node [style=Black Vertex] (0) at (0, 0) {0};
    \node [style=Black Vertex] (1) at (-3, -3) {1};
    \node [style=Black Vertex] (2) at (3, -3) {2};
    \node [style=Black Vertex] (3) at (-7, -3) {3};
    \node [style=Black Vertex] (4) at (7, -3) {4};
    \draw (0) to (1);
    \draw (0) to (2);
    \draw (1) to (2);
    \draw (1) to (3);
    \draw (2) to (4);
  \end{tikzpicture}
  & no cograph
  &
  $\frac{-X T^{2} + 3 T^{2} - T + 1 - X^{-1} T^{3} + X^{-1} T^{2} - 3 X^{-1}
    T + X^{-2} T}{(1 - XT)^4}$
  \\
  \begin{tikzpicture}[scale=0.2]
    \tikzstyle{Black Vertex}=[draw=black, fill=black, shape=circle, scale=0.2]
    \tikzstyle{Solid Edge}=[-]
    \node [style=Black Vertex] (0) at (0, 0) {0};
    \node [style=Black Vertex] (1) at (-3, 3) {1};
    \node [style=Black Vertex] (2) at (-3, -3) {2};
    \node [style=Black Vertex] (3) at (3, 0) {3};
    \node [style=Black Vertex] (4) at (-6, 0) {4};

    \draw (0) to (1);
    \draw (1) to (4);
    \draw (0) to (4);
    \draw (0) to (2);
    \draw (2) to (4);
    \draw (0) to (3);
  \end{tikzpicture}
  & $\KiteGraph(2,1,1,1)$
  &
  $\frac{(1-X^{-1}T)^2}{(1-XT)^3}$
  \\
  \begin{tikzpicture}[scale=0.2]
    \tikzstyle{Black Vertex}=[draw=black, fill=black, shape=circle, scale=0.2]
    \tikzstyle{Solid Edge}=[-]
    \node [style=Black Vertex] (0) at (0, -3) {0};
    \node [style=Black Vertex] (1) at (6, -3) {1};
    \node [style=Black Vertex] (2) at (6, 3) {2};
    \node [style=Black Vertex] (3) at (0, 3) {3};
    \node [style=Black Vertex] (4) at (3, 0) {4};
    \draw (0) to (1);
    \draw (0) to (4);
    \draw (1) to (2);
    \draw (2) to (3);
    \draw (2) to (4);
    \draw (3) to (0);
  \end{tikzpicture}
  & cograph
  &
  $\frac{T + 1 - X^{-1} T - 2 X^{-2} T - X^{-3} T + X^{-4} T^{2} + X^{-4}
    T}{(1 - X^{-1}T)(1 - XT)^2}$
  \\
  \begin{tikzpicture}[scale=0.2]
    \tikzstyle{Black Vertex}=[draw=black, fill=black, shape=circle, scale=0.2]
    \tikzstyle{Solid Edge}=[-]
    \node [style=Black Vertex] (0) at (-1, -2) {0};
    \node [style=Black Vertex] (1) at (-2, 4) {1};
    \node [style=Black Vertex] (2) at (5, -3) {2};
    \node [style=Black Vertex] (3) at (-5, -3) {3};

    \node [style=Black Vertex] (4) at (1, 1) {4};

    \draw (0) to (1);
    \draw (0) to (2);
    \draw (0) to (3);

    \draw (0) to (4);
    \draw (1) to (4);

    \draw (2) to (4);
    \draw (3) to (4);
  \end{tikzpicture}
  & $\KiteGraph(3,2)$
  &
  $\frac{(1-X^{-2}T)(1-X^{-1}T)}{(1-T)(1-XT)^2}$
  \\
  \begin{tikzpicture}[scale=0.2]
    \tikzstyle{Black Vertex}=[draw=black, fill=black, shape=circle, scale=0.2]
    \tikzstyle{Solid Edge}=[-]
    \node [style=Black Vertex] (0) at (0, 0) {0};
    \node [style=Black Vertex] (1) at (-3, 3) {1};
    \node [style=Black Vertex] (2) at (-3, -3) {2};
    \node [style=Black Vertex] (3) at (3, 0) {3};
    \node [style=Black Vertex] (4) at (-6, 0) {4};

    \draw (0) to (1);
    \draw (1) to (4);
    \draw (0) to (2);
    \draw (2) to (4);
    \draw (0) to (3);
  \end{tikzpicture}
  & no cograph
  & $\frac{-X T^{2} + T^{2} + 1 + 3 X^{-1} T^{2} - 3 X^{-1} T - X^{-2} T^{3} -
    X^{-2} T + X^{-3} T}{(1 - T)(1 - XT)^3}$
  \\
  \begin{tikzpicture}[scale=0.2]
    \tikzstyle{Black Vertex}=[draw=black, fill=black, shape=circle, scale=0.2]
    \tikzstyle{Solid Edge}=[-]
    \node [style=Black Vertex] (0) at (0, 0) {0};
    \node [style=Black Vertex] (1) at (-3, 3) {1};
    \node [style=Black Vertex] (2) at (-3, -3) {2};
    \node [style=Black Vertex] (3) at (3, 0) {3};
    \node [style=Black Vertex] (4) at (-6, 0) {4};

    \draw (0) to (1);
    \draw (1) to (4);
    \draw (1) to (2);
    \draw (0) to (2);
    \draw (2) to (4);
    \draw (0) to (3);
  \end{tikzpicture}
  & no cograph
  &
  $\frac{T + 1 - 2 X^{-1} T - 2 X^{-2} T + X^{-3} T^{2} + X^{-3} T}
  {(1 - T) (1 - XT)^2}$
  \\
  \begin{tikzpicture}[scale=0.2]
    \tikzstyle{Black Vertex}=[draw=black, fill=black, shape=circle, scale=0.2]
    \tikzstyle{Solid Edge}=[-]
    \node [style=Black Vertex] (0) at (0, 0) {0};
    \node [style=Black Vertex] (1) at (-3, 3) {1};
    \node [style=Black Vertex] (2) at (-3, -3) {2};
    \node [style=Black Vertex] (3) at (3, 0) {3};
    \node [style=Black Vertex] (4) at (-6, 0) {4};

    \draw (0) to (1);
    \draw (1) to (4);
    \draw (1) to (2);
    \draw (0) to (2);
    \draw (2) to (4);
    \draw (0) to (3);
    \draw (0) to (4);
  \end{tikzpicture}
  & $\KiteGraph(1,2,1,1)$
  &
  $\frac{(1-X^{-1}T)(1-X^{-2}T)}{(1-T)(1-XT)^2}$
  \\
  \begin{tikzpicture}[scale=0.2]
    \tikzstyle{Black Vertex}=[draw=black, fill=black, shape=circle, scale=0.2]
    \tikzstyle{Solid Edge}=[-]
    \node [style=Black Vertex] (0) at (0, 0) {0};
    \node [style=Black Vertex] (1) at (0, 6) {1};
    \node [style=Black Vertex] (2) at (4, 3) {2};
    \node [style=Black Vertex] (3) at (8, 3) {3};
    \node [style=Black Vertex] (4) at (12, 3) {4};
    \draw (0) to (1);
    \draw (0) to (2);
    \draw (1) to (2);
    \draw (2) to (3);
    \draw (3) to (4);
  \end{tikzpicture}
  & no cograph
  &
  $\frac{-X T^{2} + T^{2} + 1 + 3 X^{-1} T^{2} - 3 X^{-1} T - X^{-2} T^{3} -
    X^{-2} T + X^{-3} T}{(1 - T)(1 - XT)^3}$
  \\
  \begin{tikzpicture}[scale=0.2]
    \tikzstyle{Black Vertex}=[draw=black, fill=black, shape=circle, scale=0.2]
    \tikzstyle{Solid Edge}=[-]
    \node [style=Black Vertex] (0) at (0, 0) {0};
    \node [style=Black Vertex] (1) at (-3, 3) {1};
    \node [style=Black Vertex] (2) at (-3, -3) {2};
    \node [style=Black Vertex] (3) at (3, 3) {3};
    \node [style=Black Vertex] (4) at (3, -3) {4};
    \draw (0) to (1);
    \draw (1) to (2);
    \draw (0) to (2);
    \draw (0) to (3);
    \draw (0) to (4);
    \draw (3) to (4);
  \end{tikzpicture}
  & cograph
  &
  $\frac {T + 1 - 2 X^{-1} T - 2 X^{-2} T + X^{-3} T^{2} + X^{-3} T}{(1 - XT)^2(1 - T)}$
  \\
  \begin{tikzpicture}[scale=0.2]
    \tikzstyle{Black Vertex}=[draw=black, fill=black, shape=circle, scale=0.2]
    \tikzstyle{Solid Edge}=[-]
    \node [style=Black Vertex] (0) at (3, 0) {0};
    \node [style=Black Vertex] (1) at (0.9, 2.8) {1};
    \node [style=Black Vertex] (2) at (-2.4, 1.7) {2};
    \node [style=Black Vertex] (3) at (-2.4, -1.7) {3};
    \node [style=Black Vertex] (4) at (0.9, -2.8) {4};

    \draw (0) to (1) to (2) to (3) to (4) to (0);
  \end{tikzpicture}
  & no cograph ($\Cycle 5$)
  &
  \begin{minipage}{7.5cm}
    $\displaystyle
    \bigl(T^{2} + 3 T + 1 - 5 X^{-1} T^{2} - 5 X^{-1} T + 5 X^{-2} T^{2} - 5
    X^{-2} T + 5 X^{-3} T^{2} + 5 X^{-3} T - X^{-4} T^{3} - 3 X^{-4} T^{2} -
    X^{-4} T\bigr)/\bigl((1 - T)^3 (1 - XT)\bigr)$
  \end{minipage}
  \\
  \begin{tikzpicture}[scale=0.2]
    \tikzstyle{Black Vertex}=[draw=black, fill=black, shape=circle, scale=0.2]
    \tikzstyle{Solid Edge}=[-]
    \node [style=Black Vertex] (0) at (0, 0) {0};
    \node [style=Black Vertex] (1) at (4, 0) {1};
    \node [style=Black Vertex] (2) at (8, 0) {2};

    \node [style=Black Vertex] (3) at (2, 4) {3};
    \node [style=Black Vertex] (4) at (6, 4) {4};

    \draw (0) to (1) to (2);
    \draw (0) to (3) to (1) to (4) to (2);
    \draw (3) to (4);
  \end{tikzpicture}
  & no cograph
  &
  $\frac{1 - X^{-1} T^{2} - X^{-1} T + 3 X^{-2} T^{2} - 3 X^{-2} T + X^{-3}
    T^{2} + X^{-3} T - X^{-4} T^{3}}{(1 - T)^3(1 - XT)}$
  \\
  \begin{tikzpicture}[scale=0.2]
    \tikzstyle{Black Vertex}=[draw=black, fill=black, shape=circle, scale=0.2]
    \tikzstyle{Solid Edge}=[-]
    \node [style=Black Vertex] (0) at (0, 0) {0};
    \node [style=Black Vertex] (4) at (4, 4) {4};
    \node [style=Black Vertex] (3) at (4, -4) {3};
    \node [style=Black Vertex] (2) at (8, 0) {2};
    \node [style=Black Vertex] (1) at (4, 0) {1};

    \draw (0) to (3) to (1) to (4) to (2) to (3);
    \draw (0) to (4);
    \draw (2) to (3);
    \draw (2) to (1);
  \end{tikzpicture}
  & cograph
  &
  $\frac{1 - X^{-1} T^{2} + X^{-2} T^{2} - 3 X^{-2} T + 3 X^{-3} T^{2} - X^{-3}
    T + X^{-4} T - X^{-5} T^{3}}{(1 - XT)(1 - X^{-1}T)(1 - T)^2}$
  \\
  \begin{tikzpicture}[scale=0.2]
    \tikzstyle{Black Vertex}=[draw=black, fill=black, shape=circle, scale=0.2]
    \tikzstyle{Solid Edge}=[-]
    \node [style=Black Vertex] (0) at (0, 0) {0};
    \node [style=Black Vertex] (4) at (4, 4) {4};
    \node [style=Black Vertex] (3) at (4, -4) {3};
    \node [style=Black Vertex] (2) at (8, 0) {2};
    \node [style=Black Vertex] (1) at (14, 0) {1};

    \draw (0) to (3) to (2) to (4) to (0) to (2);
    \draw (3) to (4);
    \draw (4) to (1) to (3);
  \end{tikzpicture}
  & $\KiteGraph(1,1,1,2)$
  & $\frac{(1-X^{-2}T)^2}{(1-T)^2(1-XT)}$
  \\
  \begin{tikzpicture}[scale=0.2]
    \tikzstyle{Black Vertex}=[draw=black, fill=black, shape=circle, scale=0.2]
    \tikzstyle{Solid Edge}=[-]
    \node [style=Black Vertex] (0) at (0, -3) {0};
    \node [style=Black Vertex] (1) at (6, -3) {1};
    \node [style=Black Vertex] (2) at (6, 3) {2};
    \node [style=Black Vertex] (3) at (0, 3) {3};
    \node [style=Black Vertex] (4) at (3, 0) {4};
    \draw (0) to (1);
    \draw (0) to (4);
    \draw (1) to (2);
    \draw (2) to (3);
    \draw (2) to (4);
    \draw (3) to (0);
    \draw (1) to (4) to (3);
  \end{tikzpicture}
  & cograph
  &
  $\frac{1 + X^{-1} T - 2 X^{-2} T - 2 X^{-3} T + X^{-4} T + X^{-5}
    T^{2}}{(1 - X^{-1}T)(1 - T)(1 - XT)}$
  \\
  \begin{tikzpicture}[scale=0.2]
    \tikzstyle{Black Vertex}=[draw=black, fill=black, shape=circle, scale=0.2]
    \tikzstyle{Solid Edge}=[-]
    \node [style=Black Vertex] (2) at (-3,2) {2};
    \node [style=Black Vertex] (4) at (0,0) {0};
    \node [style=Black Vertex] (3) at (3,2) {0};
    \draw (2) to (3) to (4) to (2);

    \node [style=Black Vertex] (1) at (0,6) {1};
    \node [style=Black Vertex] (0) at (0,-3) {0};

    \draw (0) to (2);
    \draw (0) to (3);
    \draw (0) to (4);
    \draw (1) to (2);
    \draw (1) to (3);
    \draw (1) to (4);
  \end{tikzpicture}
  & $\KiteGraph(2,3)$
  &
  $\frac{(1-X^{-3}T)(1-X^{-2}T)}{(1-X^{-1}T)(1-T)(1-XT)}$
  \\
  \begin{tikzpicture}[scale=0.2]
    \tikzstyle{Black Vertex}=[draw=black, fill=black, shape=circle, scale=0.2]
    \tikzstyle{Solid Edge}=[-]
    \node [style=Black Vertex] (0) at (0,0) {0};
    \node [style=Black Vertex] (1) at (4,0) {1};
    \node [style=Black Vertex] (2) at (8,0) {2};
    \node [style=Black Vertex] (3) at (10,0) {3};
    \node [style=Black Vertex] (4) at (14,0) {4};
    \draw (0) to (1) to (2);
    \draw (3) to (4);
  \end{tikzpicture}
  & cograph
  &
  $\frac{-X^{3} T^{2} + X^{2} T^{2} - X T^{3} + 3 X T^{2} - 3 T + 1 - X^{-1} T
    + X^{-2} T}{(1 - XT)^2(1 - X^2T)^2}$
  \vspace*{0.8em}
  \\
  \begin{tikzpicture}[scale=0.2]
    \tikzstyle{Black Vertex}=[draw=black, fill=black, shape=circle, scale=0.2]
    \tikzstyle{Solid Edge}=[-]
    \node [style=Black Vertex] (0) at (0, 0) {0};
    \node [style=Black Vertex] (1) at (-3, -3) {2};
    \node [style=Black Vertex] (2) at (3, -3) {3};
    \draw (0) to (1) to (2) to (0);
    \node [style=Black Vertex] (3) at (5, 0) {3};
    \node [style=Black Vertex] (4) at (5, -3) {4};
    \draw (3) to (4);
  \end{tikzpicture}
  & cograph
  &
  $\frac{X T - T + 1 - 2 X^{-1} T + X^{-2} T^{2} - X^{-2} T + X^{-3} T}
  {(1 - T)(1 - XT) (1 - X^2T)}$
  \vspace*{0.7em}
  \\
  \begin{tikzpicture}[scale=0.2]
    \tikzstyle{Black Vertex}=[draw=black, fill=black, shape=circle, scale=0.2]
    \tikzstyle{Solid Edge}=[-]
    \node [style=Black Vertex] (0) at (0, 0) {0};
    \node [style=Black Vertex] (4) at (4, 4) {4};
    \node [style=Black Vertex] (3) at (4, -4) {3};
    \node [style=Black Vertex] (2) at (8, 0) {2};
    \node [style=Black Vertex] (1) at (4, 0) {1};

    \draw (0) to (3) to (1);
    \draw (4) to (2) to (3);
    \draw (0) to (4);
    \draw (2) to (3);
    \draw (2) to (1);
  \end{tikzpicture}
  & no cograph
  & \begin{minipage}{7cm}
    $\bigl(T + 1 - 2 X^{-1} T^{2} - 2 X^{-1} T + 4 X^{-2} T^{2} - 4 X^{-2} T + 2
    X^{-3} T^{2} + 2 X^{-3} T - X^{-4} T^{3} - X^{-4} T^{2}\bigr)/\bigl(
    (1 - T)^3 (1 - XT)\bigr)$
  \end{minipage}
  \vspace*{0.6em}
  \\
  \begin{tikzpicture}[scale=0.2]
    \tikzstyle{Black Vertex}=[draw=black, fill=black, shape=circle, scale=0.2]
    \tikzstyle{Solid Edge}=[-]
    \node [style=Black Vertex] (0) at (3, 0) {0};
    \node [style=Black Vertex] (1) at (0.9, 2.8) {1};
    \node [style=Black Vertex] (2) at (-2.4, 1.7) {2};
    \node [style=Black Vertex] (3) at (-2.4, -1.7) {3};
    \node [style=Black Vertex] (4) at (0.9, -2.8) {4};

    \draw (0) to (1) to (2) to (3) to (4) to (0);
    \draw (0) to (2) to (4) to (1) to (3) to (0);
  \end{tikzpicture}
  & $\KiteGraph(1,4) = \CG_5$ & $\frac{1-X^{-4}T}{(1-T)(1-XT)}$
\end{longtable}

\begin{table}
  \centering
  \small
  \begin{tabular}{c|m{13cm}}
    $\Gamma$ & $W_\Gamma^-(X,T)$ \\
    \hline \hline
    $\Path 1$ & $1/(1-XT)$ \\
    \hline
    $\Path 2$ & $(1-X^{-1}T)/((1-T)(1-XT))$ \\
    \hline
    $\Path 3$ & $(1-X^{-1}T)/((1 - XT)^2$ \\
    \hline
    $\Path 4$ &
    $(-X T^{2} + 3 T^{2} - T + 1 - X^{-1} T^{3} + X^{-1} T^{2} - 3 X^{-1} T +
    X^{-2} T)/((1 - XT)^3(1 - T))$\\
    \hline

    $\Path 5$ &
    $(-2 X T^{2} + X T + 4 T^{2} - 2 T + 1 - X^{-1} T^{3} + 2 X^{-1} T^{2} - 4
    X^{-1} T - X^{-2} T^{2} + 2 X^{-2} T)/(1 - XT)^4$\\

    \hline
    $\Path 6$ &
    $(X^{2} T^{4} - X^{2} T^{3} + X^{2} T^{2} - 6 X T^{4} + 11 X T^{3}
    - 13 X T^{2} + 3 X T + 7 T^{4} - 12 T^{3} + 20 T^{2} - 6 T + 1 -
    X^{-1} T^{5} + 6 X^{-1} T^{4} - 20 X^{-1} T^{3} + 12 X^{-1} T^{2}
    - 7 X^{-1} T - 3 X^{-2} T^{4} + 13 X^{-2} T^{3} - 11 X^{-2} T^{2}
    + 6 X^{-2} T - X^{-3} T^{3} + X^{-3} T^{2} - X^{-3} T)/((1-T) (1-XT)^5)$\\
    \hline
    $\Path 7$
    &
    $(X^{3} T^{3} + 3 X^{2} T^{4} - 12 X^{2} T^{3} + 7 X^{2} T^{2} - 12 X
    T^{4} + 41 X T^{3} - 40 X T^{2} + 7 X T + 9 T^{4} - 28 T^{3} + 45 T^{2} -
    12 T + 1 - X^{-1} T^{5} + 12 X^{-1} T^{4} - 45 X^{-1} T^{3} + 28 X^{-1}
    T^{2} - 9 X^{-1} T - 7 X^{-2} T^{4} + 40 X^{-2} T^{3} - 41 X^{-2} T^{2} +
    12 X^{-2} T - 7 X^{-3} T^{3} + 12 X^{-3} T^{2} - 3 X^{-3} T - X^{-4}
    T^{2})/((1 - XT)^6)$    
    \\
    \hline
    $\Path 8$
    &
    $(-X^{4} T^{5} + X^{4} T^{4} - X^{3} T^{6} + 14 X^{3} T^{5} - 28 X^{3}
    T^{4} + 14 X^{3} T^{3} + 10 X^{2} T^{6} - 84 X^{2} T^{5} + 200 X^{2} T^{4}
    - 150 X^{2} T^{3} + 31 X^{2} T^{2} - 25 X T^{6} + 185 X T^{5} - 496 X
    T^{4} + 462 X T^{3} - 161 X T^{2} + 14 X T + 11 T^{6} - 76 T^{5} + 310
    T^{4} - 374 T^{3} + 189 T^{2} - 26 T + 1 - X^{-1} T^{7} + 26 X^{-1} T^{6}
    - 189 X^{-1} T^{5} + 374 X^{-1} T^{4} - 310 X^{-1} T^{3} + 76 X^{-1} T^{2}
    - 11 X^{-1} T - 14 X^{-2} T^{6} + 161 X^{-2} T^{5} - 462 X^{-2} T^{4} +
    496 X^{-2} T^{3} - 185 X^{-2} T^{2} + 25 X^{-2} T - 31 X^{-3} T^{5} + 150
    X^{-3} T^{4} - 200 X^{-3} T^{3} + 84 X^{-3} T^{2} - 10 X^{-3} T - 14
    X^{-4} T^{4} + 28 X^{-4} T^{3} - 14 X^{-4} T^{2} + X^{-4} T - X^{-5} T^{3}
    + X^{-5} T^{2})/((1 - XT)^7(1 - T))$ 
    \\
    \hline
    $\Path 9$
    &
    $(X^{5} T^{5} - 16 X^{4} T^{5} + 26 X^{4} T^{4} - 4 X^{3} T^{6} + 110
    X^{3} T^{5} - 282 X^{3} T^{4} + 109 X^{3} T^{3} + 25 X^{2} T^{6} - 376
    X^{2} T^{5} + 1162 X^{2} T^{4} - 798 X^{2} T^{3} + 109 X^{2} T^{2} - 46 X
    T^{6} + 559 X T^{5} - 2042 X T^{4} + 1962 X T^{3} - 486 X T^{2} + 26 X T +
    8 T^{6} - 94 T^{5} + 918 T^{4} - 1526 T^{3} + 582 T^{2} - 50 T + 1 -
    X^{-1} T^{7} + 50 X^{-1} T^{6} - 582 X^{-1} T^{5} + 1526 X^{-1} T^{4} -
    918 X^{-1} T^{3} + 94 X^{-1} T^{2} - 8 X^{-1} T - 26 X^{-2} T^{6} + 486
    X^{-2} T^{5} - 1962 X^{-2} T^{4} + 2042 X^{-2} T^{3} - 559 X^{-2} T^{2} +
    46 X^{-2} T - 109 X^{-3} T^{5} + 798 X^{-3} T^{4} - 1162 X^{-3} T^{3} +
    376 X^{-3} T^{2} - 25 X^{-3} T - 109 X^{-4} T^{4} + 282 X^{-4} T^{3} - 110
    X^{-4} T^{2} + 4 X^{-4} T - 26 X^{-5} T^{3} + 16 X^{-5} T^{2} - X^{-6}
    T^{2})/((1 - XT)^8)$
  \end{tabular}
  \caption{Ask zeta functions associated with paths on at most nine vertices}
  \label{tab:paths}
\end{table}

\begin{table}
  \centering
  \small
  \begin{tabular}{c|m{13cm}}
    $\Gamma$ & $W_\Gamma^-(X,T)$ \\
    \hline \hline
    $\Cycle 3$ &
    $(1 - X^{-2} T)/((1 - XT)(1 - T))$
    \\
    \hline
    $\Cycle 4$ &
    $(T + 1 - 2 X^{-1} T - 2 X^{-2} T + X^{-3} T^{2} + X^{-3} T)/((1 -XT)(1 - T)^2)$
    \\
    \hline
    $\Cycle 5$ &
    $(T^{2} + 3 T + 1 - 5 X^{-1} T^{2} - 5 X^{-1} T + 5 X^{-2} T^{2} - 5
    X^{-2} T + 5 X^{-3} T^{2} + 5 X^{-3} T - X^{-4} T^{3} - 3 X^{-4} T^{2} -
    X^{-4} T)/((1 - XT)(1 - T)^3)$
    \\
    \hline
    $\Cycle 6$ &
    $(T^{3} + 8 T^{2} + 8 T + 1 - 6 X^{-1} T^{3} - 33 X^{-1} T^{2} - 15 X^{-1}
    T + 13 X^{-2} T^{3} + 28 X^{-2} T^{2} - 5 X^{-2} T - 5 X^{-3} T^{3} + 28
    X^{-3} T^{2} + 13 X^{-3} T - 15 X^{-4} T^{3} - 33 X^{-4} T^{2} - 6 X^{-4}
    T + X^{-5} T^{4} + 8 X^{-5} T^{3} + 8 X^{-5} T^{2} + X^{-5} T)/((1 - XT)(1 - T)^4)$
    \\
    \hline
    $\Cycle 7$ &
    $(T^{4} + 17 T^{3} + 41 T^{2} + 17 T + 1 - 7 X^{-1} T^{4} - 98 X^{-1}
    T^{3} - 168 X^{-1} T^{2} - 35 X^{-1} T + 21 X^{-2} T^{4} + 189 X^{-2}
    T^{3} + 175 X^{-2} T^{2} - 28 X^{-3} T^{4} - 70 X^{-3} T^{3} + 70 X^{-3}
    T^{2} + 28 X^{-3} T - 175 X^{-4} T^{3} - 189 X^{-4} T^{2} - 21 X^{-4} T +
    35 X^{-5} T^{4} + 168 X^{-5} T^{3} + 98 X^{-5} T^{2} + 7 X^{-5} T - X^{-6}
    T^{5} - 17 X^{-6} T^{4} - 41 X^{-6} T^{3} - 17 X^{-6} T^{2} - X^{-6}
    T)/((1 - XT) (1 - T)^5)$
    \\
    \hline
    $\Cycle 8$ &
    $(T^{5} + 33 T^{4} + 158 T^{3} + 158 T^{2} + 33 T + 1 - 8 X^{-1} T^{5} -
    236 X^{-1} T^{4} - 924 X^{-1} T^{3} - 676 X^{-1} T^{2} - 76 X^{-1} T + 28
    X^{-2} T^{5} + 660 X^{-2} T^{4} + 1884 X^{-2} T^{3} + 860 X^{-2} T^{2} +
    24 X^{-2} T - 54 X^{-3} T^{5} - 772 X^{-3} T^{4} - 1128 X^{-3} T^{3} - 12
    X^{-3} T^{2} + 46 X^{-3} T + 46 X^{-4} T^{5} - 12 X^{-4} T^{4} - 1128
    X^{-4} T^{3} - 772 X^{-4} T^{2} - 54 X^{-4} T + 24 X^{-5} T^{5} + 860
    X^{-5} T^{4} + 1884 X^{-5} T^{3} + 660 X^{-5} T^{2} + 28 X^{-5} T - 76
    X^{-6} T^{5} - 676 X^{-6} T^{4} - 924 X^{-6} T^{3} - 236 X^{-6} T^{2} - 8
    X^{-6} T + X^{-7} T^{6} + 33 X^{-7} T^{5} + 158 X^{-7} T^{4} + 158 X^{-7}
    T^{3} + 33 X^{-7} T^{2} + X^{-7} T)/((1 - XT) (1 - T)^6)$
    \\
    \hline
    $\Cycle 9$ &
    $(T^{6} + 60 T^{5} + 516 T^{4} + 1015 T^{3} + 516 T^{2} + 60 T + 1 - 9
    X^{-1} T^{6} - 504 X^{-1} T^{5} - 3798 X^{-1} T^{4} - 6192 X^{-1} T^{3} -
    2358 X^{-1} T^{2} - 153 X^{-1} T + 36 X^{-2} T^{6} + 1770 X^{-2} T^{5} +
    10974 X^{-2} T^{4} + 13896 X^{-2} T^{3} + 3603 X^{-2} T^{2} + 87 X^{-2} T
    - 84 X^{-3} T^{6} - 3141 X^{-3} T^{5} - 14154 X^{-3} T^{4} - 11760 X^{-3}
    T^{3} - 1287 X^{-3} T^{2} + 60 X^{-3} T + 117 X^{-4} T^{6} + 2268 X^{-4}
    T^{5} + 3456 X^{-4} T^{4} - 3456 X^{-4} T^{3} - 2268 X^{-4} T^{2} - 117
    X^{-4} T - 60 X^{-5} T^{6} + 1287 X^{-5} T^{5} + 11760 X^{-5} T^{4} +
    14154 X^{-5} T^{3} + 3141 X^{-5} T^{2} + 84 X^{-5} T - 87 X^{-6} T^{6} -
    3603 X^{-6} T^{5} - 13896 X^{-6} T^{4} - 10974 X^{-6} T^{3} - 1770 X^{-6}
    T^{2} - 36 X^{-6} T + 153 X^{-7} T^{6} + 2358 X^{-7} T^{5} + 6192 X^{-7}
    T^{4} + 3798 X^{-7} T^{3} + 504 X^{-7} T^{2} + 9 X^{-7} T - X^{-8} T^{7} -
    60 X^{-8} T^{6} - 516 X^{-8} T^{5} - 1015 X^{-8} T^{4} - 516 X^{-8} T^{3}
    - 60 X^{-8} T^{2} - X^{-8} T)/((1 - XT)(1 - T)^7)$ 
  \end{tabular}
  \caption{Ask zeta functions associated with cycles on at most nine vertices}
  \label{tab:cycles}
\end{table}

\begin{table}
  \begin{center}
  \begin{minipage}{14cm}
  $\displaystyle 1 + 5 X^{-2} T^{2} + 18 X^{-2} T - 4 X^{-3} T^{2} + 12 X^{-4}
  T^{3} + 4 X^{-5} T^{4} - 60 X^{-3} T + 44 X^{-4} T^{2} - 129 X^{-5} T^{3} -
  56 X^{-6} T^{4} - 21 X^{-7} T^{5} + 83 X^{-4} T - 359 X^{-5} T^{2} + 345
  X^{-6} T^{3} + 105 X^{-7} T^{4} + 28 X^{-8} T^{5} + 32 X^{-9} T^{6} - 106
  X^{-5} T + 707 X^{-6} T^{2} - 623 X^{-7} T^{3} + 210 X^{-8} T^{4} + 265
  X^{-9} T^{5} - 95 X^{-10} T^{6} + 74 X^{-11} T^{7} + 92 X^{-6} T - 477
  X^{-7} T^{2} + 1415 X^{-8} T^{3} - 626 X^{-9} T^{4} - 500 X^{-10} T^{5} +
  137 X^{-11} T^{6} - 497 X^{-12} T^{7} + 18 X^{-13} T^{8} - 43 X^{-7} T + 59
  X^{-8} T^{2} - 2179 X^{-9} T^{3} + 409 X^{-10} T^{4} - 602 X^{-11} T^{5} -
  570 X^{-12} T^{6} + 1329 X^{-13} T^{7} - 150 X^{-14} T^{8} + X^{-15} T^{9} +
  8 X^{-8} T + 33 X^{-9} T^{2} + 1872 X^{-10} T^{3} - 6 X^{-11} T^{4} + 2093
  X^{-12} T^{5} + 1360 X^{-13} T^{6} - 1815 X^{-14} T^{7} + 425 X^{-15} T^{8}
  - 100 X^{-16} T^{9} + 93 X^{-10} T^{2} - 947 X^{-11} T^{3} - 301 X^{-12}
  T^{4} - 2468 X^{-13} T^{5} - 1348 X^{-14} T^{6} + 1086 X^{-15} T^{7} - 608
  X^{-16} T^{8} + 553 X^{-17} T^{9} - 8 X^{-18} T^{10} - 124 X^{-11} T^{2} +
  35 X^{-12} T^{3} - 459 X^{-13} T^{4} + 2585 X^{-14} T^{5} + 806 X^{-15}
  T^{6} + 741 X^{-16} T^{7} + 1094 X^{-17} T^{8} - 1174 X^{-18} T^{9} + 117
  X^{-19} T^{10} - 3 X^{-20} T^{11} + 56 X^{-12} T^{2} + 379 X^{-13} T^{3} +
  2070 X^{-14} T^{4} - 1762 X^{-15} T^{5} + 268 X^{-16} T^{6} - 1682 X^{-17}
  T^{7} - 2568 X^{-18} T^{8} + 1093 X^{-19} T^{9} - 546 X^{-20} T^{10} + 10
  X^{-21} T^{11} - 9 X^{-13} T^{2} - 258 X^{-14} T^{3} - 2141 X^{-15} T^{4} -
  196 X^{-16} T^{5} - 1798 X^{-17} T^{6} + 88 X^{-18} T^{7} + 3403 X^{-19}
  T^{8} - 458 X^{-20} T^{9} + 1253 X^{-21} T^{10} + 74 X^{-22} T^{11} + 37
  X^{-15} T^{3} + 688 X^{-16} T^{4} + 2069 X^{-17} T^{5} + 2381 X^{-18} T^{6}
  + 1478 X^{-19} T^{7} - 1543 X^{-20} T^{8} + 401 X^{-21} T^{9} - 1607 X^{-22}
  T^{10} - 379 X^{-23} T^{11} - 15 X^{-24} T^{12} + 26 X^{-16} T^{3} + 315
  X^{-17} T^{4} - 2168 X^{-18} T^{5} - 1973 X^{-19} T^{6} - 1887 X^{-20} T^{7}
  - 2220 X^{-21} T^{8} - 887 X^{-22} T^{9} + 974 X^{-23} T^{10} + 428 X^{-24}
  T^{11} + 120 X^{-25} T^{12} - 9 X^{-17} T^{3} - 353 X^{-18} T^{4} + 667
  X^{-19} T^{5} + 529 X^{-20} T^{6} + 1568 X^{-21} T^{7} + 4496 X^{-22} T^{8}
  + 1568 X^{-23} T^{9} + 529 X^{-24} T^{10} + 667 X^{-25} T^{11} - 353 X^{-26}
  T^{12} - 9 X^{-27} T^{13} + 120 X^{-19} T^{4} + 428 X^{-20} T^{5} + 974
  X^{-21} T^{6} - 887 X^{-22} T^{7} - 2220 X^{-23} T^{8} - 1887 X^{-24} T^{9}
  - 1973 X^{-25} T^{10} - 2168 X^{-26} T^{11} + 315 X^{-27} T^{12} + 26
  X^{-28} T^{13} - 15 X^{-20} T^{4} - 379 X^{-21} T^{5} - 1607 X^{-22} T^{6} +
  401 X^{-23} T^{7} - 1543 X^{-24} T^{8} + 1478 X^{-25} T^{9} + 2381 X^{-26}
  T^{10} + 2069 X^{-27} T^{11} + 688 X^{-28} T^{12} + 37 X^{-29} T^{13} + 74
  X^{-22} T^{5} + 1253 X^{-23} T^{6} - 458 X^{-24} T^{7} + 3403 X^{-25} T^{8}
  + 88 X^{-26} T^{9} - 1798 X^{-27} T^{10} - 196 X^{-28} T^{11} - 2141 X^{-29}
  T^{12} - 258 X^{-30} T^{13} - 9 X^{-31} T^{14} + 10 X^{-23} T^{5} - 546
  X^{-24} T^{6} + 1093 X^{-25} T^{7} - 2568 X^{-26} T^{8} - 1682 X^{-27} T^{9}
  + 268 X^{-28} T^{10} - 1762 X^{-29} T^{11} + 2070 X^{-30} T^{12} + 379
  X^{-31} T^{13} + 56 X^{-32} T^{14} - 3 X^{-24} T^{5} + 117 X^{-25} T^{6} -
  1174 X^{-26} T^{7} + 1094 X^{-27} T^{8} + 741 X^{-28} T^{9} + 806 X^{-29}
  T^{10} + 2585 X^{-30} T^{11} - 459 X^{-31} T^{12} + 35 X^{-32} T^{13} - 124
  X^{-33} T^{14} - 8 X^{-26} T^{6} + 553 X^{-27} T^{7} - 608 X^{-28} T^{8} +
  1086 X^{-29} T^{9} - 1348 X^{-30} T^{10} - 2468 X^{-31} T^{11} - 301 X^{-32}
  T^{12} - 947 X^{-33} T^{13} + 93 X^{-34} T^{14} - 100 X^{-28} T^{7} + 425
  X^{-29} T^{8} - 1815 X^{-30} T^{9} + 1360 X^{-31} T^{10} + 2093 X^{-32}
  T^{11} - 6 X^{-33} T^{12} + 1872 X^{-34} T^{13} + 33 X^{-35} T^{14} + 8
  X^{-36} T^{15} + X^{-29} T^{7} - 150 X^{-30} T^{8} + 1329 X^{-31} T^{9} -
  570 X^{-32} T^{10} - 602 X^{-33} T^{11} + 409 X^{-34} T^{12} - 2179 X^{-35}
  T^{13} + 59 X^{-36} T^{14} - 43 X^{-37} T^{15} + 18 X^{-31} T^{8} - 497
  X^{-32} T^{9} + 137 X^{-33} T^{10} - 500 X^{-34} T^{11} - 626 X^{-35} T^{12}
  + 1415 X^{-36} T^{13} - 477 X^{-37} T^{14} + 92 X^{-38} T^{15} + 74 X^{-33}
  T^{9} - 95 X^{-34} T^{10} + 265 X^{-35} T^{11} + 210 X^{-36} T^{12} - 623
  X^{-37} T^{13} + 707 X^{-38} T^{14} - 106 X^{-39} T^{15} + 32 X^{-35} T^{10}
  + 28 X^{-36} T^{11} + 105 X^{-37} T^{12} + 345 X^{-38} T^{13} - 359 X^{-39}
  T^{14} + 83 X^{-40} T^{15} - 21 X^{-37} T^{11} - 56 X^{-38} T^{12} - 129
  X^{-39} T^{13} + 44 X^{-40} T^{14} - 60 X^{-41} T^{15} + 4 X^{-39} T^{12} +
  12 X^{-40} T^{13} - 4 X^{-41} T^{14} + 18 X^{-42} T^{15} + 5 X^{-42} T^{14}
  + X^{-44} T^{16}$
\end{minipage}
\end{center}
\caption[A numerator]{Numerator $F(X,T)$ of $W^+_{(\CG_3\oplus\CG_3)\join\CG_2}$ in \eqref{eq:K3_K3_K2}}
\label{tab:CG3_CG3_CG2}
\end{table}

\clearpage

\section{Open problems}\label{s:open.problems}

Inspired by the theoretical results and the explicit formulae in this
article, we raise a number of further questions (beyond
Questions~\ref{qu:nonneg} and~\ref{qu:half.int}) pertaining to the
topics covered here.

\subsection{The algebra of graphs}
\label{ss:algebra_of_graphs}

As we mentioned in \S\ref{ss:disjoint_union_graphs},
the effect of taking disjoint unions of graphs corresponds to
taking Hadamard products of the rational functions $W^\pm_\Gamma(X,T)$.
In particular, for arbitrary simple graphs $\Gamma_1$ and
$\Gamma_2$, the rational function $W_{\Gamma_1\oplus\Gamma_2}^{-}(X,T)$ does not
depend on the individual graphs $\Gamma_1$ and $\Gamma_2$ but only on the
rational functions attached to these.

In the special case that $\Gamma_1$ and $\Gamma_2$ are both
\itemph{cographs}, Proposition~\ref{prop:zeta_join_of_cographs} similarly
expresses $W_{\Gamma_1\join\Gamma_2}^-(X,T)$ in terms of the
$W_{\Gamma_i}^-(X,T)$.  As we mentioned in Remark~\ref{rem:funeq},
even though the formula in Proposition~\ref{prop:zeta_join_of_cographs}
involves the numbers of vertices of $\Gamma_1$ and $\Gamma_2$, these
numbers can be recovered from the corresponding rational functions via the
functional equation in Corollary~\ref{cor:feqn}.

\begin{question}[Joins]
  \label{que:join}
  Do the conclusions of Proposition~\ref{prop:zeta_join_of_cographs}
  hold for all simple graphs?
\end{question}

The smallest graph which is not a cograph is the path $\Path 4$ on four
vertices; the rational function $W^-_{\Path 4}(X,T)$ is recorded in
Table~\ref{tab:graphs4} (and also in Table~\ref{tab:paths}).
Using \textsf{Zeta}, we find that 
\begin{dmath*}
  W^-_{\Path 4 \join \Path 4}(X,T)
  = \bigl(1 + 2 X^{-3} T - 2 X^{-4} T - 6 X^{-5} T - X^{-6} T^{2} + 2 X^{-6} T
  - 2 X^{-7} T^{2} + X^{-7} T + 6 X^{-8} T^{2} + 2 X^{-9} T^{2} - 2 X^{-10}
  T^{2} - X^{-13} T^{3}\bigr)/\bigl( (1 - X^{-3}T)^2(1 - T)(1 - XT)\bigr).
\end{dmath*}

A routine calculation shows that, even though $\Path 4$ is not a cograph,
$W^-_{\Path 4\join \Path 4}(X,T)$ is indeed correctly calculated by
\eqref{equ:zeta_join_of_cographs} for $\Gamma_1 = \Gamma_2 = \Path 4$. 

\vspace*{1em}

While we defined cographs in terms of disjoint unions and joins, either one of
these two operations could be replaced by complements.
For instance,
the class of cographs is the smallest class of graphs that contains a single
vertex and which is closed under taking disjoint unions and complements.
Write $\hat\Gamma$ for the complement of a simple graph $\Gamma$.

\begin{question}[Complements]
  \label{qu:complements}
  Is there an involution $W(X,T) \mapsto \hat W(X,T)$ of rational
  generating functions in $T$ such that $W_{\hat\Gamma}^-(X,T) = \hat
  W_{\Gamma}^-(X,T)$ for each simple graph $\Gamma$?
\end{question}

\subsection{Connections with statistics on Weyl groups}
\label{ss:stats}

We noted in \S\ref{subsec:hyp.dis.uni} that the class of functions
$W_\Eta(X,T)$ attached to hypergraphs is closed under Hadamard products. 
However, it remains an open problem to exploit the Hadamard factorisation
\eqref{eq:hadamard_of_master} in a way that improves upon the
general Theorem~\ref{thm:master.intro}.
For disjoint unions $\BE_{\bfn,\bfm}$ of block hypergraphs,
we achieved this in \eqref{eq:master.dis.rat}.
Even in the special case $\bfn=\bfm$, the latter formula seems to admit
substantial further improvements.
We illustrate this for $\bfn = \bfm = (a,\dotsc,a)$.
Recall from \cite[\S 5.4]{ask} the definitions of the statistics $\mathrm{N}$
and $\mathrm{d}_{\mathrm{B}}$ on the group $\mathrm{B}_\nb
= \{\pm 1\} \wr \mathrm S_{\nb}$ of signed permutations of degree~$\nb$.

\begin{prop}
  Let $\bfa = (a,\dotsc,a) \in \NN^\nb$.
  Then
  \begin{equation}\label{eq:brenti}
    W_{\BE_{\bfa,\bfa}}(X,T)=
    \frac{\sum\limits_{\sigma\in\textup{B}_\nb}(-X^{-a})^{\textup{N}(\sigma)}T^{\textup{d}_{\textup{B}}(\sigma)}}{(1-T)^{\nb+1}}.
    \end{equation}
\end{prop}
\begin{proof}
  The proof of \cite[Corollary~5.17]{ask} of the case $a = 1$
  (a consequence of a result due to Brenti~\cite[Theorem~3.4]{Bre94})
  carries over to a general $a$; just replace $-q^{-1}$ by $-q^{-a}$.
\end{proof}

The intriguing shape of the numerator of the right-hand side of
\eqref{eq:brenti} prompts the following.
\begin{question}
  Is there an  interpretation of the
  rational functions
  \begin{enumerate}[label=(\alph*)]
  \item
    \label{qu:weyl1}
    $W_{\BE_{\bfn,\bfn}}(X,T)$
    in \eqref{eq:master.dis.rat} for $\bfn \in\N^\nb$,
  \item
    \label{qu:weyl2}
    $W_{\BE_{\bfn,\bfm}}(X,T)$
    in \eqref{eq:master.dis.rat} for $\bfn, \bfm \in\N^\nb$, or possibly even
  \item
    \label{qu:weyl3}
    $W_{\Eta(\bfmu)}(X,T)$ in \eqref{eq:master.WOhat}
    for general $\bfmu\in \N_0^{\Pow( V)}$
  \end{enumerate}
  in terms of statistics on Weyl (or more general reflection) groups?
\end{question}

Note that \ref{qu:weyl2} includes, as a special case, the
problem of finding an interpretation of the class counting zeta functions
$\zeta^{\cc}_{\GG_{\CG_{\bfn}} \otimes \fO}(s)$ in terms of permutation
statistics; see \S\ref{sss:product_F2n}.

\subsection{Analytic properties}

The determination of the rational functions $W_\Gamma^-(X,T)$ for all simple
graphs on at most seven vertices (see \cite{cico-db}) inspired us to raise the
following.

\begin{question}
  \label{qu:riemann}
  Let $\Gamma$ be a simple graph.
  Are the following properties equivalent?
  \begin{enumerate}[label={\textrm{(\alph*)}}]
  \item
    \label{qu:riemann1}
    $\Gamma$ is a threshold graph.
  \item
    \label{qu:riemann2}
    $W^-_\Gamma(X,T)$ is a product of factors of the form $(1 - X^AT^B)^{\pm 1}$.
  \end{enumerate}
\end{question}

The implication \ref{qu:riemann1}$\to$\ref{qu:riemann2} in
Question~\ref{qu:riemann} follows from Theorem~\ref{thm:kite_zeta}.
Let $ \gamma_-$ be the negative adjacency representation of $\Gamma$.
Then \ref{qu:riemann2} is equivalent to $\zeta_{\gamma_-}^\ak(s)$
factoring as a product of factors $\zeta(Bs-A)^{\pm 1}$, where $\zeta$
denotes the Riemann zeta function.
In particular, class counting zeta functions of cographical group schemes
associated with threshold graphs over rings of integers of number fields admit
meromorphic continuation to~$\CC$.

One may speculate that the condition in
Question~\ref{qu:riemann}\ref{qu:riemann2} is not just sufficient but
also necessary for the Euler product 
\[
  \zeta^{\cc}_{\GG_\Gamma\otimes \mcO_K}(s) =
  \prod_{v\in\mcV_K}W^-_\Gamma\bigl(q_v,q_v^{\card{\edges(\Gamma)}-s}\bigr)
\]
associated with an arbitrary simple graph $\Gamma$ and number field $K$ with
ring of integers~$\mcO_K$ to admit meromorphic continuation to the whole
complex plane.
Indeed, consider an Euler product 
$\prod_{v\in\mcV_K}f(q_v,q_v^{-s})$, where $f(X,T)\in\ZZ[X,T]$ is a fixed
polynomial.
Then a general conjecture based on work of Estermann~\cite{Estermann/28} and
Kurokawa~\cite{Kurokawa1/86, Kurokawa2/86} predicts that such an Euler product
admits meromorphic continuation to all of $\CC$ if and only if
$f(X,T)$ is a product of unitary polynomials;
cf.~\cite[Conjecture~1.11]{duSWoodward/08} for details and related work.

In a similar spirit, recall from Theorem~\ref{thm:ana} that for a hypergraph
$\Eta$ with incidence representation~$\eta$, we denoted the common
abscissa of convergence of $\zeta^\ak_{\eta^{\mcO_K}}(s)$ for each number
field~$K$ by~$\alpha(\Eta)$.
By \cite[Theorem~4.20]{ask}, there exists a positive real number
$\delta(\Eta)$, independent of~$K$, such that the function
$\zeta^\ak_{\eta^{\mcO_K}}(s)$ can be meromorphically continued to the domain 
$\{ s\in \C : \Real(s) > \alpha(\Eta)-\delta(\Eta)\}$. 

\begin{question}
  What is the largest value of $\delta(\Eta)$ (if such a value exists)?
  When does~$\zeta^\ak_{\eta^{\mcO_K}}(s)$ admit meromorphic continuation to
  all of $\CC$ for each number field $K$?
\end{question}

By Proposition~\ref{prop:staircase}, staircase hypergraphs have the latter property.

\phantomsection
\addcontentsline{toc}{section}{Acknowledgements}
\subsection*{Acknowledgements}

The work described here was begun while we both enjoyed the generous
hospitality of our mutual host Eamonn O'Brien and the
\href{https://www.auckland.ac.nz}{University of Auckland}.
The first author gratefully acknowledges the support of the
\href{https://www.humboldt-foundation.de}{Alexander von Humboldt Foundation}
in the form of a Feodor Lynen Research Fellowship.
We thank Angela Carnevale for mathematical discussions.
In addition, we are grateful to
\href{https://www.uni-bielefeld.de/}{Bielefeld University}
and the 
\href{https://www.nuigalway.ie/}{National University of Ireland,
  Galway}
for hospitality during reciprocal visits and to
the \href{http://research.ie/}{Irish Research Council}
and
\href{https://www.sfi.ie/}{Science Foundation Ireland}
for financial support.
Finally, we thank the anonymous referee for a very thorough
study of the text and for providing several helpful comments and suggestions.

{
  \bibliographystyle{abbrv}
  \tiny
  \phantomsection
  \addcontentsline{toc}{section}{References}
  \bibliography{cico}
}

\vspace*{2em}
\noindent
{\footnotesize
\begin{minipage}[t]{0.60\textwidth}
  Tobias Rossmann\\
  School of Mathematics, Statistics and Applied Mathematics \\
  National University of Ireland, Galway \\
  Galway \\
  Ireland \\
  \quad\\
  E-mail: \href{mailto:tobias.rossmann@nuigalway.ie}{tobias.rossmann@nuigalway.ie}
\end{minipage}
\hfill
\begin{minipage}[t]{0.38\textwidth}
  Christopher Voll\\
  Fakult\"at f\"ur Mathematik\\
  Universit\"at Bielefeld\\
  D-33501 Bielefeld\\
  Germany\\
  \quad\\
  E-mail: \href{mailto:C.Voll.98@cantab.net}{C.Voll.98@cantab.net}
\end{minipage}
}

\end{document}